\documentclass[10pt]{amsart}
\usepackage{amsthm,amsfonts,amsmath,amscd,amssymb,color,graphicx}
\usepackage{epsfig}
\usepackage{enumerate}
\usepackage{pinlabel}

\title{Knot contact homology, string topology, and the cord
  algebra}
\author{Kai Cieliebak}
\author{Tobias Ekholm}
\author{Janko Latschev}
\author{Lenhard Ng}

\theoremstyle{plain}
\newtheorem{theorem}{Theorem}[section]
\newtheorem{thm}[theorem]{Theorem}
\newtheorem{corollary}[theorem]{Corollary}
\newtheorem{cor}[theorem]{Corollary}
\newtheorem{proposition}[theorem]{Proposition}
\newtheorem{prop}[theorem]{Proposition}
\newtheorem{lemma}[theorem]{Lemma}

\theoremstyle{remark}

\newtheorem{remark}[theorem]{Remark}
\newtheorem{example}[theorem]{Example}

\theoremstyle{definition}
\newtheorem{definition}[theorem]{Definition}
\newcommand{\id}{{{\mathchoice {\rm 1\mskip-4mu l} {\rm 1\mskip-4mu l}
{\rm 1\mskip-4.5mu l} {\rm 1\mskip-5mu l}}}}

\newcommand{\ol}{\overline}
\newcommand{\wt}{\widetilde}
\newcommand{\wh}{\widehat}

\newcommand{\p}{\partial}

\newcommand{\om}{\omega}
\newcommand{\Om}{\Omega}
\newcommand{\eps}{\varepsilon}
\newcommand{\pHi}{\varphi}

\newcommand{\la}{\langle}
\newcommand{\ra}{\rangle}
\newcommand{\N}{{\mathbb{N}}}
\newcommand{\Z}{{\mathbb{Z}}}
\newcommand{\R}{{\mathbb{R}}}
\newcommand{\C}{{\mathbb{C}}}

\newcommand{\BH}{{\mathbb{H}}}
\newcommand{\Bf}{{\mathbb{F}}}

\newcommand{\D}{{\bf D}}

\newcommand{\ind}{{\rm ind}}
\newcommand{\im}{{\rm im }}        
\newcommand{\st}{{\rm st}}

\renewcommand{\min}{{\rm min}}
\renewcommand{\max}{{\rm max}}

\newcommand{\sgn}{{\rm sgn\,}}

\newcommand{\inn}{{\rm int}}

\newcommand{\ev}{{\rm ev}}
\newcommand{\sing}{{\rm sing}}
\newcommand{\sy}{{\rm sy}}
\newcommand{\EE}{\mathcal{E}}
\newcommand{\BB}{\mathcal{B}}
\newcommand{\JJ}{\mathcal{J}}
\newcommand{\MM}{\mathcal{M}}
\newcommand{\CC}{\mathcal{C}}

\newcommand{\FF}{\mathcal{F}}
\newcommand{\OO}{\mathcal{O}}
\newcommand{\HH}{\mathcal{H}}

\renewcommand{\AA}{\mathcal{A}}

\newcommand{\II}{\mathcal{I}}

\newcommand{\RR}{\mathcal{R}}
\newcommand{\TT}{\mathcal{T}}
\newcommand{\s}{{\mathfrak s}}
\newcommand{\comment}[1]{}
\newcommand{\x}{\times}
\newcommand{\str}{{\rm string}}

\newcommand{\cont}{{\rm contact}}
\newcommand{\lk}{{\rm lk}}
\newcommand{\pl}{{\rm pl}}
\newcommand{\lin}{{\rm lin}}
\newcommand{\HSN}{H^{\str}_0}
\newcommand{\HSQ}{\hat{H}^{\str}_0}
\newcommand{\HAN}{S^{\hat{\pi}\hat{\pi}}(\pi,\hat \pi)}
\newcommand{\HAQ}{S^{\pi\pi}(\pi,\hat \pi)}

\newcommand{\pa}{\partial}
\newcommand{\Ordo}{{\mathcal O}}
\newcommand{\sblv}{{\mathcal H}}

\newcommand{\Cord}{\operatorname{Cord}}
\begin{document}

\begin{abstract}
The conormal Lagrangian $L_K$ of a knot $K$ in $\R^{3}$ is the
submanifold of the cotangent bundle $T^{\ast}\R^{3}$ consisting of covectors
along $K$ that annihilate tangent vectors to $K$. By intersecting with
the unit cotangent bundle $S^{\ast}\R^{3}$, one obtains the unit
conormal $\Lambda_{K}$, and the Legendrian contact homology of
$\Lambda_{K}$ is a knot invariant of $K$, known as
knot contact homology. We
define a version of string topology for strings in $\R^{3}\cup L_K$
and prove that this is isomorphic in degree $0$ to knot contact homology.
The string topology perspective gives a topological
derivation of the cord algebra (also isomorphic to degree $0$ knot
contact homology) and relates it to the knot group. Together with the
isomorphism this gives a new proof that knot
contact homology detects the unknot.
Our techniques involve a detailed analysis of certain moduli spaces of
holomorphic disks in $T^{\ast}\R^{3}$ with boundary on $\R^{3}\cup L_K$.
\end{abstract}

\maketitle
\tableofcontents

\section{Introduction}\label{sec:intro}

To a smooth $n$-manifold $Q$ we can naturally associate a symplectic
manifold and a contact manifold: its cotangent bundle $T^*Q$ with the
canonical symplectic structure $\om=dp\wedge dq$, and its unit cotangent
bundle (with respect to any Riemannian metric) $S^*Q\subset T^*Q$ with
its canonical contact structure $\xi=\ker(p\,dq)$. Moreover, a
$k$-dimensional submanifold $K\subset Q$ naturally gives rise to a
Lagrangian and a Legendrian submanifold in $T^*Q$ resp.~$S^*Q$: its
conormal bundle $L_K=\{(q,p)\in T^*Q\mid q\in K,\;p|_{T_qK}=0\}$ and
its unit conormal bundle $\Lambda_K=L_K\cap S^*Q$. Symplectic field
theory (SFT~\cite{EGH}) provides a general framework for associating
algebraic invariants to a pair $(M,\Lambda)$ consisting of a contact manifold and
a Legendrian submanifold; when applied to $(S^*Q,\Lambda_K)$, these
invariants will be diffeotopy invariants of the manifold pair
$(Q,K)$.
The study of the resulting invariants was first suggested by Y.~Eliashberg.

In this paper we concentrate on the case where $K$ is a framed
oriented knot in $Q=\R^3$.
Moreover, we consider only the simplest SFT
invariant: {\em Legendrian contact homology}. For $Q=\R^3$, $S^*Q$ is contactomorphic to the $1$-jet space $J^1(S^2)$,
for which Legendrian contact homology has been rigorously defined in \cite{EES2}.
The Legendrian contact homology of the pair $(S^*\R^3,\Lambda_K)$ is called the \textit{knot contact homology} of $K$. We will denote it $H^\cont_*(K)$.

In its most general form (see \cite{EENStransverse,Ngtransverse}),
knot contact homology is the homology of a differential graded algebra
over the group ring $\Z[H_2(S^*\R^3,\Lambda_K)] = \Z[\lambda^{\pm
    1},\mu^{\pm 1},U^{\pm 1}]$, where the images of $\lambda,\mu$ under the connecting homomorphism generate
$H_1(\Lambda_K) = H_1(T^2)$ and $U$ generates $H_2(S^*\R^3)$. The
isomorphism class of $H^\cont_*(K)$ as a $\Z[\lambda^{\pm 1},\mu^{\pm
    1},U^{\pm 1}]$-algebra is then an isotopy invariant of the framed
oriented knot $K$.

The topological content of knot contact homology has been much studied
in recent years; see for instance \cite{AENV} for a conjectured
relation, which we will not discuss here, to colored HOMFLY-PT
polynomials and topological strings. One part of knot contact homology
that has an established topological interpretation is its $U=1$
specialization. In \cite{Ng:2b,Ng:1}, the third author constructed a
knot invariant called the \textit{cord algebra} $\Cord(K)$, whose
definition we will review in Section~\ref{ss:cordalg}. The combined
results of \cite{Ng:2b,Ng:1,EENS} then prove that the cord algebra is
isomorphic as a $\Z[\lambda^{\pm 1},\mu^{\pm 1}]$-algebra to the $U=1$
specialization of degree $0$ knot contact homology. We will assume
throughout this paper that we have set $U=1$;\footnote{However, we
  note that it is an interesting open problem to find a similar
  topological interpretation of the full degree $0$ knot contact homology as a
  $\Z[\lambda^{\pm 1},\mu^{\pm 1},U^{\pm 1}]$-algebra.
}
then the result is:

\begin{thm}[\cite{Ng:2b,Ng:1,EENS}]
$H^\cont_0(K)\cong \Cord(K)$. \label{thm:cordHC0}
\end{thm}

It has been noticed by many people that the definition of the cord algebra bears a striking resemblance to certain operations in string topology \cite{CS:1,S:1}. Indeed, Basu, McGibbon, Sullivan, and Sullivan used this observation in \cite{BMSS} to construct a theory called ``transverse string topology'' associated to any codimension $2$ knot $K\subset Q$, and proved that it determines the $U=\lambda=1$ specialization of the cord algebra.

In this paper, we present a different approach to knot contact homology and the cord algebra via string topology. Motivated by the general picture sketched by the first and third authors in \cite{CL},
we use string topology operations to define the \textit{string
  homology $H^\str_*(K)$} of $K$.
Then the main result of this paper is:

\begin{thm}\label{thm:main}
For any framed oriented knot $K\subset\R^3$, we have an isomorphism
between $U=1$ knot contact homology and string homology in degree $0$,
$$
   H^\cont_0(K)\cong H_0^\str(K),
$$
defined by a count of punctured holomorphic disks in $T^{\ast} \R^{3}$ with Lagrangian boundary condition $L_K\cup \R^{3}$.
\end{thm}

On the other hand, degree $0$ string homology is easily related to
the cord algebra:

\begin{prop}\label{prop:str-cord}
For any framed oriented knot $K\subset\R^3$, we have an isomorphism
$$
   H^\str_0(K)\cong \Cord(K).
$$
\end{prop}
\noindent
As a corollary we obtain a new geometric proof of
Theorem~\ref{thm:cordHC0}. In fact, we even prove a slight refinement
of the usual formulation of Theorem~\ref{thm:cordHC0}, as we relate
certain noncommutative versions of the two sides where the
coefficients $\lambda,\mu$ do not commute with everything;
see Section~\ref{ss:cordalg} for the version of $\Cord(K)$ and
Section~\ref{sec:leg} for the definition of $H^\cont_0(K)$ that we use.

Our proof is considerably more direct than the original proof of Theorem~\ref{thm:cordHC0}, which was rather circuitous and went as follows.
The third author constructed in \cite{Ng:2a,Ng:1} a combinatorial differential graded algebra associated to a braid whose closure is $K$, and then proved in \cite{Ng:2b,Ng:1} that the degree $0$ homology of this combinatorial complex is isomorphic to $\Cord(K)$ via a mapping class group argument. The second and third authors, in joint work with Etnyre and Sullivan \cite{EENS}, then proved that the combinatorial complex is equal to the differential graded algebra for knot contact homology, using an analysis of degenerations of holomorphic disks to Morse flow trees.

Besides providing a cleaner proof of Theorem~\ref{thm:cordHC0}, the string topology formulation also
gives a geometric explanation for the somewhat mystifying skein relations that define the cord algebra. Moreover, string homology can be directly related to the group ring $\Z\pi$ of the fundamental group $\pi = \pi_1(\R^3\setminus K)$ of the knot complement:

\begin{prop}[see Proposition~\ref{prop:subring}] \label{prop:str-fund}
For a framed oriented knot $K \subset \R^3$, $H_0^\str(K) \cong H_0^\cont(K)$ is isomorphic to the subring of $\Z\pi$ generated by $\lambda^{\pm 1}$, $\mu^{\pm 1}$, and $\operatorname{im}(1-\mu)$, where $\lambda,\mu$ are the elements of $\pi$ representing the longitude and meridian of $K$, and $1-\mu$ denotes the map $\Z\pi \to \Z\pi$ given by left multiplication by $1-\mu$.
\end{prop}

As an easy consequence of Proposition~\ref{prop:str-fund}, we recover the following result from \cite{Ng:1}:

\begin{cor}[see Section~\ref{ss:groupring}]
Knot contact homology detects the unknot: if $H_0^\cont(K) \cong
H_0^\cont(U)$ where $K$ is a framed oriented knot in $\R^3$ and $U$ is
the unknot with any framing, then $K=U$ as framed oriented knots.
\label{cor:unknot}
\end{cor}

\noindent
The original proof of Corollary~\ref{cor:unknot} in \cite{Ng:1} uses
the result that the $A$-polynomial detects the unknot \cite{DG}, which
in turn relies on results from gauge theory \cite{KM}. By contrast,
our proof of Corollary~\ref{cor:unknot} uses no technology beyond the
Loop Theorem (more precisely, the consequence of the Loop Theorem that the longitude is null-homotopic in $\R^3\setminus K$ if and only if $K$ is unknotted).

{\bf Organization of the paper. }
In Section~\ref{sec:string} we define degree $0$ string homology and
prove Proposition~\ref{prop:str-cord}, Proposition~\ref{prop:str-fund}
and Corollary~\ref{cor:unknot}. The remainder of the paper is occupied
by the proof of
Theorem~\ref{thm:main}, beginning with an outline in
Section~\ref{sec:roadmap}.
After a digression in Section~\ref{sec:holo} on the local behavior of
holomorphic functions near corners, which serves as a model for the
behavior of broken strings at switches, we define string homology in
arbitrary degrees in Section~\ref{sec:string-ref}.

The main work in proving Theorem~\ref{thm:main}
is an explicit description of
the moduli spaces of holomorphic disks in $T^*\R^3$ with boundary on
$L_K\cup\R^3$ and punctures asymptotic to Reeb chords.
In Section~\ref{sec:chain} we state the main results about these
moduli spaces and show how they give rise to a chain map from
Legendrian contact homology to string homology (in arbitrary
degrees). Moreover, we show that this chain map respects a natural
length filtration. In Section~\ref{sec:iso} we construct a length
decreasing chain homotopy and prove Theorem~\ref{thm:main}.

The technical results about moduli spaces of holomorphic disks and
their compactifications as manifolds with corners are proved in the
remaining Sections~\ref{S:mdlisp}, \ref{sec:trans}
and~\ref{sec:gluing}.

{\bf Extensions. }
The constructions in this paper have several possible extensions.
Firstly, the definition of string homology and the construction of a
homomorphism from Legendrian contact homology to string homology in
degree zero work the same way for a knot $K$ in an arbitrary
$3$-manifold $Q$ instead of $\R^3$ (the corresponding sections are
actually written in this more general setting), and more generally for
a codimension $2$ submanifold $K$ of an arbitrary manifold $Q$.\footnote{
In the presence of contractible closed geodesics in $Q$, this will
require augmentations by holomorphic planes in $T^*Q$, see e.g.~\cite{CL}.}
The fact that the ambient manifold is $\R^3$ is only used to obtain a
certain finiteness result in the proof that this map is an
isomorphism (see Remark~\ref{rem:R3}). If this result can be generalized, then
Theorem~\ref{thm:main} will hold for arbitrary codimension $2$
submanifolds $K\subset Q$.

Secondly, for knots in $3$-manifolds, the homomorphism from Legendrian
contact homology to string homology is actually constructed in arbitrary
degrees. Proving that it is an isomorphism in arbitrary degrees will
require analyzing codimension three phenomena in the space of strings
with ends on the knot, in addition to the codimension one and two
phenomena described in this paper.

\section*{Acknowledgments}

We thank Chris Cornwell, Tye Lidman, and especially Yasha Eliashberg for
stimulating conversations.
This project started when the authors met at the Workshop ``SFT 2'' in Leipzig in August 2006, and the final technical details were cleaned up when we met during the special program on ``Symplectic geometry and topology'' at the Mittag-Leffler institute in Djursholm in the fall of 2015. We would like to thank the sponsors of these programs for the opportunities to meet, as well as for the inspiring working conditions during these events.
The work of KC was supported by DFG grants CI 45/2-1 and CI 45/5-1.
The work of TE was supported by the Knut and Alice Wallenberg Foundation and by the Swedish Research Council.
The work of JL was supported by DFG grant LA 2448/2-1.
The work of LN was supported by NSF grant DMS-1406371 and a grant from
the Simons Foundation (\# 341289 to Lenhard Ng).
Finally, we thank the referee for suggesting numerous improvements.

\section{String homology in degree zero}\label{sec:string}

In this section, we introduce the degree $0$ string homology $H_0^\str(K)$.
The discussion of string homology here is only a
first approximation to the more precise approach in
Section~\ref{sec:string-ref}, but is much less technical and suffices
for the comparison to the cord algebra. We then give several
formulations of the cord algebra $\Cord(K)$ and use these to prove
that $H_0^\str(K) \cong \Cord(K)$ and that string homology detects the
unknot.
Throughout this section, $K$ denotes an oriented framed knot in
some oriented $3$-manifold $Q$.

\subsection{A string topology construction}\label{ss:string0}

Here we define $H_0^\str(K)$ for an oriented knot $K \subset Q$. 
Let $N$ be a tubular neighborhood of $K$.
For this definition we do not need a framing for the knot $K$; later, when we identify $H_0^\str(K)$ with the cord algebra, it will be convenient to fix a framing, which will in turn fix an identification of $N$ with $S^1 \times D^2$.

Any
tangent vector $v$ to $Q$ at a point on $K$ has a tangential component
parallel to $K$ and a normal component lying in the disk fiber; write
$v^{\text{normal}}$ for the normal component of $v$.
Fix a base point $x_0\in\p N$ and a unit tangent vector
$v_0\in T_{x_0}\p N$.

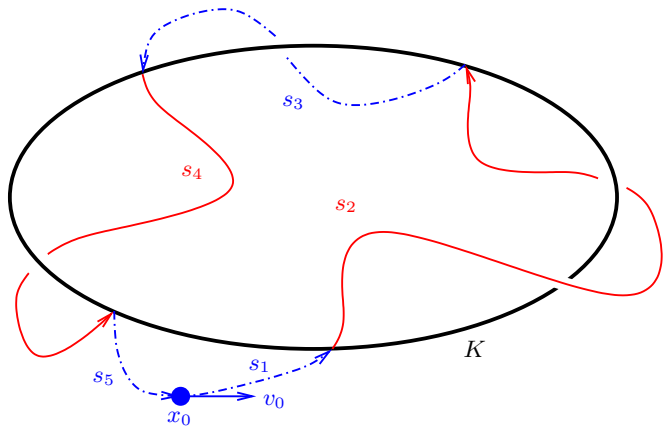
\begin{figure}
\labellist
\small\hair 2pt
\pinlabel ${\color{blue} x_0}$ at 82 -6
\pinlabel ${\color{blue} v_0}$ at 127 2
\pinlabel ${\color{blue} s_1}$ at 120 19
\pinlabel ${\color{red} s_2}$ at 161 95
\pinlabel ${\color{blue} s_3}$ at 136 144
\pinlabel ${\color{red} s_4}$ at 88 111
\pinlabel ${\color{blue} s_5}$ at 46 13
\pinlabel $K$ at 222 27
\endlabellist
\centering
\includegraphics[width=0.7\textwidth]{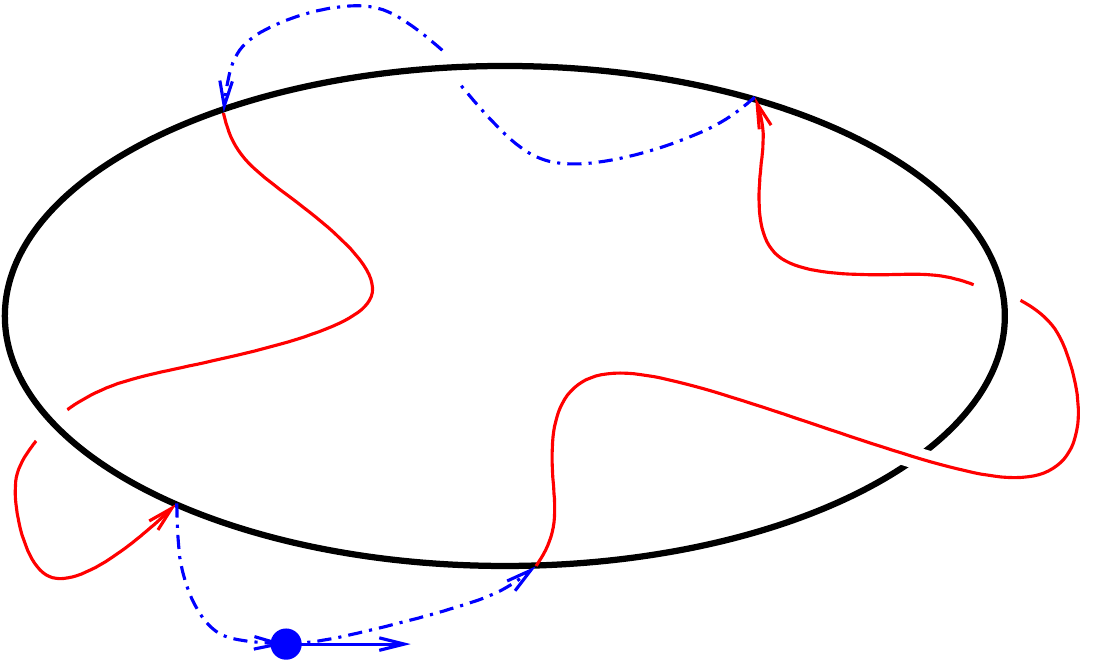}
\caption{A broken closed string with $4$ switches. Here, as in subsequent figures, we draw the knot $K$ in black, $Q$-strings ($s_2,s_4$) in red, and $N$-strings ($s_1,s_3,s_5$) in blue (dashed for clarity to distinguish from the red $Q$-strings).}
\label{fig:1}
\end{figure}

\begin{definition}\label{def:broken_string1}
A {\em broken (closed) string with $2\ell$ switches} on $K$ is a
tuple $s=(a_1,\dots,a_{2\ell+1};s_1,\dots,s_{2\ell+1})$ consisting of real
numbers $0=a_0<a_1<\dots<a_{2\ell+1}$ and $C^1$ maps
$$
   s_{2i+1}:[a_{2i},a_{2i+1}]\to N,\quad s_{2i}:[a_{2i-1},a_{2i}]\to
   Q
$$
satisfying the following conditions:
\begin{enumerate}
\item $s_0(0)=s_{2\ell+1}(a_{2\ell+1})=x_0$ and $\dot s_0(0)=\dot
s_{2\ell+1}(a_{2\ell+1})=v_0$;
\item for $j=1,\dots, 2\ell$, $s_j(a_j)=s_{j+1}(a_j) \in K$;
\item for $i=1,\dots,\ell$,
\begin{align*}
   (\dot s_{2i}(a_{2i}))^{\text{normal}} &= -(\dot s_{2i+1}(a_{2i}))^{\text{normal}} \\
   (\dot s_{2i-1}(a_{2i-1}))^{\text{normal}} &=(\dot s_{2i}(a_{2i-1}))^{\text{normal}}.
\end{align*}
\end{enumerate}
We will refer to the $s_{2i}$ and $s_{2i+1}$ as
{\em Q-strings} and {\em N-strings}, respectively.
Denote by $\Sigma^\ell$ the set of broken strings with $2\ell$
switches.
\end{definition}

\noindent
The last condition, involving normal components of the tangent vectors to the ends of the $Q$- and $N$-strings, models the boundary behavior of holomorphic disks in this context (see Subsections~\ref{ss:series} and \ref{ss:brokenstring} for
more on this point). A typical picture of a broken string is shown in
Figure~\ref{fig:1}.

We call a broken string $s=(s_1,\dots s_{2\ell+1})$ {\em generic} if
none of the derivatives $\dot s_i(a_{i-1})$, $\dot s_i(a_i)$ is tangent to $K$
and no $s_i$ intersects $K$ away from its end points. We call a smooth
1-parameter family of broken strings $s^\lambda=(s_1^\lambda,\dots
s_{2\ell+1}^\lambda)$, $\lambda\in[0,1]$, {\em generic} if $s^0$ and
$s^1$ are generic
strings, none of the derivatives $\dot s_i^\lambda(a_{i-1}^\lambda),\dot
s_i^\lambda(a_i^\lambda)$ is tangent to $K$, and for each $i$ the family
$s_i^\lambda$ intersects $K$ transversally in the interior.
The boundary of this family is given by
$$
   \p\{s^\lambda\} := s^1-s^0.
$$

\begin{figure}
\labellist
\small\hair 2pt
\pinlabel ${\color{blue} \lambda=1}$ at 107 493
\pinlabel ${\color{blue} \lambda=0}$ at 160 528
\pinlabel ${\color{blue} s^\lambda}$ at 338 302
\pinlabel ${\color{blue} 1}$ at 88 400
\pinlabel ${\color{blue} 2}$ at 88 330
\pinlabel ${\color{blue} 3}$ at 14 367
\pinlabel $K$ at 353 525
\pinlabel $K$ at 715 525
\pinlabel $\delta_N$ at 414 417
\pinlabel ${\color{red} \lambda=1}$ at 107 205
\pinlabel ${\color{red} \lambda=0}$ at 160 240
\pinlabel ${\color{red} s^\lambda}$ at 338 14
\pinlabel ${\color{red} 1}$ at 88 112
\pinlabel ${\color{red} 2}$ at 88 42
\pinlabel ${\color{red} 3}$ at 98 79
\pinlabel $K$ at 353 237
\pinlabel $K$ at 715 237
\pinlabel $\delta_Q$ at 414 129
\endlabellist
\centering
\includegraphics[width=0.9\textwidth]{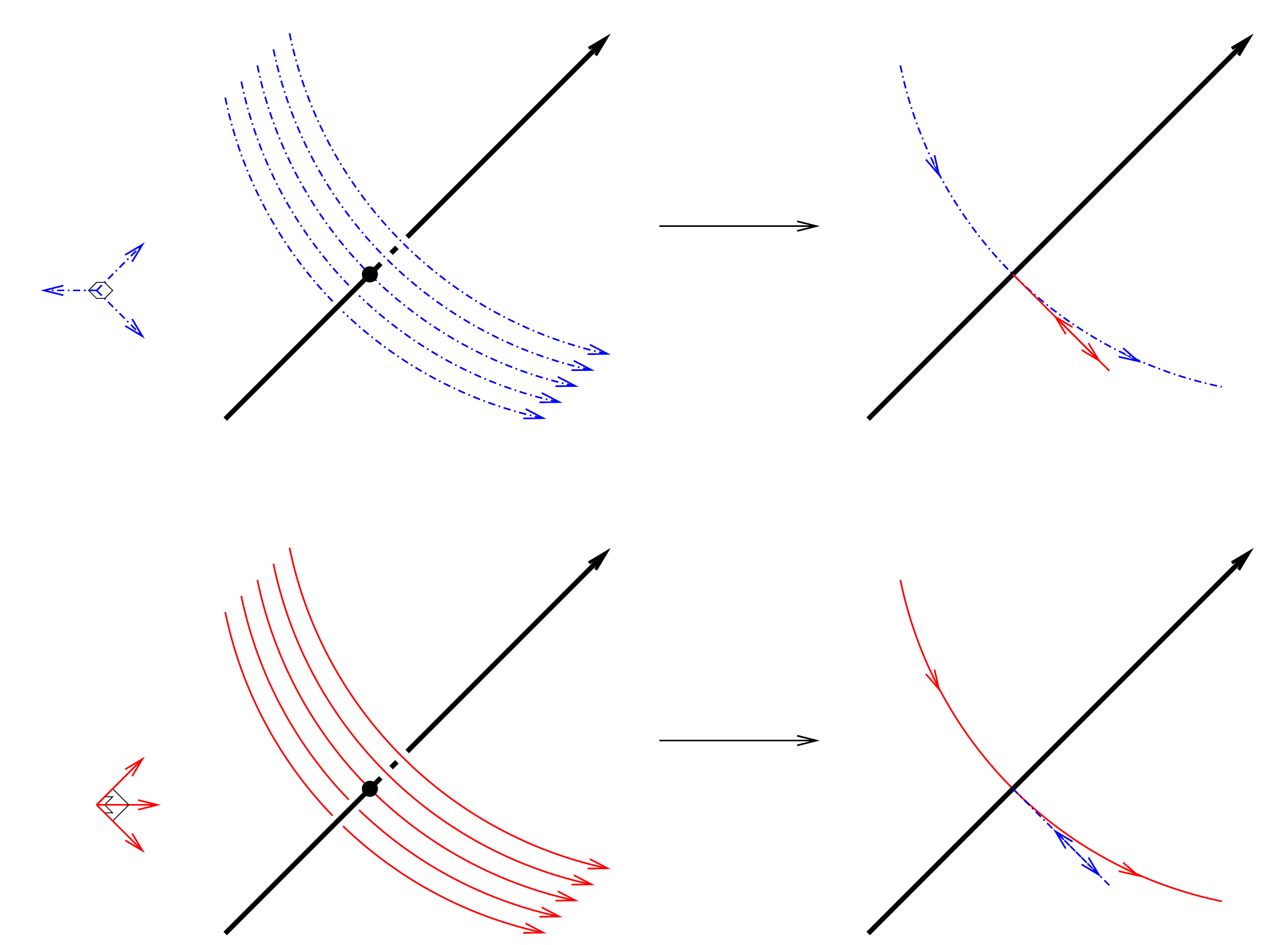}

\caption{The definition of $\delta_N$ and $\delta_Q$. The two configurations shown have sign $\varepsilon=1$.
If the orientation of the $1$-parameter family $s^\lambda$ is
switched, i.e., the $\lambda=0$ and $\lambda=1$ ends are interchanged,
then $\delta_N$ and $\delta_Q$ are still as shown, but with sign
$\varepsilon = -1$. The coordinate axes denote orientations chosen on
$N$ (top) and $Q$ (bottom).
}
\label{fig:2}
\end{figure}
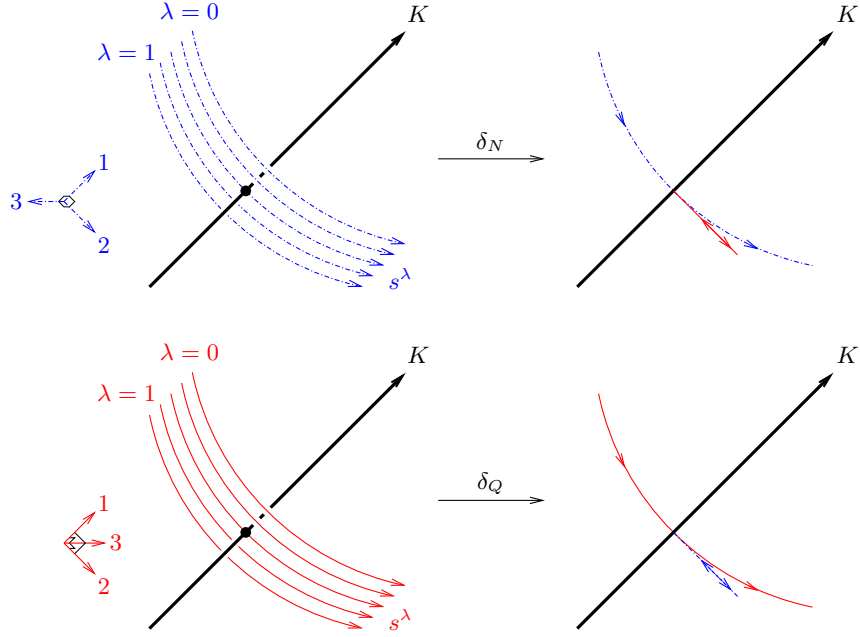

We define {\em string coproducts} $\delta_Q$ and $\delta_N$ as follows, cf.\ Section~\ref{ss:string-op}.
Fix a family of bump functions (which we will call {\em spikes})
$\s_\nu:[0,1] \to D^2$ for $\nu\in D^2$ such that
$\s_\nu^{-1}(0)=\{0,1\}$, $\dot\s_\nu(0)=\nu$ and $\dot\s_\nu(1)=-\nu$; for each $\nu$, $\s_\nu$ lies in the line joining $0$ to $\nu$.
For a
generic $1$-parameter family of broken strings $\{s^\lambda\}$ denote
by $\lambda^j,b^j$ the finitely many values for which
$s_{2i}^{\lambda^j}(b^j)\in K$ for some $i=i(j)$.
For each $j$, let $\s^j=\s_{\nu^j}(\cdot-b^j):[b^j,b^j+1]\to N$ be a
shift of the spike associated to the normal derivative
$\nu^j:=-(\dot\sigma_{2i}^{\lambda^j}(b_j))^{\rm normal}$, with
constant value $s_{2i}^{\lambda^j}(b^j)$ along $K$; interpret this as an $N$-string in the normal disk to $K$ at the point $s_{2i}^{\lambda^j}(b^j)$, traveling along the line joining $0$ to $\nu^j \in D^2$.
Now set
$$
   \delta_Q\{s^\lambda\} :=
   \sum_{j}\eps^j\Bigl(s_1^{\lambda^j},\dots,
   s^{\lambda^j}_{2i}|_{[a_{2i-1},b^j]},\s^j,
   \hat s^{\lambda^j}_{2i}|_{[b^j,a_{2i}]}, \dots, \hat
   s^{\lambda^j}_{2\ell+1}\Bigr),
$$
where the hat means shift by $1$ in the argument, and $\eps^j=\pm 1$
are signs defined as in Figure~\ref{fig:2}.\footnote{Regarding the
  signs: from our considerations of orientation bundles in
  Section~\ref{sec:trans}, we can assign the same sign (which we have
  chosen to be $\eps=1$) to both configurations shown in
  Figure~\ref{fig:2}, provided we choose orientations on $Q$ and $N$
  appropriately. More precisely, at a point $p$ on $K$, if $(v_1,v_2,v_3)$
  is a positively oriented frame in $Q$ where $v_1$ is tangent to $K$
  and $v_2,v_3$ are normal to $K$,
  then we need $(v_1,Jv_2,-Jv_3)$ to be a positively oriented frame in
  $N$, where $J$ is the almost complex structure that rotates normal
  directions in $Q$ to normal directions in $N$. As a result, if we
  give $Q$ any orientation and view $N$ as the subset of $Q$ given by
  a tubular neighborhood of $K$, then we assign the \textit{opposite}
  orientation to $N$.}
Loosely speaking,
$\delta_Q$ inserts an {\em $N$-spike}
at all points where some $Q$-string meets $K$, in such a way that
(iii) still holds. The operation $\delta_N$ is defined analogously,
inserting a {\em $Q$-spike} where an $N$-string meets $K$ (and
defining $\nu^j$ without the minus sign).

Denote by $C_0(\Sigma^\ell)$ and $C_1(\Sigma^\ell)$ the free
$\Z$-modules generated by generic broken strings and generic
1-parameter families of broken strings with $2\ell$ switches,
respectively, and set
$$
   C_i(\Sigma) := \bigoplus_{\ell=0}^\infty C_i(\Sigma^\ell),\qquad i=0,1.
$$
Concatenation of broken strings at the base point gives $C_0(\Sigma)$
the structure of a (noncommutative but strictly associative) algebra
over $\Z$. The operations defined above yield linear maps
$$
   \p:C_1(\Sigma^\ell)\to C_0(\Sigma^\ell)\subset C_0(\Sigma),\qquad
   \delta_N,\delta_Q:C_1(\Sigma^\ell)\to
   C_0(\Sigma^{\ell+1})\subset C_0(\Sigma).
$$
Define the degree zero {\em string homology} of $K$ as
$$
   H_0^\str(K) = H_0(\Sigma) := C_0(\Sigma)/\im(\p+\delta_N+\delta_Q).
$$
Since $\p+\delta_N+\delta_Q$ commutes with multiplication by elements
in $C_0(\Sigma)$, its image is a two-sided ideal in
$C_0(\Sigma)$. Hence degree zero
string homology inherits the structure of an algebra
over $\Z$. By definition, $H_0^\str(K)$ is an isotopy invariant
of the oriented knot $K$ (the framing was used only for convenience
but is not really needed for the construction,
cf.~Remark~\ref{rem:framing} below).

Considering 1-parameter families consisting of generic strings (on which
$\delta_N$ and $\delta_Q$ vanish), we see that for the computation of
$H_0^\str(K)$ we may replace the algebra $C_0(\Sigma)$ by its quotient under
homotopy of generic strings. On the other hand, if $\{s^\lambda\}$ is a generic $1$-parameter family of strings
that consists of generic strings except for an $N$-string (resp.~a $Q$-string) that passes through $K$ exactly once, then $\delta_N$ (resp.~$\delta_Q$) contributes a term to $(\partial+\delta_N+\delta_Q)$, and setting $(\partial+\delta_N+\delta_Q)(\{s^\lambda\})=0$ in these two cases yields the following ``skein relations'':
\begin{enumerate}[(a)]
\item \hspace{2ex} \label{eq:delta1}
$0 \hspace{1ex} = \hspace{1ex}
{\labellist
\small\hair 2pt
\pinlabel $K$ at 87 87 \endlabellist
\raisebox{-3ex}{\includegraphics[width=7ex]{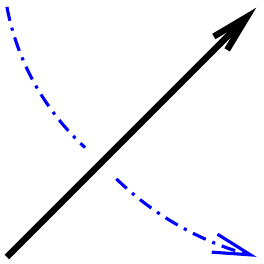}}}
\hspace{2ex} -
\hspace{1ex}
{\labellist
\small\hair 2pt
\pinlabel $K$ at 87 87 \endlabellist
\raisebox{-3ex}{\includegraphics[width=7ex]{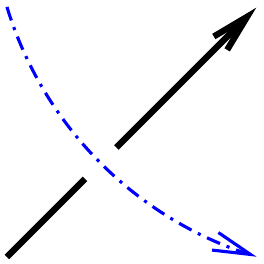}}}
\hspace{2ex} +
\hspace{1ex}
{\labellist
\small\hair 2pt
\pinlabel $K$ at 87 87 \endlabellist
\raisebox{-3ex}{\includegraphics[width=7ex]{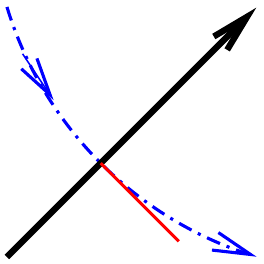}}}$
\vspace{1ex}
\item \hspace{2ex}
$0 \hspace{1ex} = \hspace{1ex}
{\labellist
\small\hair 2pt
\pinlabel $K$ at 87 87 \endlabellist
\raisebox{-3ex}{\includegraphics[width=7ex]{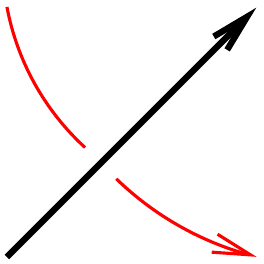}}}
\hspace{2ex} -
\hspace{1ex}
{\labellist
\small\hair 2pt
\pinlabel $K$ at 87 87 \endlabellist
\raisebox{-3ex}{\includegraphics[width=7ex]{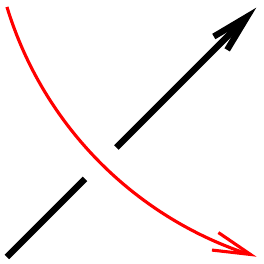}}}
\hspace{2ex} +
\hspace{1ex}
{\labellist
\small\hair 2pt
\pinlabel $K$ at 87 87 \endlabellist
\raisebox{-3ex}{\includegraphics[width=7ex]{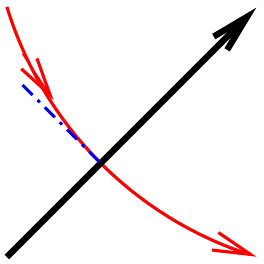}}}$ \hspace{1ex}.
\label{eq:delta2}
\end{enumerate}
Since any generic $1$-parameter family of broken closed strings can be
divided into $1$-parameter families each of which crosses $K$ at most
once, we have proved the following result.

\begin{prop}\label{prop:iso}
Let $\BB$ be the quotient of $C_0(\Sigma)$ by homotopy of generic broken
strings and let $\JJ\subset\BB$ be the two-sided ideal generated by
the skein relations \eqref{eq:delta1} and \eqref{eq:delta2}.
Then
\[
H_0^\str(K) \cong \BB/\JJ.
\]
\end{prop}

\begin{remark}\label{rem:framing}
Degree zero string homology $H_0^\str$ (as well as its higher degree
version defined later) is an invariant of an
oriented knot $K \subset Q$.
Reversing the orientation of $K$ has the result of changing the signs
of $\delta_N$ and $\delta_Q$ but not of $\p$ and gives rise to
isomorphic $H_0^\str$. More precisely, if $-K$ is $K$ with the
opposite orientation, the map $C_0(\Sigma) \to C_0(\Sigma)$ given by
multiplication by $(-1)^\ell$ on the summand $C_0(\Sigma^\ell)$
intertwines the differentials $\p+\delta_N+\delta_Q$ for $K$ and $-K$
and induces an isomorphism $H_0^\str(K) \to H_0^\str(-K)$.
Similarly, mirroring does not change $H_0^\str$ up to isomorphism: if
$\bar K$ is the mirror of $K$, then the mirror (reflection) map induces
a map $C_0(\Sigma) \to C_0(\bar\Sigma)$, and composing with the above map
$C_0(\Sigma) \to C_0(\Sigma)$ gives a chain isomorphism $C_0(\Sigma) \to
C_0(\bar\Sigma)$.

In Sections~\ref{ss:cordalg} through~\ref{ss:groupring}, we will ``improve''
$H_0^\str$ from an abstract ring to one that canonically contains
the ring $\Z[\lambda^{\pm 1},\mu^{\pm 1}]$. This requires a choice of
framing of $K$ (though for $Q=\R^3$, there is a canonical choice given
by the Seifert framing). In the improved setting, $H_0^\str$ changes
under orientation reversal of $K$ by replacing $(\lambda,\mu)$ by
$(\lambda^{-1},\mu^{-1})$; under framing change by $f\in\Z$ by
replacing $(\lambda,\mu)$ by $(\lambda\mu^f,\mu)$; and under mirroring
by replacing $(\lambda,\mu)$ by $(\lambda,\mu^{-1})$. In particular,
the improved $H_0^\str$ is very sensitive to framing change and
mirroring. For a related discussion, see \cite[\S 4.1]{Ng:1}.
\end{remark}
\smallskip

{\bf A modified version of string homology. }
The choice of the base point in $N$ rather than $Q$ in the definition
of string homology $\HSN(K)$ is dictated by the relation to Legendrian contact
homology. However, from the perspective of string topology we could
equally well pick the base point in $Q$, as we describe next.

\begin{figure}
\labellist
\small\hair 2pt
\pinlabel ${\color{red} x_0}$ at 26 109
\pinlabel ${\color{red} v_0}$ at 6 60
\pinlabel ${\color{red} s_1}$ at 65 40
\pinlabel ${\color{blue} s_2}$ at 229 40
\pinlabel ${\color{red} s_3}$ at 302 72
\pinlabel ${\color{blue} s_4}$ at 243 116
\pinlabel ${\color{red} s_5}$ at 97 135
\pinlabel $K$ at 351 2
\endlabellist
\centering
\includegraphics[width=0.6\textwidth]{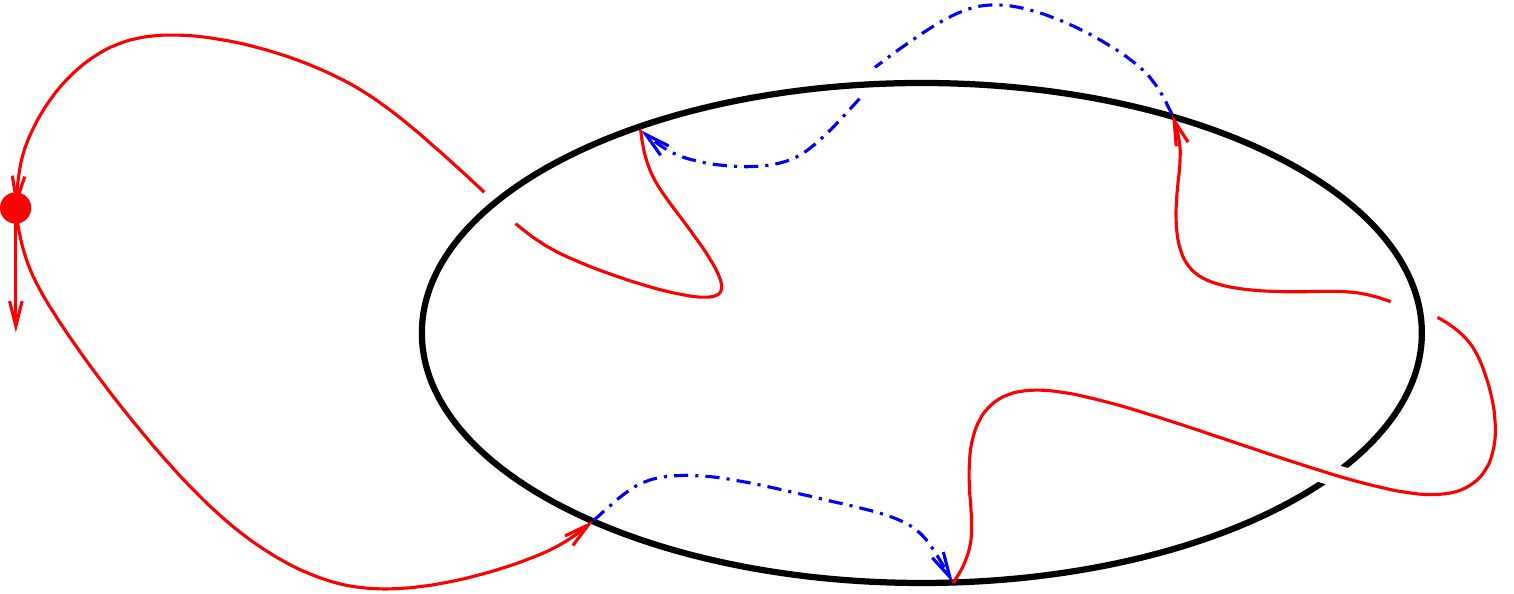}
\caption{In the alternate definition that produces modified string homology, a broken closed string with $4$ switches. As usual, $Q$-strings are in red, $N$-strings in (dashed) blue.}
\label{fig:bcstring-modified}
\end{figure}
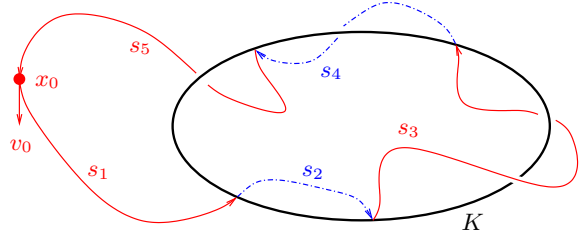

Choose a base point $x_0\in Q\setminus K$ and a tangent vector $v_0 \in
T_{x_0}Q$. Modify the definition of a broken string with $2\ell$ switches to $s=(a_0,\dots,
a_{2\ell+1};s_0,\dots,s_{2\ell})$, where now
$$
   s_{2i}:[a_{2i},a_{2i+1}]\to Q,\quad s_{2i-1}:[a_{2i-1},a_{2i}]\to
   N,
$$
and we require that $s_0(a_0)=s_{2\ell}(a_{2\ell+1})=x_0$, $\dot
s_0(a_0)=\dot s_{2\ell}(a_{2\ell+1})=v_0$ and conditions (ii) and
(iii) of Definition~\ref{def:broken_string1} hold. See Figure~\ref{fig:bcstring-modified}.

Let $\hat{C}_0(\Sigma)$ denote the ring generated as a $\Z$-module by
generic broken strings with base point $x_0\in Q$. (As usual, the
product operation on $\hat{C}_0(\Sigma)$ is given by string concatenation.)
We can define
string coproducts $\delta_N$, $\delta_Q$ as before, and then define
the degree $0$ \textit{modified string homology} of $K$ as
$$
\HSQ(K) = \hat{C}_0(\Sigma)/\im(\p+\delta_N+\delta_Q).
$$
We have the following analogue of Proposition~\ref{prop:iso}.

\begin{prop}
Let $\hat{\BB}$ be the quotient of $\hat{C}_0(\Sigma)$ by homotopy
of generic broken strings and let $\hat{\JJ} \subset \hat{\BB}$ be the
two-sided ideal generated by the skein relations \eqref{eq:delta1} and \eqref{eq:delta2}.
Then
$$
\HSQ(K) \cong \hat{\BB}/\hat{\JJ}.
$$
\end{prop}

There is one key difference between $\HSQ$ and $\HSN$.
Since any element in $\pi_1(Q\setminus K,x_0)$ can be
viewed as a pure $Q$-string, we have a canonical map $\Z\pi_1(Q\setminus K,x_0) \to
\HSQ(K)$. In fact,
we will see in Proposition~\ref{prop:groupring} that this is a ring
isomorphism. The same is not the case for $\HSN(K)$.

\subsection{The cord algebra}\label{ss:cordalg}

The definition of $\HSN(K)$ in Section~\ref{ss:string0} is very
similar to the definition of the cord algebra of a knot
\cite{Ng:2b,Ng:1,Ngsurvey}. Here we review the cord algebra, or more
precisely, present a noncommutative refinement of it, in which the
``coefficients'' $\lambda,\mu$ do not commute with the ``cords''.

Let $K \subset Q$ be an oriented knot equipped with a framing, and let
$K'$ be a parallel copy of $K$ with respect to this framing. Choose a
base point $\ast$ on $K$ and a corresponding base point $\ast$ on $K'$
(in fact only the base point on $K'$ will be needed).

\begin{definition}
A \textit{(framed) cord} of $K$ is a continuous map $\gamma :\thinspace
[0,1] \to Q$ such that $\gamma([0,1]) \cap K = \emptyset$ and
$\gamma(0),\gamma(1) \in K'\setminus\{\ast\}$. Two framed cords are
\textit{homotopic} if they are homotopic through framed cords.
\end{definition}

We now construct a noncommutative unital ring $\AA$ as follows: as a
ring, $\AA$ is freely
generated by homotopy classes of cords and four extra generators
$\lambda^{\pm 1},\mu^{\pm 1}$, modulo the relations
\[
\lambda\cdot\lambda^{-1} = \lambda^{-1}\cdot\lambda = \mu\cdot\mu^{-1}
= \mu^{-1}\cdot\mu = 1, \hspace{3ex} \lambda\cdot\mu =
\mu\cdot\lambda.
\]
Thus $\AA$ is generated as a $\Z$-module by (noncommutative) words in
homotopy classes of cords and powers of $\lambda$ and $\mu$ (and the
powers of $\lambda$ and $\mu$ commute with each other, but not with
any cords).

\begin{definition}
The \textit{cord algebra} of $K$ is the quotient ring \label{def:cordalg}
\[
\Cord(K) = \AA/\II,
\]
where $\II$ is the two-sided ideal of $\AA$ generated by the following
``skein relations'':
\begin{enumerate}
\item \label{it:cord1}
$\raisebox{-3ex}{\includegraphics[height=7ex]{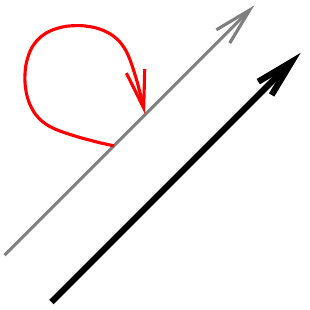}} = 1-\mu$
\item \label{it:cord0}
$\raisebox{-3ex}{\includegraphics[height=7ex]{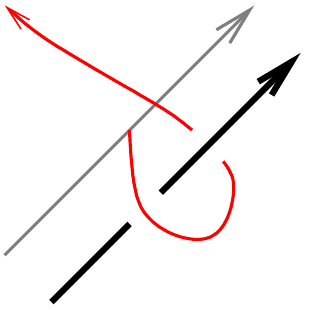}} = \mu
\cdot \raisebox{-3ex}{\includegraphics[height=7ex]{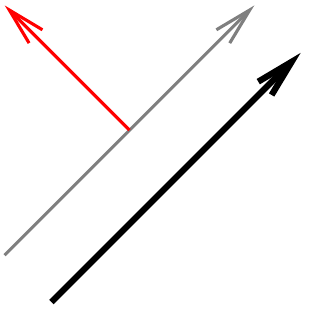}}$
\hspace{3ex} and
\hspace{3ex}
$\raisebox{-3ex}{\includegraphics[height=7ex]{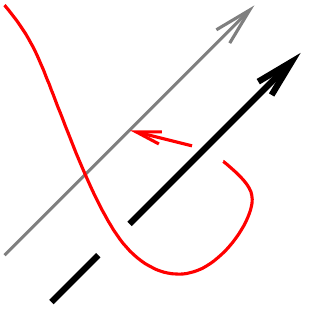}}
=\raisebox{-3ex}{\includegraphics[height=7ex]{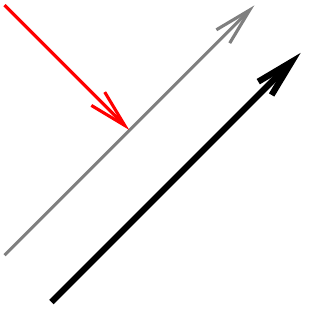}} \cdot \mu$
\item \label{it:cord2}
$\raisebox{-3ex}{\includegraphics[height=7ex]{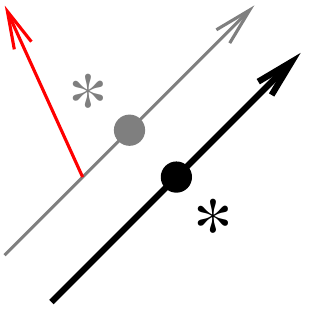}} = \lambda
\cdot \raisebox{-3ex}{\includegraphics[height=7ex]{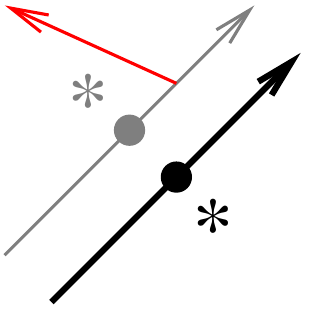}}$
\hspace{3ex} and
\hspace{3ex}
$\raisebox{-3ex}{\includegraphics[height=7ex]{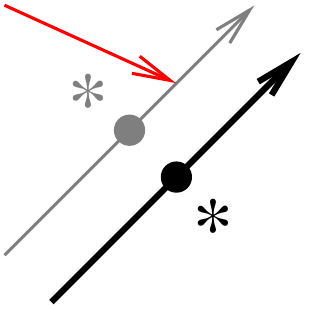}} =
\raisebox{-3ex}{\includegraphics[height=7ex]{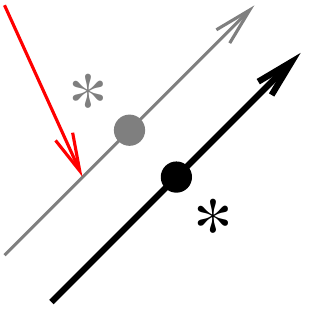}} \cdot
\lambda$
\item \label{it:cord3}
$\raisebox{-3ex}{\includegraphics[height=7ex]{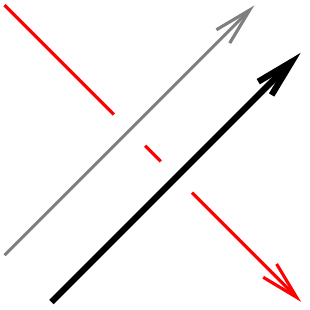}}
- \raisebox{-3ex}{\includegraphics[height=7ex]{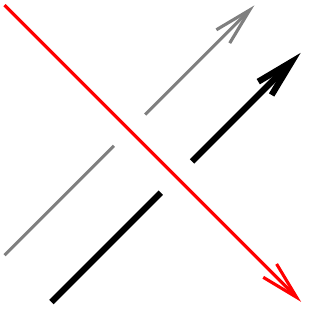}}
= \raisebox{-3ex}{\includegraphics[height=7ex]{figures/skein3c}} \cdot
\raisebox{-3ex}{\includegraphics[height=7ex]{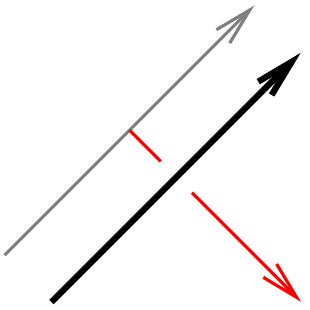}}$.
\end{enumerate}
Here $K$ is depicted in black and $K'$ parallel to $K$ in gray, and
cords are drawn in red.
\end{definition}

\begin{remark}
The skein relations in Definition~\ref{def:cordalg} depict cords in
space that
agree outside of the drawn region (except in (\ref{it:cord3}), where
either of the two cords on the left hand side of the equation splits
into the two on the right). Thus (\ref{it:cord0}) states that
appending a meridian to the beginning or end of a cord multiplies that
cord by $\mu$ on the left or right, and (\ref{it:cord3}) is equivalent
to:
\[
\raisebox{-3ex}{\includegraphics[height=7ex]{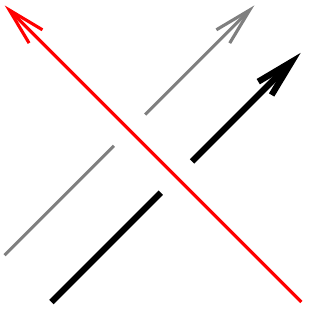}}
- \raisebox{-3ex}{\includegraphics[height=7ex]{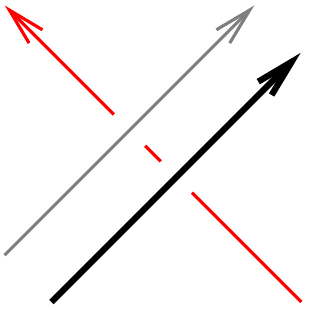}}
= \raisebox{-3ex}{\includegraphics[height=7ex]{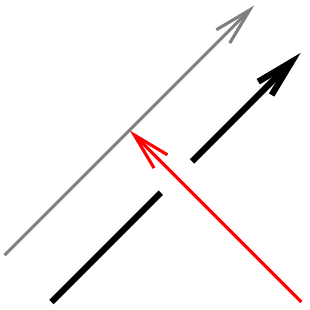}} \cdot
\raisebox{-3ex}{\includegraphics[height=7ex]{figures/skein3i}}.
\]
\end{remark}

\begin{remark} \label{rmk:commute}
Our stipulation that $\lambda,\mu$ not commute with cords necessitates a
different normalization of the cord algebra of $K \subset Q$ from
previous definitions
\cite{Ng:1,Ngsurvey}. In the definition from \cite{Ngsurvey}
(\cite{Ng:1} is the same except for a change of variables),
$\lambda,\mu$ commute with cords, and the
parallel copy $K'$ is
not used. Instead, cords are defined to be paths that begin and end on
$K$ with no interior point lying on $K$, and the skein relations are
suitably adjusted, with the key relation, the equivalent of
(\ref{it:cord3}), being:
\[
\raisebox{-3ex}{\includegraphics[height=7ex]{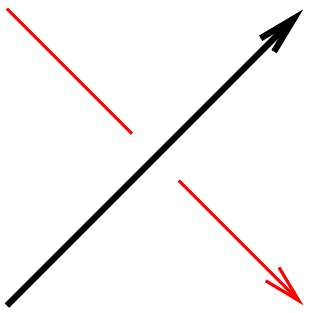}}
- \mu \raisebox{-3ex}{\includegraphics[height=7ex]{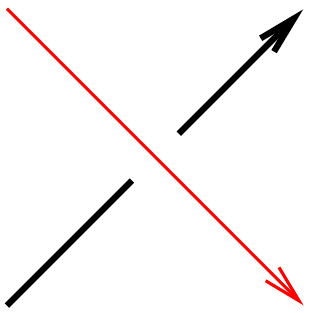}}
= \raisebox{-3ex}{\includegraphics[height=7ex]{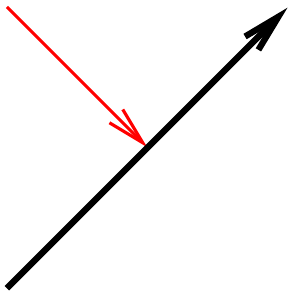}} \cdot
\raisebox{-3ex}{\includegraphics[height=7ex]{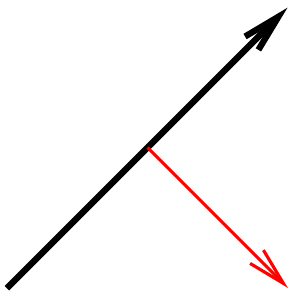}}.
\]
Let $\Cord'(K)$ denote the resulting version of cord algebra.

If we take the quotient of the cord algebra $\Cord(K)$ from
Definition~\ref{def:cordalg} where $\lambda,\mu$ commute with
everything, then the result is a $\Z[\lambda^{\pm 1},\mu^{\pm
  1}]$-algebra isomorphic to $\Cord'(K)$, as long as we take the
Seifert framing ($\operatorname{lk}(K,K') = 0$).
The isomorphism is given as follows: given a framed cord $\gamma$,
extend $\gamma$ to an oriented closed loop $\widetilde{\gamma}$ in $Q\setminus
K$ by joining the
endpoints of $\gamma$ along $K'$ in a way that does not pass through
the base point $\ast$, and map $\gamma$ to
$\mu^{-\operatorname{lk}(\widetilde{\gamma},K)} \gamma$. This is a
well-defined map on $\Cord(K)$ and sends the relations for $\Cord(K)$ to
the relations for $\Cord'(K)$. See also the proof of Theorem 2.10 in
\cite{Ng:1}.
\end{remark}

We now show that the cord algebra is exactly equal to degree $0$
string homology. This follows from the observation that the
$Q$-strings in a generic broken closed string are each
a framed cord of $K$, once we push the endpoints of the $Q$-string off
of $K$; and thus a broken closed string can be thought of as a product
of framed cords.

\begin{proposition}
Let $K \subset Q$ be a framed oriented knot. Then we have a ring isomorphism \label{prop:HS0cord}
\[
\Cord(K) \cong \HSN(K).
\]
\end{proposition}

\begin{proof}
Choose a normal vector field $v$ along $K$ defining the framing and let $K'$ be the pushoff of $K$ in the direction of $v$, placed so that $K'$ lies on the boundary of the tubular neighborhood $N$ of $K$. Fix a base point $p \neq \ast$ on $K$, and let $p'$ be the corresponding point on $K'$, so that $v(p)$ is mapped to $p'$ under the diffeomorphism between the normal bundle to $K$ and $N$. Identify $p'$ with $x_0\in \partial N$ from Definition~\ref{def:broken_string1} (the definition of broken closed string). We can homotope any cord of $K$ so that it begins and ends at $p'$, by pushing the endpoints of the cord along $K'$, away from $\ast$, until they reach $p'$.

Every generator of $\Cord(K)$ as a $\Z$-module has the form $\alpha_1 x_1 \alpha_2 x_2 \cdots x_\ell \alpha_{\ell+1}$, where $\ell \geq 0$, $x_1,\ldots,x_\ell$ are cords of $K$, and $\alpha_1,\ldots,\alpha_{\ell+1}$ are each of the form $\lambda^a \mu^b$ for $a,b\in\Z$. We can associate a broken closed string with $2\ell$ switches as follows. Assume that each cord $x_1,\ldots,x_\ell$ begins and ends at $p'$. Fix paths $\gamma_Q,\tilde{\gamma}_Q$ in $Q$ from $p,p'$ to $p',p$ respectively, and paths $\gamma_N,\tilde{\gamma}_N$ in $N$ from $p,p'$ to $p',p$ respectively, as shown in Figure~\ref{fig:cordtostring}: these are chosen so that the derivative of $\gamma_Q,\tilde{\gamma}_Q,\gamma_N,\tilde{\gamma}_N$ at $p$ is $-v(p),-v(p),v(p),-v(p)$, respectively. For $k=1,\ldots,\ell$, let $\overline{x}_k$ be the $Q$-string with endpoints at $p$ given by the concatenation $\gamma_Q \cdot x_k \cdot \tilde{\gamma}_Q$ (more precisely, smoothen this string at $p'$). Similarly, for $k=1,\ldots,\ell+1$, identify $\alpha_k \in \pi_1(\partial N) = \pi_1(T^2)$ with a loop in $\partial N$ with basepoint $p'$ representing this class; then define $\overline{\alpha}_k$ to be the $N$-string $\gamma_N \cdot \alpha_k \cdot \tilde{\gamma}_N$ for $k=1,\ldots,\ell$, $\alpha_1 \cdot \tilde{\gamma}_N$ for $k=0$, and $\gamma_N \cdot \alpha_{\ell+1}$ for $k=\ell+1$. (If $\ell=0$, then $\overline{\alpha}_1 = \alpha_1$.) Then the concatenation
\[
\overline{\alpha}_1 \cdot \overline{x}_1 \cdot \overline{\alpha}_2 \cdot \overline{x}_2 \cdots \overline{x}_\ell \cdot \overline{\alpha}_{\ell+1}
\]
is a broken closed string with $2\ell$ switches.

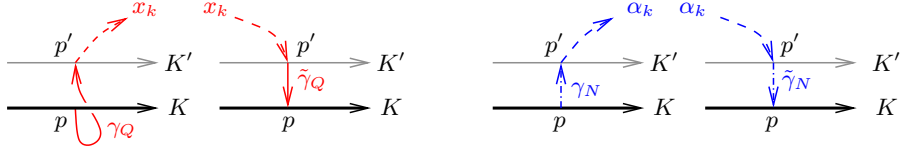
\begin{figure}
\labellist
\small\hair 2pt
\pinlabel $K$ at 104 24
\pinlabel $K$ at 230 24
\pinlabel $K$ at 392 24
\pinlabel $K$ at 524 24
\pinlabel $K'$ at 104 51
\pinlabel $K'$ at 230 51
\pinlabel $K'$ at 392 51
\pinlabel $K'$ at 524 51
\pinlabel $p$ at 34 14
\pinlabel $p$ at 169 14
\pinlabel $p$ at 330 14
\pinlabel $p$ at 456 14
\pinlabel $p'$ at 36 61
\pinlabel $p'$ at 180 61
\pinlabel $p'$ at 327 61
\pinlabel $p'$ at 467 61
\pinlabel ${\color{red} x_k}$ at 84 82
\pinlabel ${\color{red} x_k}$ at 126 82
\pinlabel ${\color{red} \gamma_Q}$ at 70 10
\pinlabel ${\color{red} \tilde{\gamma}_Q}$ at 183 40
\pinlabel ${\color{blue} \alpha_k}$ at 378 82
\pinlabel ${\color{blue} \alpha_k}$ at 409 82
\pinlabel ${\color{blue} \gamma_N}$ at 346 34
\pinlabel ${\color{blue} \tilde{\gamma}_N}$ at 471 40
\endlabellist
\centering
\includegraphics[width=0.9\textwidth]{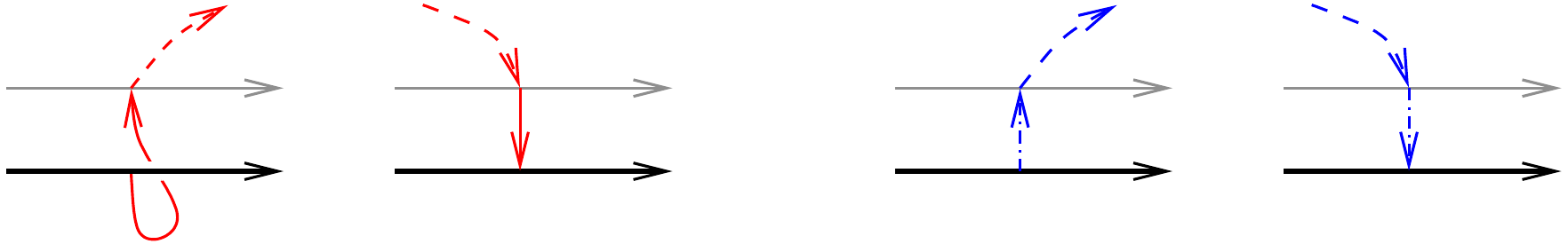}
\caption{Turning an element of the cord algebra into a broken closed string.}
\label{fig:cordtostring}
\end{figure}

Extend this map from generators of $\Cord(K)$ to broken closed strings to a map on $\Cord(K)$ by $\Z$-linearity. We claim that this induces the desired isomorphism $\phi :\thinspace \Cord(K) \to \HSN(K)$. Recall that $\Cord(K)$ is defined by skein relations (\ref{it:cord1}), (\ref{it:cord0}), (\ref{it:cord2}), (\ref{it:cord3}) from Definition~\ref{def:cordalg}, while $\HSN(K)$ is defined by skein relations \eqref{eq:delta1}, \eqref{eq:delta2}
from Proposition~\ref{prop:iso}.

To check that $\phi$ is well-defined, we need for the skein relations (\ref{it:cord1}), (\ref{it:cord0}), (\ref{it:cord2}), (\ref{it:cord3}) to be preserved by $\phi$. Indeed, (\ref{it:cord1}) maps under $\phi$ to
\[
\raisebox{-3ex}{\includegraphics[height=7ex]{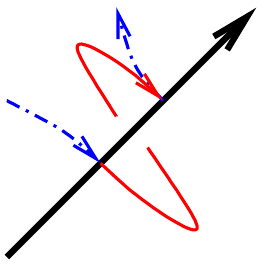}} = \raisebox{-3ex}{\includegraphics[height=7ex]{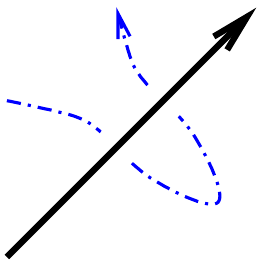}} -
\raisebox{-3ex}{\includegraphics[height=7ex]{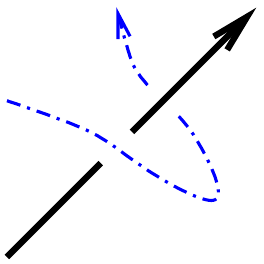}},
\]
which holds in $\HSN(K)$ since both sides are equal to $\raisebox{-3ex}{\includegraphics[height=7ex]{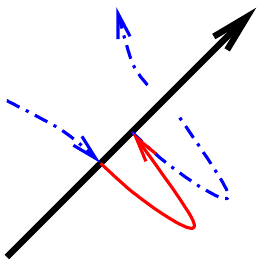}}$: the left hand side by rotating the end of the red $Q$-string and the beginning of the blue $N$-string around $K$ at their common endpoint, the right hand side by skein relation \eqref{eq:delta1}. 
Skein relation (\ref{it:cord3}) maps under $\phi$ to
\[
\raisebox{-3ex}{\includegraphics[height=7ex]{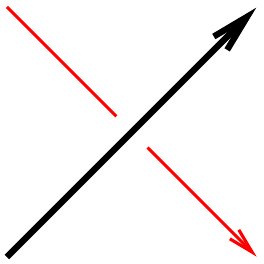}} - \raisebox{-3ex}{\includegraphics[height=7ex]{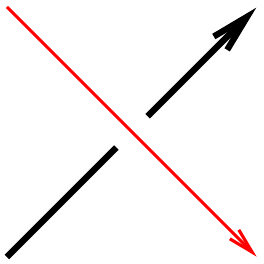}} =
\raisebox{-3ex}{\includegraphics[height=7ex]{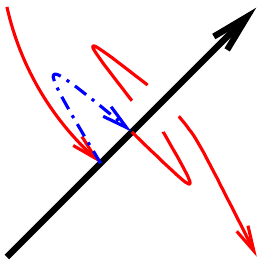}},
\]
which holds by \eqref{eq:delta2}. 
Finally, (\ref{it:cord0}) and (\ref{it:cord2}) map to homotopies of broken closed strings: for instance, the left hand relation in (\ref{it:cord0}) maps to
\[
\raisebox{-3ex}{\includegraphics[height=7ex]{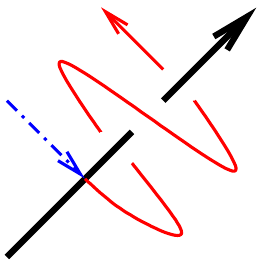}} = \raisebox{-3ex}{\includegraphics[height=7ex]{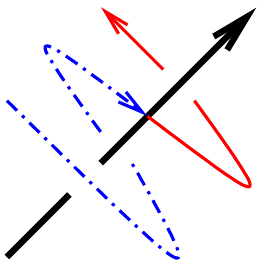}}.
\]

To show that $\phi$ is an isomorphism, we simply describe the inverse map from broken closed strings to the cord algebra. Given any broken closed string, homotope it so that the switches all lie at $p$, and so that the tangent vector to the endpoint of all strings ending at $p$ is $-v(p)$; then the result is in the image of $\phi$ by construction. There is more than one way to homotope a broken closed string into this form, but any such form gives the same element of the cord algebra: moving the switches along $K$ to $p$ in a different way gives the same result by (\ref{it:cord2}), while moving the tangent vectors to $-v(p)$ in a different way gives the same result by (\ref{it:cord0}). The two skein relations \eqref{eq:delta1} and \eqref{eq:delta2} 
are satisfied in the cord algebra because of (\ref{it:cord1}) and (\ref{it:cord3}).
\end{proof}

As mentioned in the Introduction, when $Q=\R^3$, it is an immediate consequence of Theorem~\ref{thm:main} and Proposition~\ref{prop:HS0cord} that the cord algebra is isomorphic to degree $0$ knot contact homology:
\[
H^\cont_0(K) \cong \HSN(K) \cong \Cord(K).
\]
This recovers a result from the literature (see Theorem~\ref{thm:cordHC0}), modulo one important point.
Recall (or see Section~\ref{sec:leg}) that $H^\cont_0(K)$ is the
degree $0$ homology of a differential graded algebra
$(\AA,\partial)$. In much of the literature on knot contact homology,
cf.\ \cite{EENStransverse,Ng:1,Ngtransverse}, this DGA is an algebra
over the coefficient ring $\Z[\lambda^{\pm 1},\mu^{\pm 1}]$ (or
$\Z[\lambda^{\pm 1},\mu^{\pm 1},U^{\pm 1}]$, but in this paper we set
$U=1$): $\AA$ is generated by a finite collection of noncommuting
generators (Reeb chords) along with powers of $\lambda,\mu$ that
commute with Reeb chords. By contrast, in this paper $(\AA,\partial)$
is the \textit{fully noncommutative} DGA in which the coefficients $\lambda,\mu$
commute with each other but not with the Reeb chords; see
\cite{EENS,Ngsurvey}.

The isomorphism $\Cord(K) \cong H^\cont_0(K)$ in
Theorem~\ref{thm:cordHC0} is stated in the existing literature as an
isomorphism of $\Z[\lambda^{\pm 1},\mu^{\pm 1}]$-algebras, i.e., the
coefficients $\lambda,\mu$ commute with everything for both
$H^\cont_0(K)$ and $\Cord(K)$. However, an inspection of the proof of
Theorem~\ref{thm:cordHC0} from \cite{EENS,Ng:2b,Ng:1} shows that it
can be lifted to the fully noncommutative setting, in which
$\lambda,\mu$ do not commute with Reeb chords (for $H^\cont_0(K)$) or
cords (for $\Cord(K)$). We omit the details here, and simply note that
our results give a direct proof of Theorem~\ref{thm:cordHC0} in the
fully noncommutative setting.

\begin{remark} \label{rmk:commute1}
Besides being more natural from the viewpoint of string homology, the stipulation that $\lambda,\mu$ do not commute with cords (in the cord algebra) or Reeb chords (in the DGA) is essential for our construction, in Section~\ref{ss:groupring} below, of a map from degree $0$ homology to the group ring of $\pi$, the fundamental group of the knot complement. This in turn is what allows us to (re)prove that knot contact homology detects the unknot, among other things. If we pass to the quotient where $\lambda,\mu$ commute with everything, then there is no well-defined map to $\Z\pi$.
\end{remark}

\begin{remark}
As already mentioned in the introduction, in \cite{BMSS} Basu, McGibbon, 
Sullivan and Sullivan have given a string topology description of a version 
of the cord algebra for a codimension 2 submanifold $K \subset Q$ of some 
ambient manifold $Q$, proving a theorem which formally looks quite similar to 
Proposition~\ref{prop:HS0cord}. In the language we use here, the main 
difference in their work is the absence of $N$-strings, so that for 
knots $K \subset \R^3$ the version of $\HSN(K)$ they define only recovers the 
specialization at $\lambda=1$ of (the commutative version of) $\Cord(K)$.
\end{remark}

\subsection{Homotopy formulation of the cord algebra}\label{ss:htpy}

We now reformulate the cord algebra in terms of fundamental groups,
more precisely the knot group and its peripheral subgroup, along the
lines of the Appendix to \cite{Ng:2b}.
In light of Proposition~\ref{prop:HS0cord}, we will henceforth denote
the cord algebra as $\HSN(K)$.

We first introduce some notation. Let $K$ be an oriented knot in an
oriented $3$-manifold $Q$ (in fact we only need an orientation and
coorientation of $K$). Let $N$ be a tubular neighborhood of $K$;
as suggested by the notation, we will identify this neighborhood with
the conormal bundle $N \subset T^*Q$ via the tubular neighborhood
theorem.  We write
\begin{align*}
\pi &= \pi_1(Q\setminus K) \\
\hat{\pi} &= \pi_1(\p N);
\end{align*}
note that the inclusion $\p N \hookrightarrow N$ induces a map
$\hat{\pi} \to \pi$, typically an injection.
Let $\Z\pi$,
$\Z\hat{\pi}$ denote the group rings of $\pi,\hat{\pi}$.
We fix a framing on $K$; this, along with the orientation and
coorientation of $K$, allows us to specify two
elements $\mu,\lambda$ for $\hat{\pi}$ corresponding to
the meridian and longitude, and to write
\[
\Z\hat{\pi} =
\Z[\lambda^{\pm 1},\mu^{\pm 1}].
\]

The group ring $\Z\pi$ and the cord algebra $\HSN(K)$ both have natural maps from $\Z[\lambda^{\pm 1},\mu^{\pm 1}]$ (which are injective unless $K$ is the unknot). This motivates the following definition, where ``NC'' stands for ``noncommutative''.

\begin{definition}
Let $R$ be a ring. An \textit{$R$-NC-algebra} is a ring $S$ equipped
with a ring homomorphism $R \to S$. Two $R$-NC-algebras $S_1,S_2$ are
\textit{isomorphic} if there is a ring isomorphism $S_1\to S_2$ that
commutes with the maps $R\to S_1$, $R\to S_2$.
\end{definition}

Note that when $R$ is commutative, the notion of an $R$-NC-algebra
differs from the usual notion of an $R$-algebra; for example, an
$R$-algebra $S$ requires $s_1(rs_2) = rs_1s_2$ for $r\in R$ and
$s_1,s_2\in S$, while an $R$-NC-algebra does not. (One can quotient an
$R$-NC-algebra by commutators involving elements of $R$ to obtain an
$R$-algebra.) If $R$ and $S$ are
both commutative, however, then the notions agree. Also note that any
$R$-NC-algebra is automatically an $R$-bimodule, where $R$ acts on
the left and on the right by multiplication.

By the construction of the cord algebra $\Cord(K)$ from
Section~\ref{ss:cordalg}, $\HSN(K)$ is a $\Z\hat{\pi}$-NC-algebra. We
now give an alternative definition of $\HSN(K)$ that uses $\pi$ and
$\hat{\pi}$ in place of cords.

A \textit{broken word} in $\pi,\hat{\pi}$ is a nonempty word in elements of
$\pi$ and $\hat{\pi}$ whose letters alternate between elements in
$\pi$ and $\hat{\pi}$. For clarity, we use Roman letters for elements
in $\pi$ and Greek for $\hat{\pi}$, and enclose elements in $\pi,\hat{\pi}$ by
square and curly brackets, respectively. Thus examples of broken
words are $\{\alpha\}$, $[x]$, $[x]\{\alpha\}$, and
$\{\alpha_1\}[x_1]\{\alpha_2\}[x_2]\{\alpha_3\}$.

Consider the $\Z$-module freely generated by broken words in
$\pi,\hat{\pi}$, divided by the following
\textit{string relations}:
\begin{enumerate}
\item \label{it:str1}
$\cdots_1 [x\alpha_1]\{\alpha_2\} \cdots_2 = \cdots_1 [x]\{\alpha_1\alpha_2\}
\cdots_2 $
\item \label{it:str2}
$\cdots_1 \{\alpha_1\}[\alpha_2 x] \cdots_2  = \cdots_1 \{\alpha_1\alpha_2\}[x] \cdots_2$
\item \label{it:str3}
$(\cdots_1 [x_1x_2] \cdots_2) - (\cdots_1 [x_1\mu x_2] \cdots_2) = \cdots_1
[x_1]\{1\}[x_2] \cdots_2$
\item \label{it:str4}
$(\cdots_1 \{\alpha_1\alpha_2\} \cdots_2) - (\cdots_1
\{\alpha_1\mu\alpha_2\} \cdots_2 ) = \cdots_1 \{\alpha_1\}[1]\{\alpha_2\} \cdots_2$.
\end{enumerate}
Here $\cdots_1$ is understood to represent the same (possibly empty)
subword each time it appears, as is $\cdots_2$. We denote the
resulting quotient by $S(\pi,\hat\pi)$.

The $\Z$-module $S(\pi,\hat \pi)$ splits into a direct sum
corresponding to the four possible beginnings and endings for broken
words:
\[
S(\pi,\hat\pi) =  S^{\hat{\pi}\hat{\pi}}(\pi,\hat\pi) \oplus
S^{\hat{\pi}\pi}(\pi,\hat\pi) \oplus
S^{\pi\hat\pi}(\pi,\hat\pi) \oplus S^{\pi\pi}(\pi,\hat\pi),
\]
where the superscripts denote which of $\pi$ and $\hat{\pi}$ contain the first
and last letters in the broken word. Thus
$S^{\hat{\pi}\hat{\pi}}(\pi,\hat\pi)$ is generated by broken words
beginning and ending with curly brackets (elements of $\hat{\pi}$)---
$\{\alpha\}$, $\{\alpha_1\}[x]\{\alpha_2\}$, etc.---while
$S^{\pi\pi}(\pi,\hat\pi)$ is generated by $[x]$, $[x]\{\alpha\}[y]$, etc.
We think of these broken words as broken strings with base point on
$N\cap Q$ beginning and ending with $N$-strings (for
$S^{\hat{\pi}\hat{\pi}}(\pi,\hat\pi)$) or $Q$-strings (for
$S^{\pi\pi}(\pi,\hat\pi)$). The other two summands
$S^{\hat{\pi}\pi}(\pi,\hat\pi), S^{\pi\hat\pi}(\pi,\hat\pi)$ can
similarly be interpreted in terms of broken strings, but we will not
consider them further.

On $S^{\hat{\pi}\hat{\pi}}(\pi,\hat\pi)$ and $S^{\pi\pi}(\pi,\hat\pi)$, we can define
 multiplications by
\[
(\cdots_1 \{\alpha_1\})(\{\alpha_2\}\cdots_2) = \cdots_1
\{\alpha_1\alpha_2\} \cdots_2
\]
and
\[
(\cdots_1 [x_1])([x_2] \cdots_2) = \cdots_1 [x_1x_2] \cdots_2,
\]
respectively.
These turn $S^{\hat{\pi}\hat{\pi}}(\pi,\hat\pi)$ and
$S^{\pi\pi}(\pi,\hat\pi)$ into rings. Note for future reference that
$S^{\hat{\pi}\hat{\pi}}(\pi,\hat\pi)$ is generated as a ring by $\{\alpha\}$ and
$\{1\}[x]\{1\}$ for $\alpha \in \hat\pi$ and $x\in\pi$.

\begin{proposition}
$\HAN$ is a $\Z\hat{\pi}$-NC-algebra, while $\HAQ$ is a
$\Z\pi$-NC-algebra and hence a $\Z\hat{\pi}$-NC-algebra as well. Both
$\HAN$ and $\HAQ$ are knot invariants as NC-algebras.
\end{proposition}

\begin{proof}
We only need to specify the ring homomorphisms $\Z\hat{\pi} \to \HAN$
and $\Z\pi \to \HAQ$; these are given by $\alpha \mapsto \{\alpha\}$
and $x \mapsto [x]$, respectively.
\end{proof}

\begin{remark}
View $\Z\pi$ as a $\Z\hat{\pi}$-bimodule via the map $\hat{\pi} \to
\pi$. Then $\HAN$ and $\HAQ$ can alternatively be defined as
follows. Let $\mathcal{A},\hat{\mathcal{A}}$ be defined by
\begin{align*}
\mathcal{A} &= \Z\hat{\pi} \oplus \Z\pi \oplus (\Z\pi
\otimes_{\Z\hat{\pi}} \Z\pi) \oplus
(\Z\pi
\otimes_{\Z\hat{\pi}} \Z\pi \otimes_{\Z\hat{\pi}} \Z\pi) \oplus \cdots \\
\hat{\mathcal{A}} &= \Z\pi \oplus (\Z\pi
\otimes_{\Z\hat{\pi}} \Z\pi) \oplus
(\Z\pi
\otimes_{\Z\hat{\pi}} \Z\pi \otimes_{\Z\hat{\pi}} \Z\pi) \oplus
\cdots.
\end{align*}
Each of $\mathcal{A},\hat{\mathcal{A}}$ has a multiplication operation
given by concatenation (e.g. $a \cdot (b\otimes c) = a\otimes b\otimes
c$); multiplying by an element of $\Z\hat{\pi} \subset \mathcal{A}$
uses the $\Z\hat{\pi}$-bimodule structure on $\Z\pi$. There are
two-sided ideals $\mathcal{I}\subset\mathcal{A},
\hat{\mathcal{I}}\subset\hat{\mathcal{A}}$ generated by
\begin{gather*}
x_1x_2 - x_1\mu x_2-x_1\otimes x_2 \\
1_{\hat{\pi}}-(1-\mu)_{\pi}
\end{gather*}
where $x_1,x_2\in\pi$, $x_1x_2,x_1\mu x_2$ are viewed as elements in
$\Z\pi$, and $1_{\hat{\pi}}$ denotes the element $1\in\Z\hat\pi$ while
$(1-\mu)_{\pi}$ denotes the element $1-\mu\in\Z\pi$. Then
\begin{align*}
\HAN &\cong \mathcal{A}/\mathcal{I} \\
\HAQ &\cong \hat{\mathcal{A}}/\hat{\mathcal{I}}.
\end{align*}
\end{remark}

We conclude this subsection by noting that $\HAN$ is precisely the
cord algebra of $K$.

\begin{proposition}
We have the following isomorphism of $\Z\hat\pi$-NC-algebras:
\[
\HSN(K) \cong \HAN.
\]
\end{proposition}

\begin{proof}
We use the cord-algebra formulation of $\HSN(K) \cong \Cord(K)$ from
Definition~\ref{def:cordalg}. Let $K'$ be the parallel copy of $K$,
and choose a base point $p'$ for $\pi =
\pi_1(Q\setminus K)$ with $p' \in K'\setminus \{\ast\}$. Given a cord
$\gamma$ of $K$, define
$\widetilde{\gamma} \in \pi$ as in
Remark~\ref{rmk:commute}: extend $\gamma$ to a closed loop
$\widetilde{\gamma}$ in $Q\setminus K$ with endpoints at $p'$ by connecting
the endpoints of $\gamma$ to $p'$ along $K'\setminus\{\ast\}$. Then the
isomorphism $\phi :\thinspace \Cord(K) \to \HAN$ is the ring
homomorphism defined by:
\begin{align*}
\phi(\gamma) &= \{1\}[\widetilde{\gamma}]\{1\} &
\phi(\alpha) &= \{\alpha\},
\end{align*}
for $\gamma$ any cord of $K$ and $\alpha$ any element of $\Cord(K)$ of
the form $\lambda^a \mu^b$.

The skein relations in $\Cord(K)$ from Definition~\ref{def:cordalg}
are mapped by $\phi$ to:
\begin{enumerate}
\item
$\{1\}[1]\{1\} = \{1\}-\{\mu\}$
\item
$\{1\}[\mu\widetilde{\gamma}]\{1\} = \{\mu\}[\widetilde{\gamma}]\{1\}$
and
$\{1\}[\widetilde{\gamma}\mu]\{1\} = \{1\}[\widetilde{\gamma}]\{\mu\}$
\item
$\{1\}[\lambda\widetilde{\gamma}]\{1\} = \{\lambda\}[\widetilde{\gamma}]\{1\}$
and
$\{1\}[\widetilde{\gamma}\lambda]\{1\} = \{1\}[\widetilde{\gamma}]\{\lambda\}$
\item
$\{1\}[\widetilde{\gamma}_1\widetilde{\gamma}_2]\{1\} -
\{1\}[\widetilde{\gamma}_1\mu\widetilde{\gamma}_2]\{1\} =
\{1\}[\widetilde{\gamma}_1]\{1\}[\widetilde{\gamma}_2]\{1\}$.
\end{enumerate}
In $\HAN$, these follow from string relations (\ref{it:str4}),
(\ref{it:str1}) and (\ref{it:str2}),
(\ref{it:str1}) and (\ref{it:str2}),
and (\ref{it:str3}), respectively.

Thus $\phi$ is well-defined. It is straightforward to check that
$\phi$ is an isomorphism (indeed, the string relations are constructed
so that this is the case), with inverse $\phi^{-1}$ defined by
\begin{align*}
\phi^{-1}(\{\alpha\}) &= \alpha \\
\phi^{-1}(\{1\}[\widetilde{\gamma}]\{1\}) &= \widetilde{\gamma},
\end{align*}
for $\alpha \in \hat\pi$ and $\widetilde{\gamma}\in\pi$: note that
a closed loop at $p' \in K'\setminus\{\ast\}$ representing
$\widetilde{\gamma}$ is also by definition a cord of $K$.
\end{proof}

\begin{remark}
Similarly, one can show that $\HSQ(K) \cong \HAQ$ as
$\Z\pi$-NC-algebras. In the same vein, there is also a cord formulation for modified string homology $\HSQ(K)$ (as introduced at the end of Section~\ref{ss:string0}), along the lines of Definition~\ref{def:cordalg}: this is $\hat{\mathcal{A}}/\hat{\mathcal{I}}$, where $\hat{\mathcal{A}}$ is the non-unital algebra generated by nonempty products of cords (the difference from $\mathcal{A}$ being that $\hat{\mathcal{A}}$ does not contain words of the form $\lambda^a \mu^b$, which have no cords), and $\hat{\mathcal{I}}$ is the ideal of $\hat{\mathcal{A}}$ generated by skein relations \eqref{it:cord0} through \eqref{it:cord3} from Definition~\ref{def:cordalg}, without \eqref{it:cord1}.
\end{remark}

\subsection{The cord algebra and group rings}\label{ss:groupring}
Having defined the cord algebra in terms of homotopy groups, we can now
give an even more explicit interpretation not involving broken
words, in terms of the group ring $\Z\pi$. Notation is as in
Section~\ref{ss:htpy}: in particular, $K\subset Q$ is a framed oriented knot with
tubular neighborhood $N$, $\pi = \pi_1(Q\setminus K)$, and
$\hat\pi = \pi_1(\partial N)$. When $Q=\R^3$, we assume for simplicity that the framing on $K$ is the Seifert framing.

Before addressing the cord
algebra $\HSN(K) \cong \HAN$ itself, we first note that the modified
version $\HSQ(K) \cong \HAQ$ is precisely $\Z\pi$.

\begin{proposition}
For a knot $K\subset Q$, we have an isomorphism as $\Z\pi$-NC-algebras
$$
\HAQ \cong \Z\pi.
$$
\label{prop:groupring}
\end{proposition}

\begin{proof}
The map $\Z\pi \to \HAQ$, $x \mapsto [x]$, has inverse $\phi$ given by
\begin{align*}
\phi([x_1]\{\alpha_1\}[x_2]\{\alpha_2\}&\cdots
[x_{n-1}]\{\alpha_{n-1}\}[x_n])
\\
&= x_1(1-\mu)\alpha_1x_2(1-\mu)\alpha_2\cdots
x_{n-1}(1-\mu)\alpha_{n-1}x_n;
\end{align*}
note that $\phi$ is well-defined (just check the string
relations) and preserves ring structure.
\end{proof}

The corresponding description of the cord algebra $\HAN$ is a bit more
involved, and
we give two interpretations.

\begin{proposition}
For a knot $K \subset Q$, we have a $\Z$-module isomorphism
$$
\HAN \cong \Z[\lambda^{\pm 1}] \oplus \Z\pi.
$$
For any $\alpha\in\Z[\lambda^{\pm 1}]$, the left and right actions of
$\alpha$ on $\HAN$ induced from the $\hat{\pi}$-NC-algebra
structure on $\HAN$ coincide under this isomorphism
with the actions of $\alpha$ on the factors of $\Z[\lambda^{\pm 1}]
\oplus \Z\pi$ by left and right multiplication.
\label{prop:directsum}
\end{proposition}

\begin{proof}
The isomorphism $\Z[\lambda^{\pm 1}] \oplus \Z\pi \to \HAN$ sends
$(\alpha,0)$ to $\{\alpha\}$ and $(0,x)$ to $\{1\}[x]\{1\}$. Note that
this map commutes with left and right multiplication by powers of
$\lambda$; for example, $\{\lambda^k\alpha\} = \lambda^k\{\alpha\}$ and
$\{1\}[\lambda^k x]\{1\} = \{\lambda^k\}[x]\{1\} = \lambda^k\{1\}[x]\{1\}$.

To see
that the map is a bijection, note that the generators of $\HAN$ can be
separated into ``trivial'' broken words of the form $\{\alpha\}$ and
``nontrivial'' broken words of length at least $3$. Using the string
relations, we can write any trivial broken word uniquely as a sum of
some $\{\lambda^a\}$ and some nontrivial broken words:
$$
\{\lambda^a \mu^b\} = \{\lambda^a\} - \sum_{i=0}^{b-1}
\{\lambda^a\}[\mu^{i-1}]\{1\}$$
if $b \geq 0$, and similarly for $b<0$. On the other hand, any
nontrivial broken word in $\HAN$ can be written uniquely as a $\Z$-linear
combination of words of the form $\{1\}[x]\{1\}$, $x\in\pi$: just use
the map $\phi$ from the proof of Proposition~\ref{prop:groupring} to
reduce any nontrivial broken word to broken words of length $3$, and
then apply the identity $\{\alpha_1\}[x]\{\alpha_2\} =
\{1\}[\alpha_1x\alpha_2]\{1\}$.
\end{proof}

\begin{proposition}
For knots $K \subset \R^3$, string homology $\HSN(K)\cong\HAN$, and
thus knot contact homology,
detects the unknot $U$.
More
precisely, left multiplication by $\lambda-1$ on $\HAN$ has nontrivial kernel
if and only if $K$ is unknotted.
\label{prop:distinguish}
\end{proposition}

\begin{proof}
First, if $K = U$, then $\lambda = 1$ in $\pi$, and so
$$
(\lambda-1)\{1\}[1]\{1\} = \{\lambda\}[1]\{1\}-\{1\}[1]\{1\}
= \{1\}[\lambda]\{1\}-\{1\}[1]\{1\} = 0
$$
in $\HSN(U)$, while $\{1\}[1]\{1\} \neq 0$ by the proof of
Proposition~\ref{prop:directsum}.

Next assume that $K \neq U$, and consider the effect of multiplication
by $\lambda-1$ on $\Z[\lambda^{\pm 1}] \oplus \Z\pi$. Clearly this map
is injective on the $\Z[\lambda^{\pm 1}]$ summand; we claim that it is
injective on the $\Z\pi$ summand as well. Indeed, suppose that some
nontrivial sum $\sum a_i [x_i] \in \Z\pi$ is unchanged by
multiplication by $\lambda$. Then $[\lambda^k x_1]$ must appear in
this sum for all $k$, whence the sum is infinite since $\hat{\pi}$
injects into $\pi$ by the Loop Theorem.
\end{proof}

\begin{remark}
It was first shown in \cite{Ng:1} that the cord algebra detects the
unknot. That proof uses a relationship between the cord algebra and the
$A$-polynomial, along with the fact that the $A$-polynomial detects
the unknot \cite{DG}, which in turn relies on gauge-theoretic results
of Kronheimer and Mrowka \cite{KM}. As noted previously, by contrast,
the above proof that string homology detects the unknot uses only the
Loop Theorem. Either proof shows that knot contact homology detects
the unknot. However, we emphasize that for our argument, unlike the
argument in \cite{Ng:1}, it is crucial that we use the fully
noncommutative version
of knot contact homology.
\end{remark}

We can recover the multiplicative structure on $\HAN$ under the
isomorphism of Proposition~\ref{prop:directsum} as follows. On
$\Z[\lambda^\pm] \oplus \Z\pi$, define a multiplication operation $*$
by
$$
(\lambda^{k_1},x_1) * (\lambda^{k_2},x_2)
=
(\lambda^{k_1+k_2}, \lambda^{k_1} x_2 + x_1 \lambda^{k_2} + x_1x_2 -
x_1\mu x_2).
$$
It is easy to check that $*$ is associative, and that the isomorphism
$\HAN \cong (\Z[\lambda^{\pm}] \oplus \Z\pi,*)$ now becomes an
isomorphism of $\Z\hat{\pi}$-NC-algebras, where $\Z[\lambda^{\pm 1}]
\oplus \Z\pi$ is viewed as a $\Z\hat{\pi}$-NC-algebra via the map
$\Z[\hat{\pi}] \to \Z[\lambda^{\pm 1}] \oplus \Z\pi$ sending $\lambda$
to $(\lambda,0)$ and $\mu$ to $(1,0)-(0,1)$.

We now turn to another formulation of string homology in terms of the
group ring $\Z\pi$. This formulation is a bit cleaner than the
one in Proposition~\ref{prop:directsum}, as the multiplication
operation is easier to describe.

\begin{proposition}
For a knot $K\subset Q$,
let $\mathfrak{R}$ denote the subring of $\Z\pi$ generated by
$\hat{\pi}$ and $\im(1-\mu)$, where $1-\mu$ denotes the map
$\Z\pi\to\Z\pi$ given by left multiplication by $1-\mu$. There is a
ring homomorphism
\[
\psi :\thinspace \HSN(K) \to \mathfrak{R}
\]
determined by
$\psi(\{\alpha\}) = \alpha$ and $\psi(\{1\}[x]\{1\}) = x-\mu x$.

If $\hat\pi \to \pi$ is an injection (in particular, if $K
\subset\R^3$ is nontrivial), then $\psi$ is an isomorphism of
$\Z\hat{\pi}$-NC-algebras.
\label{prop:subring}
\end{proposition}

\begin{proof}
It is easy to check that $\psi$ respects all of the string relations
defining $\HAN \cong \HSN(K)$: the key relation $[x_1x_2] - [x_1\mu x_2] -
[x_1]\{1\}[x_2]$ is sent
to $(1-\mu)x_1x_2-(1-\mu)x_1\mu x_2-(1-\mu)x_1(1-\mu)x_2 = 0$. Thus
$\psi$ is well-defined as a map $\HSN(K) \to \mathfrak{R}$. This map acts
as the identity on $\hat{\pi}$ and thus is a $\Z\hat{\pi}$-NC-algebra
map.

Since $\psi$ is surjective by construction, it remains only to show
that $\psi$ is injective when $\hat\pi\to\pi$ is injective. Suppose that
\begin{equation}
0 = \psi\left(\sum_i a_i\{\lambda^i\} + \sum_j b_j\{1\}[x_j]\{1\}\right)
= \sum_i a_i\lambda^i + \sum_j b_j (1-\mu)x_j
\label{eq:injective}
\end{equation}
for some $a_i,b_j\in\Z$ and $x_j\in\pi$. We claim that $b_j = 0$ for
all $j$, whence $a_i = 0$ for all $i$ since $\hat{\pi}$ injects into
$\pi$ for $K$ nontrivial.
Assume without loss of
generality that the framing on $K$ is the $0$-framing (changing
framing simply replaces $\lambda$ by $\lambda \mu^k$ for some $k$).
Then the linking number with $K$ gives a homomorphism
$\lk\colon\thinspace\pi\to\Z$ satisfying $\lk(\lambda) = 0$ and
$\lk(\mu) = 1$.
If $\sum_j b_j
x_j$ is not a trivial sum, then let $x_\ell$ be the contributor to this
sum of maximal linking number. The term $-b_\ell\mu x_\ell$ in $\sum_j
b_j (1-\mu) x_j$ cannot be canceled by any other term in that sum;
thus for (\ref{eq:injective}) to hold, $x_\ell$ must have linking
number $-1$. But a similar argument shows that the contributor to $\sum_j b_j
x_j$ of minimal linking number must have linking number $0$, contradiction.
We conclude that $\sum_j b_j x_j$ must be a trivial sum, as claimed.
\end{proof}

\begin{remark}
To be clear, as a knot invariant derived from knot contact homology,
the cord algebra $\HSN(K)$ (for $K\neq U$) is the ring
$\mathfrak{R} \subset \Z\pi$
along with the map $\Z\hat{\pi} = \Z[\lambda^{\pm 1},\mu^{\pm 1}] \to
\mathfrak{R}$.
Proposition~\ref{prop:subring} implies that the
$\Z[\lambda^{\pm 1},\mu^{\pm 1}]$-NC-algebra structure on
$\Z\pi = \HSQ(K)$ completely determines
$\HSN(K)$. We do not know if $\HSN(K)$ determines $\HSQ(K)$ as well, nor
whether $\HSN(K)$ is a
complete knot invariant.\footnote{Added in revision: it has now been proven by Shende \cite{Shende}, and then reproven in \cite{ENSh}, that the Legendrian isotopy type of $\Lambda_K$ completely determines the knot $K$. The proof in \cite{ENSh} relies on the present paper and shows that an enhanced version of knot contact homology (or of $\HSN(K)$) determines $K$. The question of whether $\HSN(K)$ is a complete invariant remains open.}

On the other hand, $\HSQ(K)$ as a ring is a complete knot invariant for prime knots in $\R^3$ up to mirroring, as we can see as follows. By Proposition~\ref{prop:groupring}, $\HSQ(K) \cong \Z[\pi]$, and for prime knots $K$, Gordon and Luecke \cite{GL} show that $\pi = \pi_1(\R^3 \setminus K)$ determines $K$ up to mirroring. On the other hand, $\pi$ is a left-orderable group, and the ring isomorphism type of $\Z[G]$ when $G$ is left-orderable is determined by the group isomorphism type of $G$ \cite{LR}.
We thank Tye Lidman for pointing this out to us.
\end{remark}

We conclude this section with two examples.

\begin{example}
When $K$ is the unknot $U$, then $\HSN(U) \cong
\Z[\lambda^{\pm 1},\mu^{\pm 1}]/((\lambda-1)(\mu-1))$, while $\HSQ(U)
\cong \Z[\mu^{\pm 1}]$.
\label{ex:unknot}
The ring homomorphism $\psi$ from
Proposition~\ref{prop:subring}, which is not injective, is given by
$\psi(\lambda)=1$, $\psi(\mu)=\mu$. The isomorphism from
Proposition~\ref{prop:directsum} is the (inverse of the) map
\begin{align*}
\Z[\lambda^{\pm 1}]\oplus\Z[\mu^{\pm 1}] &\to \Z[\lambda^{\pm
  1},\mu^{\pm 1}]/((\lambda-1)(\mu-1)) \\
(\alpha,\beta) &\mapsto \alpha+(\mu-1)\beta.
\end{align*}

As noticed by Lidman, this computation of $\HSN(U)$ along with
Proposition~\ref{prop:subring} gives an
alternative (and shorter) proof that knot contact homology detects the
unknot (Proposition~\ref{prop:distinguish}), and more generally that
this continues to hold even if the knot is not assumed to be Seifert
framed (Corollary~\ref{cor:unknot}).

\begin{proof}[Proof of Corollary~\ref{cor:unknot}]
Suppose that $\HSN(K) \cong
\HSN(U)$ where $K$ is a framed
oriented knot and $U$ is the unknot with some framing. By changing the
framing of both, we can assume that $K$ has its Seifert framing. If
$K$ is knotted, then $\Z\pi$ has no zero divisors since $\pi$ is
left-orderable, and thus $\HSN(K) \subset \Z\pi$ also has no zero
divisors by Proposition~\ref{prop:subring}. On the other hand,
$\HSN(U)
\cong \Z[\lambda^{\pm 1},\mu^{\pm
  1}]/((\lambda\mu^f-1)(\mu-1))$ for some $f\in\Z$. Thus $K$ must be
the unknot and must further have the same framing as $U$.
\end{proof}

In \cite{GLid}, Gordon and Lidman extend this line of argument (i.e., applying Proposition~\ref{prop:subring}) to prove that knot contact homology detects torus knots as well as cabling and compositeness.
\end{example}

\begin{example}
When $K$ is the right-handed trefoil $T$, a slightly more elaborate
version of the calculation of the cord algebra from \cite{Ng:1} (see
also \cite{Ngsurvey})
gives the following expression for $\HSN(T)$: it is generated by
$\lambda^{\pm 1}$,
$\mu^{\pm 1}$, and one more generator $x$, along with the relations:
\begin{align*}
\lambda\mu &= \mu\lambda \\
\lambda\mu^6 x &= x \lambda\mu^6 \\
-1+\mu+x-\lambda\mu^5 x \mu^{-3} x \mu^{-1} &= 0 \\
1-\mu-\lambda\mu^4 x \mu^{-2} - \lambda\mu^5 x \mu^{-2} x \mu^{-1} &=
0.
\end{align*}
On the other hand,
$\HSQ(T) = \Z\pi$ is the ring generated by $\mu^{\pm 1}$ and $a^{\pm 1}$
modulo the relation $\mu a \mu = a \mu a$; the longitudinal class is
$\lambda = a\mu a^{-1}\mu a\mu^{-3}$. The explicit map from $\HSN(T)$
to $\Z\pi$ is given by:
\begin{align*}
\mu & \mapsto \mu \\
\lambda & \mapsto \lambda = a\mu a^{-1}\mu a\mu^{-3} \\
x & \mapsto (1-\mu) a \mu^{-1} a^{-1}.
\end{align*}
It can be checked that this map preserves the relations in $\HSN(T)$.
\end{example}

\section{Roadmap to the proof of
  Theorem~\ref{thm:main}}\label{sec:roadmap}

The remainder of this paper is devoted to the proof of
Theorem~\ref{thm:main}. To avoid getting lost in the details, we give
here a roadmap to the proof and explain the technical issues to be
addressed along the way.

The proof follows the scheme that is described for a different
situation in~\cite{CL} and consists of 3 steps. Let $\AA$ be the free
$\Z\hat\pi$-NC-algebra generated by Reeb chords and
$\p_\Lambda:\AA\to\AA$ the boundary operator for
Legendrian contact homology. For a Reeb chord $a$
and an integer $\ell\geq 0$ denote by $\MM_\ell(a)$ the moduli space
of $J$-holomorphic disks in $T^*Q$ with one positive puncture
asymptotic to $a$ and boundary on $Q\cup L_K$ with $2\ell$ corners at
which it switches between $L_K$ and $Q$.

{\bf Step 1. }Show that $\MM_\ell(a)$ can be compactified to a
manifold with corners $\ol\MM_\ell(a)$ and that the generating
functions $\phi(a):=\sum_{\ell=0}^\infty\ol\MM_\ell(a)$ (extended
as algebra maps to $\AA$) satisfy the relation
$$
   \p\phi = \phi\p_\Lambda - \delta\phi,
$$
where $\delta\ol\MM_\ell(a)$ is the subset of elements in
$\ol\MM_\ell(a)$ that intersect $K$ at the interior of some boundary
string.

{\bf Step 2. }Construct a chain complex $(C_*(\Sigma),\p+\delta)$ of suitable
chains of broken strings such that $\phi$ induces a chain map
$$
   \Phi:(\AA,\p_\Lambda)\to(C(\Sigma),\p+\delta),
$$
and the homology $H_0(\Sigma,\p+\delta)$ agrees with the string homology
$H_0^\str(K)$ as defined in Section~\ref{ss:string0}.

{\bf Step 3. }Prove that $\Phi$ induces an isomorphism on homology in
degree zero.
\medskip

Step 1 occupies Sections~\ref{S:mdlisp} to~\ref{sec:gluing}. It
involves detailed descriptions of
\begin{itemize}
\item the behavior of holomorphic disks at corner points;
\item compactifications of moduli spaces of holomorphic disks;
\item transversality and gluing of moduli spaces.
\end{itemize}

In Step 2 (Sections~\ref{sec:holo} to~\ref{sec:chain}) we encounter the following problem: The direct approach to
setting up the complex $(C(\Sigma),\p+\delta)$ would involve chains in
spaces of broken strings with varying number of switches. These spaces
could probably be given smooth structures using the polyfold theory by
Hofer, Wysocki and Zehnder~\cite{HWZ1}. Here we choose a different
approach, keeping the number of switches fixed and inserting small
``spikes'' in the definition of the string operation $\delta=\delta_Q+\delta_N$. Since
this involves non-canonical choices, one does not expect identities
such as $\p\delta+\delta\p=0$ to hold strictly but only up to
homotopy, thus leading to an $\infty$-structure as described by
Sullivan in~\cite{Su}. We avoid $\infty$-structures by carefully
defining $\delta$ via induction over the dimension of chains such that
all identities hold strictly on the chain level.

Step 3 (Section~\ref{sec:iso}) follows the scheme described in~\cite{CL}. This involves
\begin{itemize}
\item a length estimate for the boundary of holomorphic disks, which
  implies that $\Phi$ respects the filtrations of $\AA$ and
  $C(\Sigma)$ by the actions of Reeb chords and the total lengths of
  $Q$-strings, respectively.
\item construction of a length-decreasing chain homotopy deforming
  $C(\Sigma)$ to chains $C(\Sigma_\lin)$ of broken strings all of
  whose $Q$-strings are {\em linear straight line segments} (at this point we specialize to $Q=\R^3$);
\item Morse-theoretical arguments on the space $\Sigma_\lin$ to prove
  that $\Phi$ induces an isomorphism on degree zero homology.
\end{itemize}

\section{Holomorphic functions near corners}\label{sec:holo}

In this section, we call a function $f:R\to\C$ on a subset
$R\subset\C$ with piecewise smooth boundary {\em holomorphic} if it is
continuous on $R$ and holomorphic in the interior of $R$.

\subsection{Power series expansions}\label{ss:series}
Denote by $D\subset\C$ the open unit disk and set
\begin{align*}
   D^+ &:= \{z\in D\mid \Im(z)\geq 0\}, \cr
   Q^+ &:= \{z\in D\mid \Re(z)\geq 0,\,\Im(z)\geq 0\}.
\end{align*}
Consider a holomorphic function $f:Q^+\to\C$ (in the above sense,
i.e.~continuous on $Q^+$ and holomorphic in the interior) with
$f(0)=0$. We distinguish four cases according to their boundary
conditions.

{\em Case 1: }$f$ maps $\R_+$ to $\R$ and $i\R_+$ to  $i\R$.\\
In this case, we extend $f$ to a map $f:D^+\to\C$ by the formula
$$
   f(z):=-\overline{f(-\bar z)},\qquad \Re(z)\leq 0,\Im(z)\geq 0,
$$
and then to a map $f:D\to\C$ by the formula
$$
   f(z):=\overline{f(\bar z)},\qquad \Im(z)\leq 0.
$$
The resulting map $f$ is continuous on $D$ and holomorphic outside
the axes $\R\cup i\R$, hence holomorphic on $D$, and it maps $\R$ to
$\R$ and $i\R$ to $i\R$. Thus it has a power series expansion
$$
   f(z) = \sum_{j=1}^\infty a_{2j-1}z^{2j-1},\qquad a_j\in\R.
$$
This shows that each holomorphic function $f:Q^+\to\C$
mapping $\R_+$ to $\R$ and $i\R_+$ to $i\R$ is uniquely the
restriction of such a power series. In particular, $f$ has an isolated
zero at the origin unless it vanishes identically. Similar discussions
apply in the other cases.

{\em Case 2: }$f$ maps $(\R_+,i\R_+)$ to $(i\R,\R)$. Then it has a power series expansion
$$
   f(z) = i\sum_{j=1}^\infty a_{2j-1}z^{2j-1},\qquad a_j\in\R.
$$
{\em Case 3: }$f$ maps $(\R_+,i\R_+)$ to $(\R,\R)$. Then it has a power series expansion
$$
   f(z) = \sum_{j=1}^\infty a_{2j}z^{2j},\qquad a_j\in\R.
$$
{\em Case 4: }$f$ maps $(\R_+,i\R_+)$ to $(i\R,i\R)$. Then it has a power series expansion
$$
   f(z) = i\sum_{j=1}^\infty a_{2j}z^{2j},\qquad a_j\in\R.
$$

\begin{remark}\label{rem:cases1-4}
We can summarize the four cases by saying that $f:Q^+\to\C$ is given by a
power series
$$
   f(z) = \sum_{k=1}^\infty a_kz^k
$$
with either only odd (in Cases 1 and 2) or only even (in Cases 3 and
4) indices $k$, and with the $a_k$ either all real (in Cases 1 and 3)
or all imaginary (in Cases 2 and 4).
Such holomorphic functions $f$ will appear as projections onto a
normal direction of the holomorphic curves considered in
Section~\ref{ss:switching} near switches. Then Case 1 corresponds to
a switch from $Q$ to $N$, Case 2 to a switch from $N$ to $Q$,
Case 3 to a switch from $N$ to $N$, and Case 4 to a switch from $Q$ to
$Q$.
\end{remark}

\begin{remark}\label{rem:series}
It will sometimes be convenient to switch from the positive quadrant
to other domains. For example, the map $\psi(z):=\sqrt{z}$ maps the
upper half disk $D^+$ biholomorphically onto $Q^+$. Thus in Case 1 the
composition $f\circ\psi$ is a holomorphic function on
$D^+$ which maps $\R_+$ to $\R$ and $\R_-$ to $i\R$,
and it has an expansion in powers of $\sqrt{z}$ by
$$
   f\circ\psi(z) = \sum_{j=1}^\infty
   a_{2j-1}z^{j-1/2},\qquad a_j\in\R.
$$
As another example, the map $\phi(s,t):=ie^{-\pi(s+it)/2}$ maps the strip
$(0,\infty)\times[0,1]$ biholomorphically onto
$Q^+\setminus\{0\}$. Thus in Case 1 the
composition $f\circ\phi$ is a continuous function on
$\R_+\times[0,1]$ which is holomorphic in the interior and maps
$\R_+ \times \{0\}$ to $i\R$ and $\R_+\times \{1\}$ to
$\R$, and it has a power series expansion
$$
   f\circ\phi(s,t) = -i\sum_{j=1}^\infty(-1)^j
   a_{2j-1}e^{-(2j-1)\pi(s+it)/2},\qquad a_j\in\R.
$$
Similar discussions apply to the other cases.
\end{remark}

Let us consider once more the function $f:Q^+\to\C$ of Case 1 mapping
$(\R_+,i\R_+)$ to $(\R,i\R)$. Its restrictions to
$i\R_+$ resp.~$\R_+$ naturally give rise to functions
$f_-:(-1,0]\to\R$ resp.~$f_+:[0,1)\to\R$ via
$$
   f_-(t):=(-i)f(-it),\ t\leq 0,\qquad
   f_+(t):=f(t),\ t\geq 0.\qquad
$$
Here and in the sequel we always use the isomorphism
$(-i)=i^{-1}:i\R\to\R$ to identify $i\R$ with $\R$ in the target.
So $f_\pm$ are related by $f_-=r_*f_+$, where the {\em reflection}
$r_*f$ of a complex valued power series $f(t)=\sum_{k=1}^\infty
a_kt^k$, $a_k\in\C^n$, is defined by
$$
   r_*f(t) := (-i)f(-it) = \sum_{k=1}^\infty(-i)^{k+1}a_kt^k.
$$
(Note that the domain $\mathbb{C}$ and the target $\mathbb{C}$ play different roles here: multiplication by $(-i)$ on the domain comes from opening up the positive quadrant to the 
upper half plane, while multiplication by $(-i)$ in the target corresponds to the
canonical rotation by $-J$ from $i\mathbb{R}\subset Q$ to $\mathbb{R}\subset N$.)

The effect of $r_*$ on the power series expansion $f(t) =
\sum_{j=1}^\infty a_{2j-1}t^{2j-1}$ in Case 1 is as follows:
$$
   r_*f(t) = (-i)\sum_{j=1}^\infty a_{2j-1}(-it)^{2j-1}
   = \sum_{j=1}^\infty (-1)^ja_{2j-1}t^{2j-1},
$$
so the coefficient $a_{2j-1}$ is changed to $(-1)^ja_{2j-1}$. Note
that $a_1$ is changed to $-a_1$, which justifies the name
``reflection''.

Now consider $f$ as in Case 2 mapping $(\R_+,i\R_+)$ to
$(i\R,\R)$. Here the restrictions
to $i\R_+$ resp.~$\R_+$ naturally give rise to functions
$f_-:(-1,0]\to\R$ resp.~$f_+:[0,1)\to\R$ via
$$
   f_-(t):=f(-it),\ t\leq 0,\qquad
   f_+(t):=(-i)f(t),\ t\geq 0.\qquad
$$
So $f_\pm$ are related by $f_-=-r_*f_+$, and the coefficient $a_{2j-1}$
in the power series expansion of $f_+$ is changed to
$(-1)^{j+1}a_{2j-1}$. In particular, $a_1$ is unchanged so that
$f_-$ and $f_+$ fit together to a function $(-1,1)\to\R$ of class
$C^2$ (but not $C^3$).

\subsection{Winding numbers}\label{ss:winding}

Consider a holomorphic function $f:Q^+\to\C$ given by a power series
$f(z)=\sum_{k=1}^\infty a_kz^k$ as in Cases 1-4 of the previous subsection.
In each of these cases we define its {\em winding number} at $0$ as
$$
   w(f,0) := \frac{1}{2}\inf\{k\mid a_k\neq 0\}.
$$
Note that the winding number is a half-integer in the first two cases
and an integer in the last two cases. Also note that the
winding number is given by
$$
   w(f,0) = \frac{1}{\pi}\int_{\gamma}f^*d\theta,
$$
where $\gamma$ is a small arc in $Q^+$ connecting $(0,1)$ to
$i(0,1)$. This can be seen, for example, by choosing $\gamma$ as a
small quarter circle $Q^+\cap\p D_\eps$; then the symmetry of $f$ with
respect to reflections at the coordinate axes implies
$$
   \frac{1}{\pi}\int_{\gamma}f^*d\theta = \frac{1}{4\pi}\int_{\p
     D_\eps}f^*d\theta = \frac{1}{4\pi}\cdot 2\pi\inf\{k\mid a_k\neq
   0\} = w(f,0).
$$

Next let $r>1$, denote by $D_r$ the open disk of radius $r$, by
$$
   H^+:=\{z\in\C\mid \Im(z)\geq 0\}
$$
the upper half plane, and set $D_r^+:=D_r\cap H^+$. Consider a
nonconstant continuous map $f:D_r^+\to\C$ which is holomorphic in the
interior and maps the interval $(-r,r)$
to $\R\cup i\R$. Suppose that $f$ has no zeroes on the semi-circle
$\p D_1\cap H^+$. Then $f$ has finitely
many zeroes $s_1,\dots,s_k$ in the interior of $D_1^+$ as well as
finitely many zeroes $t_1,\dots,t_\ell$ in $(-1,1)$.
(Finiteness holds because the holomorphic function $z\mapsto f(z)^2$
maps $\R$ to $\R$, and thus can only have finitely many zeroes by
the Schwarz reflection principle and unique continuation.)
Denote by $w(f,s_i)\in\N$ resp.~$w(f,t_j)\in\frac{1}{2}\N$ the
winding numbers at the zeroes. Thus with the closed angular
form $d\theta$ on $\C\setminus\{0\}$,
$$
   w(f,s_i) := \frac{1}{\pi}\int_{\alpha_i}f^*d\theta,\qquad
   w(f,t_j) := \frac{1}{\pi}\int_{\beta_j}f^*d\theta,
$$
where $\alpha_i$ is a small circle around $s_i$ and $\beta_j$ is a
small semi-circle around $t_j$ in $D_1^+$, both oriented in the
counterclockwise direction. (Thus the $w(f,s_i)$ are even integers and
the $w(f,t_j)$ are integers or half-integers). Denote by $\gamma$ the
semi-circle $\p D_1\cap H^+$ oriented in the counterclockwise
direction. Then Stokes' theorem yields
$$
   \frac{1}{\pi}\int_\gamma f^*d\theta = \sum_{i=1}^k w(f,s_i) +
   \sum_{j=1}^\ell w(f,t_j).
$$
Since all winding numbers are nonnegative, we have shown the following result.

\begin{lemma}\label{lem:wind}
Consider a nonconstant continuous map
$f:D_r^+\to\C$ which is holomorphic in the interior and maps $(-r,r)$
to $\R\cup i\R$. Suppose that $f$ has no zeroes on the semi-circle
$\gamma=\p D_1\cap H^+$ and zeroes at $t_1,\dots,t_m\in(-1,1)$ (plus
possibly further zeroes in $D_1^+$). Then
$$
   \frac{1}{\pi}\int_\gamma f^*d\theta \geq \sum_{j=1}^m w(f,t_j).
$$
\end{lemma}

More generally, for $n\geq 1$ consider a nonconstant continuous map
$f:D_r^+\to\C^n$ which is holomorphic in the interior and maps $(-r,r)$
to $\R^n\cup i\R^n$. Suppose that $f$ has no zeroes on the semi-circle
$\p D_1\cap H^+$ and zeroes $z_1,\dots,z_m$ in $D_1^+$ (in the interior
or on the boundary). For each direction $v\in S^{n-1}\subset\R^n$
we obtain a holomorphic map $f_v:=\pi_v\circ f$, where $\pi_v$ is the
projection onto the complex line spanned by $v$. Fix a positive volume
form $\Om$ on $S^{n-1}$ of total volume $1$. Then there exists an
open subset $V\subset S^{n-1}$ of measure $1$ such that for all $v\in
V$, $f_v$ has zeroes precisely at the $z_j$ and their winding numbers
are independent of $v\in V$. So we can define
$$
   w(f,z_j) := \int_V w(f_v,z_j)\Om(v) = w(f_{v_0},z_j)
$$
for any $v_0\in V$ and obtain

\begin{corollary}\label{cor:wind}
Consider a nonconstant continuous map
$f:D_r^+\to\C^n$ which is holomorphic in the interior and maps $(-r,r)$
to $\R^n\cup i\R^n$. Suppose that $f$ has no zeroes on the semi-circle
$\gamma=\p D_1\cap H^+$ and zeroes at $t_1,\dots,t_m\in(-1,1)$ (plus
possibly further zeroes in $D_1^+$). Then there exists an open subset
$V\subset S^{n-1}$ of measure $1$ such that for every $v_0\in V$,
$$
   \frac{1}{\pi}\int_\gamma f_{v_0}^*d\theta =
   \int_V\left(\frac{1}{\pi}\int_\gamma f_v^*d\theta\right)\Om(v)
   \geq \sum_{j=1}^m w(f,t_j).
$$
\end{corollary}

\subsection{Spikes}\label{ss:spikes}
Consider again the upper half disk $D^+=\{z\in D\mid \Im(z)\geq 0\}$
and real points $-1<b_1<b_2<\cdots<b_\ell<1$. We are interested in
holomorphic functions $f:D^+\setminus\{b_1,\dots,b_\ell\}$, continuous
on $D^+$, mapping the intervals $[b_{i-1},b_i]$ alternatingly to $\R$
and $i\R$. We wish to describe models of 1- resp.~2-parameter families
in which 2 resp.~3 of the $b_i$ come together.
A model for such a 1-parameter family is
\begin{equation}\label{eq:1-dimmodel}
 f_\eps(z) := \sqrt{z(z-\eps)},\qquad \eps\geq 0
\end{equation}
with zeroes at $0,\eps$. A model for a 2-parameter family is
\begin{equation}\label{eq:2-dimmodel}
   f_{\delta,\eps}(z) := \sqrt{z(z+\delta)(z-\eps)},\qquad
   \eps,\delta\geq 0
\end{equation}
with zeroes at $-\delta,0,\eps$. Here we choose appropriate branches
of the square root so that the functions become continuous.
The images of these functions are
shown in Figure~\ref{fig:spike-families}.
\begin{figure}
\labellist
\small\hair 2pt
\pinlabel $f_\eps$ at 323 533
\pinlabel $f_{\delta,\eps}$ at 323 191
\pinlabel ${\color{blue} 0}$ at 140 453
\pinlabel ${\color{blue} \eps}$ at 181 453
\pinlabel ${\color{blue} 0}$ at 140 114
\pinlabel ${\color{blue} \eps}$ at 181 114
\pinlabel ${\color{blue} -\delta}$ at 73 114
\endlabellist
\centering
\includegraphics[width=0.8\textwidth]{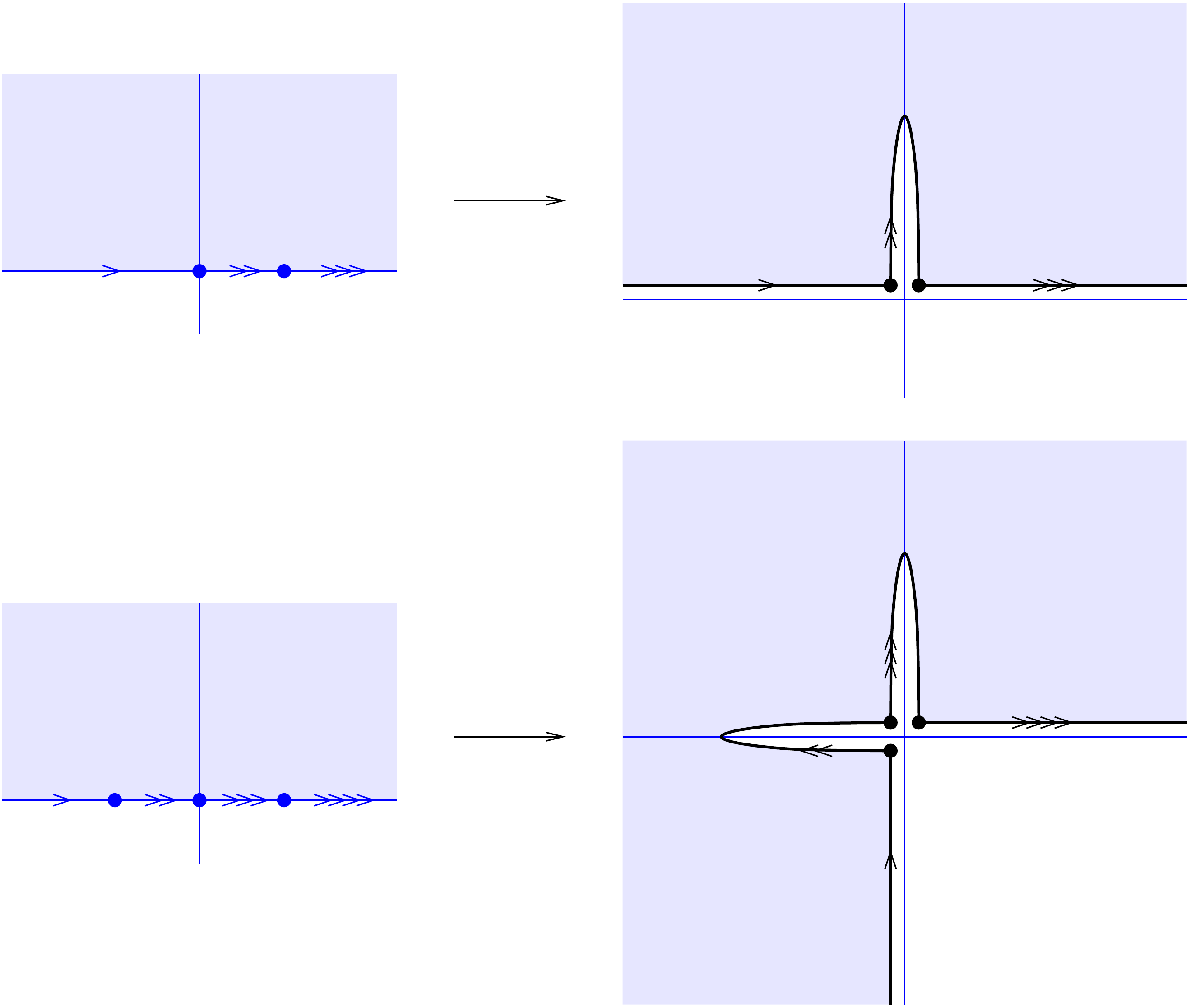}
\caption{
Spikes in the model families $f_\eps$ and $f_{\delta,\eps}$.
}
\label{fig:spike-families}
\end{figure}
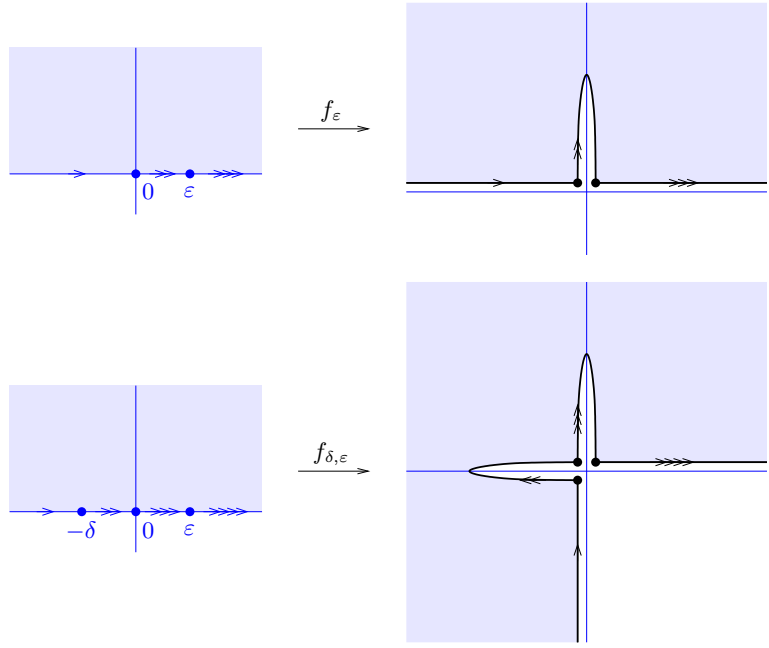
They show that $f_\eps$ has a ``spike'' in the
direction $i\R_+$ which disappears as $\eps\to 0$, and
$f_{\delta,\eps}$ has two ``spikes'' in the directions $\R_-$
resp.~$i\R_+$ which disappear as $\delta$ resp.~$\eps$ approaches
zero. Based on these models, the notion of a ``spike'' will be
formalized in Section~\ref{sec:string-ref}.

In the following section, functions with two spikes will appear in the
following local model. Consider the $1$-parameter family of functions
$f_a:Q^+\to\C$,
$$
   f_a(z)=i(az-z^3),\qquad a\in\R.
$$
They map $(\R_+,i\R_+)$ to $(i\R,\R)$ and thus correspond to
Case 2 in Section~\ref{ss:series}. Via the identifications in that
section, $f_a$ induces functions
\begin{align*}
   f_-(a,t) &:= f_a(-it) = at+t^3, \qquad t\neq 0,\cr
   f_+(a,t) &:= (-i)f_a(t) = at-t^3, \qquad t\geq 0,
\end{align*}
which fit together to a $C^2$ (though not $C^3$) function
$$
   f(a,t) = at - \sgn(t)t^3,\qquad t\in\R.
$$
In Case 1, one considers the functions $f_a(z)=-az+z^3$
mapping $(\R_+,i\R_+)$ to $(\R,i\R)$. Here the induced functions
\begin{align*}
   f_-(a,t) &:= (-i)f_a(-it) = at+t^3, \qquad t\neq 0,\cr
   f_+(a,t) &:= f_a(t) = -at+t^3, \qquad t\geq 0
\end{align*}
do not fit together to a $C^1$ function, but when we replace $f_+$ by
$-f_+$ they fit together to the function $f(a,t)$ above.

\section{String homology in arbitrary degree}\label{sec:string-ref}

\subsection{Broken strings}\label{ss:brokenstring}
Let $K$ be a framed oriented knot in some oriented 3-manifold $Q$. Fix a tubular
neighborhood $N$ of $K$ and a diffeomorphism $N\cong S^1\times D^2$.

Fix an 
integer $m\geq 3$ and a base point $x_0\in\p N$. We also fix an $m$-jet of a curve passing through $x_0$ in $N$. Using the diffeomorphism $N \cong S^1 \times D^2$, this is equivalent to specifying suitable vectors $v_0^{(k)}\in\R^3$, $1\leq
k\leq m$. The following definition refines the one given in
Section~\ref{sec:string}, which corresponds to the case $m=1$.

\begin{definition}\label{def:string}
A {\em broken (closed) string with $2\ell$} switches on $K$ is a tuple
$s=(a_1,\dots,a_{2\ell+1};s_1,\dots,s_{2\ell+1})$ consisting of real
numbers $0=a_0<a_1<\dots<a_{2\ell+1}$ and $C^m$-maps
$$
   s_{2i+1}:[a_{2i},a_{2i+1}]\to N,\quad s_{2i}:[a_{2i-1},a_{2i}]\to
   Q
$$
satisfying the following matching conditions at the end points $a_i$:

(i) $s_1(0)=s_{2\ell+1}(a_{2\ell+1})=x_0$ and
$s_1^{(k)}(0)=s_{2\ell+1}^{(k)}(a_{2\ell+1})=v_0^{(k)}$ for $1\leq
k\leq m$.

(ii) For $i=1,\dots,\ell$,
$$
   s_{2i}(a_{2i}) = s_{2i+1}(a_{2i})\in K,\qquad
   s_{2i-1}(a_{2i-1})=s_{2i}(a_{2i-1})\in K.
$$
(iii) Denote by $\sigma_i$ the $D^2$-component of $s_i$ near its end
points. Then for $i=1,\dots,\ell$ and $1\leq k\leq m/2$ (for the left
hand side) resp.~$1\leq k\leq (m+1)/2$ (for the right hand side)
\begin{gather*}
   \sigma_{2i}^{(2k)}(a_{2i}) = \sigma_{2i+1}^{(2k)}(a_{2i})=0,\qquad
   \sigma_{2i}^{(2k-1)}(a_{2i}) = (-1)^{k}\sigma_{2i+1}^{(2k-1)}(a_{2i}),\cr
   \sigma_{2i-1}^{(2k)}(a_{2i-1})=\sigma_{2i}^{(2k)}(a_{2i-1})=0,\qquad
   \sigma_{2i-1}^{(2k-1)}(a_{2i-1})=(-1)^{k+1}\sigma_{2i}^{(2k-1)}(a_{2i-1}).
\end{gather*}
\end{definition}
We will refer to the $s_{2i}$ and $s_{2i+1}$ as
{\em Q-strings} and {\em N-strings}, respectively.
A typical picture of a broken string is shown in
Figure~\ref{fig:1} on page~\pageref{fig:1}.
Conditions (i) and (ii) in Definition~\ref{def:string} mean that the
$s_i$ fit together to a continuous loop $s:[0,a_{2\ell+1}]\to Q$
with end points at $x_0$ (which fit together in $C^m$). 

Condition (iii) is motivated as follows: In Section~\ref{S:mdlisp} below, 
we consider almost complex structures 
$J$ on $T^*Q$ which are particularly well adapted to the immersed Lagrangian 
submanifold $Q \cup L_K \subset T^*Q$. For such a $J$, Lemma~\ref{l:knotnbhd} 
then provides a holomorphic embedding of a neighborhood $\OO$ of 
$S^1 \times \{0\} \subset \C \times \C^2$ onto a neighborhood of 
$K \subset T^*Q$ mapping $\OO \cap (S^1 \times i\R^2)$ to $Q$ and  
$\OO \cap (S^1 \times \R^2)$ to $L_K$. Condition (iii) requires that 
the normal component $\sigma$ of $s$ at the switching points $a_i$ behaves 
like the boundary values of a holomorphic disk with boundary on $Q \cup L_K$ 
when projected to $\C^2$ in these coordinates near $K$. 

To see this, let us reformulate condition (iii). As in
Section~\ref{ss:series}, to a complex valued polynomial
$p(t)=\sum_{k=1}^mp_kt^k$, $p_k\in\C^2$, we associate its reflection
$$
   r_*p(t) = (-i)p(-it) = \sum_{k=1}^m(-i)^{k+1}p_kt^k.
$$
Then two {\em real valued} polynomials $p(t)=\sum_{k=1}^mp_kt^k$ and
$q(t)=\sum_{k=1}^mq_kt^k$, $p_k,q_k\in\R^2$, satisfy
$r_*p=q$ if and only if for $1\leq k\leq m/2$ (on the left hand side)
resp.~$1\leq k\leq (m+1)/2$ (on the right hand side)
$$
   p_{2k}=q_{2k}=0 \text{ and } p_{2k-1}=(-1)^kq_{2k-1}.
$$
So in terms of the {\em normal Taylor polynomials} at the switching points
$$
   T^m\sigma_i(a_{i-1})(t) :=
   \sum_{k=1}^m\frac{\sigma_i^{(k)}(a_{i-1})}{k!}t^k, \qquad
   T^m\sigma_i(a_i)(t) := \sum_{k=1}^m\frac{\sigma_i^{(k)}(a_i)}{k!}t^k,
$$
condition (iii) is equivalent to the conditions
$$
   T^m\sigma_{2i}(a_{2i}) = r_*T^m\sigma_{2i+1}(a_{2i}),\qquad
   T^m\sigma_{2i-1}(a_{2i-1})=-r_*T^m\sigma_{2i}(a_{2i-1}).
$$
These are precisely the conditions in Section~\ref{ss:series}
describing the boundary behavior of holomorphic disks at a corner
going from the imaginary to the real axis (Case 1, corresponding to a
switch from $Q$ to $N$), resp.~from the real to the imaginary axis
(Case 2, corresponding to a switch from $N$ to $Q$).

\begin{remark}
(a) The case $m=3$ suffices for the purposes of this paper. In fact,
for $0$- and $1$-parametric families of strings we only need the
conditions on the first derivatives (the case $m=1$ considered in
Section~\ref{sec:string}), while for $2$-parametric families we also
need the conditions on the second and third derivatives). Explicitly,
condition (iii) for $m=3$ reads
\begin{equation}
\begin{gathered}\label{eq:m=3}
   \sigma_{2i}'(a_{2i}) = -\sigma_{2i+1}'(a_{2i}),\qquad
   \sigma_{2i-1}'(a_{2i-1})=\sigma_{2i}'(a_{2i-1}), \cr
   \sigma_{2i}''(a_{2i}) = \sigma_{2i+1}''(a_{2i}) =
   \sigma_{2i-1}''(a_{2i-1})=\sigma_{2i}''(a_{2i-1})=0, \cr
   \sigma_{2i}'''(a_{2i}) = \sigma_{2i+1}'''(a_{2i}),\qquad
   \sigma_{2i-1}'''(a_{2i-1})=-\sigma_{2i}'''(a_{2i-1}).
\end{gathered}
\end{equation}
(b) In Definition~\ref{def:string} one could add the condition that
all derivatives of the tangent components agree at switches (as it is
the case for boundaries of holomorphic disks). However, we will not
need such a condition and thus chose not to include it. Similarly, one
could have required all the $s_j$ to be $C^\infty$ rather than $C^m$.
\end{remark}

We denote by $\Sigma^\ell$ the space of broken strings with $2\ell$
switches. We make it a Banach manifold by equipping it
with the topology of $\R$ on the $a_j$ and the $C^m$-topology on the
$s_j$.
It comes with interior evaluation maps
$$
   \ev_i:(0,1)\times \Sigma^\ell\to Q \text{ resp.~}N,\qquad (t,s)\mapsto
   s_i\bigl((1-t)a_{i-1}+ta_i\bigr)
$$
and corner evaluation maps
$$
   T_i:\Sigma^\ell\to (\R^2)^{\lfloor \frac{m+1}{2}\rfloor},\qquad s\mapsto T^m\sigma_i(a_i) \cong
   \Bigl(\sigma_i^{(2k-1)}(a_i)\Bigr)_{1\leq k\leq \lfloor \frac{m+1}{2}\rfloor}.
$$
Moreover, concatenation at the base point $x_0$ yields a smooth map
$$
   \Sigma^\ell \times \Sigma^{\ell'}\mapsto \Sigma^{\ell+\ell'}.
$$

\subsection{Generic chains of broken strings}\label{ss:chains}

Now we define the generators of the string chain complex in degrees
$d\in\{0,1,2\}$. Set $\Delta_0:=\{0\}$ and let
$\Delta_d=\{(\lambda_1,\dots,\lambda_d)\in\R^d\mid \lambda_i\geq
0,\lambda_1+\dots+\lambda_d\leq 1\}$ denote the $d$-dimensional
standard simplex for $d\geq 1$. It is stratified by the sets where
some of the inequalities are equalities. Fix $m\geq 3$ as in the
previous subsection.

\begin{definition}\label{def:generic-chain}
A {\em generic $d$-chain} in $\Sigma^\ell$ is a smooth map $S:\Delta_d\to
\Sigma^\ell$ such that the maps $\ev_i\circ S:(0,1)\times \Delta_d\to Q$ and
$T_i\circ S:\Delta_d\to (\R^2)^{\lfloor \frac{m+1}{2} \rfloor}$ are jointly transverse to $K$
resp.~jet-transverse to $0$ (on all strata of $\Delta_d$).
\end{definition}

Let us spell out what this means for $m=3$ in the
cases $d=0,1,2$.

\underline{$d=0$}: A generic $0$-chain is a broken string $s=(s_1,\dots
s_{2\ell+1})$ such that

(0a) $\dot\sigma_i(a_i)\neq 0$ for all $i$;

(0b) $s_i$ intersects $K$ only at its end points.
\smallskip

\underline{$d=1$}: A generic 1-chain of broken strings is a smooth map
$$
   S:[0,1]\to\Sigma^\ell,\qquad \lambda\mapsto
   s^\lambda=(s_1^\lambda,\dots s_{2\ell+1}^\lambda)
$$
such that

(1a) $s^0$ and $s^1$ are generic strings;

(1b) $\dot\sigma_i^\lambda(a_i^\lambda)\neq 0$ for all $i,\lambda$;

(1c) for each $i$ the map
$$
   (0,1)\times(0,1)\to Q \text{ resp.~}N,\qquad (t,\lambda)\to
   s_i^\lambda\bigl((1-t)a_{i-1}^\lambda+ta_i^\lambda\bigr)
$$
meets $K$ transversely in finitely many points $(t_a,\lambda_a)$. Moreover, distinct such intersections (even for different $i$) appear at distinct parameter values $\lambda_a$.
\smallskip

\underline{$d=2$}: A generic 2-chain of broken strings is a smooth map
$$
   S:\Delta_2\to\Sigma^\ell,\qquad \lambda\mapsto
   s^\lambda=(s_1^\lambda,\dots s_{2\ell+1}^\lambda)
$$
such that

(2a) the $s^\lambda$ at vertices $\lambda\in \Delta_2$ are generic strings;

(2b) the restrictions of $S$ to edges of $\Delta_2$ are generic 1-chains;

(2c) for each $i$ the map
$$
   (0,1)\times\inn \Delta_2\to Q \text{ resp.~}N,\qquad (t,\lambda)\to
   s_i^\lambda\bigl((1-t)a_{i-1}^\lambda+ta_i^\lambda\bigr)
$$
is transverse to $K$; moreover, we assume that the projection of the
preimage of $K$ to $\Delta_2$ is an immersed submanifold $D_i \subset \Delta_2$ with
transverse double points;

(2d) for all $i,j$ the submanifolds $D_i,D_j\subset \Delta_2$ from
(2c) meet transversely in finitely many points;

(2e) for each $i$ the map
$$
   \inn \Delta_2\to\R^2,\qquad \lambda\mapsto
   \dot\sigma_i^\lambda(a_i^\lambda)
$$
meets $0$ transversely in finitely many points satisfying
$(\sigma_i^\lambda)^{(3)}(a_i^\lambda)\neq 0$; moreover, these points
do not meet the $D_j$.
\smallskip

We will see in the next subsection that the points in (2e) are limit
points of both $D_i$ and $D_{i+1}$.

\subsection{String operations}\label{ss:string-op}

\begin{figure}
\labellist
\small\hair 2pt
\pinlabel $q_1$ at 72 28
\pinlabel ${\color{blue} p_1}$ at 193 145
\endlabellist
\centering
\includegraphics[width=0.5\textwidth]{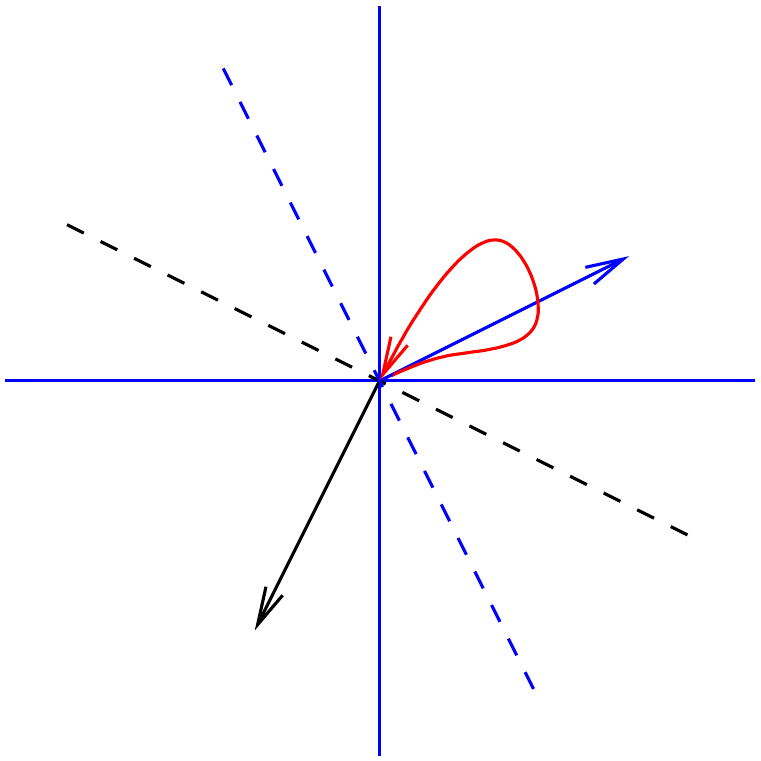}
\caption{
A spike with ends $(p,q)$.
}
\label{fig:spike}
\end{figure}

Now we define the relevant operations on generic chains of broken
strings. Let $\p$ denote the singular boundary operator, thus
$$
   \p\{s^\lambda\} := s^1-s^0,\qquad \p S := S|_{\p\Delta_2}
$$
for $1$-resp.~$2$-chains. For the definition of string coproducts we
need the following
\begin{definition}\label{def:spike}
Let $p(t)=\sum_{k=1}^mp_kt^k$ and $q(t)=\sum_{k=1}^mq_kt^k$,
$p_k,q_k\in\R^2$, be real polynomials with $\la p_1,q_1\ra<0$. A {\em
spike} with ends $(p,q)$ is a $C^m$-function $f:[a,b]\to D^2$ with the
following properties (see Figure~\ref{fig:spike}):

(S1) the Taylor polynomials to order $m$ of $f$ at $a$ resp.~$b$ agree
with $p$ resp.~$q$;

(S2) $\la f(t),p_1\ra>0$ and $\la f(t),q_1\ra<0$ for all
$t\in(a,b)$.
\end{definition}
\begin{remark}\label{rem:spikes-convex}
Note that the spikes with fixed ends $(p,q)$ and fixed or varying
$a<b$ form a convex (hence contractible) space.
\end{remark}

We choose a family of {\em preferred spikes
$\s_{p,q}:[0,1]\to D^2$} for all $(p,q)$ depending smoothly (with respect to
the $C^m$-topology) on the coefficients of $p$ and $q$.
Now we are ready to define the {\em string coproducts}
$\delta_N,\delta_Q$ on generic $d$-chains for $d\leq 2$.

\underline{$d=0$}: On $0$-chains set $\delta_N=\delta_Q=0$.

\underline{$d=1$}: For a $1$-chain
$\{s^\lambda\}_{\lambda\in[0,1]}$ let $(\lambda^j,b^j)$ be the finitely
many values for which $s_{2i}^{\lambda^j}(b^j)\in K$ for some $i=i(j)$.
Set
$$
   \delta_Q\{s^\lambda\} :=
   \sum_{j}\eps^j\Bigl(s_1^{\lambda^j},\dots,
   s^{\lambda^j}_{2i}|_{[a_{2i-1},b^j]},\s^j,
   \hat s^{\lambda^j}_{2i}|_{[b^j,a_{2i}]}, \dots, \hat
   s^{\lambda^j}_{2\ell+1}\Bigr),
$$
where $\s^j=\s(\cdot-b^j):[b^j,b^j+1]\to N$ is a shift of the preferred
spike $\s$ with ends $\bigl(r_*T^m\sigma_{2i}^{\lambda^j}(b_j),
T^m\sigma_{2i}^{\lambda^j}(b_j)\bigr)$ in the normal directions, with
constant value $s_{2i}^{\lambda^j}(b^j)$ along $K$.
The hat means shift by $1$ in the argument, and $\eps^j=\pm 1$
are the signs defined in Figure~\ref{fig:2}.
Loosely speaking, $\delta_Q$ inserts an {\em $N$-spike}
at all points where some $Q$-string meets $K$. The operation
$\delta_N$ is defined analogously, inserting a {\em $Q$-spike} where an
$N$-string meets $K$. Note that
by Definition~\ref{def:spike} the spikes stay in $N$ and meet $K$
only at their end points.
\comment{
\begin{figure}[h]
\begin{center}
\epsfbox{cord_02.eps}
\caption{The signs in the definition of $\delta_N$ and $\delta_Q$.}
\label{fig:2}
\end{center}
\end{figure}
}

\begin{figure}
\labellist
\small\hair 2pt
\pinlabel $U^j$ at 57 306
\pinlabel $\lambda^j$ at 37 276
\pinlabel $D_\delta$ at 307 292
\pinlabel $\R^2$ at 31 93
\pinlabel $D_\eps$ at 323 101
\pinlabel $\times (-\gamma,\gamma)$ at 100 265
\pinlabel $\times (-\gamma,\gamma)$ at 347 265
\pinlabel $\times (-\gamma,\gamma)$ at 361 65
\pinlabel $\psi\times\id$ at 199 286
\pinlabel $\cong$ at 199 257
\pinlabel $\sigma$ at 41 182
\pinlabel $\tilde{f}$ at 183 192
\pinlabel $\cong$ at 319 174
\pinlabel $\Phi\times\id$ at 362 174
\pinlabel $\Phi\circ\Psi$ at 198 67
\pinlabel $f$ at 181 4
\endlabellist
\centering
\includegraphics[width=250pt]{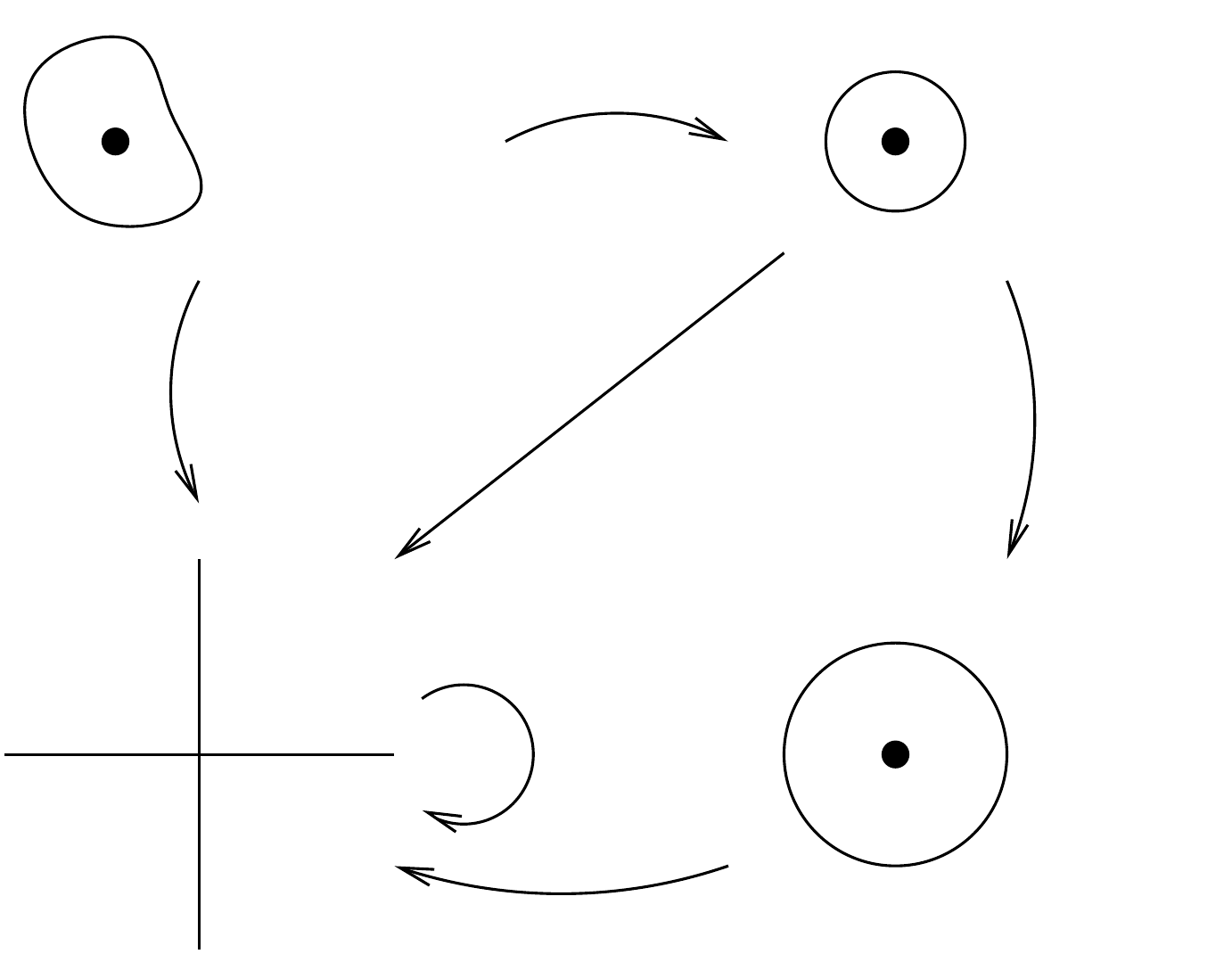}

\caption{
Construction of the map $f$.
}
\label{fig:f}
\end{figure}

\underline{$d=2$}: Finally, consider a generic $2$-chain $S:\Delta_2\to\Sigma^\ell$. Let
$\lambda^j\in\inn \Delta_2$ be the finitely many points where
$\dot\sigma_i^{\lambda^j}(a_i^{\lambda^j})=0$ for some $i=i(j)$.
For the following construction see Figure~\ref{fig:f}.
Let $\delta>0$ be
a number $\leq 1$ such that the map
$\psi:\lambda\mapsto \dot\sigma_i^{\lambda}(a_i^{\lambda})$ is a
diffeomorphism from a neighborhood $U^j$ of $\lambda^j$ onto the
$\delta$-disk $D_\delta\subset\R^2$ (such $\delta$ exists by condition (2e)
in Section~\ref{ss:chains}). We choose $U^j$ so small that it contains
no other $\lambda^i$.
Let $\gamma>0$ be
a number $\leq 1$ such that
$|\sigma^\lambda(t+a_i^\lambda)|\leq 1$ for all $|t|\leq\gamma$.
Consider the function $\sigma:U^j\times(-\gamma,\gamma)\to\R^2$ defined by
$$
   \sigma(\lambda,t) :=
   \begin{cases}
   \sigma_i^\lambda(t+a_i^\lambda) :& t<0, \cr
      -\sigma_{i+1}^\lambda(t+a_i^\lambda) :& t\geq 0 \text{ if $i$ is
      even}, \cr
      \sigma_{i+1}^\lambda(t+a_i^\lambda) :& t\geq 0 \text{ if
      $i$ is odd}.
   \end{cases}
$$
According to conditions~\eqref{eq:m=3}, the function $\sigma(\lambda,t)$ is
smooth in $\lambda$ and of class $C^2$ but not $C^3$ in $t$. Define
the function
$$
   \tilde f:D_\delta\times (-\gamma,\gamma)\to\R^2,\qquad (a,b,t)\mapsto
   \sigma\bigl(\psi^{-1}(a,b),t\bigr).
$$
By construction we have $\frac{\p\tilde f}{\p t}(a,b,0)=(a,b)$ for all
$(a,b)$. Moreover, by condition (2e) in Section~\ref{ss:chains} we
have $v^j:=(\sigma_i^{\lambda^j})^{(3)} (a_i^{\lambda^j})\neq 0$. Let
$\Psi$ be the
rotation of $\R^2$ which maps $v^j$ onto a vector
$(\mu,0)$ with $\mu>0$, let $\Phi:\R^2\to\R^2$ be
multiplication by $6/\mu$, and set $\eps:=6\delta/\mu$. Then the map
\begin{equation}\label{eq:f}
   f:=\Phi\circ\Psi\circ\tilde f\circ(\Phi^{-1}\times\id):D_\eps\times
   (-\gamma,\gamma)\to\R^2
\end{equation}
satisfies
\begin{gather*}
   f(a,b,0)=(0,0),\quad
   \frac{\p f}{\p t}(a,b,0)=(a,b),\cr
   \frac{\p^2 f}{\p t^2}(a,b,0)=(0,0),\quad
   \frac{\p f^3}{\p t^3}(0,0,0)=\pm(6,0)
\end{gather*}
for all $(a,b)$. Here the map $f$ is $C^2$ but not $C^3$, and the
statement about the third derivative $\frac{\p f^3}{\p t^3}(0,0,0)$
means that it equals $+(6,0)$ from the left and $-(6,0)$ from the
right. Therefore, $f$ has a Taylor expansion (again considered for
$t\leq 0$ and $t\geq 0$ separately)
\begin{equation}\label{eq:taylor}
   f(a,b,t) = \Bigl(at-\sgn(t)t^3,bt\Bigr) +
   O(|a||t|^3+|b||t|^3+|t|^4).
\end{equation}
Here to simplify notation we tacitly assume that the restrictions of
$f$ to $t\leq 0$ and $t\geq 0$ are $C^4$ rather than $C^3$. The
following argument carries over to the $C^3$ case if we replace
throughout $O(|t|^4)$ by $o(|t|^3)$.

Consider first the model case without higher order terms, i.e.~the
function
$$
   f^0(a,b,t) = \Bigl(at-\sgn(t)t^3,bt\Bigr).
$$
Note that the first component $at-\sgn(t)t^3$ of $f^0$ is exactly the
function that we encountered at the end of Section~\ref{ss:spikes}.
The zero set of $f^0$ consists of three strata
$$
   \{t=0\}\cup \{b=0,a>0,t=\sqrt{a}>0\}\cup
   \{b=0,a<0,t=-\sqrt{-a}<0\}.
$$
For $a>0$ and $b=0$ the function
$$
   f_a:[0,\sqrt{a}]\to\R^2,\qquad t\mapsto f^0(a,0,t)=(at-t^3,0)
$$
is a spike with ends satisfying
$$
   f_a'(0)=(a,0),\quad f_a'(\sqrt{a})=(-2a,0),\quad
   f_a'''(0)=f_a'''(\sqrt{a})=-6.
$$
Similarly, for $a<0$ the function
$$
   f_a:[-\sqrt{-a},0]\to\R^2,\qquad t\mapsto f^0(a,0,t)=(at+t^3,0)
$$
is a spike with ends satisfying
$$
   f_a'(0)=(a,0),\quad f_a'(-\sqrt{-a})=(-2a,0),\quad
   f_a'''(0)=f_a'''(-\sqrt{-a})=+6.
$$
So two families of spikes pointing in the same directions come
together from both sides along the $a$-axis $\{b=0\}$ and vanish at
$(a,b)=(0,0)$, see Figure~\ref{fig:spikes-vanishing}.
The following lemma states that this qualitative picture persists in
the presence of higher order terms.

\begin{figure}
\labellist
\small\hair 2pt
\pinlabel $\R\times\{0\}$ at 289 385
\pinlabel $-\eps$ at 165 171
\pinlabel $\eps$ at 407 171
\pinlabel $a$ at 522 182
\pinlabel $b$ at 369 345
\endlabellist
\centering
\includegraphics[width=0.8\textwidth]{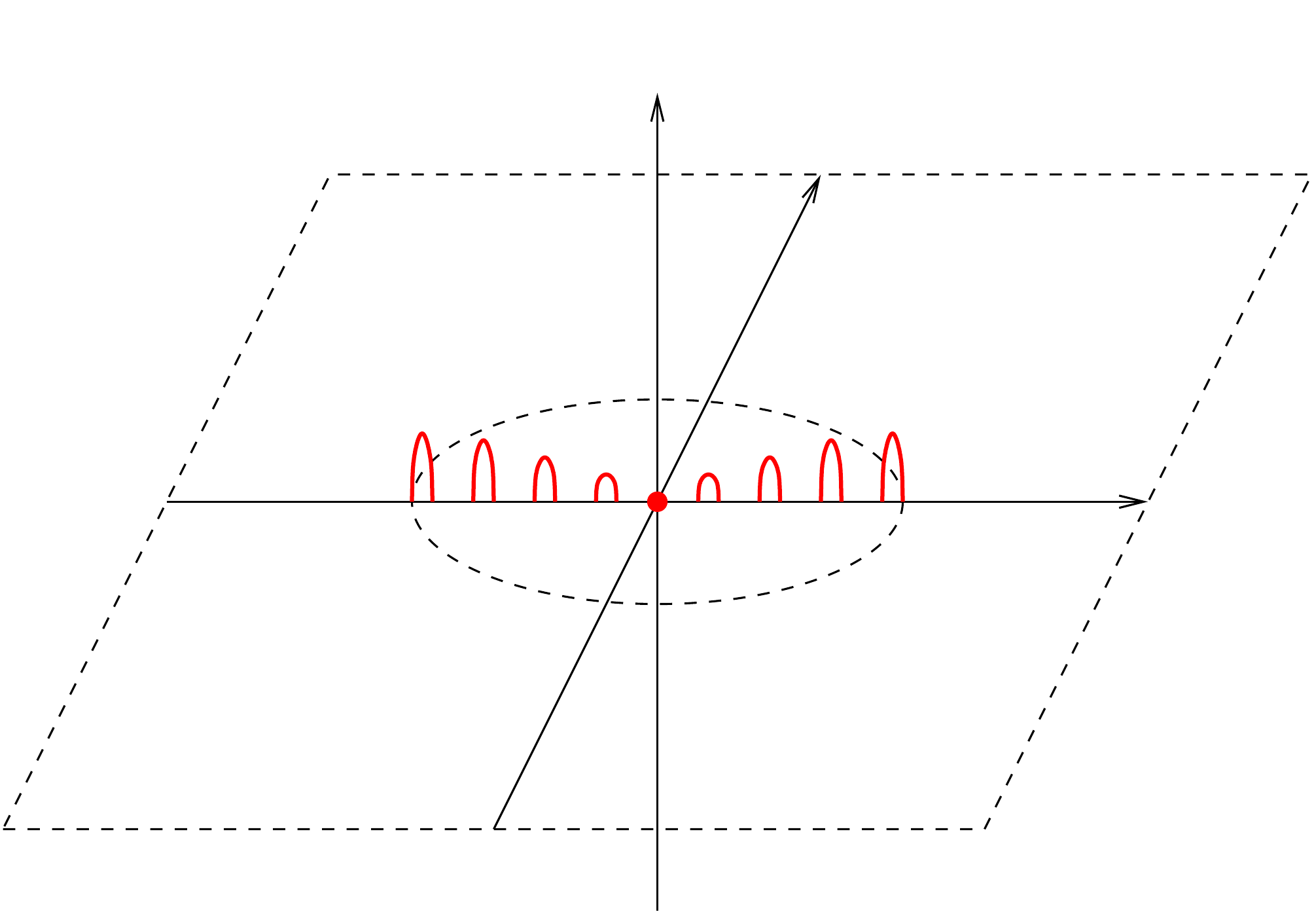}
\caption{
Two families of spikes vanishing at the origin.
}
\label{fig:spikes-vanishing}
\end{figure}

\begin{lemma}\label{lem:spike}
Let $f:D_\eps\times(-\gamma,\gamma)\to\R^2$ be a function
satisfying~\eqref{eq:taylor}. Then for $\eps$ and $\gamma$
sufficiently small there exist smooth functions $\beta(a,t)$ and
$\tau(a)$ for $a\in[-\eps,\eps]\setminus\{0\}$ such
that with $\beta(a):=\beta\bigl(a,\tau(a)\bigr)$ the zero set of $f$
in $D_\eps\times(-\gamma,\gamma)$ consists of three strata
$$
   \{t=0\}\cup \{b=\beta(a),a>0,t=\tau(a)>0\}\cup
   \{b=\beta(a),a<0,t=\tau(a)<0\}.
$$
The functions $\beta,\tau$ satisfy the estimates
$$
   \beta(a,t) = O(|a|t^2+|t|^3),\quad \tau(a)^2-a = O(a^{3/2}),\quad
   \beta(a) = O(a^{3/2}).
$$
Moreover, the functions
\begin{gather*}
   f_a:[0,\tau(a)]\to\R^2,\qquad t\mapsto
   f\bigl(a,\beta(a),t\bigr),\quad a>0,\cr
   f_a:[\tau(a),0]\to\R^2,\qquad t\mapsto
   f\bigl(a,\beta(a),t\bigr),\quad a<0
\end{gather*}
are spikes with ends satisfying
\begin{gather*}
   f_a'(0)=(a,0)+O(a^{3/2}),\quad f_a'\bigl(\tau(a)\bigr)=(-2a,0)+O(a^{3/2}).
\end{gather*}
\end{lemma}

\begin{proof}
We consider the case $a,t>0$, the case $a,t<0$ being
analogous. Setting the second component in~\eqref{eq:taylor} to zero
and dividing by $t$ yields $b=O(at^2+bt^2+t^3)$, which for
$t$ sufficiently small can be solved for $b=\beta(a,t)$ satisfying the
estimate $\beta(a,t) = O(at^2+t^3)$. Inserting this into the first
component in~\eqref{eq:taylor}, setting it to zero and dividing by $t$
yields
$$
   a-t^2=O\Bigl(at^2+\beta(a,t)t^2+t^3\Bigr) = O(at^2+t^3),
$$
which for $(a,t)$ sufficiently small can be solved for $t=\tau(a)$
satisfying the estimate $\tau(a)^2-a = O(a^{3/2})$. Inserting
$t=\tau(a)$ in $\beta(a,t)$ we obtain the estimate $\beta(a) =
O(a^{3/2})$. This proves the first assertions.

Now consider the function $f_a(t)=f\bigl(a,\beta(a),t\bigr)$ for
$t\in[0,\tau(a)]$ and $a>0$. Inserting $\beta(a) = O(a^{3/2})$ we find
\begin{align*}
   f_a(t) &= (at-t^3,0) + O\Bigl(\beta(a)t+at^3+\beta(a)t^3+t^4\Bigr)\cr
   &= (at-t^3,0) + O(a^{3/2}t+at^3+t^4),
\end{align*}
and therefore $f_a'(t)=(a-3t^2,0)+O(a^{3/2}+at^2+t^3)$. This immediately
gives $f_a'(0)=(a,0)+O(a^{3/2})$ and, using $\tau(a)=O(a^{1/2})$, also
$f_a'\bigl(\tau(a)\bigr)=(-2a,0)+O(a^{3/2})$.

It remains to prove that the functions $f_a$ are spikes in the sense
of Definition~\ref{def:spike}. Write in
components $f=(f^1,f^2)$ and $f_a=(f^1_a,f^2_a)$ and abbreviate
$\tau:=\tau(a)$. We claim that there exist constants $\delta,D>0$
independent of $a,t$ such
that for all $t\in[0,\tau]$ we have
$$
   f^1_a(t)\geq2\delta t(\tau^2-t^2),\qquad |f_a^2(t)|\leq
   Dt(a+t)(\tau-t).
$$
For the first estimate, note that
$$
   \frac{1}{t}f^1_a(t) = a-t^2+O(a^{3/2}+at^2+t^3),
$$
viewed as a function of $t^2$, has transversely cut out zero locus
$t=\tau$ and is therefore $\geq2\delta(\tau^2-t^2)$ for
some $\delta>0$. The second estimate holds because
$$
   \frac{1}{t}f^2_a(t) = O(a^{3/2}+at^2+t^3)
$$
vanishes at $t=\tau$, so $|f^2_a(t)|\leq Dt(a+t)(\tau-t)$ for some
constant $D$. Using these
estimates as well as $f_a'(0)=(a,0)+O(a^{3/2})$ and $\tau=O(a^{1/2})$
we compute with a generic constant $C$ (independent of $a,t$):
\begin{align*}
   \la f_a'(0),f_a(t)\ra &= \bigl(a+O(a^{3/2})\bigr)f_a^1(t) + \la O(a^{3/2}),f_a^2(t)\ra \cr
   &\geq a\delta t(\tau^2-t^2) - Ca^{3/2}t(\tau-t)(a+t) \cr
   &= at(\tau-t)\Bigl(\delta(\tau+t)-Ca^{1/2}(a+t)\Bigr) \cr
   &\geq a^{3/2}t(\tau-t)\bigl(\delta-C(a+t)\bigr),
\end{align*}
which is positive for $0<t<\tau$ and $a$ sufficiently small. An
analogous computation, using $f_a'(\tau)=(-2a,0)+O(a^{3/2})$, shows
$\la f_a'(\tau),f_a(t)\ra<0$, so $f_a$ is a spike.
\end{proof}

\begin{remark}\label{rem:spike-interpol}
The spikes from Lemma~\ref{lem:spike} can be connected to the spike
of the model function $f^0$ without higher order terms by rescaling:
For $s\in(0,1]$ set
\begin{align*}
   f^s(a,b,t) &:= \frac{1}{s^3}f(s^2a,s^2b,st) \cr
   &= \Bigl(at-\sgn(t)t^3,bt\Bigr) + sO(|a||t|^3+|b||t|^3+|t|^4) \cr
   & \overset{s\to 0}{\longrightarrow} \Bigl(at-\sgn(t)t^3,bt\Bigr).
\end{align*}
Thus for $|a|\leq \eps$ the corresponding family of spikes
$(f^s_a)_{s\in[0,1]}$ connects $f_a$ to the spike $f_a^0$.
\end{remark}

Now we return to the points $\lambda^j\in U^j$ and the corresponding maps
$f:D_\eps\times(-\gamma,\gamma)\to\R^2$ defined by~\eqref{eq:f}.
After shrinking $\eps,\gamma>0$ and replacing $U^j$ by
$(\Phi\circ\psi)^{-1}(D_\eps)\subset \Delta_2$ (where $\psi,\Phi$ are
the maps defined above), we may assume that $\eps,\gamma$ satisfy the
smallness requirement in Lemma~\ref{lem:spike} for each $j$.
Define
\begin{gather*}
   M_{\tilde\delta_Q} := \bigcup_i(\ev_{2i}\circ S)^{-1}(K)
   \setminus\bigcup_j\bigl(U^j\times(0,1)\bigr),\cr
   M_{\tilde\delta_N} := \bigcup_i(\ev_{2i-1}\circ S)^{-1}(K)
   \setminus\bigcup_j\bigl(U^j\times(0,1)\bigr).
\end{gather*}
By construction, $M_{\tilde\delta_Q}$ and $M_{\tilde\delta_N}$ are
1-dimensional submanifolds with boundary of $\Delta_2\times(0,1)$.
Define $\tilde\delta_QS:M_{\tilde\delta_Q}\to\Sigma^{\ell+1}$ by
inserting preferred $N$-spikes at all points where some $Q$-string
meets $K$ (via the same formula as the one above for $\delta_Q$ on
$1$-chains), and similarly for $\tilde\delta_NS$. See Figure~\ref{fig:string-relations}.

Note that the boundary $\p M_{\tilde\delta_Q}$ consists of
intersections with $\p\Delta_2$ and with the boundaries $\p U^j$. Thus each
$j$ contributes a unique point $\lambda^j_Q$ to $\p
M_{\tilde\delta_Q}$, which
corresponds in the above coordinates to $a=+\eps$ if the
associated index $i$ is odd and to $a=-\eps$ if $i$ is even. Similarly,
each $j$ contributes a unique point $\lambda^j_N$ to $\p
M_{\tilde\delta_N}$ which corresponds in the above coordinates to
$a=-\eps$ if the associated index $i$ is odd and to $a=+\eps$ if $i$ is
even. The broken strings
$\tilde\delta_QS(\lambda^j_Q)$ and $\tilde\delta_NS(\lambda^j_N)$ are
$C^m$-close for $|t|\geq\gamma$, and by Lemma~\ref{lem:spike} for
$|t|<\gamma$ they both have a $Q$-spike and an $N$-spike {\em with the
  same first derivatives at the ends}.
So, using convexity of the space of spikes with fixed ends
(Remark~\ref{rem:spikes-convex}, see also
Remark~\ref{rem:spike-interpol}), we can
connect them by a short $1$-chain $S^j:[0,1]\to\Sigma^{\ell+1}$ with
spikes in $[-\gamma,\gamma]$ (which we regard as $Q$-spikes.)

We define $\delta_QS:M_{\delta_Q}\to\Sigma^{\ell+1}$ to be $\tilde\delta_QS$
together with the $1$-chains $S^j$, and we set
$\delta_NS:=\tilde\delta_NS:M_{\delta_N}=M_{\tilde\delta_N}\to\Sigma^{\ell+1}$.
Recall that the $1$-dimensional submanifold
$M_{\tilde\delta_Q}\subset\Delta_2\times(0,1)$ is the union of the
transversely cut out preimages of $K$ under the evaluation maps
$\ev_{2i}\circ S:\Delta_2\times(0,1)\to Q$. Hence the coorientation of
$K\subset Q$ and the orientation of $\Delta_2\times(0,1)$ induce an
orientation on $M_{\tilde\delta_Q}$, and similarly for
$M_{\tilde\delta_N}$. (The induced orientations depend on orientation
conventions which will be fixed in the proof of
Proposition~\ref{prop:string-relations} below.)
We parametrize each connected component of $M_{\tilde\delta_Q}$ and
$M_{\tilde\delta_N}$ by the interval $\Delta_1=[0,1]$ proportionally
to arclength (with respect to the standard metric on
$\Delta_2\times(0,1)$ and in the direction of the orientation, where
for components diffeomorphic to $S^1$ we choose an arbitrary initial
point). So we can view $\delta_QS:M_{\delta_Q}\to\Sigma^{\ell+1}$ and
$\delta_NS:M_{\delta_N}\to\Sigma^{\ell+1}$ as generic $1$-chains,
where we orient the $1$-chains $S^j$ such that the points
$\tilde\delta_QS(\lambda^j_Q)$ appear with opposite signs in the
boundary of $S^j$ and $M_{\tilde\delta_Q}$.

\begin{prop}\label{prop:string-relations}
On generic chains of degree $2$, the operations $\p$, $\delta_Q$ and
$\delta_N$ satisfy the relations
\begin{gather*}
   \p^2=\delta_Q^2=\delta_N^2=\delta_Q\delta_N+\delta_N\delta_Q=0 ,\cr
   \p\delta_Q+\delta_Q\p + \p\delta_N+\delta_N\p = 0.
\end{gather*}
In particular, these relations imply
$$
   (\p+\delta_Q+\delta_N)^2=0.
$$
\end{prop}

\begin{proof}
Consider a generic $2$-chain $S:\Delta_2\to\Sigma^\ell$. We continue
to use the notation above and denote by
$\pi:\Delta_2\times(0,1)\to\Delta_2$ the projection. The relation
$\p^2S=0$ is clear. Points in $\delta_Q^2S$ correspond to
transverse self-intersections of $\pi(M_{\tilde\delta_Q})$, so each
point appears twice with opposite signs, hence $\delta_Q^2S=0$ and
similarly $\delta_N^2S=0$. Points in
$\delta_Q\delta_NS+\delta_N\delta_QS$ correspond to transverse
intersections of $\pi(M_{\tilde\delta_Q})$ and
$\pi(M_{\tilde\delta_N})$, so again each point appears twice with
opposite signs and the expression vanishes. Note that the broken
strings corresponding to these points have two preferred spikes
inserted at different places, so due to the uniqueness of preferred
spikes with given end points the broken strings do not depend on the
order in which the spikes are inserted.

\begin{figure}
\labellist
\small\hair 2pt
\pinlabel $S^j$ at 68 102
\pinlabel $\tilde\delta_NS$ at 42 35
\pinlabel $\tilde\delta_QS$ at 140 65
\pinlabel ${\color{blue} \Delta_2}$ at 99 227
\pinlabel ${\color{blue} U^j}$ at 63 139
\pinlabel ${\color{blue} \lambda^j_N}$ at 48 75
\pinlabel ${\color{blue} \lambda^j_Q}$ at 102 116
\endlabellist
\centering
\includegraphics[width=0.6\textwidth]{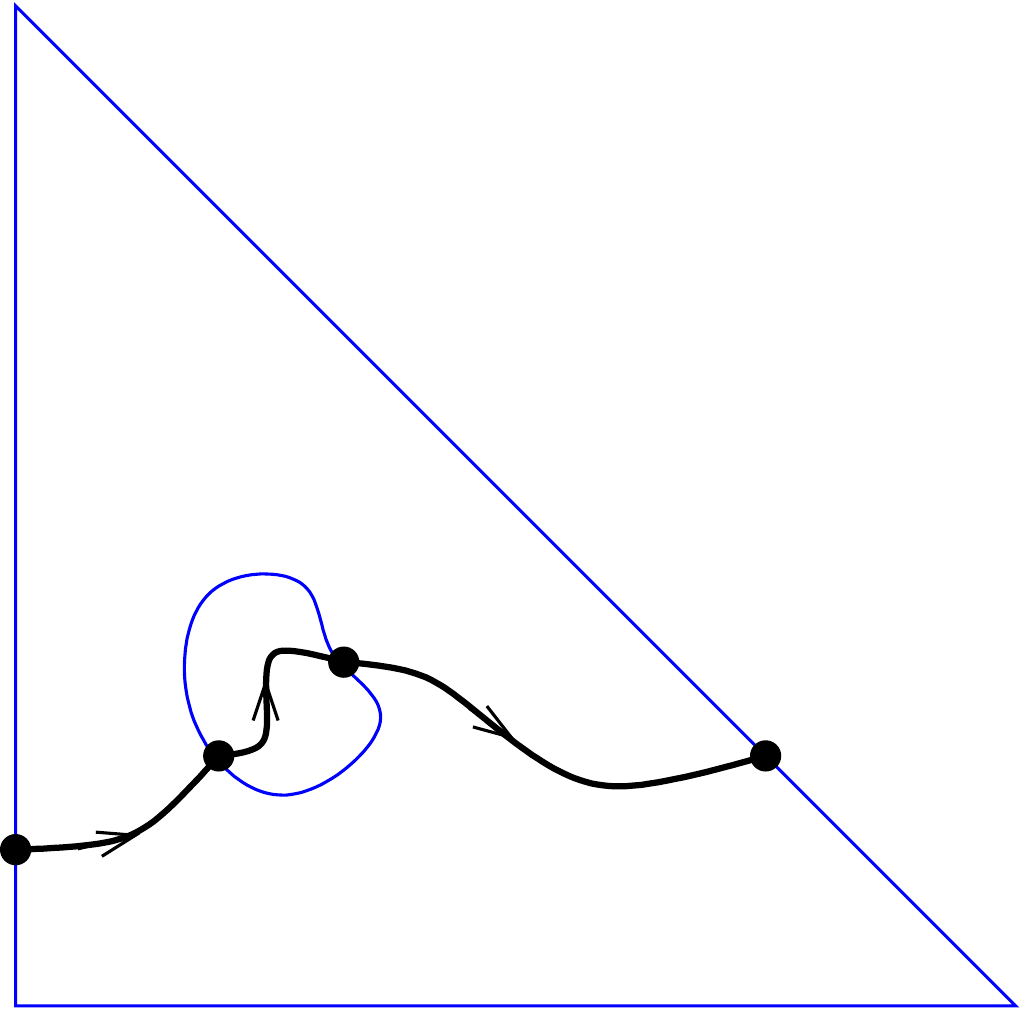}
\caption{
The definition of $\delta_QS=\tilde\delta_QS+S^j$ and $\delta_NS=\tilde\delta_NS$.
}
\label{fig:string-relations}
\end{figure}

In order to achieve $\p\delta_Q+\delta_Q\p + \p\delta_N+\delta_N\p =
0$, we choose the orientation conventions for $M_{\tilde\delta_Q}$ and
$M_{\tilde\delta_N}$ such that (see Figure~\ref{fig:string-relations}):
\begin{itemize}
\item $\p\tilde\delta_QS+\tilde\delta_Q\p S$ corresponds to the intersection
  points $\lambda^j_Q$ of $M_{\tilde\delta_Q}$ with the boundaries of the regions
  $U^j$, and similarly for $\p\tilde\delta_NS+\tilde\delta_N\p S$;
\item the sign of $\lambda^j_Q$ as a boundary point of
  $M_{\tilde\delta_Q}$ is opposite to the sign of $\lambda^j_N$ as a
  boundary point of $M_{\tilde\delta_N}$.
\end{itemize}
Due to the choice of the $1$-chains $S^j$, it follows that
$\p\delta_QS+\delta_Q\p S$ is the sum of the points
$\delta_NS(\lambda^j_N)$ with suitable signs, and
$\p\delta_NS+\delta_N\p S$ is the same sum with opposite signs, so the
total sum equals zero.
\end{proof}

\subsection{The string chain complex}

For $d=0,1,2$ and $\ell\geq 0$ let $C_d(\Sigma^\ell)$ be the free
$\Z$-module generated by generic $d$-chains in $\Sigma^\ell$, and set
$$
   C_d(\Sigma) := \bigoplus_{\ell=0}^\infty C_d(\Sigma^\ell),\qquad
   d=0,1,2.
$$
The string operations defined in Subsection~\ref{ss:string-op} yield
$\Z$-linear maps
$$
   \p:C_d(\Sigma^\ell)\to C_{d-1}(\Sigma^\ell),\qquad
   \delta_N,\delta_Q:C_d(\Sigma^\ell)\to
   C_{d-1}(\Sigma^{\ell+1}).
$$
The induced maps $\p,\delta_Q,\delta_N:C_d(\Sigma)\to C_{d-1}(\Sigma)$
satisfy the relations in Proposition~\ref{prop:string-relations}, in
particular
$$
   D:=\p+\delta_Q+\delta_N.
$$
satisfies $D^2=0$.
We call $\bigl(C_*(\Sigma),\p+\delta_Q+\delta_N\bigr)$ the {\em string
chain complex} of $K$, and we define the {\em degree $d$ string homology} of
$K$ as the homology of the resulting complex,
$$
   H_d^\str(K) :=
   H_d\Bigl(C_*(\Sigma),\p+\delta_Q+\delta_N\Bigr),\qquad d=0,1,2.
$$
Concatenation of broken strings at the base point $x_0$ (and the
canonical subdivision of $\Delta_1\times\Delta_1$ into two
$2$-simplices) yields products
$$
   \x:C_d(\Sigma^\ell)\times C_{d'}(\Sigma^{\ell'})\to
   C_{d+d'}(\Sigma^{\ell+\ell'}), \qquad d+d'\leq 2
$$
satisfying the relations
\begin{equation}\label{eq:x}
   (a\x b)\x c = a\x (b\x c),\qquad D(a\x b)=Da\x b+(-1)^{\deg a} a\x Db
\end{equation}
whenever $\deg a+\deg b+\deg c\leq 2$.
In particular, this gives $C_0(\Sigma)$ the structure of a
(noncommutative but strictly associative) algebra
over $\Z$ and $C_1(\Sigma),C_2(\Sigma)$ the structure of
bimodules over this algebra. These structures induce on
homology the structure of a $\Z$-algebra on $H_0^\str(K)$, and of
bimodules over this algebra on $H_1^\str(K)$ and $H_2^\str(K)$.
By definition, the isomorphism classes of the algebra $H_0^\str(K)$
and the modules $H_1^\str(K),H_2^\str(K)$ are clearly isotopy invariants
of the framed oriented knot $K$.

We can combine these invariants into a single graded algebra as follows.
For $d>2$, we define $C_d(\Sigma^\ell)$ to be the free $\Z$-module
generated by products $S_1\x \dots\x S_r$ of generic chains $S_i$ of
degrees $1\leq d_i\leq 2$ in $\Sigma^{\ell_i}$ such that
$d_1+\dots+d_r=d$ and $\ell_1+\dots+\ell_r=\ell$, modulo the submodule
generated by
$$
   S_1\x \dots\x S_r - S'_1\x \dots\x S'_{r'}
$$
for different decompositions of the same $d$-chain. Put differently,
this submodule is generated by
$$
   S_1\x\dots S_i\x S_{i+1}\x\dots \x S_r - S_1\x\dots (S_i\x
   S_{i+1})\x\dots \x S_r,
$$
where $S_i$ and $S_{i+1}$ are generic $1$-chains and $(S_i\x S_{i+1})$
is the associated generic $2$-chain. Note that for $d=2$ this
definition of $C_2(\Sigma^\ell)$ agrees with the earlier one.
We define $D=\p+\delta_Q+\delta_N$ on
$$
   C_d(\Sigma) := \bigoplus_{\ell=0}^\infty C_d(\Sigma^\ell),\qquad
   d\geq 0
$$
by the Leibniz rule. This is well-defined in view of the second equation
in~\eqref{eq:x} and satisfies $D^2=0$. Together with the product $\x$
this gives $C_*(\Sigma)$ the structure of a differential graded
$\Z$-algebra. The {\em total string homology}
$$
   H_*^\str(K) := H_*\Bigl(C_*(\Sigma),D\Bigr)
$$
inherits the structure of a graded $\Z$-algebra whose isomorphism
class is an invariant of the framed oriented knot $K$.

\begin{remark}
Our definition of string homology of $K$ in degrees $>2$ in terms of
product chains is motivated by Legendrian contact homology of $\Lambda
K$ when $Q=\R^{3}$ which is then generated by elements of degrees $\leq 2$. From the point
of view of string topology, it would appear more natural to define
string homology in arbitrary degrees in terms of higher dimensional
generic chains of broken strings in the sense of
Definition~\ref{def:generic-chain}. Similarly, for knot contact
homology in other ambient manifolds, e.g. for $Q=S^{3}$, there are
higher degree Reeb chords that contribute to the (linearized) contact
homology. It would be interesting to see whether such constructions
would carry additional information.  
\end{remark}

\subsection{Length filtration}\label{ss:length-filt}
Up to this point, the constructions have been fairly symmetric in the
$Q$-and $N$-strings. However, as we will see below, the relation to
Legendrian contact homology leads us to assign to $Q$-strings $s_{2i}$
their geometric length $L(s_{2i})$, and to $N$-strings length
zero. Thus we define the {\em length} of a broken string
$s=(s_1,\dots,s_{2\ell+1})$ by
$$
   L(s) := \sum_{i=1}^\ell L(s_{2i}),
$$
where we do not include in the sum those $s_{2i}$ that are
$Q$-spikes in the sense of Definition~\ref{def:spike}. 
We define the length of a generic $i$-chain $S:K\to\Sigma$ by
$$
   L(S) := \max_{k\in K}L\bigl(S(k)\bigr).
$$
Then the subspaces
$$
   \FF^\ell C_i(\Sigma):=\left\{\sum a_jS_j\in C_i(\Sigma)\mid
   L(S_j)\leq\ell \text{ whenever }a_j\neq 0\right\}
$$
define a filtration in the sense that $\FF^k C_i(\Sigma)\subset
\FF^\ell C_i(\Sigma)$ for $k\leq\ell$ and
$$
   D\Bigl(\FF^\ell C_i(\Sigma)\Bigr)\subset \FF^\ell C_{i-1}(\Sigma).
$$
This {\em length filtration} will play an important role in the proof
of the isomorphism to Legendrian contact homology in Section~\ref{sec:iso}.

\begin{remark}\label{rem:length-spikes}
The omission of the length of $Q$-spikes from the length of a broken
string ensures that the operation $\delta_N$, which inserts
$Q$-spikes, does not increase the length. Since $Q$-spikes do not
intersect the knot in their interior, they are not affected by
$\delta_Q$ and it follows that $D$ preserves the length filtration. 
\end{remark}


\section{The chain map from Legendrian contact homology to string homology}\label{sec:chain}

In this section we define a chain map $\Phi\colon
C_*(\RR)\to C_{\ast}(\Sigma)$ from a complex computing Legendrian
contact homology to the string chain complex defined in the previous section.
The boundary operator on $C_*(\RR)$ is
defined using moduli spaces of holomorphic disks in $\R\times
S^{\ast}Q$ with Lagrangian boundary condition $\R\times\Lambda_{K}$
and the map $\Phi$ is defined using moduli spaces of holomorphic disks
in $T^{\ast}Q$ with Lagrangian boundary condition $Q\cup L_{K}$, where
the boundary is allowed to switch back and forth between the two
irreducible components of the Lagrangian at corners as in Lagrangian
intersection Floer homology. We will describe these spaces and their
properties, as well as define the algebra and the chain map. In order
not to obscure the main lines of argument, we postpone the
technicalities involved in detailed proofs to Sections \ref{S:mdlisp}
-- \ref{sec:gluing}.

\subsection{Holomorphic disks in the symplectization}\label{sec:symp-disks}

Consider a contact $(2n-1)$-manifold $(M,\lambda)$ with a closed
Legendrian $(n-1)$-submanifold $\Lambda$. For the purposes of this
paper we only consider the case that $M=S^*Q$ is the cosphere bundle of
$Q=\R^3$ with its standard contact form $\lambda=p\,dq$
and $\Lambda=\Lambda_K$ is the unit conormal bundle of an oriented
framed knot $K\subset Q$, but the construction works more generally
for any pair $(M,\Lambda)$ for which 
$M$ has no contractible closed Reeb orbits, 
see Remark~\ref{rem:cont-hom} below.

Denote by $R$ the Reeb vector field of $\lambda$. A {\em Reeb chord}
is a solution $a\colon [0,T]\to M$ of $\dot a=R$ with $a(0),a(T)\in\Lambda$.
Reeb chords correspond bijectively to
{\em binormal chords} of $K$, i.e., geodesic segments meeting $K$
orthogonally at their endpoints. As usual, we assume throughout that
$\Lambda$ is chord generic, i.e., each Reeb chord corresponds to a
Morse critical point of the distance function on $K \x K$.

In order to define Maslov indices, one usually chooses for each Reeb
chord $a\colon [0,T]\to M$ {\em capping paths} connecting $a(0)$ and
$a(T)$ in $\Lambda$ to a base point $x_0\in\Lambda$. Then one can
assign to each $a$ completed by the capping paths a {\em Maslov index}
$\mu(a)$, see \cite[Appendix A]{CEL}. In the case under consideration
($M=S^{\ast}\R^3$ and $\Lambda=\Lambda_{K}$) the Maslov
class of $\Lambda$ equals $0$, so the Maslov index does not depend on
the choice of capping paths. It is given by $\mu(a)=\ind(a)+1$, where
$\ind(a)$ equals the index of $a$ as a critical point of the
distance function on $K\times K$, see \cite{EENS}.
We define the {\em degree} of a Reeb chord $a$ as
\[
   |a|:=\mu(a)-1=\ind(a),
\]
and the degree of a word $\mathbf{b}=b_1b_2\cdots b_m$ of Reeb
chords as
$$
   |\mathbf{b}| := \sum_{j=1}^{m}|b_j|.
$$
Given $a$ and $\mathbf{b}$, we write $\MM^{\rm sy}(a;\mathbf{b})$
for the moduli space of $J$-holomorphic disks $u\colon (D,\partial
D)\to(\R\times M,\R\times\Lambda)$ with one positive boundary puncture
asymptotic to the Reeb chord strip over $a$ at the positive end of the
symplectization, and $m$ negative boundary punctures asymptotic to the
Reeb chord strips over $b_1,\dots, b_m$ at the negative end of the
symplectization. Here $J$ is an $\R$-invariant almost complex
structure on $\R\times M$ compatible with $\lambda$.
For generic $J$, the moduli space $\MM^{\rm sy}(a;\mathbf{b})$ is a
manifold of dimension
\[
\dim(\MM^{\rm sy}(a;\mathbf{b}))=|a|-|\mathbf{b}|=|a|-\sum_{j=1}^{m}|b_j|,
\]
see Theorem \ref{t:sy}. In fact, the moduli spaces correspond to the
zero set of a Fredholm section of a Banach bundle that can be made
transverse by perturbing the almost complex structure, and there exist
a system of coherent (or gluing compatible) orientations of the
corresponding index bundles over the configuration spaces and this
system induces orientations on all the moduli spaces.

By our choice of almost complex structure, $\R$ acts on $\MM^{\rm
  sy}(a;\mathbf{b})$ by translations in the target $\R\times M$ and we write $\MM^{\rm
  sy}(a;\mathbf{b})/\R$ for the quotient, which is then an oriented
manifold of dimension $|a|-|\mathbf{b}|-1$.

Finally, we discuss the compactness properties of $\MM^{\rm
  sy}(a;\mathbf{b})/\R$. The moduli space $\MM^{\rm
  sy}(a;\mathbf{b})/\R$ is generally not compact but admits a
compactification by multilevel disks, where a multilevel disk is
a tree of disks with a top level disk in $\MM^{\rm
  sy}(a,\mathbf{b}^{1})$,
$\mathbf{b}^{1}=b^{1}_{1},\dots,b_{m_1}^{1}$, second level disks in $\MM^{\rm
  sy}(b_{i}^{1};\mathbf{b}^{2,i})$ attached at the negative
punctures of the top level disk, etc. See Figure~\ref{fig:sy-sy} below.
It follows from the
dimension formula above that the formal dimension of the total disk
that is the union of the levels in a multilevel disk is the sum of
dimensions of all its components. Consequently, for generic almost
complex structure, if $\dim(\MM^{\rm sy}(a;\mathbf{b}))=1$ then
$\MM^{\rm sy}(a;\mathbf{b})/\R$ is a compact 0-dimensional manifold,
and if $\dim(\MM^{\rm sy}(a;\mathbf{b}))=2$ then the boundary of
$\MM^{\rm sy}(a;\mathbf{b})/\R$ consists of two-level disks where each
level is a disk of dimension $1$ (and possibly trivial Reeb chord strips).

The simplest version of Legendrian contact homology would be defined
by the free $\Z$-algebra generated by the Reeb chords, with
differential counting rigid holomorphic disks. In the following
subsection we will define a refined version which also incorporates
the boundary information of holomorphic disks.

\subsection{Legendrian contact homology}\label{sec:leg}
In this subsection we define a version of Legendrian contact homology
that will be directly related to the string homology of Section~\ref{sec:string-ref}, see \cite{E_rsft} for a similar construction in rational symplectic field theory.
The usual definition of Legendrian contact homology is a quotient of
our version.
We keep the notation from Section~\ref{sec:symp-disks}.

Fix an integer $m\geq 3$. For points $x,y\in\Lambda$ we denote by
$P_{x,y}\Lambda$ the space of $C^m$ paths $\gamma:[a,b]\to\Lambda$ with
$\gamma(a)=x$ and $\gamma(b)=y$ whose first $m$ derivatives vanish at
the endpoints. Here the interval $[a,b]$ is allowed to vary. The
condition at the endpoints ensures that concatenation of such paths
yields again $C^m$ paths. Fix a base point $x_0\in\Lambda$ and denote
by $\Om_{x_0}\Lambda=P_{x_0x_0}\Lambda$ the Moore loop space based at $x_0$.

\begin{figure}
\labellist
\small\hair 2pt
\pinlabel ${\color{red} a_1}$ at 22 148
\pinlabel ${\color{red} a_2=a}$ at 111 148
\pinlabel ${\color{red} a_3}$ at 200 148
\pinlabel ${\color{red} b_1}$ at 38 4
\pinlabel ${\color{red} b_2}$ at 110 4
\pinlabel ${\color{red} b_3}$ at 182 4
\pinlabel ${\color{blue} \alpha_1}$ at 29 298
\pinlabel ${\color{blue} \alpha_2}$ at 46 219
\pinlabel ${\color{blue} \alpha_3}$ at 162 223
\pinlabel ${\color{blue} \alpha_4}$ at 140 318
\pinlabel ${\color{blue} \beta_1}$ at 26 114
\pinlabel ${\color{blue} \beta_2}$ at 75 87
\pinlabel ${\color{blue} \beta_3}$ at 146 87
\pinlabel ${\color{blue} \beta_4}$ at 196 114
\pinlabel ${\color{blue} x_0}$ at 93 359
\pinlabel ${\color{blue} u}$ at 109 110
\endlabellist
\centering
\includegraphics[width=0.5\textwidth]{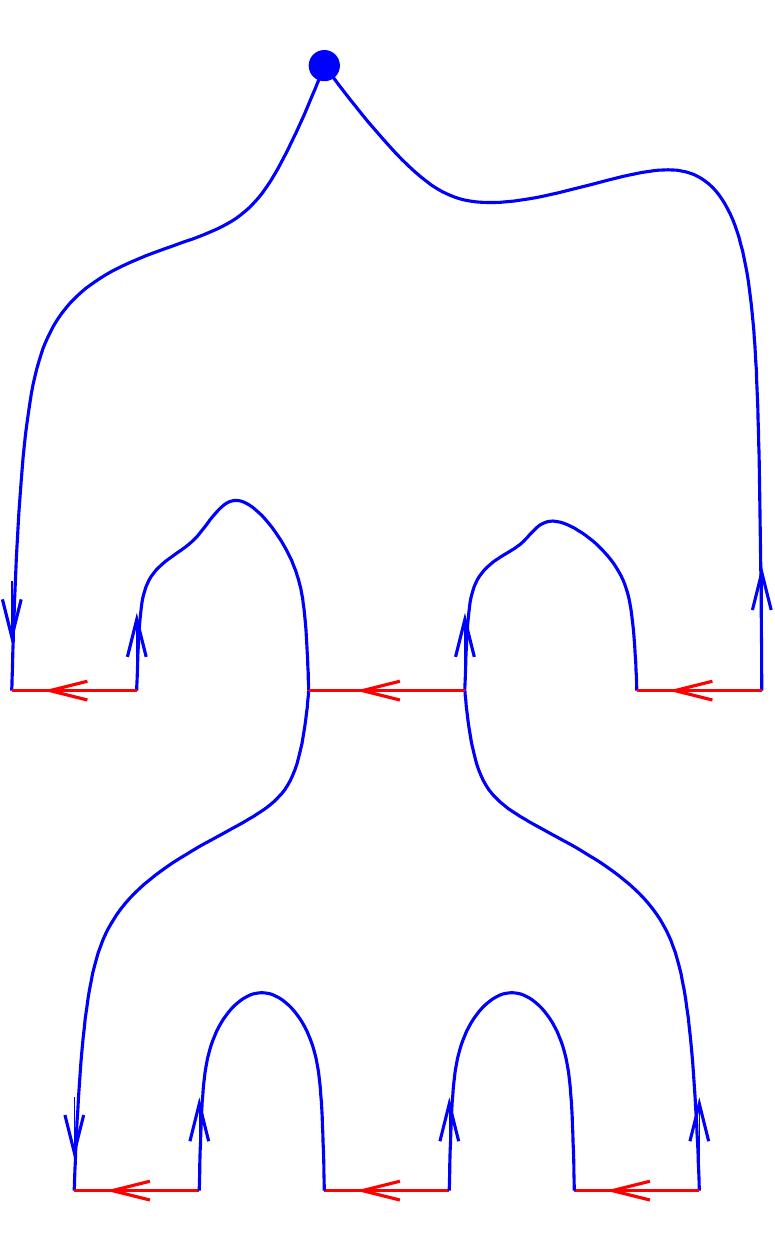}
\caption{The definition of $\p(u)$ and $\p(u) \,\cdot_{i}\, \mathbf{a}$.}
\label{fig:Reeb-string}
\end{figure}

\begin{definition}
A {\em Reeb string with $\ell$ chords} is an expression
$\alpha_1a_1\alpha_2a_2\cdots \alpha_\ell a_\ell \alpha_{\ell+1}$, where the
$a_i\colon[0,T_i]\to M$ are Reeb chords and the $\alpha_i$ are elements
in the path spaces
$$
   \alpha_1\in P_{x_0a_1(T_1)},\qquad
   \alpha_i\in P_{a_{i-1}(0)a_i(T_i)} \text{ for }2\leq i\leq \ell,\qquad
   \alpha_{\ell+1}\in P_{a_\ell(0)x_0}.
$$
\end{definition}
See the top of Figure~\ref{fig:Reeb-string}.
Note that the $\alpha_i$ and the {\em negatively traversed} Reeb chords
$a_i$ fit together to define a loop in $M$ starting and ending at $x_0$.
Concatenating all the $\alpha_i$ and $a_i$ in a Reeb string with the
appropriate capping paths, we can view each $\alpha_i$ as an element in
the based loop space $\Om_{x_0}\Lambda$. However, we will usually not take
this point of view.

Boundaries of holomorphic disks in the symplectization give rise to
Reeb strings as follows. Consider a holomorphic disk $u$ belonging to a moduli space
$\MM^{\rm sy}(a;\mathbf{b})$ as above, with Reeb chords
$a:[0,T]\to M$ and $b_i:[0,T_i]\to M$, $i=1,\dots,\ell$. Its boundary arcs in
counterclockwise order and orientation projected to $\Lambda$ define paths
$\beta_1,\dots,\beta_{\ell}$ in $\Lambda$ as shown in
Figure~\ref{fig:Reeb-string}, i.e.
$$
   \beta_1\in P_{a(T)b_1(T_1)},\qquad
   \beta_i\in P_{b_{i-1}(0)b_i(T_i)} \text{ for }2\leq i\leq \ell,\qquad
   \beta_{\ell+1}\in P_{b_\ell(0)a(0)}.
$$
We denote the alternating word of paths and Reeb chords obtained in
this way as the boundary of $u$ by
\begin{equation}\label{eq:pu}
   \p(u) := \beta_1b_1\beta_2b_2\cdots\beta_\ell b_\ell\beta_{\ell+1}.
\end{equation}
Note that the $\beta_i$ and the negatively traversed Reeb chords
$b_i$ fit together to define a path in $M$ from $a(T)$ to $a(0)$.
We obtain from $\p(u)$ a Reeb string if we extend $\beta_1$ and
$\beta_{\ell+1}$ to the base point $x_0$ by the capping paths of $a$.

For $\ell\geq 0$ we denote by $\RR^\ell$ the space of Reeb strings with $\ell$
chords, equipped with the $C^m$ topology on the path spaces. Note that
different collections of Reeb chords correspond to different components.
Concatenation at the base point gives
$$
   \RR := \amalg_{\ell\geq 0}\RR^\ell
$$
the structure of an H-space. Note that the sub-H-space
$\RR^0=\Om_{x_0}\Lambda$ agrees with the Moore based loop space with
its Pontrjagin product. Let
$$
   C(\RR) = \bigoplus_{d\geq 0}C_d(\RR)
$$
be singular chains in $\RR$ with integer coefficients. It carries two
gradings: the degree $d$ as a singular chain, which we will refer to
as the {\em chain degree}, and the degree $\sum_{i=1}^{\ell}|b_i|$ of
the Reeb chords, which we will refer to as the {\em chord degree}. For
sign rules we think of the {\em chain coming first and the Reeb chords last}.
The total grading is given by the sum of the two degrees.
Recall that it does not depend on the choice of capping paths.
Concatenation of Reeb strings at the base point and product of chains
gives $C(\RR)$ the structure of a (noncommutative but strictly
associative) graded ring. Note that it contains the subring
$$
   C(\RR^0) = C(\Om_{x_0}\Lambda).
$$
Next we define the differential
$$
   \p_\Lambda=\p^\sing+\p^\sy\colon C(\RR)\to C(\RR).
$$
Here $\p^\sing$ is the singular boundary and $\p^\sy$ is defined as
follows. Pick a generic compatible cylindrical almost complex
structure $J$ on the symplectization $\R\times M$. Consider a
punctured $J$-holomorphic disk $u\colon D\to\R\times M$ in $\MM^{\rm
  sy}(a;\mathbf{b})$. If the Reeb chord $a=a_i$ appears in a Reeb string
$\mathbf{a}=\alpha_{1}a_1\dots a_m\alpha_{m+1}$, then we can replace
$a_i$ by $\p(u)$ to obtain a new Reeb string which we denote by
$$
   \p(u) \,\cdot_{i}\, \mathbf{a} := \alpha_1a_1\cdots
   \wt\alpha_i\p(u)\wt\alpha_{i+1}\cdots a_\ell \alpha_{\ell+1}.
$$
Here $\p(u)$ is defined in~\eqref{eq:pu}
and the paths $\wt\alpha_i,\wt\alpha_{i+1}$ are the concatenations of
$\alpha_i,\alpha_{i+1}$ with the paths $\beta_1,\beta_{\ell+1}$ in $\p(u)$,
respectively. See Figure~\ref{fig:Reeb-string}.
For a chain $\mathbf{a}\in C(\RR)$ of Reeb strings of type
$\mathbf{a}=\alpha_{1}a_1\dots a_m\alpha_{m+1}$ we now define
$$
   \p^{\rm sy}(\mathbf{a})  :=  \quad
   \sum_{i=1}^\ell \!\! \sum_{\begin{smallmatrix}
   |a_i|-|\mathbf{b}|=1 \\ u\in\MM^{\rm sy}(a_i;\mathbf{b})/\R
   \end{smallmatrix}}
   \eps (-1)^{d+|a_1|+\cdots+|a_{i-1}|} \p(u)\,\cdot_{i}\,\mathbf{a},
$$
where $d$ is the chain degree of $\mathbf{a}$ and $\eps$ is the sign from the
orientation of $\MM^{\rm sy}(a_i;\mathbf{b})/\R$ as a
compact oriented 0-manifold (i.e., points with signs).
Note that $\p^{\rm sy}$ preserves the chain degree and
decreases the chord degree by $1$, whereas $\p^{\rm sing}$
preserves the chord degree and decreases the chain
degree by $1$. In particular, $\p_\Lambda$
has degree $-1$ with respect to the total grading.
The main result about the contact homology algebra that we need is
summarized in the following theorem.

\begin{thm}\label{thm:comt-hom}
The differential $\p_\Lambda\colon C(\RR)\to C(\RR)$ satisfies
$\p_\Lambda^2=0$ and the {\em Legendrian contact homology}
$$
   H^\cont(\Lambda) := \ker\p_\Lambda/\im\p_\Lambda
$$
is independent of all choices.
\end{thm}

\begin{proof}
In the case that we use it, for $M=S^*\R^{3}$ and $\Lambda=\Lambda_K$,
the proof is an easy adaptation of the one in \cite{EES2,EESori} and
\cite{Rizell}, see also \cite{EENS}.

Consider first the equation for the differential. The equation $\p_\Lambda^2=0$ follows from our description of the boundary of the moduli spaces $\MM^{\rm sy}(a;\mathbf{b})$ of dimension $2$ in Section~\ref{sec:symp-disks}, which shows that contributions to $(\p^{\rm  sy})^{2}$ are in oriented one-to-one correspondence with the boundary of an oriented 1-manifold and hence cancel out. The relations
$(\p^{\rm sing})^2=0$ and $\p^{\rm sing}\p^{\rm sy} + \p^{\rm  sy}\p^{\rm sing} =0$ are clear.

To prove the invariance statement we use a bifurcation method similar to \cite[section 4.3]{EESori}. Consider a generic $1$-parameter family $(\Lambda_{s},J_{s})$, $s\in S=[0,1]$, of Legendrian submanifolds and almost complex structures. By genericity of the family there is a finite set of points $s_{1}<s_{2}<\dots < s_{m}$ such that in $S\setminus\{s_{1},\dots,s_{m}\}$ all Reeb chords of $\Lambda_{s}$ are transverse, all Reeb chords have distinct actions, and all holomorphic disks determined by $(\Lambda_{s},J_{s})$ have dimension at least $1$ (i.e.~if we write $\MM^{\sy}_{s}$ for moduli spaces determined by $(\Lambda_{s},J_{s})$ then $\dim\MM^{\sy}_{s}(a,\mathbf{b})\ge 1$ if the moduli space is nonempty). Furthermore, the points $s_{j}$ are of three kinds:
\begin{itemize}
\item handle slides, where all Reeb chords are nondegenerate but where there is a transversely cut out disk of formal dimension $0$ (i.e., there exists one $\MM^{\sy}(a;\mathbf{b})$, with $\dim\MM^{\sy}(a;\mathbf{b})=0$ which contains one $\R$-family of $J_{s_{j}}$-holomorphic disks with boundary on $\Lambda_{s_{j}}$, and this disk is transversely cut out as a solution of the parameterized problem);
\item action switches, where two nondegenerate Reeb chords have the same action and their actions interchange;
\item birth/death moments where there is one degenerate Reeb chord at which two Reeb chords cancel through a quadratic tangency. 
\end{itemize}

To show invariance we first observe that if $[s',s'']\subset S$ is an interval which does not contain any $s_{j}$, then the Reeb chords of $\Lambda_{s}$, $s\in[s',s'']$ form 1-manifolds canonically identified with $[s',s'']$ and the actions of the different Reeb chord manifolds do not cross. Thus for Reeb chords $a,b_{1},\dots,b_{m}$ of $\Lambda_{s'}$ we get corresponding chords on $\Lambda_{s}$ for each $s\in [s',s'']=S'$ which we denote by the same symbols, suppressing the $s$-dependence below. We next define a chain map
\[
\Phi\colon C(\RR_{s'})\to C(\RR_{s''})
\]
which counts geometrically induced chains as follows. We introduce the notion of a {\em disk with lines of Reeb chords}. Such an object has a positive puncture at the Reeb chord $a$ over $s'$ and negative punctures at Reeb chords according to $\mathbf{b}$ over $s''$ and is given by a collection of disks $u_{1},\dots,u_{m}$ where the disk $u_{j}$ is a disk at $\sigma_{j}$, where $s'\le \sigma_{1}\le \sigma_{2}\le\dots\le \sigma_{m}\le s''$ and if $\sigma_{j}> \sigma_{k}$ for some $k$ then its positive Reeb chord is connected by a line in a Reeb chord manifold to a Reeb chord at the negative puncture of some $u_{r}$ for $\sigma_{r}< \sigma_{j}$. The collection of such objects naturally forms a moduli space, $\MM^{\sy}_{S'}(a;\mathbf{b})$, where we glue two disks when the length of the line connecting them goes to zero. We define the chain map $\Phi$ as
\[ 
\Phi(a) = \left[\MM^{\sy}_{S'}(a,\mathbf{b})\right].
\]  
The chain map equation $\p_{\Lambda_{s''}}\Phi=\Phi\p_{\Lambda_{s'}}$ follows immediately once one notices that the codimension one boundary of the moduli space consists of disks over the endpoints with lines of Reeb chords over $[s',s'']$ attached. (We point out that this construction is inspired by Morse-Bott arguments, compare \cite{EK}.)

Consider the filtration in $C(\RR)$ which associates to a chain of Reeb strings the sum of actions of its Reeb chords. By Stokes' theorem the differential respects the filtration.
The pure lines of Reeb chords (without disks) contribute to the map and show that
\[ 
\Phi(a) = a + \Phi_{0}(a),
\]
where the action of $\Phi_{0}(a)$ is strictly smaller than that of $a$. It follows that $\Phi$ induces an isomorphism on the $E_{2}$-page of the action spectral sequence and hence is a quasi-isomorphism.

In order to show invariance at the bifurcation moments we consider the deformation in a small interval around $[s_{j}-\epsilon,s_{j}+\epsilon]$. In this case we can construct a Lagrangian cobordism $L$ in the symplectization $\R\times M$ interpolating between the cylinders on $\Lambda_{s_{j}-\epsilon}$ and $\Lambda_{s_{j}+\epsilon}$, see \cite[Lemma A.2]{E_rsft}. If $a$ is a Reeb chord of $\Lambda_{s_{j}+\epsilon}$ and $\mathbf{b}$ is a word of Reeb chords of $\Lambda_{s_{j}-\epsilon}$ then let $\MM^{\sy,L}(a;\mathbf{b})$ denote the moduli space of holomorphic disks defined as $\MM^{\sy}(a,\mathbf{b})$, see Section \ref{sec:symp-disks}, but with boundary condition given by $L$ instead of $\R\times\Lambda$. (Note that since $L$ is not $\R$-invariant, $\R$ does generally not act on $\MM^{\sy,L}(a;\mathbf{b})$.)
We define a chain map 
\[ 
\Phi\colon C(\RR_{+})\to C(\RR_{-})
\]
between the algebras at the positive and the negative ends
as follows: $\Phi$ is the identity map on chains, and on Reeb chords $a$ of $\Lambda_{s_{j}-\epsilon}$ $\Phi$ is given by
\[ 
\Phi(a) = \sum_{\mathbf{b}}[\MM^{\sy,L}(a;\mathbf{b})],
\]
where $\mathbf{b}$ runs over all words of Reeb chords of $\Lambda_{s_j+\epsilon}$ and $[\MM^{\sy,L}(a;\mathbf{b})]$ denotes the chain of Reeb strings carried by the moduli space. SFT compactness and gluing as in \cite{E_rsft} shows that the chain map equation $\p_{\Lambda_{s_j+\epsilon}}\Phi=\Phi\p_{\Lambda_{s_j-\epsilon}}$ holds. It remains to show that $\Phi$ is a quasi-isomorphism. 

Consider first the case that $s_{j}$ is a handle slide. Taking $\epsilon$ sufficiently small we find that for each Reeb chord $a$ on $\Lambda_{s_{j}-\epsilon}$ there is a unique holomorphic strip connecting it to the corresponding Reeb chord $a$ on $\Lambda_{s_{j}+\epsilon}$. (These strips converge to trivial strips as $\epsilon\to 0$.) It follows that for each generator $c$ (chord or chain),
\[ 
\Phi(c)= c + \Phi_{0}(c),
\]
where the filtration degree of $\Phi_{0}(c)$ is strictly smaller than that of $c$. Thus $\Phi$ induces an isomorphism on the $E_{2}$-page of the action filtration spectral sequence and is hence a quasi-isomorphism.

Consider second the case of an action switch. In this case we find exactly as in the handle slide case that
\[ 
\Phi(c) = c + \Phi_{0}(c)
\]
for each $c$. The only difference is that now one action window contains two generators. Since the two Reeb chords have the same action but lie at a positive distance apart, it follows by monotonicity and Stokes' theorem that the chain map induces an isomorphism also in this action window. We find as above that $\Phi$ is a quasi-isomorphism.

Finally consider the case that $s_{j}$ is a birth moment where two new Reeb chords $a$ and $b$ are born (the death case is analogous). For $\epsilon>0$ sufficiently small we have
\[ 
\p^{\rm sy} a = b + \p^{\rm sy}_{0}(a),
\]
where the action of $\p^{\rm sy}_{0}(a)$ is strictly smaller than the action of $b$, see \cite[Lemma 2.14]{EES1}. As above we find that for any Reeb chord $c$ of $\Lambda_{s_{j}+\epsilon}$ we have
\[ 
\Phi(c) = c + \Phi_{0}(c).
\]
If we filter by small action windows that contain one Reeb chord each, except for one that contains both $a$ and $b$ (note that the action of $a$ approaches the action of $b$ as $\epsilon\to 0$) we find again that $\Phi$ gives an isomorphism on the $E_{2}$-page and hence is an isomorphism. We conclude that we can subdivide the interval $S$ into pieces with endpoints with quasi-isomorphic algebras. The theorem follows.
\end{proof}

According to Theorem~\ref{thm:comt-hom}, $(C(\RR),\p_\Lambda)$ is a (noncommutative but
strictly associative) differential graded (dg) ring containing the dg
subring
$$
   \Bigl(C(\RR^0),\p_\Lambda\Bigr) = \Bigl(C(\Om_{x_0}\Lambda),\p^{\rm sing}\Bigr).
$$
Thus $(C(\RR),\p_\Lambda)$ is a $(C(\RR^0),\p_\Lambda)$-NC-algebra in the
sense of the following definition.

\begin{definition}
Let $(R,\p)$ be a dg ring. An {\em $(R,\p)$-NC-algebra} is a dg ring
$(S,\p_S)$ together with a dg ring homomorphism $(R,\p)\to(S,\p_S)$.
\end{definition}

It follows that the Legendrian contact homology $H^\cont(\Lambda)$ is
an NC-algebra over the graded ring
$$
   H_*(\Om_{x_0}\Lambda,\p^{\rm sing}) \cong \Z\pi_1(\Lambda) \cong
   \Z[\lambda^{\pm 1},\mu^{\pm 1}].
$$
Here we have used that in our situation $\Lambda\cong T^2$ is a
$K(\pi,1)$, so all the homology of its based loop space is
concentrated in degree zero and agrees with the group ring of its
fundamental group $\pi_1(\Lambda)\cong\Z^2$.

{\bf Relation to standard Legendrian contact homology. }
Recall that $C(\RR)$ is a double complex with bidegree (chain degree,
chord degree), horizontal differential $\p^\sing$, and vertical
differential $\p^\sy$. As observed above, the first page of the
spectral sequence corresponding to the chord degree is concentrated in
the $0$-th column and given by
$$
   \Bigl(\AA:=H_0(\RR,\p^{\rm sing}),\p^{\rm sy}\Bigr).
$$
Generators of $\AA$ are words $\alpha_1a_1\alpha_2a_2\cdots
\alpha_\ell a_\ell \alpha_{\ell+1}$ consisting of Reeb chords $a_i$ and
{\em homotopy classes} of paths $\alpha_i$ satisfying the same boundary
conditions as before. Note that $\AA$ is an NC-algebra over the
subring $\AA^0=H_0(\RR^0)\cong\Z\pi_1(\Lambda)$ (on which $\p_\Lambda$
vanishes), and $\AA^k=H_0(\RR^k)$ is the $k$-fold tensor product of the
bimodule $\AA^1$ over the ring $\AA^0$.

We denote by
$$
   \bar\AA := \AA/\II
$$
the quotient of $\AA$ by the ideal $\II$ generated by the commutators
$[a,\beta]$ of Reeb chords $a$ and $\beta\in\pi_1(\Lambda)$. Since
$\p_\Lambda(\II)\subset\II$, the differential descends to a differential
$\bar\p^{\rm sy}:\bar\AA\to\bar\AA$ whose homology
$$
   \bar H^\cont(\Lambda) := \ker\bar\p^{\rm sy}/\im\bar\p^{\rm sy}
$$
is the usual Legendrian contact homology as defined in~\cite{EES1}.

{\bf Length filtration. }
The complex $(C(\RR),\p_\Lambda)$ is filtered by the {\em length}
$$
   L(\alpha_1a_1\alpha_2a_2\cdots \alpha_\ell a_\ell \alpha_{\ell+1}) :=
   \sum_{i=1}^\ell L(a_i),
$$
where $L(a)=\int_a\lambda$ denotes the action of a Reeb chord $a$,
which agrees with its period and also with the length of the
corresponding binormal cord. The length is preserved by the singular
boundary operator $\p^{\rm sing}$ and strictly decreases under
$\p^{\rm sy}$.

\begin{remark}\label{rem:cont-hom}
The construction of Legendrian contact homology in this subsection
works for any pair $(M,\Lambda)$ such that $M$ has no contractible
closed Reeb orbits. 
Examples include cosphere bundles $S^*Q$ of
$n$-manifolds $Q$ with a metric of nonpositive curvature that are
convex at infinity, with $\Lambda=\Lambda_K$ the unit conormal bundle
of a closed connected submanifold $K\subset Q$.
However, if $\Lambda$ is not a $K(\pi,1)$, then the coefficient ring
$H_*(\Om_{x_0}\Lambda,\p^{\rm sing})$ will not be equal to the group
ring of its fundamental group but contain homology in higher degrees.
\end{remark}

\subsection{Switching boundary conditions, winding numbers, and length}\label{ss:switching}
We continue to consider $Q=\R^{3}$ equipped with the flat metric
and an oriented framed knot $K\subset Q$. In addition, we assume from
now on that {\em $K$ is real analytic}; this can always be achieved
by a small perturbation of $K$ not changing its knot type. We equip $T^{\ast} Q$
with an almost complex structure $J$ which agrees with an $\R$-invariant
almost complex structure on the symplectization of $S^{\ast}Q$ outside
a finite radius disk sub-bundle of $T^{\ast} Q$ and with the standard
almost complex structure $J_\st$ on $T^{\ast}Q$ inside the disk
sub-bundle of half that radius. An explicit formula for such $J$ is given
in Section~\ref{sec:length-estimates2}. We point out that the
canonical isomorphism $(T^{\ast}Q,J_\st)\cong (\C^{3},i)$ identifies
the fibre with $\R^3$ and the zero section with $i\R^3$. Recall that $L=Q\cup L_K$.

Let $D$ be the closed unit disk with a boundary puncture at
$1\in\partial D$ and let $u\colon (D,\partial D)\to (T^{\ast}Q,L)$ be
a holomorphic disk with one positive puncture and switching boundary
conditions. This means that the map $u$ is asymptotic to a Reeb chord
at infinity at the positive puncture $1$ and that it is smooth outside an
additional finite number of boundary punctures where the boundary
switches, i.e., jumps from one irreducible component of $L$ to
another (which may be the same one). At these additional boundary punctures, the holomorphic disk
is asymptotic to some point in the clean intersection $K\subset L$,
i.e., it looks like a corner of a disk in Lagrangian intersection
Floer homology.

The real analyticity of $K$ allows us to get explicit local forms for
holomorphic disks near corners. We show in Lemma~\ref{l:knotnbhd} that
there are holomorphic coordinates
\[
   \R\times (0,0)\subset U\subset \C \times \C^{2},
\]
in which $K$ corresponds to $\R\times (0,0)$,
the $0$-section $Q$ corresponds to $\R\times \R^{2}$, and the conormal
$L_K$ to $\R\times i\R^{2}$.

Consider now a neighborhood of a switching point of a holomorphic disk
$u$ on the boundary of $D$, where we use $z$ in a half-disk
$D_\eps^+$ around $0$ in the upper half-plane as a local coordinate
around the switching point in the source. According to
Section~\ref{ss:series}, $u$ admits a Taylor
expansion around $0$, with $u=(u_1,u_2)\in \C\times\C^{2}$:
\begin{equation}\label{eq:Taylorswitch}
   u_1(z) = \sum_{k\in\N} b_k z^{k},\qquad
   u_{2}(z) = \sum_{k\in\frac12\N} c_k z^{k}.
\end{equation}
Here compared to Section~\ref{ss:series} we have divided the indices
by $2$, so the $b_k$ and $c_k$ correspond to the $a_{2k}$ in
Section~\ref{ss:series}.
The coefficients $b_j$
are real constants, reflecting smoothness of the tangent component of
$u$. The $c_k$ satisfy one of the conditions in
Remark~\ref{rem:cases1-4}, i.e., they are either all real or all
purely imaginary vectors in $\C^{2}$, and the indices are either all
integers or all half-integers.

Equivalently (and more adapted to the analytical study in Sections
\ref{S:mdlisp} -- \ref{sec:gluing}) one can use $z$ in a neighborhood
of infinity in the strip $\R\times [0,1]$ as a local coordinate in the
source. Composing the Taylor expansions~\eqref{eq:Taylorswitch} with
the biholomorphism
\begin{equation}\label{eq:chi}
   \chi:\R_{\geq 0}\times[0,1]\stackrel{\cong}\longrightarrow D^+, \qquad z\mapsto -\exp(-\pi z)
\end{equation}
(see Figure~\ref{fig:exp})
\begin{figure}
\labellist
\small\hair 2pt
\pinlabel $1$ at 6 92
\pinlabel $-1$ at 307 6
\pinlabel $1$ at 451 6
\endlabellist
\centering
\includegraphics[width=\textwidth]{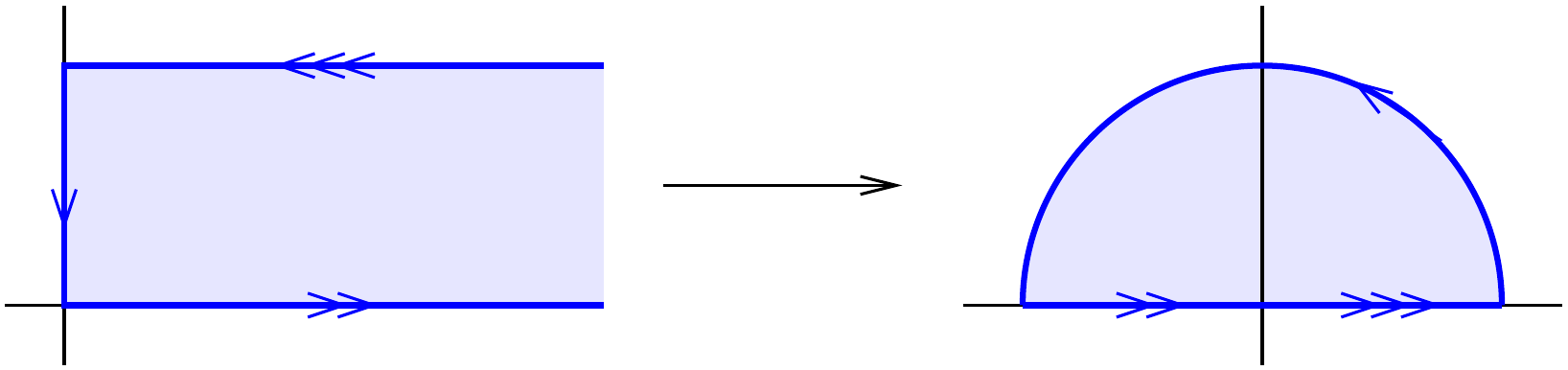}
\caption{
The biholomorphism $\chi$.
}
\label{fig:exp}
\end{figure}
one gets instead the Fourier expansions
\begin{equation}\label{eq:Fourierswitch}
   u_1(z) = \sum_{k\in\N} (-1)^kb_k e^{-k\pi z},\qquad
   u_{2}(z) =  \sum_{k\in\frac12\N} (-1)^kc_k e^{-k\pi z}.
\end{equation}
Recall from Section~\ref{ss:winding} that the \emph{local winding
  number} at the switch is the positive half-integer or integer which is the
index of the first non-vanishing Fourier coefficients in the expansion
of $u_{2}$ in~\eqref{eq:Fourierswitch}.
The sum of the local winding numbers at all switching boundary
punctures is the \emph{total winding number} of the disk.
Since the number of switches from $L_K$ to $Q$ equals that from $Q$ to
$L_K$, the total winding number is an integer.

The following technical result, which is a special case of
\cite[Theorem 1.2]{CEL}, will play a crucial role in the sequel.

\begin{thm}[\cite{CEL}]\label{thm:finite}
For a cord-generic real analytic knot $K\subset\R^3$ the total winding number, 
and in particular the number of switches, of any holomorphic disk $u\colon
(D,\partial D)\to (T^{\ast}Q,L)$ with one positive puncture is
uniformly bounded by a constant $\kappa$.
\end{thm}
\begin{remark}
The necessary energy bound appearing in the corresponding statement 
in \cite{CEL} is automatic here, since in our present situation the energy 
is given by the action of the Reeb chord at the positive puncture, which only
varies in a finite set.
\end{remark}

In view of this result, when we discuss compactness we need only
consider sequences of holomorphic disks with a {\em fixed} finite number of
switches, each of fixed winding number. As we prove in
Section~\ref{S:mdlisp}, each moduli space of such holomorphic disks is
for generic data a manifold that admits a natural
compactification as a manifold with boundary with corners. We
will specifically need such moduli spaces of dimension $0$, $1$, or
$2$ and we give brief descriptions in these cases.

Let $a$ be a Reeb chord of $\Lambda_{K}$. Let $q_{1},\dots,q_{m}$ be
punctures in $\partial D$ and let $\mathbf{n}=(n_1,\dots,n_{m})$ be a
vector of local winding numbers, so
$n_j\in\left\{\tfrac12,1,\tfrac32,2,\dots\right\}$ is the local
winding number at $q_j$. We write $\MM(a;\mathbf{n})$ for the moduli
space of holomorphic disks with positive puncture at the Reeb chord
$a$ and switching punctures at $q_1,\dots,q_{m}$ with winding numbers
according to $\mathbf{n}$. Define the nonnegative integer
\[
   |\mathbf{n}|:=\sum_{j=1}^{m}2(n_j-\tfrac12)\ge 0.
\]

\begin{theorem}\label{t:dim-moduli}
For generic almost complex structure $J$, the moduli space
$\MM(a;\mathbf{n})$ is a manifold of dimension
\[
   \dim \MM(a;\mathbf{n})=|a|-|\mathbf{n}|.
\]
Furthermore, the choice of a spin structure on $L_{K}$ together with
the spin structure on $\R^{3}$ induces a natural orientation on
$\MM(a;\mathbf{n})$.
\end{theorem}

\begin{proof}
This is a consequence of \cite[Theorem A.1]{CEL} and Lemma \ref{l:tv} below.
\end{proof}

Note that, due to Theorem~\ref{thm:finite}, any moduli space
$\MM(a;\mathbf{n})$ is empty if $\mathbf{n}$ has more than $\kappa$
components, i.e., there are more than $\kappa$ switches.

\subsection{Moduli spaces of dimension zero and one}\label{s:mswitchdim0to1}
For moduli spaces of dimension $\le 1$ with positive puncture at a
Reeb chord of degree $\le 1$, we have the following. Theorem
\ref{t:dim-moduli} implies that if $|a|=0$ then $\MM(a;\mathbf{n})$ is
empty if $|\mathbf{n}|>0$ and is otherwise a compact oriented
$0$-manifold. Likewise, if $|a|=1$ then $\MM(a;\mathbf{n})$ is empty
if $|\mathbf{n}|> 1$ and is an oriented $0$-manifold if
$|\mathbf{n}|=1$. Note that $|\mathbf{n}|=1$ implies that there is
exactly one switch with winding number $1$ and that the winding
numbers at all other switches equal $\tfrac12$. Finally, if all
entries in $\mathbf{n}$ equal $\tfrac12$ then
$\dim(\MM(a;\mathbf{n}))=1$.

It follows by Theorem \ref{t:[1,0]} that the 1-dimensional moduli
spaces of disks with switching boundary condition admit natural
compactifications to 1-manifolds with boundary. The next result
describes the disk configurations corresponding to the boundary of
these compact intervals.

\begin{prop}\label{prop:simpleboundary}
If $a$ is a Reeb chord of degree $|a|=1$ and if all entries of
$\mathbf{n}$ equal $\tfrac12$, then the oriented boundary of
$\MM(a;\mathbf{n})$ consists of the following:
\begin{description}
\item[$(Lag)$] Moduli spaces $\MM(a;\mathbf{n}')$, where $\mathbf{n}'$ is
  obtained from $\mathbf n$ by removing two consecutive
  $\tfrac12$-entries and inserting in their place a $1$.
\item[$(sy)$] Products of moduli spaces
\[
\MM^{\rm sy}(a;\mathbf{b})/\R \;\times \; \Pi_{b_j\in\mathbf{b}}\, \MM(b_j;\mathbf{n}_j),
\]
where $\mathbf{n}$ equals the concatenation of the $\mathbf{n}_j$.
\end{description}
\end{prop}

\begin{figure}
\labellist
\small\hair 2pt
\pinlabel $\frac{1}{2}$ at 18 14
\pinlabel $\frac{1}{2}$ at 54 14
\pinlabel $\frac{1}{2}$ at 90 14
\pinlabel $\frac{1}{2}$ at 126 14
\pinlabel $\frac{1}{2}$ at 162 14
\pinlabel $\frac{1}{2}$ at 198 14
\pinlabel $\frac{1}{2}$ at 342 14
\pinlabel $1$ at 387 14
\pinlabel $\frac{1}{2}$ at 432 14
\pinlabel $\frac{1}{2}$ at 486 14
\pinlabel $\frac{1}{2}$ at 522 14
\pinlabel ${\color{blue} N}$ at 34 150
\pinlabel ${\color{blue} N}$ at 181 150
\pinlabel ${\color{blue} N}$ at 358 150
\pinlabel ${\color{blue} N}$ at 505 150
\pinlabel ${\color{blue} N}$ at 71 102
\pinlabel ${\color{blue} N}$ at 143 102
\pinlabel ${\color{blue} N}$ at 459 111
\pinlabel ${\color{red} Q}$ at 36 54
\pinlabel ${\color{red} Q}$ at 108 54
\pinlabel ${\color{red} Q}$ at 180 54
\pinlabel ${\color{red} Q}$ at 387 54
\pinlabel ${\color{red} Q}$ at 504 54
\pinlabel ${\color{red} a}$ at 108 231
\pinlabel ${\color{red} a}$ at 432 231
\endlabellist
\centering
\includegraphics[width=\textwidth]{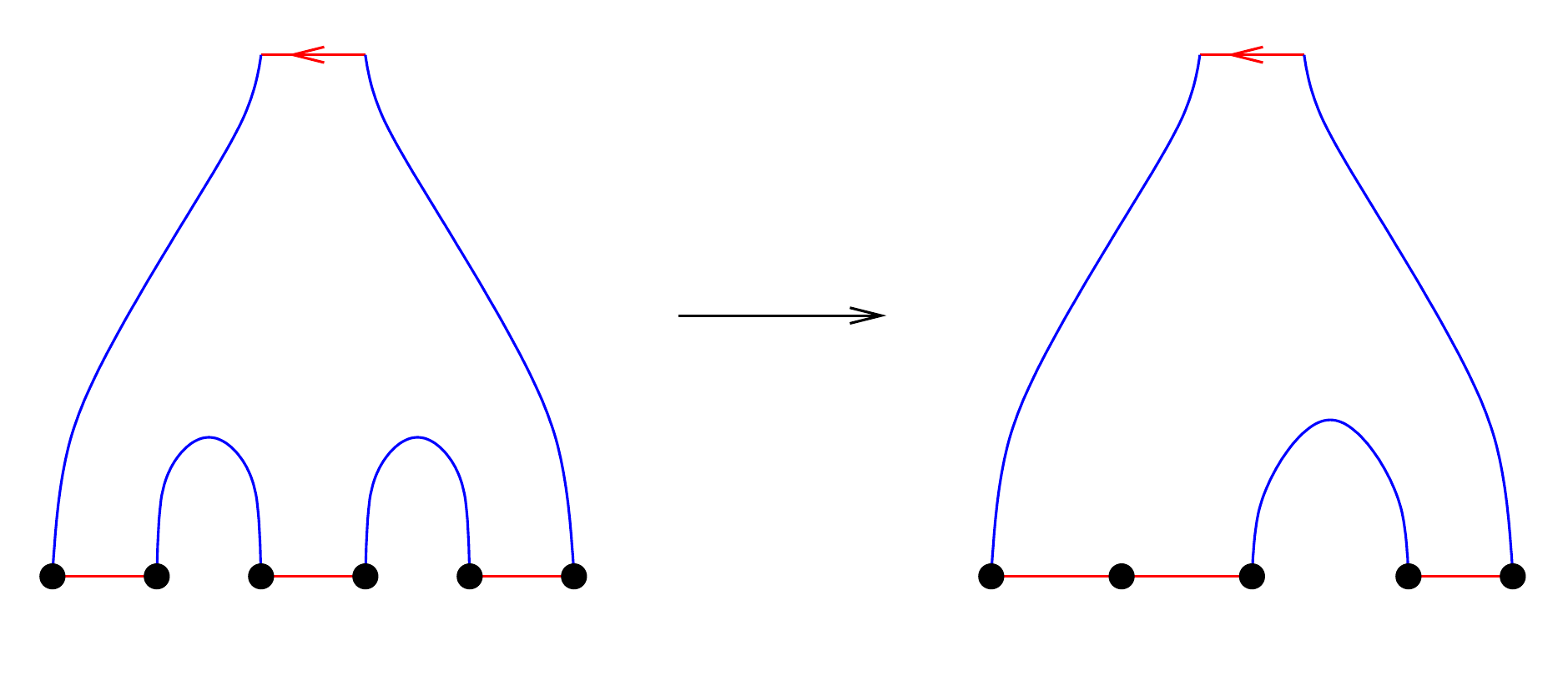}
\caption{
Type $(Lag)$ boundary where an $N$-string disappears.
}
\label{fig:Q-boundary}
\end{figure}

\begin{figure}
\labellist
\small\hair 2pt
\pinlabel $\frac{1}{2}$ at 18 14
\pinlabel $\frac{1}{2}$ at 54 14
\pinlabel $\frac{1}{2}$ at 90 14
\pinlabel $\frac{1}{2}$ at 126 14
\pinlabel $\frac{1}{2}$ at 162 14
\pinlabel $\frac{1}{2}$ at 198 14
\pinlabel $\frac{1}{2}$ at 342 14
\pinlabel $\frac{1}{2}$ at 378 14
\pinlabel $\frac{1}{2}$ at 486 14
\pinlabel $\frac{1}{2}$ at 522 14
\pinlabel $1$ at 432 73
\pinlabel ${\color{blue} N}$ at 34 150
\pinlabel ${\color{blue} N}$ at 181 150
\pinlabel ${\color{blue} N}$ at 358 150
\pinlabel ${\color{blue} N}$ at 505 150
\pinlabel ${\color{blue} N}$ at 71 102
\pinlabel ${\color{blue} N}$ at 143 102
\pinlabel ${\color{blue} N}$ at 432 111
\pinlabel ${\color{red} Q}$ at 36 54
\pinlabel ${\color{red} Q}$ at 108 54
\pinlabel ${\color{red} Q}$ at 180 54
\pinlabel ${\color{red} Q}$ at 360 54
\pinlabel ${\color{red} Q}$ at 504 54
\pinlabel ${\color{red} a}$ at 108 231
\pinlabel ${\color{red} a}$ at 432 231
\endlabellist
\centering
\includegraphics[width=\textwidth]{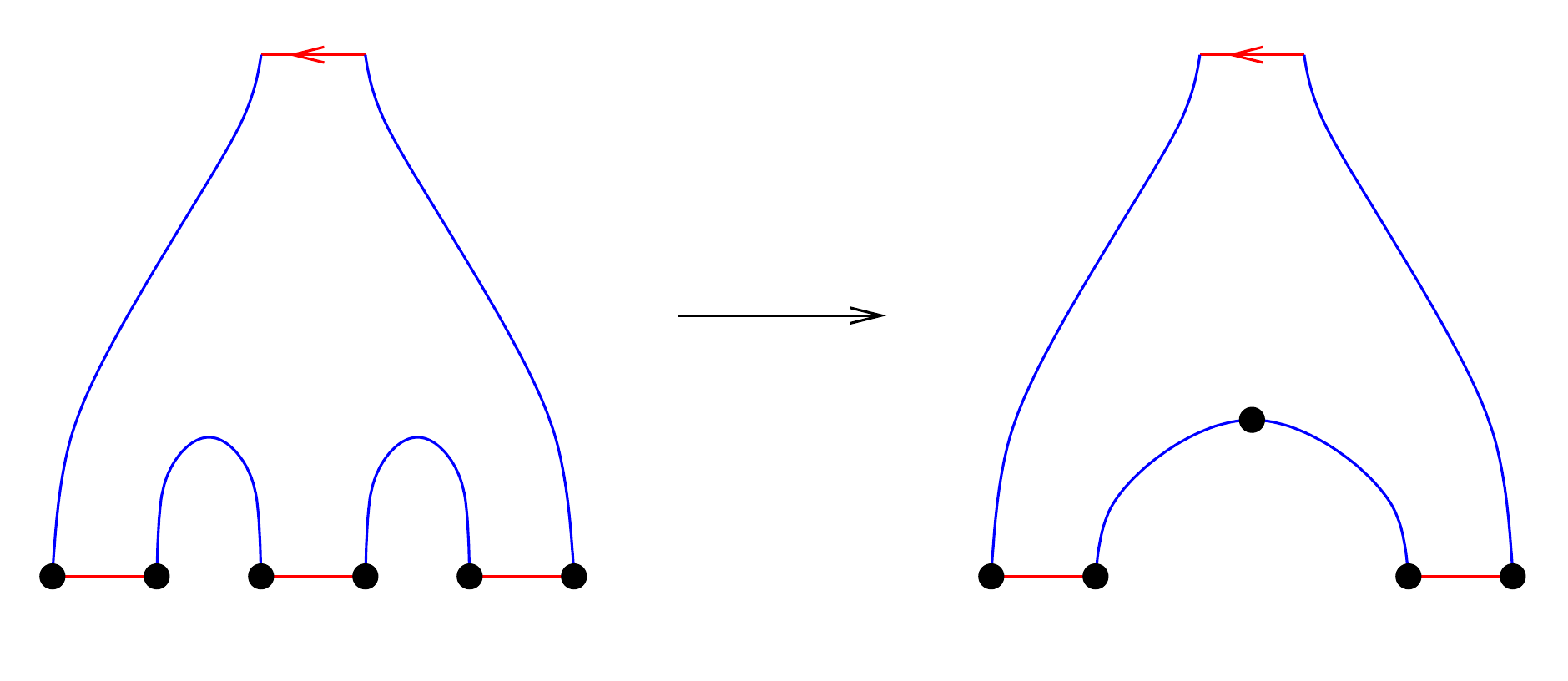}
\caption{
Type $(Lag)$ boundary where a $Q$-string disappears.
}
\label{fig:N-boundary}
\end{figure}

\begin{figure}
\labellist
\small\hair 2pt
\pinlabel $\frac{1}{2}$ at 18 49
\pinlabel $\frac{1}{2}$ at 54 49
\pinlabel $\frac{1}{2}$ at 90 49
\pinlabel $\frac{1}{2}$ at 126 49
\pinlabel $\frac{1}{2}$ at 162 49
\pinlabel $\frac{1}{2}$ at 198 49
\pinlabel $\frac{1}{2}$ at 342 14
\pinlabel $\frac{1}{2}$ at 378 14
\pinlabel $\frac{1}{2}$ at 414 14
\pinlabel $\frac{1}{2}$ at 450 14
\pinlabel $\frac{1}{2}$ at 486 14
\pinlabel $\frac{1}{2}$ at 522 14
\pinlabel ${\color{blue} N}$ at 39 184
\pinlabel ${\color{blue} N}$ at 71 136
\pinlabel ${\color{blue} N}$ at 143 136
\pinlabel ${\color{blue} N}$ at 178 184
\pinlabel ${\color{blue} N}$ at 348 120
\pinlabel ${\color{blue} N}$ at 444 120
\pinlabel ${\color{blue} N}$ at 465 99
\pinlabel ${\color{blue} N}$ at 526 99
\pinlabel ${\color{blue} N}$ at 396 99
\pinlabel ${\color{blue} \R \times \Lambda_K}$ at 364 240
\pinlabel ${\color{blue} \R \times \Lambda_K}$ at 507 240
\pinlabel ${\color{blue} \R \times \Lambda_K}$ at 440 218
\pinlabel ${\color{red} Q}$ at 37 90
\pinlabel ${\color{red} Q}$ at 108 90
\pinlabel ${\color{red} Q}$ at 179 90
\pinlabel ${\color{red} Q}$ at 359 54
\pinlabel ${\color{red} Q}$ at 431 54
\pinlabel ${\color{red} Q}$ at 503 54
\pinlabel ${\color{red} a}$ at 108 270
\pinlabel ${\color{red} a}$ at 432 304
\pinlabel ${\color{red} b_1}$ at 400 184
\pinlabel ${\color{red} b_2}$ at 485 184
\endlabellist
\centering
\includegraphics[width=\textwidth]{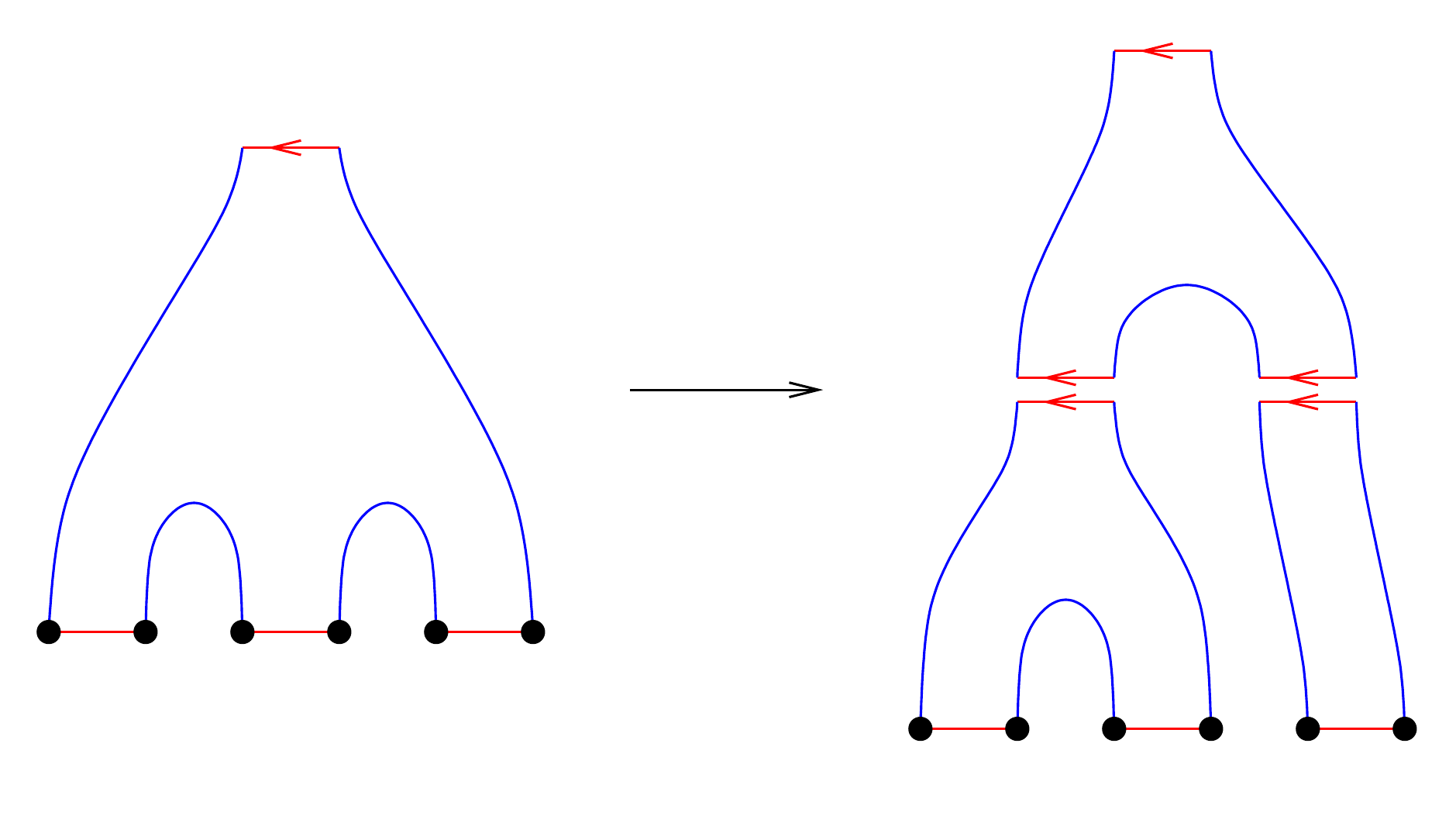}
\caption{
Type $(sy)$ boundary.
}
\label{fig:symp-boundary}
\end{figure}

\begin{proof}
This is a consequence of Theorem~\ref{t:[1,0]}. To motivate the result,
note that the first type of boundary corresponds to two switches
colliding, see Figures~\ref{fig:Q-boundary} and~\ref{fig:N-boundary}.
The second type corresponds to a splitting into
a two level curve with one $\R$-invariant level (of dimension 1) in
the symplectization and one rigid curve (of dimension 0) in
$T^{\ast}Q$, see Figure~\ref{fig:symp-boundary}.
By transversality, compactness, and the dimension formula
this accounts for all the possible boundary phenomena, and by a gluing
argument we find that any such configuration corresponds to a unique
boundary point.
\end{proof}

We conclude this subsection by giving an alternate interpretation of
the first boundary phenomenon in Proposition~\ref{prop:simpleboundary}.
Let $\MM^{\ast}(a;\mathbf{n})$ denote the moduli space corresponding
to $\MM(a;\mathbf{n})$, but with one extra marked point on
the boundary of the disk. Then $\MM^{\ast}(a;\mathbf{n})$ fibers over
$\MM(a;\mathbf{n})$ with fiber $\partial D-\{1,q_1,\dots,q_{m}\}$ and
there is an evaluation map $\ev\colon\MM^{\ast}(a;\mathbf{n})\to
L$. It follows from Theorem \ref{t:emb} that for $|a|=1$ and $|\mathbf{n}|=0$
(and generic data),
$\ev^{-1}(K)$ is a transversely cut out oriented $0$-manifold that
projects injectively into $\MM(a;\mathbf{n})$. We denote its image
by
$$
   \delta\MM(a;\mathbf{n}).
$$
As the notation suggests, this space will be the natural domain for
the string operations $\delta=\delta_Q+\delta_N$.

\begin{prop}\label{prop:b=intersdim1}
If $a$ is a Reeb chord of degree $|a|=1$ and if all entries
$\mathbf{n}$ equal $\tfrac12$, then there is a natural
orientation preserving identification
between $\delta\MM(a;\mathbf{n})$ and $\MM(a;\mathbf{n}'')$, where
$\mathbf{n}''$ is obtained from $\mathbf{n}$ by inserting in
$\mathbf{n}$ a new entry equal to $1$ at the position given by the
marked point.
\end{prop}

\begin{proof}
This is a consequence of Theorem~\ref{t:emb}. Here is the idea.
Consider local coordinates around the marked point in the source and
around $K$ in the target. Then the Taylor
expansions~\eqref{eq:Taylorswitch} with $c_{\frac12}=0$ and $c_{1}\ne
0$ give the map in $\delta\MM(a;\mathbf{n})$ with the marked point
corresponding to $0$. The corresponding Fourier expansions
\eqref{eq:Fourierswitch} present the map as an element in
$\MM(a;\mathbf{n}'')$, where the marked point is replaced by a
puncture. Conversely, translating the Fourier picture to the Taylor
picture proves the other inclusion and hence equality holds. See Section \ref{ss:signsandchainmap} for a discussion of orientations of the moduli spaces involved.
\end{proof}

\subsection{Moduli spaces of dimension two}\label{s:mswitchdim2}
For moduli spaces $\MM(a;\mathbf{n})$ with positive puncture at a Reeb
chord $a$ of degree $|a|=2$, Theorem~\ref{t:dim-moduli} implies the
following:
\begin{itemize}
\item If $|\mathbf{n}|>2$ then $\MM(a;\mathbf{n})=\varnothing$.
\item If $|\mathbf{n}|=2$ then $\MM(a;\mathbf{n})$ is a compact
  $0$-dimensional manifold. This can happen in two ways: either
  exactly one entry in $\mathbf{n}$ equals $\frac32$, or exactly two
  entries equal $1$ and all others equal $\frac12$.
\item If $|\mathbf{n}|=1$ then $\MM(a;\mathbf{n})$ is an oriented
  1-manifold, exactly one entry in $\mathbf{n}$ equals $1$ and all others
  equal $\frac12$.
\item If $|\mathbf{n}|=0$ then $\MM(a;\mathbf{n})$ is an oriented
  2-manifold and all entries in $\mathbf{n}$ equal $\frac12$.
\end{itemize}

It follows by Theorem \ref{t:[2,0]} that the 2-dimensional moduli
spaces of disks with switching boundary condition admit natural
compactifications to 2-manifolds with boundary and corners. The next
result describes the disk configurations corresponding to the boundary
and corner points of these compact surfaces, see
Figures~\ref{fig:Lag-Lag-1}, \ref{fig:Lag-Lag-2}, \ref{fig:sy-Lag}
and~\ref{fig:sy-sy}.

\begin{figure}
\labellist
\small\hair 2pt
\pinlabel $1$ at 164 172
\pinlabel $1$ at 88 95
\pinlabel $1$ at 232 95
\pinlabel $1$ at 190 10
\endlabellist
\centering
\includegraphics[height=0.5\textwidth]{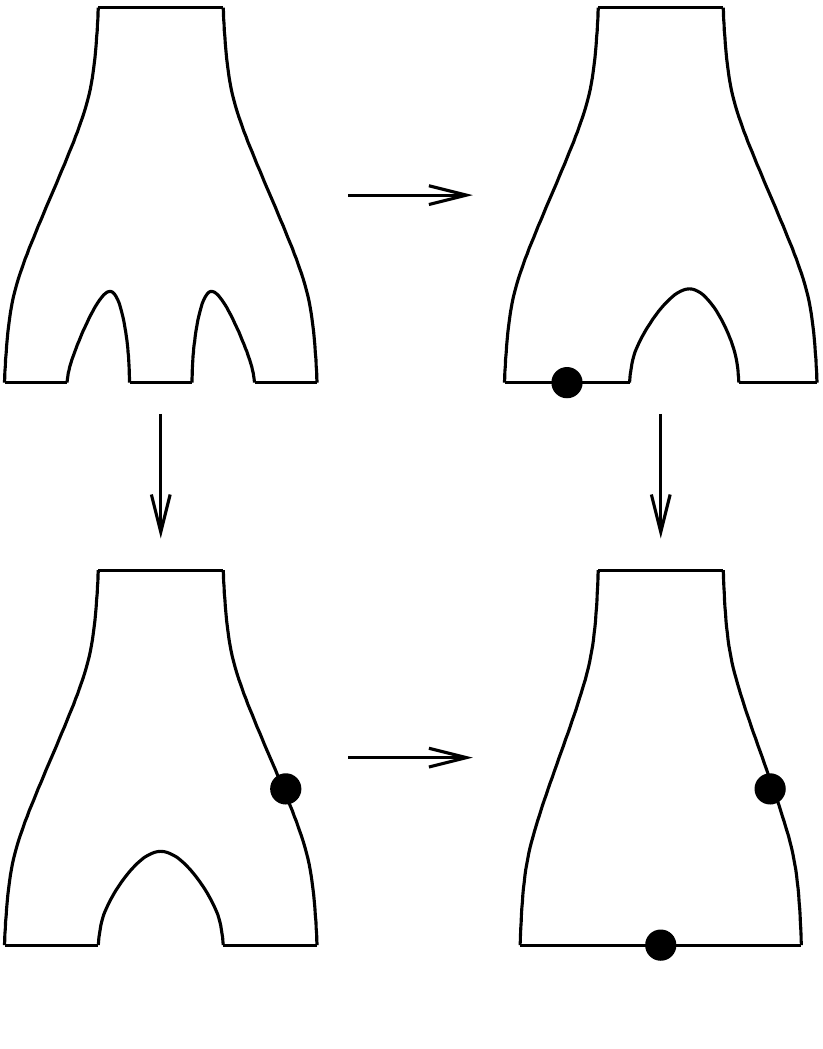}
\caption{
Type $(Lag|Lag)^1$ corner.
}
\label{fig:Lag-Lag-1}
\end{figure}

\begin{figure}
\labellist
\small\hair 2pt
\pinlabel $1$ at 164 172
\pinlabel $1$ at 47 42
\pinlabel $3/2$ at 181 10
\endlabellist
\centering
\includegraphics[height=0.5\textwidth]{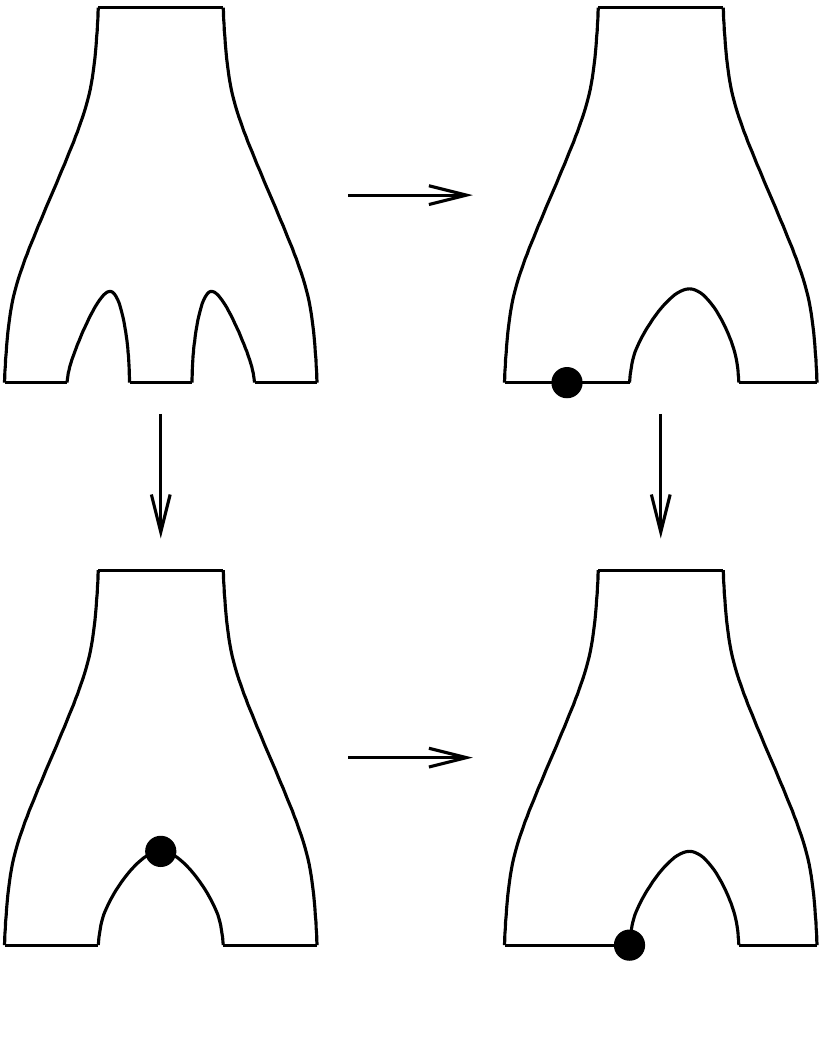}
\caption{
Type $(Lag|Lag)^2$ corner.
}
\label{fig:Lag-Lag-2}
\end{figure}

\begin{figure}
\labellist
\small\hair 2pt
\pinlabel $1$ at 19 20
\pinlabel $1$ at 167 3
\endlabellist
\centering
\includegraphics[height=0.6\textwidth]{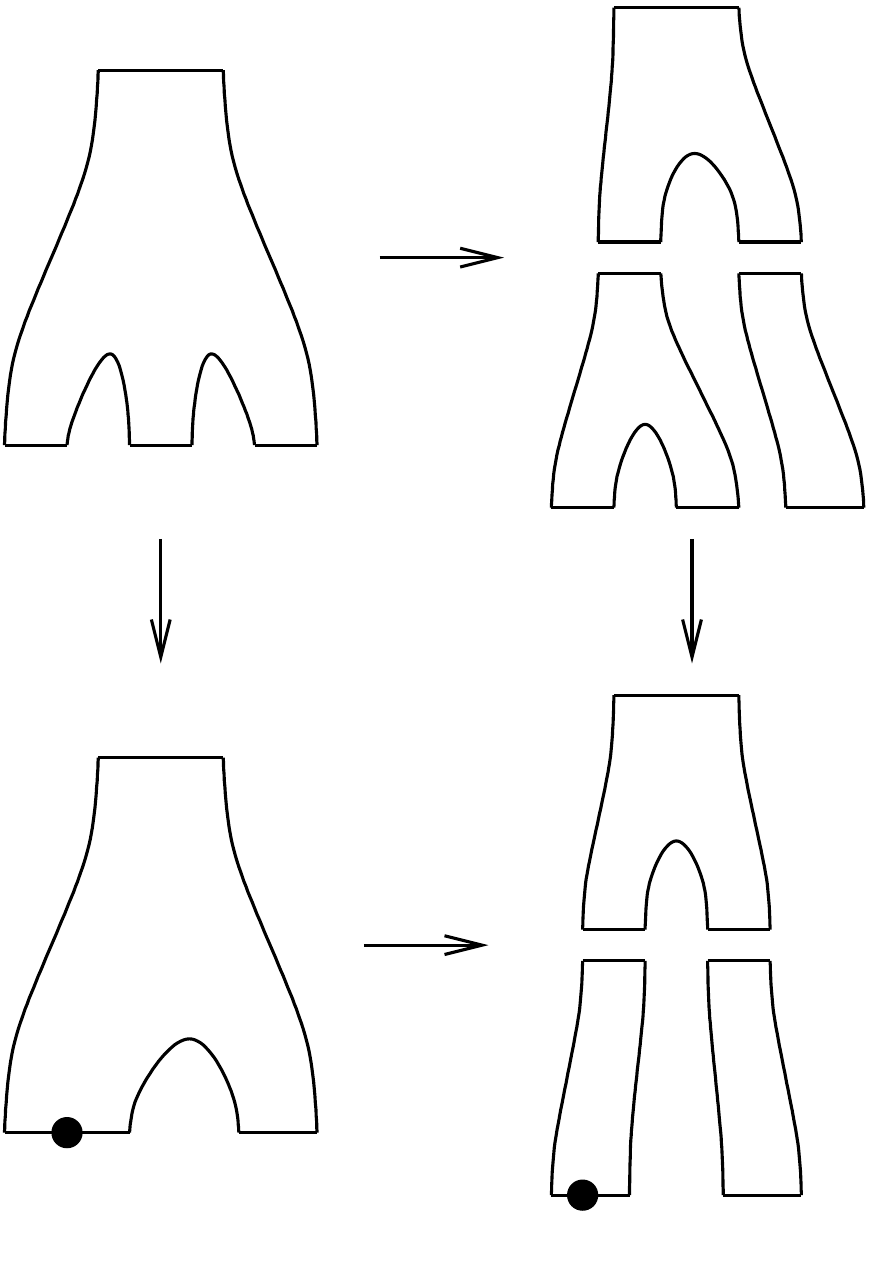}
\caption{
Type $(sy|Lag)$ corner.
}
\label{fig:sy-Lag}
\end{figure}

\begin{figure}
\labellist
\small\hair 2pt
\pinlabel $T$ at 232 84
\endlabellist
\centering
\includegraphics[height=0.7\textwidth]{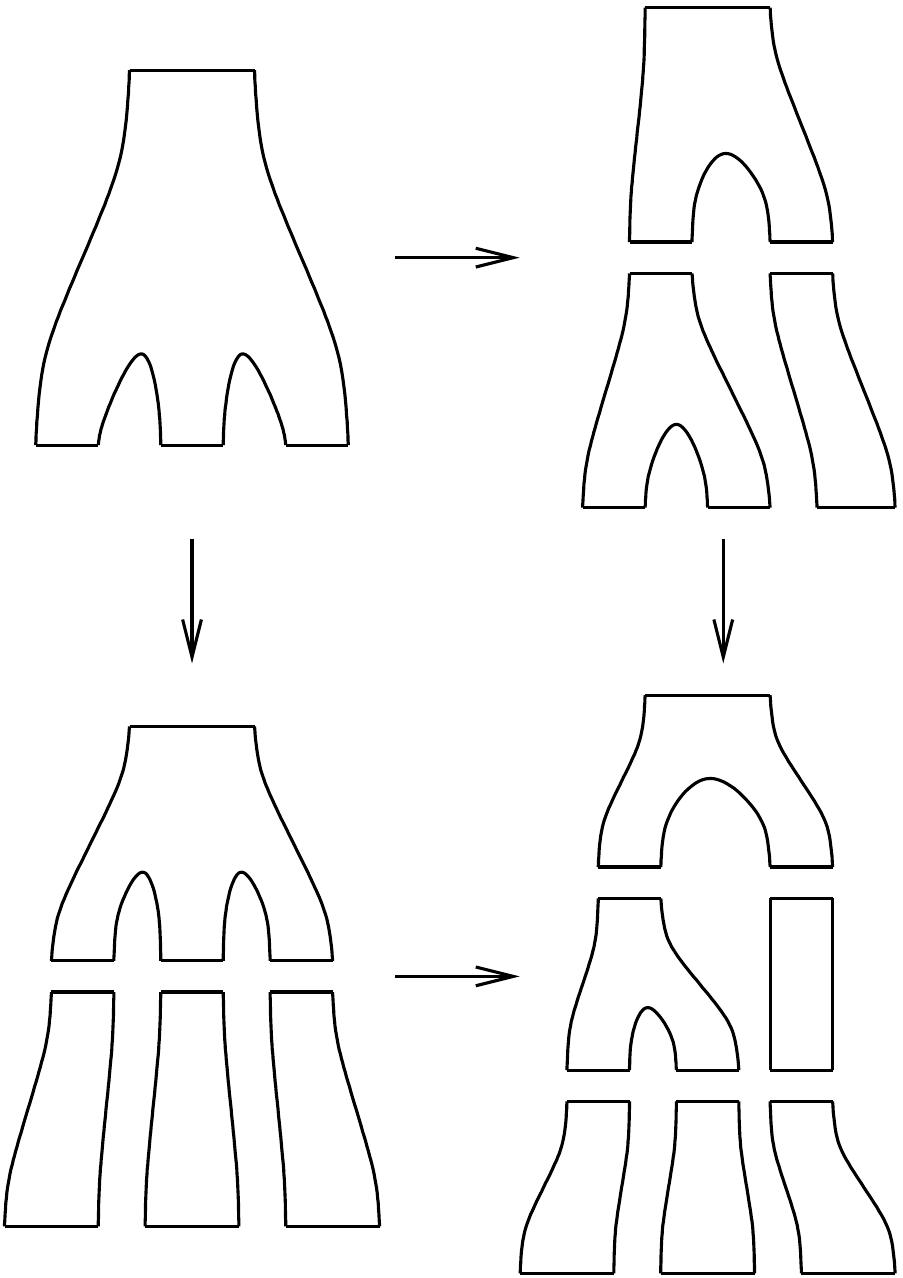}
\caption{
Type $(sy|sy)$ corner.
}
\label{fig:sy-sy}
\end{figure}

\begin{prop}\label{prop:hardboundary}
If $a$ is a Reeb chord of degree $|a|=2$ and if all entries of
$\mathbf{n}$ equal $\tfrac12$, then the 1-dimensional boundary segments
in the boundary of $\MM(a;\mathbf{n})$ consist of the following
configurations:
\begin{description}
\item[$(Lag)$] Moduli spaces $\MM(a;\mathbf{n}')$, where $\mathbf{n}'$ is
  obtained from $\mathbf{n}$ by removing two consecutive
  $\tfrac12$-entries and inserting in their place a $1$.
\item[$(sy)$] Products of moduli spaces
\[
\MM^{\rm sy}(a;\mathbf{b})/\R \;\times \; \Pi_{b_j\in\mathbf{b}}\, \MM(b_j;\mathbf{n}_j),
\]
where $\mathbf{n}$ equals the concatenation of the $\mathbf{n}_j$.
\end{description}
The corner points in the boundary consists of the following configurations:
\begin{description}
\item[$(Lag|Lag)^1$] Moduli spaces $\MM(a;\mathbf{n}')$, where $\mathbf{n}'$ is
  obtained from $\mathbf{n}$ by removing two pairs of consecutive
  $\tfrac12$-entries and inserting $1$'s in their places.
\item[$(Lag|Lag)^2$] Moduli spaces $\MM(a;\mathbf{n}'')$, where $\mathbf{n}''$ is
  obtained from $\mathbf{n}$ by removing three consecutive
  $\frac12$-entries and inserting a $\frac32$ in their place.
\item[$(sy|Lag)$] Products of moduli spaces
\[
\MM^{\rm sy}(a;\mathbf{b})/\R \;\times \; \Pi_{b_j\in\mathbf{b}}\, \MM(b_j;\mathbf{n}_j),
\]
where the concatenation of the $\mathbf{n}_j$ gives $\mathbf{n}$ with
one consecutive pair of $\frac12$-entries removed and a $1$ inserted
in their place.
\item[$(sy|sy)$] Products of moduli spaces
\[
\MM^{\rm sy}(a;\mathbf{b})/\R \;\times \; \prod_{b_j\in\mathbf{b}}\,
\Bigl(\MM^{\rm sy}(b_j;\mathbf{c}_j)/\R\;\times\prod_{c_{jk}\in\mathbf{c_j}}\,
\MM(c_{jk};\mathbf{n}_{jk})\Bigr),
\]
where $\mathbf{n}$ equals the concatenation of the $\mathbf{n}_{jk}$,
and all but one of the $\MM^{\rm sy}(b_j;\mathbf{c}_j)$ are trivial
strips over the Reeb chords $b_j$.
\end{description}
\end{prop}

\begin{proof}
This is a consequence of Theorem~\ref{t:[2,0]}. The descriptions of the
boundary segments are analogous to the boundary phenomena of
Proposition~\ref{prop:simpleboundary}. At a type $(Lag|Lag)^1$ corner we
have two pairs of switches colliding. Local coordinates in the moduli
space around this configuration can be taken as the lengths of the
corresponding short boundary segments, which is a product of two
half-open intervals. At a type $(Lag|Lag)^2$ corner there are likewise
two short boundary segments that give local coordinates on the moduli
space, see Figure \ref{fig:Lag-Lag-2}.
At a type $(sy|Lag)$ corner the two
parameters are the length of the short boundary segment and the gluing
parameter for the two-level curve.
Finally, at a type $(sy|sy)$ corner the two
parameters are the two gluing parameter for the three-level curve.
\end{proof}

We next give alternate interpretations of the boundary phenomena in
Proposition~\ref{prop:hardboundary}.
Recall the notation $\MM^{\ast}(a;\mathbf{n})$ for the moduli space
corresponding to $\MM(a;\mathbf{n})$ in which the disks have an
additional free marked point $*$ on the boundary. It comes with an
evaluation map $\ev\colon \MM^{\ast}(a;\mathbf{n})\to L$ and a
projection $\pi:\MM^{\ast}(a;\mathbf{n})\to \MM(a;\mathbf{n})$
forgetting the marked point, and we denote
$\delta\MM(a;\mathbf{n})=\ev^{-1}(K)$.

\begin{prop}\label{prop:b=intersdim2}
If $a$ is a Reeb chord of degree $|a|=2$ and if all entries
$\mathbf{n}$ equal $\tfrac12$, then there is a natural orientation
preserving identification
between $\delta\MM(a;\mathbf{n})$ and $\MM(a;\mathbf{n}'')$, where
$\mathbf{n}''$ is obtained from $\mathbf{n}$ by inserting in
$\mathbf{n}$ a new entry equal to $1$ at the position given by the
marked point.

The moduli space $\delta\MM(a;\mathbf{n})\subset
\MM^\ast(a;\mathbf{n})$ is an embedded curve with boundary. Its
boundary consists of transverse intersections with the boundary of
$\MM^\ast(a;\mathbf{n})$, corresponding to degenerations of type
$(sy|Lag)$ and $(Lag|Lag)^1$ involving the marked point $*$, and to points
in the interior of $\MM^\ast(a;\mathbf{n})$, corresponding to
degenerations of type $(Lag|Lag)^2$ involving the marked point $*$.

The projection $\pi(\delta\MM(a;\mathbf{n}))\subset
\MM(a;\mathbf{n})$ is an immersed curve with boundary and transverse
self-intersections. Its boundary consists of transverse intersections
with the boundary of $\MM^\ast(a;\mathbf{n})$. See Figure~\ref{fig:delta-M}.
\end{prop}

\begin{figure}
\labellist
\small\hair 2pt
\pinlabel ${\color{blue} (sy|sy)}$ at 0 162
\pinlabel ${\color{blue} (sy)}$ at 69 229
\pinlabel ${\color{blue} (sy|Lag)}$ at 108 318
\pinlabel ${\color{blue} (Lag)}$ at 213 294
\pinlabel ${\color{blue} (Lag|Lag)^2}$ at 323 316
\pinlabel ${\color{blue} (sy)}$ at 68 91
\pinlabel ${\color{blue} (sy|Lag)}$ at 107 12
\pinlabel ${\color{blue} (Lag)}$ at 215 30
\pinlabel ${\color{blue} (Lag|Lag)^1}$ at 324 12
\pinlabel ${\color{red} (Lag|Lag)^2}$ at 146 115
\endlabellist
\centering
\includegraphics[width=0.7\textwidth]{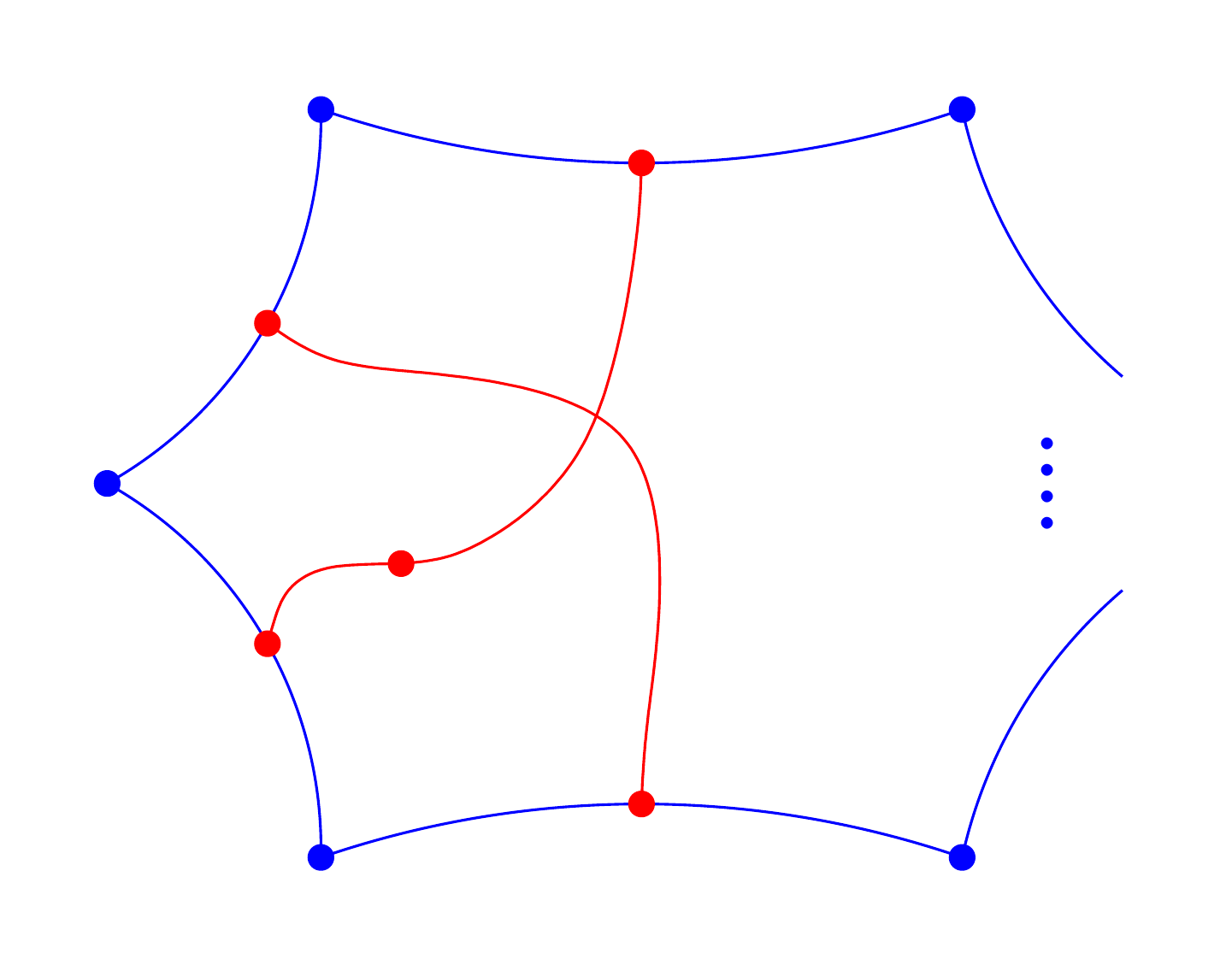}
\caption{
The $2$-dimensional moduli space $\MM(a;\mathbf{n})$ and the immersed
curve $\pi(\delta\MM(a;\mathbf{n}))$.
}
\label{fig:delta-M}
\end{figure}

\begin{proof}
This is a consequence of Theorem~\ref{t:imm}. Here is a sketch.
The proof of the first statement is analogous to that of
Proposition~\ref{prop:b=intersdim1}, looking at Taylor and Fourier
expansions. That $\delta\MM(a;\mathbf{n})$ is an embedded curve with
boundary follows from transversality of the evaluation map $\ev\colon
\MM^{\ast}(a;\mathbf{n})\to L$ to the knot $K$, which holds for
generic almost complex structure. More refined transversality
arguments show that the projection $\pi(\delta\MM(a;\mathbf{n}))$ is
an immersed curve with transverse self-intersections corresponding to
holomorphic disks that meet the knot twice at non-corner points on their
boundary.

For the other statements, note that each stratum of
$\delta\MM(a;\mathbf{n})$ corresponds to a moduli space
$\MM(a;\mathbf{n'})$, where $\mathbf{n'}$ is obtained from
$\mathbf{n}$ by inserting an entry $1$ corresponding to the marked
point $*$. It follows from Proposition~\ref{prop:hardboundary} that
boundary points of $\delta\MM(a;\mathbf{n})$ correspond to
degenerations of types $(sy|Lag)$, $(Lag|Lag)^1$ and $(Lag|Lag)^2$
involving the point $*$. The first two correspond to transverse
intersections of $\pi(\delta\MM(a;\mathbf{n}))$ with boundary strata
of $\MM(a;\mathbf{n})$ of types $(sy)$ and $(Lag)$, respectively.
A dimension argument shows that degenerations of type $(Lag|Lag)^2$
involving the point $*$ cannot meet the boundary of
$\MM^\ast(a;\mathbf{n})$, so they correspond to boundary points of
$\delta\MM(a;\mathbf{n})$ in the interior of
$\MM^\ast(a;\mathbf{n})$. They appear in pairs corresponding
to holomorphic disks in which the marked point $*$ has approached a corner
from the left or right to form a new corner of weight $3/2$. In
$\delta\MM(a;\mathbf{n})$ the two configurations on a pair are
distinct (formally, they are distinguished by the position of the
marked point $*$ on the $3$-punctured constant disk attached at the
weight $3/2$ corner), so they give actual boundary points. In the
projection $\pi(\delta\MM(a;\mathbf{n}))$ the two configuration become
equal and thus give an interior point, hence
$\pi(\delta\MM(a;\mathbf{n}))$ has no boundary points in the interior
of $\MM(a;\mathbf{n})$.

See Section \ref{ss:signsandchainmap} for a discussion of orientations of the moduli spaces involved in these arguments.
\end{proof}

\subsection{The chain map}
We can summarize the description of the moduli spaces of punctured
holomorphic disks with switching boundary conditions in the preceding
subsections as follows. For all Reeb chords $a$ and all integers $\ell\geq
0$ the compactified moduli spaces
$$
   \ol\MM_\ell(a) :=
   \ol\MM(a;\underbrace{\tfrac12,\dots,\tfrac12}_{2\ell})
$$
are compact oriented manifolds with boundary and corners of dimension $|a|$ whose
codimension $1$ boundaries satisfy the relations
\begin{equation}\label{eq:domain-ell}
      \p\ol\MM_\ell(a) = \ol\MM_\ell(\p_\Lambda a) \cup
      -\delta\ol\MM_{\ell-1}(a),
\end{equation}
where $\p_\Lambda a=\p^{\rm sy}a$ and $\delta\ol\MM_{\ell-1}(a)$ is the closure in $\ol\MM_{\ell-1}(a)$ of the set 
$$
\delta\MM_{\ell-1}(a) := \delta\MM(a;\underbrace{\tfrac12,\dots,\tfrac12}_{2\ell-2})
$$ 
introduced in Proposition~\ref{prop:b=intersdim2}.
Again we refer to Section \ref{ss:signsandchainmap} for a description of the orientations involved.

\begin{prop}\label{prop:chain-map-ell}
There exist smooth triangulations of the spaces $\ol\MM_\ell(a)$ and
generic chains of broken strings
$$
   \Phi_\ell(a): \ol\MM_\ell(a)\to\Sigma^\ell
$$
(understood as singular chains by summing up their restrictions to the
simplices of the triangulations) satisfying the relations
\begin{equation}\label{eq:chain-map-ell}
      \p\Phi_\ell(a) = \Phi_\ell(\p_\Lambda a)
      -(\delta_Q+\delta_N)\Phi_{\ell-1}(a).
\end{equation}
\end{prop}

\begin{proof}
The idea of the proof is very simple: After connecting the end points
of $a$ to the base point $x_0$ by capping paths, a suitable
parametrization (explained below)
of the boundary of a holomorphic disk $u\in\MM_\ell(a)$ determines a
broken string $\p(u)\in\Sigma^\ell$. Thus we get maps
$$
   \wt\Phi_\ell(a):\MM_\ell(a) \to \Sigma^\ell,\qquad u\mapsto\p(u)
$$
and the relations~\eqref{eq:chain-map-ell} should follow
from~\eqref{eq:domain-ell}. However, the map $\wt\Phi_\ell(a)$ in general does not
extend to the compactification $\ol\MM_\ell$ as a map to $\Sigma^\ell$
because on the boundary some $Q$- or $N$-string can disappear in the
limit. We will remedy this by suitably modifying the maps
$\wt\Phi_\ell(a)$ near the boundaries (inserting spikes).

Before doing this, let us discuss parametrizations of the broken
string $\p(u)$ for $u\in\MM_\ell(a)$. Near a switch we can pick
holomorphic coordinates on the domain (with values in the upper
half-disk) and the target (provided by Lemma~\ref{l:knotnbhd})
in which the normal projection of $u$ consists of two holomorphic
functions near a corner as in Section~\ref{sec:holo}. The discussion
in that section shows that in these coordinates $\p(u)$ satisfies the
matching conditions on the $m$-jets required in the definition a
broken string. We take near each corner a parametrization of $\p(u)$
induced by such holomorphic coordinates and extend them arbitrarily
away from the corners to make $\p(u)$ a broken string in the sense of
Definition~\ref{def:string}. Note that the space of such
parametrizations is contractible.

Now we proceed by induction over $|a|=0,1,2$.

{\bf Case $|a|=0$:}
In this case $\ol\MM_\ell(a)$ consists of finitely many oriented points and we
set $\Phi_\ell(a)(u):=\p(u)$ (picking a parametrization of the
boundary as above).
\smallskip

{\bf Case $|a|=1$:}
We proceed by induction on $\ell=0,1,\dots$. For $\ell=0$, on the
boundary $\p\ol\MM_0(a) = \ol\MM_0(\p_\Lambda a)$ we are already
given the map $\Phi_0(\p_\Lambda a)$. We extend it to a map
$\Phi_0(a):\ol\MM_0(a) \to \Sigma^0$ by sending $u$ to $\p(u)$ with
parametrizations matching the given ones on $\p\ol\MM_0(a)$, so that
$\p\Phi_0(a) = \Phi_0(\p_\Lambda a)$ holds.

Now suppose that we have already defined $\Phi_0(a),\dots,\Phi_{\ell-1}(a)$
such that the relations~\eqref{eq:chain-map-ell} hold up to
$\ell-1$. 
According to \eqref{eq:domain-ell}, the boundary $\p \ol\MM_\ell(a)$ 
is identified with the union of domains of the maps on the right hand side of~\eqref{eq:chain-map-ell}. On the other hand, on the interior $\MM_\ell(a)$ we are given the map $\wt{\Phi}_\ell(a)$ described above. Furthermore, by Proposition~\ref{prop:simpleboundary} and Remark \ref{r:breakingmodelclose},
elements $u$ close to the boundary points in $\delta\ol\MM_{\ell-1}(a) \subset 
\p\ol\MM_\ell(a)$ have spikes (shrinking as $u$ tends to the boundary)
roughly in the same direction as those on the boundary. So near $\p\ol\MM_\ell(a)$ we can interpolate 
between the map on the boundary given by the right hand side of~\eqref{eq:chain-map-ell} and the map $\wt\Phi_\ell(a)$ on the 
interior to obtain a map $\Phi_\ell(a):\ol\MM_\ell(a)\to\Sigma^\ell$
satisfying~\eqref{eq:chain-map-ell}. Since the modification of
$\wt\Phi_\ell(a)$ can be done away from the finite set
$\delta\ol\MM_\ell(a) \subset \MM_\ell(a)$, 
$\Phi_\ell(a)$ is a generic
$1$-chain of broken strings. This concludes the inductive step. Since
we are dealing with $1$-chains, a smooth triangulation just
amounts to a parametrization of the components of $\ol\MM_\ell(a)$ by
intervals whose boundary points avoid the set $\delta\ol\MM_\ell(a)$.
\smallskip

{\bf Case $|a|=2$:}
We proceed again by induction on $\ell=0,1,\dots$. For $\ell=0$, we
again define $\Phi_0(a):\ol\MM_0(a) \to \Sigma^0$ by sending $u$ to
$\p(u)$, with parametrizations matching the given ones on
$\p\ol\MM_0(a)$, so that $\p\Phi_0(a) = \Phi_0(\p_\Lambda a)$ holds.

Now suppose that we have already defined $\Phi_0(a),\dots\Phi_{\ell-1}(a)$
and triangulations of their domains such that they are generic
$2$-chains of broken strings and the relations~\eqref{eq:chain-map-ell} hold up to $\ell-1$. 
As in the case of 1-chains, the boundary $\p\ol\MM_\ell(a)$ is identified via~\eqref{eq:domain-ell} with the union of domains of the right hand side of~\eqref{eq:chain-map-ell}, so as before we define $\Phi_\ell(a)$ on that boundary via the maps $\Phi_\ell(\p_\Lambda a)$ resp.~$\delta\Phi_{\ell-1}(a)$. By induction hypothesis, these maps coincide at corner points. Note that the map
$\delta\Phi_{\ell-1}(a)$ inserts spikes at the intersection points
with the knot. According to Proposition~\ref{prop:hardboundary} and Remark \ref{r:breakingmodelclose},
elements $u$ close to the codimension one boundary strata $\delta\ol\MM_{\ell-1}(a)$ have spikes roughly in the same direction as those on the boundary (shrinking in size as $u$ tends to the boundary). Elements close to a corner point where
two boundary strata of $\delta\ol\MM_{\ell-1}(a)$ meet have two spikes
roughly in the same directions as those on the nearby boundary strata (which
both shrink as $u$ tends to the corner point), see Remark \ref{r:breakingmodelclose}.
So we can interpolate between the given map on the boundary $\p\ol\MM_\ell(a)$ and the map $\wt\Phi_\ell(a)$ on the interior $\MM_\ell(a)$ to
obtain a map $\Phi_\ell(a):\ol\MM_\ell(a)\to\Sigma^\ell$
satisfying~\eqref{eq:chain-map-ell}.

Recall that $\delta\ol\MM_\ell(a)$ is an immersed $1$-dimensional
submanifold with finitely many transverse self-intersections in the
interior, and which meets the boundary transversely away from the
corners. The modification of $\wt\Phi_\ell(a)$ can be done away from
the finite set of self-intersections of $\delta\ol\MM_\ell(a)$ in the
interior. Moreover, the modification of $\wt\Phi_\ell(a)$ near the
boundary only involves inserting spikes at switching points of broken
strings, which can be performed away from the finitely many interior
intersection points of the broken strings with the knot and thus does
not affect $\delta\ol\MM_\ell(a)$.

We pick a smooth triangulation of $\ol\MM_\ell(a)$ transverse to
$\delta\ol\MM_\ell(a)$ (i.e., transverse to its $1$-dimensional strata
as well as its self-intersection points) and inducing the given
triangulation on the boundary. By the discussion in the preceding
paragraph, $\Phi_\ell(a)$ (interpreted as the sum over its restriction
to simplices) is a generic $2$-chain of broken strings.
This concludes the inductive step and thus the proof of
Proposition~\ref{prop:chain-map-ell}.
\end{proof}

Given a Reeb chord $a$, we define
$$
   \Phi(a) := \sum_{\ell=0}^\kappa\Phi_\ell(a)\in
   C(\Sigma)=\bigoplus_{\ell=0}^\infty C(\Sigma^\ell).
$$
Here $\kappa$ is the constant from the Finiteness Theorem~\ref{thm:finite}.
The relation~\eqref{eq:chain-map-ell} for the chains $\Phi_\ell(a)$
translates into
\begin{equation}\label{eq:chain-map}
      \p\Phi(a) = \Phi(\p^\sy a)
      -\delta\Phi(a),\qquad \delta=\delta_Q+\delta_N.
\end{equation}
Given a $d$-simplex of Reeb strings $\mathbf{a}=\alpha_{1}a_1\dots
a_m\alpha_{m+1}:\Delta\to\RR^m$ we define
$$
   \Phi(\mathbf{a}) := \alpha_{1}\Phi(a_1)\dots
   \alpha_m\Phi(a_m)\alpha_{m+1}\in C(\Sigma).
$$
Here the boundary arcs are concatenated in the obvious way to obtain
broken strings. For singular simplices $\Delta_i$ appearing as
domains in $\Phi(a_i)$, the corresponding term in $\Phi(\mathbf{a})$
has by our orientation convention the domain
$$
   \Delta\times\Delta_1\times\cdots\times\Delta_m
$$
in this order of factors.

\begin{thm}\label{t:Phichainmap}
The map $\Phi$ is a chain map from $(C_*(\RR),\p_\Lambda)$ to
$(C_*(\Sigma),\partial+\delta_{Q}+\delta_{N})$.
\end{thm}

\begin{proof}
Using~\eqref{eq:chain-map} we compute for $\mathbf{a}\in C_d(\RR)$
as above, with $*=d+|a_1|+\cdots+|a_{i-1}|$:
\begin{align*}
   \p\Phi(\mathbf{a})
   &= \Phi(\p^\sing\mathbf{a}) +
   \sum_{i=1}^m(-1)^{*} \alpha_{1}\Phi(a_1)\alpha_2\cdots
   \p\Phi(a_i)\cdots\alpha_m\Phi(a_m)\alpha_{m+1} \cr
   &= \Phi(\p^\sing\mathbf{a}) +
   \sum_{i=1}^m(-1)^{*} \alpha_{1}\Phi(a_1)\alpha_2\cdots
   \Bigl(\Phi(\p^\sy
   a_i)-\delta\Phi(a_i)\Bigr)\cdots\alpha_{m+1} \cr
   &= \Phi(\p^\sing\mathbf{a}) +
   \Phi(\p^\sy \mathbf{a})-\delta\Phi(\mathbf{a}).
\end{align*}
Since $\p_\Lambda=\p^\sing+\p^\sy$, this proves the theorem.
\end{proof}

{\bf Compatibility with length filtrations. }
Holomorphic disks with switching boundary conditions have a length
decreasing property that leads to the chain map $\Phi$ respecting the
length (or action) filtration, which is central for our isomorphism
proof. Let $u\in\MM(a;\mathbf{n})$ be a holomorphic disk with $k$
boundary segments that map to $Q$. Let $\sigma_1,\dots,\sigma_{k}$ be
the corresponding curves in $Q$ and let $L(\sigma_{i})$ denote the
length of $\sigma_{i}$. Recall that the Reeb chord $a$ is the lift of
a binormal chord on the link $K$ and that the action $\int_{a} pdq$ of
$a$ equals the length of the underlying chord in $Q$, which we write
as $L(a)$. In Section~\ref{sec:length-estimates2} we utilize the
positivity of a scaled version of the contact form on holomorphic
disks to show the following result (Proposition~\ref{prop:length-estimate}).

\begin{prop}\label{prop:disklength}
If $u\in \MM(a;\mathbf{n})$ is as above then
\[
   \sum_{i=1}^{k} L(\sigma_{i}) \leq L(a),
\]
with equality if and only if $u$ is a trivial half strip over a binormal chord.
\end{prop}

Recall that both chain complexes $(C_*(\RR),\p_\Lambda)$ and
$(C_*(\Sigma),\partial+\delta_{Q}+\delta_{N})$ carry length
filtrations that were defined in Sections~\ref{sec:leg}
and~\ref{ss:length-filt}, respectively. Recall also that the length
filtration on $C_*(\Sigma)$ does not count the lengths of
$Q$-spikes. Hence the insertion of $Q$-spikes in the definition of the
chain map $\Phi$ does not increase length and 
Proposition~\ref{prop:disklength} implies

\begin{cor}\label{cor:respect-length}
The chain map $\Phi$ in Theorem~\ref{t:Phichainmap} respects the
length filtrations, i.e., it does not increase length.
\end{cor}


\section{Proof of the isomorphism in degree zero}\label{sec:iso}

In the previous section we have constructed a chain map
$\Phi:(C_*(\RR),\p_\Lambda)\to (C_{\ast}(\Sigma),D)$, where $D=\partial+\delta_{Q}+\delta_{N}$. 
In this section we finish the proof of Theorem~\ref{thm:main} by showing that
the induced map $\Phi_*:H_0(\RR,\p_\Lambda)\to H_0(\Sigma,D)$ in degree zero is
an isomorphism.
Whereas the results in the previous section hold for any
$3$-manifold $Q$ with a metric of nonpositive curvature which is convex at
infinity, in this section we need to restrict to the case $Q=\R^3$
with its Euclidean metric. This restriction will allow us to obtain crucial 
control over the straightening procedure for $Q$-strings described in Proposition~\ref{prop:pl-lin} (see the comment in Remark~\ref{rem:R3} below).

As a first step, we will slightly extend the definition of broken
strings to include piecewise linear $Q$-strings.
A relatively simple approximation result will show that the inclusion
of broken strings with piecewise linear $Q$-strings into all broken
strings induces an isomorphism on string homology in degree 0.

The central piece of the argument will then consist of deforming the
complex of broken strings with piecewise linear $Q$-strings into the
subcomplex of those with {\em linear} $Q$-strings.

It is important that both of these reduction steps can be done {\em
  preserving the length filtration on $Q$-strings}. The final step of
the argument then consists of comparing the contact homology
$H_0(\RR,\p_\Lambda)$ with the homology of the chain
complex of broken strings with linear $Q$-strings. At this stage, we
will use the length filtrations to reduce to the comparison of homology in
degrees 0 and 1 in small length windows containing at most one
critical value.

\subsection{Approximation by piecewise linear $Q$-strings}\label{ss:pl}

In the following we enlarge the space of broken $C^m$-strings
$\Sigma$, keeping the same notation, to allow for $Q$-strings to be
{\em piecewise $C^m$}. 
Here a $Q$-string $s_{2i}:[a_{2i-1},a_{2i}]\to Q$ is called piecewise $C^m$
if there exists a subdivision $a_{2i-1}=b_0<b_1<\cdots<b_r=a_{2i}$
such that the restriction of $s_{2i}$ to each subinterval $[b_{j-1},b_j]$ is $C^m$.  
For a generic $d$-chain $S:\Delta_d\to\Sigma^\ell$ ($d=0,1,2$) we require that
the number of subdivision points on each $Q$-string is constant over
the simplex $\Delta_d$. The subdivision points can vary smoothly over
$\Sigma_d$ but have to remain distinct. If for some subdivision point $b_j$
the two $C^m$-strings meeting at $b_j$ fit together in a
$C^m$-fashion for all $\lambda\in\Delta_d$, then we identify
$S$ with the generic $d$-chain obtained by removing the subdivision
point $b_j$. 

We allow $Q$-strings in a generic $d$-chain $S$ to meet the knot $K$
at a subdivision point $b_j$, provided at such a point the
derivatives from both sides satisfy the genericity conditions in
Definition~\ref{def:generic-chain}. 
If this occurs for some parameter value $\lambda^*\in\Delta_2$ in a
generic $2$-chain, then we require in addition 
that the corresponding $Q$-string meets $K$ at the subdivision point $b_j(\lambda)$ for all
$\lambda$ in the component of $\lambda^*$ in the domain $M_{\delta_Q}$
of $\delta_QS$ defined in Section~\ref{ss:string-op}. 
These conditions ensure that the operator 
$D=\partial+\delta_{Q}+\delta_{N}$ extends to generic chains of piecewise
$C^m$ strings satisfying the relations in Proposition~\ref{prop:string-relations}.

The subspace $\Sigma_\pl\subset\Sigma$ of broken strings whose $Q$-strings are {\em piecewise linear} give rise to an inclusion of a $D$-subcomplex
\begin{equation}\label{eq:ipl}
   C_*(\Sigma_\pl)\stackrel{i_\pl}\hookrightarrow C_*(\Sigma).
\end{equation}
For this to hold, we choose the $Q$-spikes inserted under the map
$\delta_N$ to be degenerate $3$-gons, i.e., short segments orthogonal
to the knot traversed back and forth. Then $C_*(\Sigma_\pl)$ becomes a
$D$-subcomplex. 

We will also consider the subspace $\Sigma_\lin \subset \Sigma_\pl$ of broken 
closed strings whose $Q$-strings are (essentially) {\em linear}: any two 
points $x_1,x_2 \in K$ determine a unique line segment $[x_1,x_2]$ in $\R^3$ 
connecting them. For technical reasons, special care has to be taken when 
such a linear $Q$-string becomes very short. Indeed, near the diagonal 
$\Delta \subset K \times K$  we deform the segments to piecewise linear 
strings with one corner in such a way that at each 
point of the diagonal, instead of a segment of length zero we have a 
degenerate 3-gon as above, i.e., a short spike in direction of the 
curvature of the knot (which we assume vanishes nowhere). 
Now $\Sigma_\lin \subset \Sigma_\pl$ consists of all broken 
closed strings whose $Q$-strings are constant speed parametrizations of such 
(possibly deformed) segments. In this way, 
\begin{equation}\label{eq:ilin}
   C_*(\Sigma_\lin)\stackrel{i_\lin}\hookrightarrow C_*(\Sigma_\pl)
\end{equation}
will be an inclusion of a $D$-subcomplex.

Recall from Section~\ref{ss:length-filt} that these complexes are
filtered by the length $L(\beta)$, i.e.~the maximum of the total
length of $Q$-strings over all parameter values of the chain,
where in the length we do not count $Q$-spikes.
With these notations, we have the following approximation result.

\begin{prop}\label{prop:pl}
There exist maps
$$
   \Bf_0:C_0(\Sigma) \to C_0(\Sigma_\pl),\qquad
   \Bf_1:C_1(\Sigma) \to C_1(\Sigma_\pl)
$$
and
$$
   \BH_0:C_0(\Sigma) \to C_1(\Sigma),\qquad
   \BH_1:C_1(\Sigma) \to C_2(\Sigma)
$$
satisfying with the map $i_\pl$ from~\eqref{eq:ipl}:
\begin{enumerate}
\item $\Bf_0i_\pl=\id$ and $D\BH_0= i_\pl\Bf_0-\id$;
\item $\Bf_1i_\pl=\id$ and $\BH_0D + D\BH_1 = i_\pl\Bf_1-\id$;
\item $\Bf_0$, $\BH_0$, $\Bf_1$ and $\BH_1$ are (not necessarily strictly) length-decreasing.
\end{enumerate}
\end{prop}

\begin{proof}
We first define $\Bf_0$ and $\BH_0$. Given $\beta\in C_0(\Sigma)$, we
pick finitely many subdivision points $p_i$ on the $Q$-strings in $\beta$ (which
include all end points) and define $\BH_0\beta$ to be the straight
line homotopy from $\beta$ to the broken string $\Bf_0\beta$ whose
$Q$-strings are the piecewise linear strings connecting the $p_i$. We
choose the subdivision so fine that the $Q$-strings in $\BH_0\beta$
remain transverse to $K$ at the end points and do not meet $K$ in the
interior. The $N$-strings are just slightly rotated near the end
points to match the new $Q$-strings, without creating intersections
with $K$. Then $\BH_0\beta$ is a generic $1$-chain in $\Sigma$
satisfying
$$
   \p\BH_0\beta=\Bf_0\beta-\beta,\qquad
   \delta_Q\BH_0\beta=\delta_N\BH_0\beta=0.
$$
If $\beta$ is already piecewise linear we include the corner points in
the subdivision to ensure $\Bf_0\beta=\beta$, so that condition (i)
holds.

To define $\Bf_1$ and $\BH_1$, consider a generic $1$-simplex
$\beta:[0,1]\to\Sigma$. We pick finitely many smooth paths of
subdivision points $p_i(\lambda)$ on the $Q$-strings in
$\beta(\lambda)$ (which
include all end points) and define $\BH_1\beta$ to be the straight
line homotopy from $\beta$ to the $1$-simplex $\Bf_1\beta$ whose
$Q$-strings are the piecewise linear strings connecting the
$p_i(\lambda)$.
Here we choose the $p_i(\lambda)$ to agree with the ones in the
definition of $\BH_0$ at $\lambda=0,1$ as well as at the finitely many
values $\lambda_j$ where some $Q$-string intersects the knot in its
interior (so at such $\lambda_j$ the intersection point with $K$ is
included among the $p_i(\lambda_j)$). 
Note that for this we may first have to add new subdivision points on
the $Q$-strings on $\beta(\lambda)$ for $\lambda=0,1,\lambda_j$, which
is allowed due to the identification above. 
Moreover, we
choose the subdivision so fine that the $Q$-strings in $\BH_1\beta$
remain transverse to $K$ at the end points and meet $K$ in the
interior exactly at the values $\lambda_j$ above. The $N$-strings are
just slightly rotated near the end points to match the new
$Q$-strings, without creating new intersections with $K$ besides the
ones already present in $\beta$ that are continued along the
homotopy. Then $\BH_1\beta$ is a generic $2$-chain in $\Sigma$
satisfying
$$
   (\p\BH_1+\BH_0\p)\beta=\Bf_1\beta-\beta,\qquad
   (\delta_Q\BH_1+\BH_0\delta_Q)\beta=(\delta_N\BH_1+\BH_0\delta_N)\beta=0.
$$
If $\beta$ is already piecewise linear we include the corner points in
the subdivision to ensure $\Bf_1\beta=\beta$, so that condition (ii)
holds.
\end{proof}


\subsection{Properties of triangles for generic knots}\label{ss:triangles}
In our arguments, we will assume that the knot $K$ is generic. In particular,
we will use that it has the properties listed in the following lemma.

\begin{lemma}\label{lem:generic.knots}
A generic knot $K\subset\R^3$ has the following properties:
\begin{enumerate}
\item \label{planes} There exists an $S\in\N$ such that each plane
  intersects $K$ at most $S$ times.
\item The set $T \subset K$ of points whose tangent lines meet the
  knot again is finite (and each such tangent line meets the knot in
  exactly one other point).
\end{enumerate}
\end{lemma}

\begin{proof}
We prove part (i). For a generic knot $K$ parametrized by
$\gamma:S^1=\R/L\Z\to\R^3$, the first
four derivatives $(\dot\gamma,\gamma^{(2)},\gamma^{(3)},\gamma^{(4)})$
span $\R^3$ at each $t\in S^1$. (For this, use the jet transversality
theorem~\cite[Chapter 3]{Hir} to make the corresponding map $S^1\to(\R^3)^4$ transverse to
the codimension two subset consisting of quadruples of vectors that
lie in a plane.) It follows that there exists an $\eps>0$ such that
$\gamma$ meets each plane at most $4$ times on a time interval of
length $\eps$. (Otherwise, taking a limit of quintuples of times
mapped into the same plane whose mutual distances shrink to zero, we
would find in the limit an order four tangency of $\gamma$ to a plane,
which we have excluded.) Hence $\gamma$ can meet each plane at most
$4L/\eps$ times.

The proof of part (ii) is contained in the proof of
Lemma~\ref{lem:2-gons}(b) below. It relies on choosing $K$ such that
its curvature vanishes nowhere.
\end{proof}

\comment{
In part (ii) we introduced the finite subset $T \subset K$ of points
whose tangent lines meet the knot again. We denote by $B \subset K$
the subset of these ``other points'' on tangent lines to $K$.
Standard transversality arguments yield

\begin{lemma}
Let $\sigma:[0,1] \to \Sigma^\ell$ be a $1$-simplex. With an arbitrarily small
perturbation, we can achieve that for each $Q$-string $s_{2i}$ of $\sigma$ the evaluation map
$$
\lambda \mapsto \bigl(s_{2i}(a_{2i-1}(\lambda)),s_{2i}(a_{2i}(\lambda))\bigr)
$$
is transverse to the diagonal in $K \times K$ as well as to the sets
$T \times K \cup K \times T$ and $B \times K \cup K \times B$ of end
points of segments which are tangent to the knot at one end point.
\hfill$\square$
\end{lemma}
}

Now we consider the space of triangles in $\R^3$ with pairwise
distinct corners $x_1,x_2,x_3$ such that $x_1$ and $x_3$ lie on the
knot $K$. Using an arclength parametrization $\gamma:S^1=\R/L\Z\to K$
we identify this space with the open subset
$$
   \TT=\{(s,x_2,r)\in S^1\times\R^3\times S^1\mid
   x_1=\gamma(s),\;x_2,\;x_3=\gamma(r)\text{ are distinct}\}.
$$
We parametrize each triangle $[x_1,x_2,x_3]$ by the map (see
Figure~\ref{fig:triangle-par})
$$
   [0,1]^2\to\R^3,\qquad (u,t)\mapsto (1-t)x_1+t\bigl((1-u)x_2+ux_3\bigr).
$$
\begin{figure}
\labellist
\small\hair 2pt
\pinlabel ${\color{blue} x_1 = \gamma(s)}$ at 60 109
\pinlabel ${\color{blue} v_2}$ at 145 182
\pinlabel ${\color{blue} x_2}$ at 234 235
\pinlabel ${\color{blue} v_3}$ at 188 79
\pinlabel ${\color{blue} x_3 = \gamma(r)}$ at 328 56
\pinlabel ${(1-u)x_2+u x_3}$ at 315 172
\pinlabel ${(1-t)x_1+t((1-u)x_2+u x_3)}$ at 127 27
\endlabellist
\centering
\includegraphics[width=0.8\textwidth]{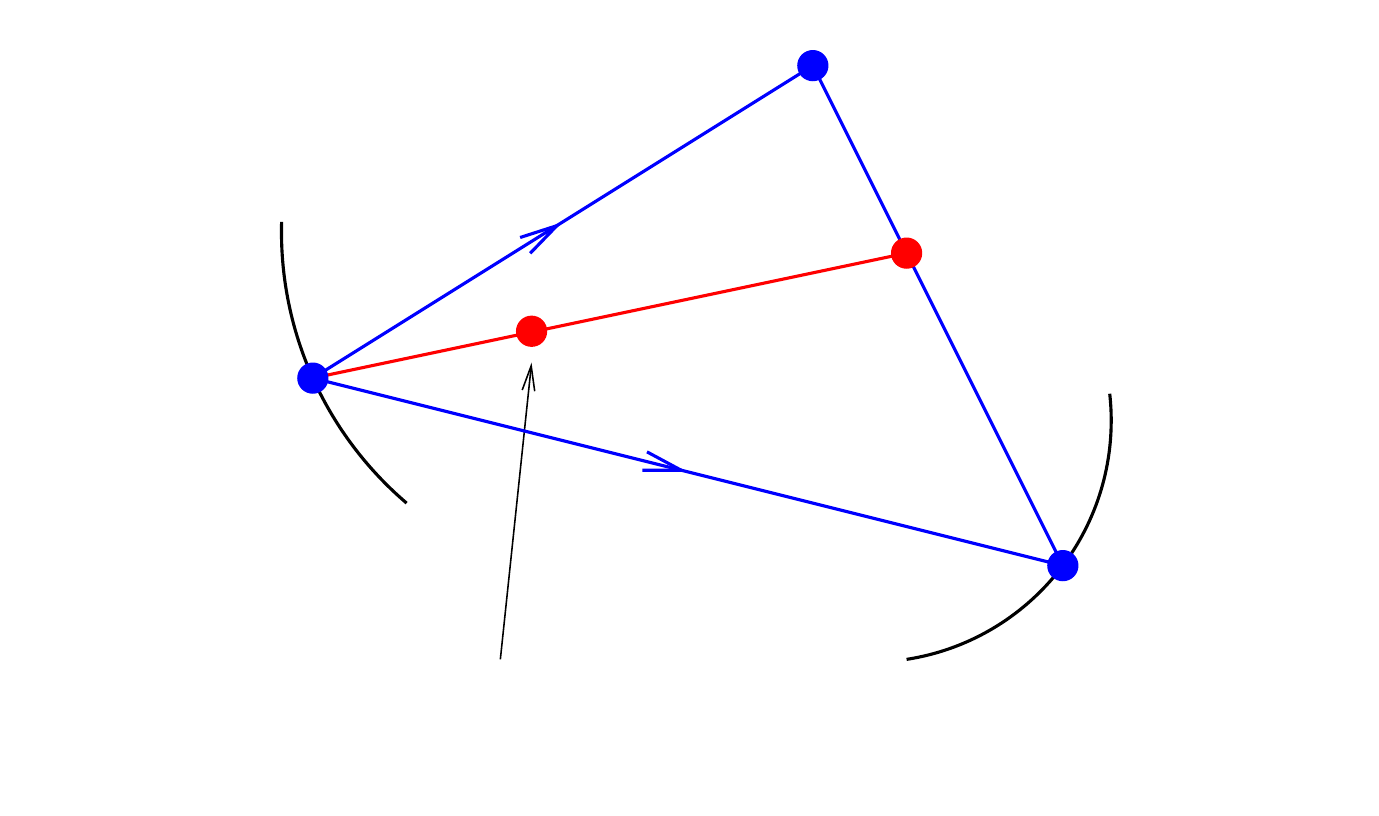}
\caption{
Parametrization of a triangle.
}
\label{fig:triangle-par}
\end{figure}

\begin{lemma}\label{lem:triangles}
For a generic $1$-parameter family of triangles $\beta:[0,1]\to\TT$,
$\lambda\mapsto(s^\lambda,x_2^\lambda,r^\lambda)$ the following
holds.

(a) The evaluation map
$$
   \ev_\beta:[0,1]^3\to\R^3,\qquad (\lambda,u,t)\mapsto
   (1-t)x_1^\lambda+t\bigl((1-u)x_2^\lambda+ux_3^\lambda\bigr)
$$
is transverse to $K$ on its interior, where we have set
$x_1^\lambda=\gamma(s^\lambda)$ and $x_3^\lambda=\gamma(r^\lambda)$.

(b) The map $(\lambda,u)\mapsto \frac{\p\ev_\beta}{\p
  t}(\lambda,u,0)$ meets the tangent bundle to $K$ transversely in
finitely many points. At these points the triangle is tangent to the
knot at $x_1^\lambda$ but not contained in its osculating plane.

(c) The points in (b) compactify the set 
$\ev_\beta^{-1}(K) \cap [0,1]^2 \times (0,1]$ 
to an embedded curve in $[0,1]^3$ transverse to the boundary. Its
image in $[0,1]^2$ under the projection
$(\lambda,u,t)\mapsto(\lambda,u)$ is an immersed curve with transverse
self-intersections.
\end{lemma}

\begin{proof}
Part (a) follows from standard transversality arguments. For part (b)
we introduce
$$
   v_2:=x_2-x_1,\qquad v_3:=x_3-x_1,\qquad \nu:=\frac{v_2\times
     v_3}{|v_2\times v_3|}.
$$
Thus $v_2,v_3$ are tangent to the sides of the triangle at $x_1$ and
$\nu$ is a unit normal vector to the triangle. So the space of
triangles that are tangent to the knot at $x_1$ is the zero set of the
map
$$
   F:\TT\to\R,\qquad (s,x_2,r)\mapsto\la\dot\gamma(s),\nu\ra =
   \frac{\la v_2,v_3\times\dot\gamma(s)\ra}{|v_2\times v_3|}.
$$
The last expression shows that along the zero set the variation of $F$
in direction $x_2$ (or equivalently $v_2$) is nonzero provided that
$v_3\times\dot\gamma(s)\neq 0$. So $F^{-1}(0)$ is a transversely cut
out hypersurface in $\TT$ outside the set $\TT_0$ where
$v_3=\gamma(r)-\gamma(s)$ and $\dot\gamma(s)$ are collinear. By
Lemma~\ref{lem:generic.knots}(ii) the set $\TT_0$ has codimension
$2$. Hence a generic curve $\beta:[0,1]\to\TT$ avoids the set $\TT_0$
and intersects $F^{-1}(0)$ transversely, which implies the first
statement in (b). The second statement in (b) follows similarly from
the fact that the set of triangles contained in the osculating plane
at $x_1$ has codimension $2$ in $\TT$.

For part (c), consider a point $(\lambda_0,u_0)$ as in (b). To
simplify notation, let us shift the parameter interval such that
$\lambda_0=0$ is an interior point. Then with the obvious notation
$\nu^\lambda$ etc the following conditions hold at $\lambda=0$:
$$
   a:=\la\dot\gamma(s^0),\nu^0\ra = 0,\quad
   b:=\la\ddot\gamma(s^0),\nu^0\ra \neq 0,\quad
   c:=\frac{d}{d\lambda}\bigl|_{\lambda=0}\la\dot\gamma(s^\lambda),\nu^\lambda\ra \neq 0.
$$
Here the first condition expresses the fact that the triangle is
tangent to the knot at $x_1^0$, the second on that the triangle is not
contained in the osculating plane, and the third one the
transversality of the map in (b) to the tangent bundle of
$K$. Intersections of $K$ with triangles $\beta(\lambda)$ for
$\lambda$ close to zero can be written in the form
$\gamma(s^\lambda+s)$ with $s=O(\lambda)$ and must satisfy the
equation
$$
   0 = \Bigl\la \gamma(s^\lambda+s)-\gamma(s^\lambda),\nu^\lambda\Bigr\ra
   = \Bigl\la s\dot\gamma(s^\lambda)+\frac{1}{2}s^2\ddot\gamma(s^\lambda) +
   O(s^3),\nu^\lambda\Bigr\ra.
$$
Ignoring the trivial solution $s=0$, we divide by $s$ and obtain using $s=O(\lambda)$:
\begin{align*}
   0 &= \Bigl\la \dot\gamma(s^\lambda)+\frac{1}{2}s\ddot\gamma(s^\lambda) +
   O(s^2),\nu^\lambda\Bigr\ra \cr
   &= \Bigl\la \dot\gamma(s^0)+\ddot\lambda\gamma(s^0)+\frac{1}{2}s\ddot\gamma(s^0) +
   O(\lambda^2),\nu^0+\lambda\dot\nu^0+O(\lambda^2)\Bigr\ra \cr
   &= \la\dot\gamma(s^0),\nu^0\ra +
   \lambda\Bigl[\la\ddot\gamma(s^0),\nu^0\ra +
     \la\dot\gamma(s^0),\dot\nu^0\ra + O(\lambda)\Bigr] +
   s\Bigl[\frac{1}{2}\la\ddot\gamma(s^0),\nu^0\ra + O(\lambda)\Bigr]\cr
   &= a + \lambda\Bigl[b+O(\lambda)\Bigr] +
   s\Bigl[\frac{1}{2}c+O(\lambda)\Bigr].
\end{align*}
Since $a=0$ and and $b,c$ are nonzero, this equation has for each
$\lambda$ a unique solution $s$ of the form
$$
   s = -\frac{2b}{c}\lambda + O(\lambda^2).
$$
Now recall that by hypothesis $\dot\gamma(s^0)$ is a multiple of
$(1-u^0)v_2^0+u^0v_3^0$. If it is a positive (resp.~negative)
multiple, then only solutions with $s>0$ (resp.~$s<0$) will lie in the
triangle. So in either case the solutions describe a curve with
boundary and part (c) follows.
\end{proof}

\begin{remark}\label{rem:trianges-par}
Lemma~\ref{lem:triangles} shows that, given a generic $1$-parameter
family of triangles $\beta:[0,1]\to\TT$, the associated $2$-parameter
family $(\lambda,u)\mapsto\ev_\beta(\lambda,u,\cdot)$ can be
reparametrized in $t$ to look like the $Q$-strings in a generic
$2$-chain of broken strings. To see the last condition (2e) in
Definition~\ref{def:generic-chain}, consider a parameter value
$(\lambda,u)$ as in Lemma~\ref{lem:triangles}(b). Since the triangle
is not contained in the osculating plane at $x_1^\lambda$, the linear
string $t\mapsto\ev_\beta(\lambda,u,t)$ deviates quadratically from
the knot, so its projection normal to the knot has nonvanishing second
derivative at $t=0$. Hence we can reparametrize it to make its second
derivative vanish and its third derivative nonzero as required in
condition (2e). We will ignore these reparametrizations in the
following.
\end{remark}

\begin{remark}
Lemma~\ref{lem:triangles} remains true (with a simpler proof) if in
the definition of the space of triangles $\TT$ we allow $x_3$ to move
freely in $\R^3$ rather than only on the knot; this situation will
also occur in the shortening process in the next subsection.

Let us emphasize that in the space $\TT$ we require the points
$x_1,x_2,x_3$ to be distinct. Now in a generic $1$-parameter family of
triples $(x_1,x_2,x_3)$ with $x_1,x_3\in K$ the points $x_1,x_3$ may
meet for some parameter values, so this situation is not covered by
Lemma~\ref{lem:triangles}. See Remark~\ref{rem:crossing} below on how to
deal with this situation.
\end{remark}

\subsection{Reducing piecewise linear $Q$-strings to linear ones}
In this subsection we deform chains in $\Sigma_\pl$ to chains in
$\Sigma_\lin$, not increasing the length of $Q$-strings in the
process. The main result of this subsection is

\begin{prop}\label{prop:pl-lin}
For a generic knot $K$ there exist maps
$$
   \Bf_0:C_0(\Sigma_\pl) \to C_0(\Sigma_\lin),\qquad
   \Bf_1:C_1(\Sigma_\pl) \to C_1(\Sigma_\lin)
$$
and
$$
   \BH_0:C_0(\Sigma_\pl) \to C_1(\Sigma_\pl),\qquad
   \BH_1:C_1(\Sigma_\pl) \to C_2(\Sigma_\pl)
$$
satisfying with the map $i_\lin$ from~\eqref{eq:ilin}:
\begin{enumerate}
\item $\Bf_0i_\lin=\id$ and $D\BH_0= i_\lin\Bf_0-\id$;
\item $\Bf_1i_\lin=\id$ and $\BH_0D + D\BH_1 = i_\lin\Bf_1-\id$;
\item $\Bf_0$, $\BH_0$, $\Bf_1$ and $\BH_1$ are (not necessarily strictly) length-decreasing.
\end{enumerate}
\end{prop}

\begin{proof}
We assume that $K$ satisfies the genericity properties in
Section~\ref{ss:triangles}.
We first construct the maps $\BH_0$ and $\Bf_0$.

For each simplex $\beta \in C_0^\pl(\Sigma)$ we denote by $M(\beta)$
the total number of corners in the $Q$-strings of $\beta$, not
counting the corners in $Q$-spikes (which are by definition $3$-gons). Connecting
each corner to the starting point of its $Q$-string, we obtain $M(\beta)$
triangles connecting the various $Q$-strings to the segments between
their end points.
We define the {\em complexity} of $\beta \in C_0^\pl(\Sigma)$ to
be the pair of nonnegative integers
$$
   c(\beta) := (M(\beta), I(\beta)),
$$
where $I(\beta)$ is the number of interior intersection points of the
first triangle with $K$ (we set $I(\beta)=0$ in the case $M(\beta)=0$,
i.e. if there are no triangles). Note that by part (\ref{planes}) of
Lemma~\ref{lem:generic.knots} we know that $I$ is bounded a priori by
a fixed constant $S=S(K)$. We define the maps $\BH_0$ and $\Bf_0$ by
induction on the lexicographical order on complexities $c(\beta)$. For
$c(\beta)=(0,0)$ we set $\Bf_0\beta=\beta$ and $\BH_0\beta=0$.

For the induction step, let $\beta\in C^\pl_0(\Sigma)$ be a
$0$-simplex and assume that $\Bf_0$ and $\BH_0$ satisfying (i) and (iii)
have been defined for all simplices of complexities
$c<c(\beta)$. Let the first triangle of $\beta$ have vertices
$x_1,x_2,x_3$, where $x_1$ is the starting point of the first
$Q$-string which is not a segment, and $x_2$ and $x_3$ are the next
two corners on that $Q$-string ($x_3$ might also be the end point).
Since there are only finitely many intersections of the knot
$K$ with the interior of the triangle (and none with its sides), we
can find a segment connecting $x_2$ to a point $x_3'$ on the segment
$x_1x_3$ which is so close to $x_3$ that the triangle $x_2x_3'x_3$
does not contain any intersection points with the knot.
Let $h\beta\in C_1^\pl(\Sigma)$ be the $1$-simplex obtained by
sweeping the first triangle by the family of segments from $x_1$ to a
varying point $(1-u)x_2+ux_3'$ on the segment $[x_2,x_3']$, followed by the segment from that
point to $x_3$ and the remaining segments to $x_4$ etc. See
Figure~\ref{fig:triangle-new} (the point $y$ and the shaded region
play no role here and are included for later use).
\begin{figure}[h]
\labellist
\small\hair 2pt
\pinlabel $x_2$ at 97 339
\pinlabel $\beta$ at 52 208
\pinlabel $x_1$ at 7 9
\pinlabel $x_3'$ at 302 117
\pinlabel $x_3$ at 406 182
\pinlabel $x_4$ at 410 99
\pinlabel ${\color{blue} (1-u)x_2+ux_3'}$ at 308 278
\pinlabel ${\color{blue} f\beta}$ at 218 81
\pinlabel ${\color{red} y}$ at 118 186
\endlabellist
\centering
\includegraphics[width=0.5\textwidth]{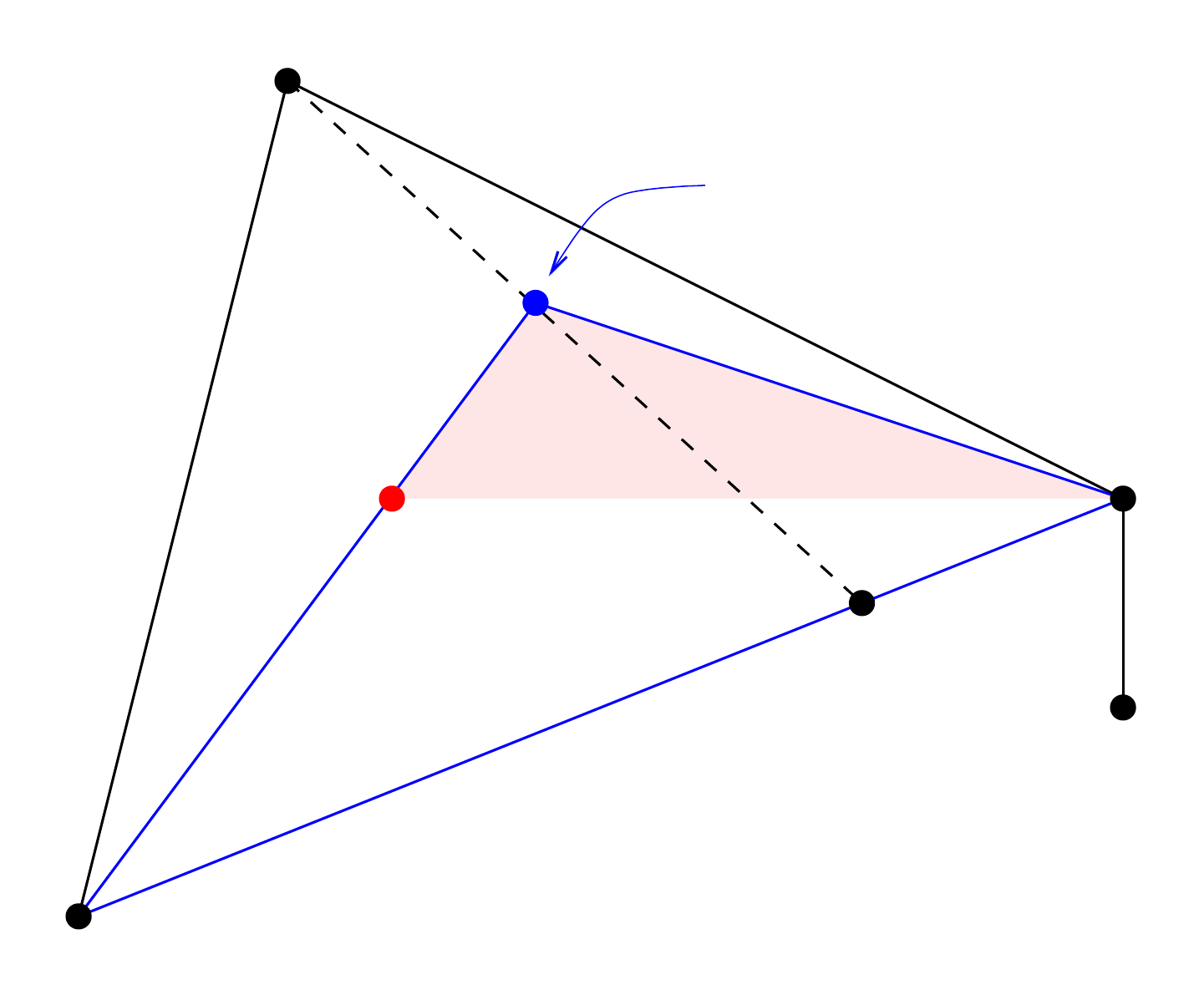}
\caption{Reducing the number of corner points.}
\label{fig:triangle-new}
\end{figure}
The $N$-string ending at $x_1$ (and if there is one, also the
$N$-string starting at $x_3$) is ``dragged along'' without creating
intersections with $K$, and all remaining $N$-and $Q$-strings remain unchanged
in the process.

The $1$-simplex $h\beta$ has boundary $\p(h\beta)=\beta'-\beta$, where
$\beta'$ is the $0$-simplex at the end of the sweep with first segment
$[x_1,x_3]$. We define
$$
   f\beta:= Dh\beta+\beta = \beta'+\delta_Qh\beta+\delta_Nh\beta.
$$
By construction we have $\delta_Nh\beta=0$ and $M(\beta')<M(\beta)$,
hence $c(\beta')<c(\beta)$. The domain of $\delta_Qh\beta$ consists of
those finitely many points where the triangle intersects $K$ in its
interior, so that $\delta_Qh\beta$ consists of broken strings with one
more $Q$-string (which is linear) and with the same total number of
corners as $\beta$. But since the new first triangle is contained in
the original first triangle for $\beta$, and one of the intersection
points is now the starting point of the new $Q$-string, we have
$I(\delta_Qh\beta)<I(\beta)$. Altogether we see that
$c(f\beta)<c(\beta)$, so by induction hypothesis $\Bf_0$ and $\BH_0$ are
already defined on $f\beta$. We set
$$
   \Bf_0\beta:= \Bf_0f\beta \quad \text{\rm and} \quad \BH_0\beta:= \BH_0f\beta +h\beta
$$
and verify that indeed (using condition (i) on $f\beta$)
$$
D\BH_0\beta = D\BH_0f\beta +Dh\beta = \Bf_0f\beta -f\beta+f\beta-\beta = \Bf_0\beta-\beta,
$$
so condition (i) continues to hold.
Condition (iii) holds by induction hypothesis in view of $L(f\beta) \le L(\beta)$
and $L(h\beta)\leq L(\beta)$. Since every $\beta\in C_0^\pl(\Sigma)$
has finite complexity, this finishes the definition of $\Bf_0$ and
$\BH_0$.

We next construct the maps $\BH_1$ and $\Bf_1$, following the same
strategy. For this, we first extend the notion of complexity $c=(M,I)$
to 1-chains with piecewise linear $Q$-strings. For a $1$-simplex
$\beta:[0,1]\to\Sigma^\pl$, we set
$$
   M(\beta) := \max_{\lambda\in[0,1]}M(\beta(\lambda)),\qquad
   I(\beta) := \max_{\lambda\in[0,1]}I(\beta(\lambda)).
$$
Note that $I(\beta)$ is still bounded by the constant $S=S(K)$ in
Lemma~\ref{lem:generic.knots}. Note also that, according to our
definition of chains of piecewise linear strings, the number
$M(\beta(\lambda))$ of corner points of $Q$-strings in
$\beta(\lambda)$ is actually constant equal to the maximal number
$M(\beta)$.
Observe that with this definition of complexity for $1$-chains, the
maps $h_0:=h$ and $\BH_0$ used in the argument for $0$-chains do not
increase complexity.

Again our definition of $\Bf_1$ and $\BH_1$ proceeds by induction on
the lexicographic order on complexity. For simplices $\beta\in
C_1^\pl(\Sigma)$ with $M=0$ we set $\Bf_1 \beta=\beta+ \BH_0D\beta$
and $\BH_1 \beta=0$. Then (ii)
holds by construction,
and (iii) holds since $\BH_0$ and $D$ are length-decreasing.

For the induction step, let $\beta\in C_1^\pl(\Sigma)$ be a
$1$-simplex, and assume that $\Bf_1$ and $\BH_1$ satisfying (ii)
and (iii) have been defined for all $1$-simplices of complexity
$c<c(\beta)$. Using a parametrized version of sweeping the first
triangle, we obtain a $2$-chain $h_1\beta\in C_2^\pl(\Sigma)$. By
construction its boundary satisfies $\p
h_1\beta+h_0\p\beta=\beta'-\beta$, where $\beta'$ is the $1$-simplex at
the end of the sweep with first segment $[x_1,x_3]$, see Figure~\ref{fig:Hbeta}.
\begin{figure}[h]
\labellist
\small\hair 2pt
\pinlabel $0$ at 12 297
\pinlabel $h_1\partial\beta$ at 1 160
\pinlabel $1$ at 12 64
\pinlabel $u$ at 12 32
\pinlabel $\beta$ at 417 297
\pinlabel $h_1\partial\beta$ at 410 179
\pinlabel $\beta'$ at 412 40
\pinlabel $Z\beta$ at 100 194
\pinlabel ${\color{blue} \delta_N\beta}$ at 206 297
\pinlabel ${\color{blue} \delta_N h_1 \beta}$ at 105 125
\pinlabel ${\color{blue} \delta_N h_1 \beta}$ at 235 169
\pinlabel ${\color{red} \delta_Q \beta}$ at 134 297
\pinlabel ${\color{red} \delta_Q \beta}$ at 294 297
\pinlabel ${\color{red} \delta_Q h_1 \beta}$ at 88 250
\pinlabel ${\color{red} \delta_Q h_1 \beta}$ at 283 228
\pinlabel ${\color{red} \delta_Q h_1 \partial \beta}$ at 418 209
\pinlabel ${\color{red} \delta_Q h_1 \partial \beta}$ at 418 151
\pinlabel ${\color{red} \delta_Q h_1 \beta}$ at 312 139
\pinlabel ${\color{red} \delta_Q \beta'}$ at 279 40
\endlabellist
\centering
\includegraphics[width=0.9\textwidth]{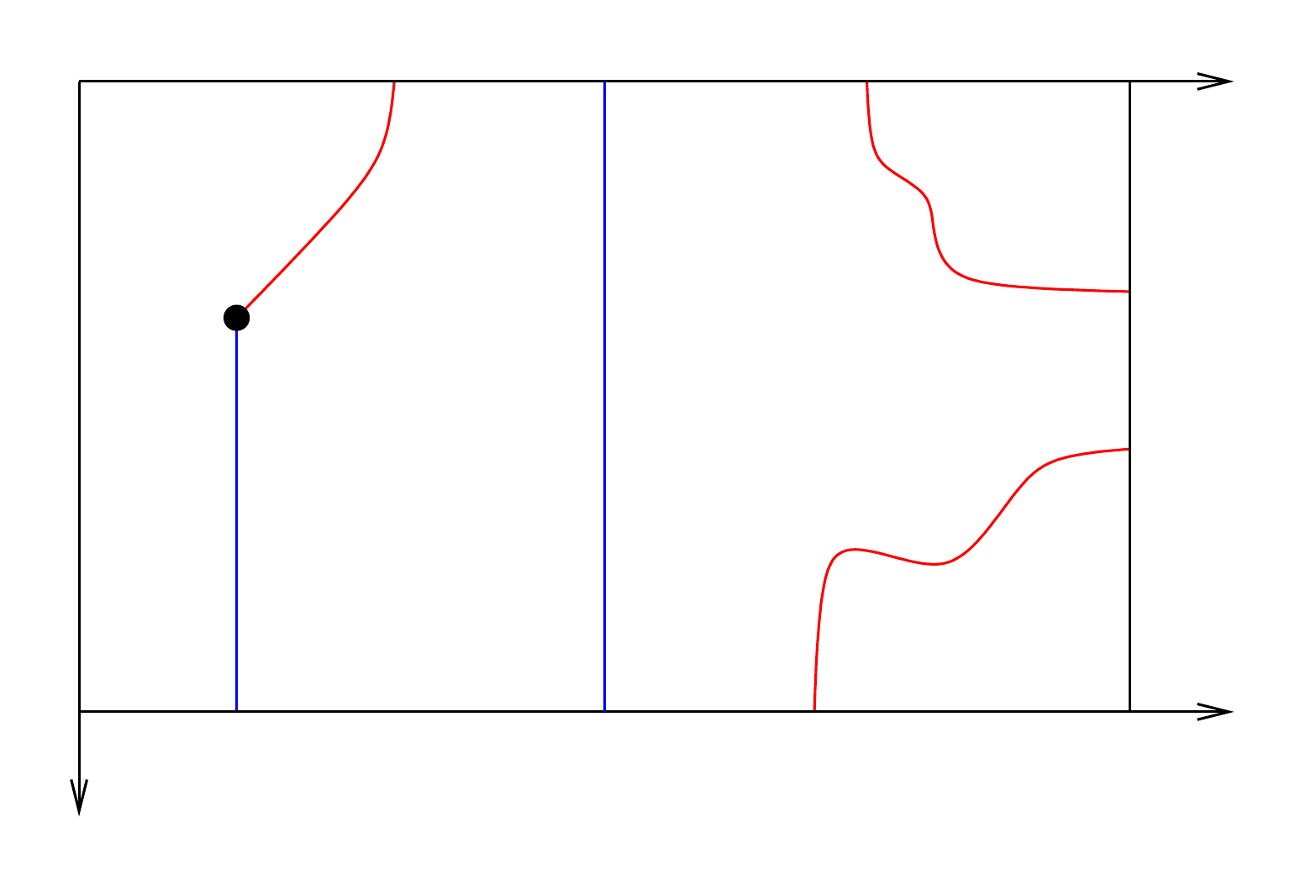}
\caption{The domain of $h_1\beta$.}
\label{fig:Hbeta}
\end{figure}
We now define
\begin{align*}
   f_1\beta :&= Dh_1\beta + h_0D\beta + \beta \cr
   &= \beta' + (\delta_Qh_1+h_0\delta_Q)\beta + (\delta_Nh_1+h_0\delta_N)\beta.
\end{align*}
We claim that $c(f_1\beta)<c(\beta)$. To see this, we need to show
that the three terms on the right hand side of the last displayed
equation have complexity lower that $c(\beta)$. For $\beta'$ this
holds because its $Q$-strings have one fewer corner,
i.e.~$M(\beta')<M(\beta)$.
The domain of $(\delta_Qh_1+h_0\delta_Q)\beta$ consists of
the finitely many curves in which the first triangle intersects $K$ at
an interior point $y$,
so that $(\delta_Qh_1+h_0\delta_Q)\beta$ consists of broken strings
with one more $Q$-string (which is linear) and with the same total number of
corners as $\beta$. But {\em since the new first triangle (the shaded region
in Figure~\ref{fig:triangle-new}) is contained in
the original first triangle} for each parameter value in $\beta$, and one of the intersection
points is now the starting point of the new $Q$-string, we have
$I((\delta_Qh_1+h_0\delta_Q)\beta)<I(\beta)$.
The domain of $(\delta_Nh_1+h_0\delta_N)\beta$ consists of
the finitely many straight line segments $[u,1]\times\{\lambda\}$
emanating from the parameter values $(u,\lambda)$ corresponding to the
tangencies of the triangle $[x_1,x_2,x_3]$ to the knot at $x_1$, see
Figure~\ref{fig:Hbeta} where one such point of tangency is shown as
$Z\beta$. So $(\delta_Nh_1+h_0\delta_N)\beta$ consists of broken strings
with one more $Q$-spike
and with the same total number of corners as $\beta$. But since the
new triangle with corners $x_1,(1-u)x_2+ux_3',x_3$ is contained in
the original first triangle at parameter value $\lambda$, and one of
the intersection points with the knot is the corner point $x_1$ of the new
triangle (which does not count towards $I$), we have
$I((\delta_Qh_1+h_0\delta_Q)\beta)<I(\beta)$ and the claim is proved.

According to the claim, $\Bf_1$ and $\BH_1$ are defined on $f_1\beta$ and we set
$$
\Bf_1\beta := \Bf_1 f_1\beta \quad \text{and} \quad \BH_1\beta :=\BH_1f_1\beta +h_1\beta.
$$
To distinguish the proposed extensions from the maps given by
induction hypothesis, we temporarily call the extended versions
$\HH_1$ and $\FF_1$, so we can write
$$
\FF_1 := \Bf_1 f_1 \quad \text{and} \quad \HH_1 :=\BH_1f_1 +h_1
$$
without ambiguity. Recall also that in this notation $\HH_0=\BH_0
f_0+h_0$. Now using $f_1=h_0D+Dh_1+\id$ we compute
\begin{align*}
D\HH_1 +\HH_0D &= D\BH_1 f_1 +Dh_1+\BH_0 f_0D +h_0D\\
&= (\Bf_1f_1-f_1-\BH_0Df_1) + (f_1-\id-h_0D) + \BH_0f_0D+h_0D\\
&= \FF_1-\id +\BH_0(f_0D-Df_1).
\end{align*}
Using $f_1=h_0D+Dh_1+\id$ again and $f_0=Dh_0+\id$, we find
$Df_1=Dh_0D+D = (Dh_0+\id)D = f_0D$,
so that the last term in the displayed equation vanishes and the
extensions $\HH_1,\FF_1$ have the required properties.
This completes the induction step and hence the proof of
Proposition~\ref{prop:pl-lin}.
\end{proof}

\begin{remark}\label{rem:crossing}
If in a $1$-simplex $\beta$ as in the preceding proof the third point
$x_3$ of the first triangle is the end point of the corresponding
$Q$-string and thus constrained to lie on the knot, then the points
$x_1$ and $x_3$ can cross each other for some parameter values
$\lambda$ in the chain. The homotopy $h_1\beta$ then shrinks the
corresponding degenerate triangle at parameter $\lambda$ to a constant
$Q$-string, which according to our convention from Section~\ref{ss:pl}
we interpret as a linear $Q$-spike in the direction of the degenerate triangle.
Incidentally, the segment $[x_2,x_3]$ is always short throughout the
shortening process, so if $x_1$ and $x_3$ agree then the triangle is
already a linear $Q$-spike without further shrinking.
\end{remark}

\begin{remark}
Definition~\ref{def:spike} implies that if a $Q$-string in $\beta$ in
the preceding proof is a (piecewise linear) $Q$-spike, then it never
intersects the knot in its interior and remains a $Q$-spike throughout
the shortening process (which ends with a degenerate triangle as in
Remark~\ref{rem:crossing}). This property ensures that $\BH_0$ and
$\BH_1$ indeed do not increase length, which does not count $Q$-spikes.
\end{remark}

\begin{remark}\label{rem:R3}
The proof relies crucially on the (trivial) fact that the new
triangle $[y,(1-u)x_2+ux_3',x_3]$ (the shaded region
in Figure~\ref{fig:triangle-new})
obtained by splitting the $Q$-string at an intersection point $y$ with $K$
is contained in the old triangle $[x_1,x_2,x_3]$. {\em This is the only place where we use that
the metric is Euclidean}; the rest of the proof works equally well for
any metric of nonpositive curvature.
\end{remark}

\subsection{Properties of linear $Q$-strings for generic knots}
Now we consider the space of {\em $2$-gons}, i.e., straight line
segments starting and ending on the knot. This
space is canonically identified with $K\times K$ by
associating to each $2$-gon its endpoints on $K$. We consider the
squared distance function
$$
   E:K\times K\to\R,\qquad E(x,y)=\frac{1}{2}|x-y|^2.
$$

\begin{lemma}\label{lem:2-gons}
For a generic knot $K\subset\R^3$ the following holds for the space
$K\times K$ of $2$-gons (see Figure~\ref{fig:2-gons}).
\begin{figure}[h]
\labellist
\small\hair 2pt
\pinlabel $0$ at 3 3
\pinlabel $L$ at 268 3
\pinlabel $L$ at 3 268
\pinlabel $s$ at 317 19
\pinlabel $t$ at 17 323
\pinlabel $K\times K$ at 294 288
\pinlabel ${\color{blue} -\nabla E}$ at 72 207
\pinlabel ${\color{red} S_Q}$ at 124 230
\pinlabel ${\color{red} S_Q}$ at 231 129
\endlabellist
\centering
\includegraphics[width=0.5\textwidth]{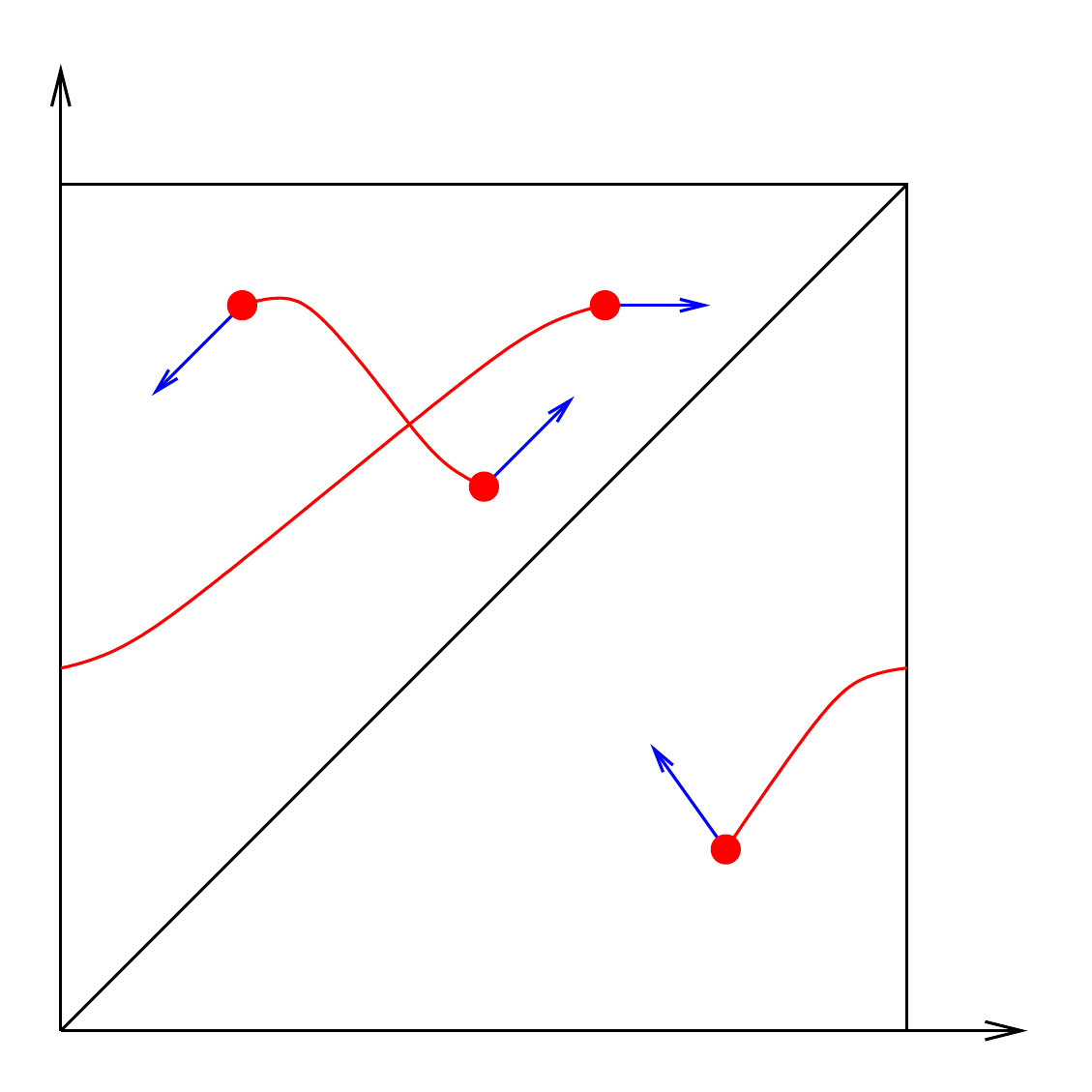}
\caption{The space of $2$-gons.}
\label{fig:2-gons}
\end{figure}

(a) $E$ attains its minimum $0$ along the diagonal, which is a
Bott nondegenerate critical manifold; the other critical points are
nondegenerate binormal chords of index $0,1,2$.

(b) The subset $S_Q\subset K\times K$ of $2$-gons meeting $K$ in their
interior is a $1$-dimensional submanifold with boundary consisting of
finitely many $2$-gons tangent to $K$ at one endpoint, and with
finitely many transverse self-intersections consisting of finitely
many $2$-gons meeting $K$ twice in their interior.

(c) The negative gradient $-\nabla E$ is not pointing into $S_Q$ at
the boundary points.
\end{lemma}

\begin{proof}
(a) In terms of an arclength parametrization $\gamma$ of $K$ we write
the energy as a function $E(s,t)=\frac{1}{2}|\gamma(s)-\gamma(t)|^2$.
We compute its partial derivatives
\begin{equation}\label{eq:partial-E}
\begin{gathered}
   \frac{\p E}{\p s}= \la\gamma(s)-\gamma(t),\dot\gamma(s)\ra ,\qquad
   \frac{\p E}{\p t}= \la\gamma(t)-\gamma(s),\dot\gamma(t)\ra ,\cr
   \frac{\p^2E}{\p s^2}= |\dot\gamma(s)|^2+\la\gamma(s)-\gamma(t),\ddot\gamma(s)\ra ,\qquad
   \frac{\p^2E}{\p s\p t}= -\la\dot\gamma(s),\dot\gamma(t)\ra ,\cr
   \frac{\p^2E}{\p t^2}=|\dot\gamma(t)|^2+\la\gamma(t)-\gamma(s),\ddot\gamma(t)\ra.
\end{gathered}
\end{equation}
We see that critical points of $E$ are points on the diagonal $s=t$
and binormal chords (where $s\neq t$), and the Hessian of $E$ at $s=t$ equals
$\left(\begin{smallmatrix}1&-1\\-1&1\end{smallmatrix}\right)$. Its
kernel is the tangent space to the diagonal and it is positive
definite in the transverse direction. This proves Bott nondegeneracy
of the diagonal. Nondegeneracy of the binormal chords is achieved by a
generic perturbation of $K$.

\begin{figure}
\labellist
\small\hair 2pt
\pinlabel $K$ at 240 304
\pinlabel $K$ at 390 317
\pinlabel $p$ at 143 144
\pinlabel $p_\xi$ at 81 194
\pinlabel $p_\eta$ at 200 182
\pinlabel $\dot{\gamma}(s)$ at 189 144
\pinlabel $q$ at 342 162
\pinlabel ${\color{blue} P}$ at 85 75
\pinlabel ${\color{blue} Q}$ at 337 18
\pinlabel ${\color{red} \ell_{\xi,\eta}}$ at 261 208
\endlabellist
\centering
\includegraphics[width=0.7\textwidth]{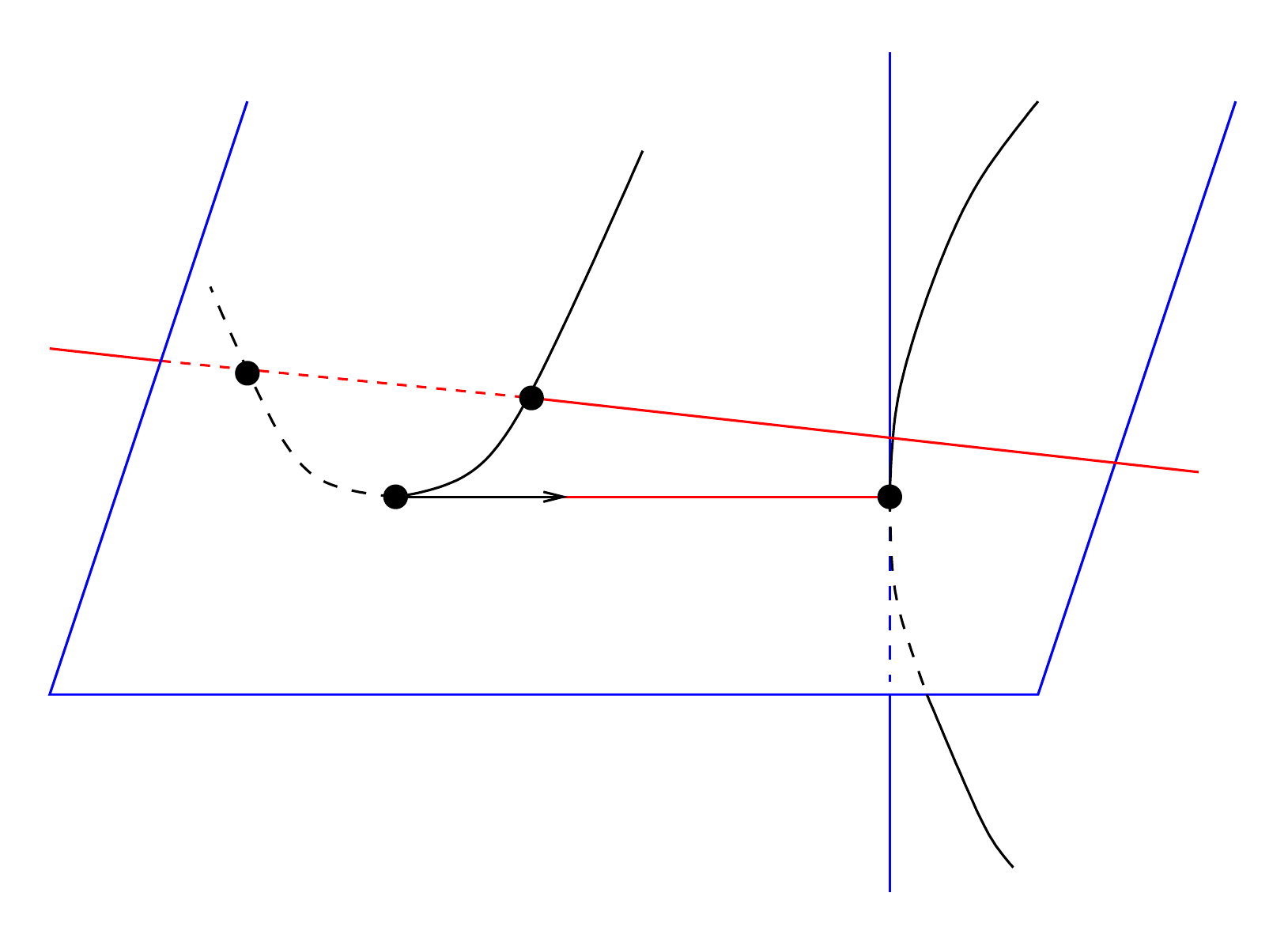}
\caption{A $2$-gon becoming tangent to $K$ at an endpoint.}
\label{fig:tangency}
\end{figure}

(b) We choose $K$ so that its curvature is nowhere $0$ (which holds
generically). Then there exists $\delta>0$ such that no $2$-gon of
positive length $<\delta$ intersects the knot in an interior
point. Consider the tangential variety $\tau_{K}$ of $K$ (where
$\gamma\colon [0,L]\to\R^{3}$ is a parametrization of $K$)
\[
\tau_{K}:=\left\{\gamma(s)+r\dot \gamma(s)\mid s\in[0,L], r\in\R\right\}\subset\R^3.
\]
Since the curvature of $K$ is nowhere zero, there exists $\delta>0$
such that for each $s$ the line segment $\{\gamma(s)+r\dot \gamma(s)\mid
r\in(-\delta,\delta)\}$ intersects $K$ only at $r=0$. Let $N(\delta)$
denote the union of these line segments. After small
perturbation, the surface $\tau_{K}\setminus N(\delta)$ intersects $K$
transversely. This shows that there are finitely many 2-gons that are
tangent to $K$ at one endpoint and that this is a transversely cut out
$0$-manifold. Moreover, transversality implies that for each $2$-gon
that is tangent to $K$ at one endpoint $p$, the tangent line $Q$ to $K$
at the other endpoint $q$ does not lie in the osculating plane $P$ (the
plane spanned by the first two derivatives of $\gamma$) at $p$; see
Figure~\ref{fig:tangency}.

We claim that the $2$-gon $[p,q]$ is the boundary point of a unique
local embedded curve of $2$-gons intersecting $K$ in their interior.
To see this, we choose affine coordinates $(x,y,z)$ on $\R^3$ in which
$p=(0,0,0)$, $q=(1,0,0)$, $P$ is the $(x,y)$-plane, and $Q$ is
parallel to the $z$-axis. Then $K$ can be written near $p$ as a graph
over the $x$-axis in the form
$$
   y=\kappa x^2+O(x^3),\quad z=O(x^3),
$$
and near $q$ as a graph over the $z$-axis in the form
$$
   x=1+O(z^2),\quad y=O(z^2). 
$$
Here $2\kappa\neq0$ is the curvature of $K$ at $p$, and after a
further reflection we may assume that $\kappa>0$. We fix a small
$\eps>0$ (to be chosen later) and consider points $\xi,\eta$ on the
$x$-axis with $-\eps<\xi<\eta<2\eps$. Let $p_\xi,p_\eta$ be the points
of $K$ near $p$ with $x$-coordinates $\xi,\eta$ and let $\ell_{\xi,\eta}$ be the line
through $p_\xi$ and $p_\eta$. Let $\pi(x,y,z)=(x,z)$ be the projection
onto the $(x,z)$-plane. Since the line $\ell_{\xi,\eta}$ is close to
the $x$-axis and $K$ is tangent to the $z$-axis at $q$, the
projected curves $\pi(\ell_{\xi,\eta})$ and $\pi(K)$ intersect in a
unique point $r_{\xi,\eta}$ in the $(x,z)$-plane near $\pi(q)=(1,0)$.  
Let $f_\xi(\eta)$ denote the difference in the $y$-values between the
points of $K$ and $\ell_{\xi,\eta}$ lying over $r_{\xi,\eta}$. Thus
$f_\xi(\eta)$ is the ``distance in the $y$-direction'' between
$\ell_{\xi,\eta}$ and $K$ near $q$. To compute the function
$f_\xi(\eta)$, note that the slope of the line
through the points $(\xi,\kappa\xi^2)$ and $(\eta,\kappa\eta^2)$ on
the parabola $y=\kappa x^2$ equals $\kappa(\xi+\eta)$, so the $y$-value of this line
at $x=1$ is of the form $\kappa(\xi+\eta) + O(\xi^2+\eta^2)$. The
linear term persists for the function $f_\xi(\eta)$, hence
$$
   f_\xi(\eta) = \kappa(\xi+\eta) + O(\xi^2+\eta^2).
$$
For $\eps$ sufficiently small, we see that if $\xi\geq 0$, then
$f_\xi(\eta)>0$ for all $\eta\in(\xi,2\eps)$. Suppose therefore that
$\xi<0$. Then for $\eps$ sufficiently small we have 
$f_\xi(0)=\kappa\xi+O(\xi^2)<0$, $f_\xi(-2\xi)=-\kappa\xi+O(\xi^2)>0$,
and $f_\xi'(\eta)=\kappa+O(|\xi|+|\eta|)>0$. Thus for every
$\xi\in(-\eps,0)$ there exists a unique $\eta(\xi)\in(\xi,2\eps)$ such that
$f_\xi(\eta(\xi))=0$, i.e., the line $\ell_{\xi,\eta(\xi)}$ intersects
$K$ near $q$. Moreover, the point $\eta(\xi)$ depends smoothly on
$\xi$ and satisfies $0<\eta(\xi)<-2\xi$. This shows that the $2$-gons
with endpoints near $p,q$ intersecting $K$ in their interior form a
smooth curve parametrized by $\xi\in(-\eps,0)$, consisting of the
corresponding segments of the lines $\ell_{\xi,\eta(\xi)}$. As this curve
extends smoothly to $\xi=0$ by the $2$-gon $[p,q]$, the claim is proved.

So we have shown that the subset $S_Q\subset K\times K$ avoids a
neighborhood of the diagonal and is a $1$-manifold with boundary near
the finitely many $2$-gons that are tangent to $K$ at an
endpoint. Away from these sets, a generic perturbation of $K$ makes
the evaluation map at the interior of the $2$-gons transverse to
$K$. Since the condition that a chord meets $K$ in the interior is
codimension one, and the condition that the tangent line at the
intersection is parallel to the chord is of codimension three and can
thus be avoided for generic $K$, we conclude that (b) holds.

(c) Consider a boundary point of $S_Q$, i.e., a $2$-gon $[p,q]$
tangent to $K$ at one endpoint, say at $p$. Let $p=\gamma(s)$ and
$q=\gamma(t)$ for an arclength parametrization of $K$ such that
$\dot\gamma(s)$ is a positive multiple of $q-p$; see Figure~\ref{fig:tangency}.
By equation~\eqref{eq:partial-E} we have $\frac{\p E}{\p s}= \la
p-q,\dot\gamma(s)\ra<0$, so the parameter $s$ strictly increases in the direction of
$-\nabla E$. On the other hand, the description in (b) shows that $s$
strictly decreases as we move into $S_Q$. Hence $-\nabla E$ is not
pointing into $S_Q$ at $[p,q]$.
\end{proof}

\comment{
The next lemma describes the familiar Morse theoretic properties of
the space of $2$-gons.

\begin{lemma}\label{lem:eps}
For a generic knot $K\subset\R^3$, for there exists a constant $\eps>0$
with the following properties.

(a) For each index $0$ binormal chord $c$, the connected component of
$c$ in the space of $2$-gons of
length in $[L(c)-\eps,L(c)+\eps]$ deformation retracts under the flow
of $-\nabla E$ onto $c$.

(b) For each index $1$ binormal chord $c$, the connected component of
the unstable manifold $U_c$ of $c$ in the space of paths of $2$-gons of
length $\leq L(c)+\eps$ with boundary of length $\leq L(c)-\eps$
deformation retracts under the flow of $-\nabla E$ onto $U_c$.
\end{lemma}

\begin{proof}
According to Lemma~\ref{lem:2-gons}, for generic $K$, the function
$E:K\times K\to\R$ has a Bott nondegenerate minimum along the diagonal and Morse
critical points outside the diagonal. Since the length $L=\sqrt{2E}$
is strictly decreasing along the flow of $-\nabla E$, the lemma
follows by standard Morse theoretic arguments.
\end{proof}
}

More generally, for an integer $\ell\geq 1$ we consider the space
$(K\times K)^\ell$ of $\ell$-tuples of $2$-gons with the energy and
length functions $E^\ell,L^\ell:(K\times K)^\ell\to\R$,
\begin{gather*}
   E^\ell(x_1,y_1,\dots,x_\ell,y_\ell) :=
   \frac{1}{2}\sum_{i=1}^\ell|x_i-y_i|^2, \cr
   L^\ell(x_1,y_1,\dots,x_\ell,y_\ell) := \sum_{i=1}^\ell|x_i-y_i|.
\end{gather*}
As a consequence of Lemma~\ref{lem:2-gons}, $E^\ell$ is a Morse-Bott
function whose critical manifolds are products $C_1\times\dots\times
C_\ell$ of critical manifolds of $E$, so each $C_i$ is either a
binormal chord or the corresponding diagonal. Note that the symmetric
group $S_\ell$ acts on $(K\times K)^\ell$ preserving $E^\ell$ as well as
the product metric.

For $a>0$ we denote by $M^a\subset(K\times K)^\ell$ the collection of
tuples $c=(c_1,\dots,c_\ell)$ of binormal chords of total length
$L(c)=a$, and by $W^a$ the disjoint union of the unstable manifolds of
points in $M^a$ under the flow of $-\nabla E^\ell$ (here $M^a$ and
thus $W^a$ may be empty). Let $\phi^T:(K\times K)^\ell\to (K\times
K)^\ell$ be the time-$T$ map of the flow of $-\nabla E^\ell$.

\begin{lemma}\label{lem:eps}
For a generic knot $K\subset\R^3$ and each $a>0$ there exist
$\eps_a>0$ and $T_a>0$ with the following property. For each
$\eps<\eps_a$, $T\geq T_a$ and $\ell\in\N$ we have
$$
   \phi^T(\{L^\ell\leq a+\eps\})\subset\{L^\ell\leq a-\eps\}\cup V^a,
$$
where $V^a$ is a tubular neighborhood of $W^a\cap\{L^\ell\geq a-\eps\}$ in
$\{a-\eps\leq L^\ell\leq a+\eps\}$. Moreover, tuples of $Q$-strings in
$V^a$ do not intersect the knot $K$ in their interior.
\end{lemma}

\begin{proof}
Note that on $K\times K$ the length and energy are related by
$L=\sqrt{2E}$, so they have the same critical points and $L$
is strictly decreasing under the flow of $-\nabla E$ outside the
critical points. Since the flow of $-\nabla E^\ell$ is the product of
the flows of $E$ in each factor, the same relation holds for any
$\ell\in\N$: $L^\ell$ and $E^\ell$ have the same critical points and
$L^\ell$ is strictly decreasing under the flow of $-\nabla E^\ell$
outside the critical points.

Next recall from above that $E^\ell$ is a Morse-Bott function. In
particular, the set of critical values of $E^\ell$, and thus also of
$L^\ell$, is discrete. Given $a\in\R$, we pick $\eps_a>0$ such that
$a$ is the only critical value of $L^\ell$ in the interval
$[a-\eps_a,a+\eps_a]$. (Since only finitely many binormal chords can
appear in tuples of critical points of total length $a$, the constant
$\eps_a$ can be chosen independently of $\ell$.) For $\eps<\eps_a$,
the familiar argument from Morse theory shows that
$\phi^T(\{L^\ell\leq a+\eps\}\subset\{L^\ell\leq a-\eps\}\cup V^a_{\eps,T}$,
where $V^a_{\eps,T}$ for large $T$ are tubular neighborhoods of
$W^a\cap\{L^\ell\geq a-\eps\}$ in $\{a-\eps\leq L^\ell\leq a+\eps\}$
that shrink to $W^a\cap\{L^\ell\geq a-\eps\}$ as $T\to\infty$.

For the last statement recall that, for a generic knot $K$, binormal
chords do not meet $K$ in their interior. So for each $\ell\in\N$
there exists a neighborhood $U^a$ of $M^a$ in $(K\times K)^\ell$ such
that tuples of $Q$-strings in $U^a$ do not intersect $K$ in their
interior. We pick $T_a$ large enough and $\eps_a$ small enough so
that $V^a_{\eps_a,T_a}$ is contained in $U^a$. By the argument as in
the previous paragraph, the constants $\eps_a$ and $T_a$ can be chosen
independently of $\ell$ and the lemma is proved.
\end{proof}

\comment{
Lemma~\ref{lem:eps} now implies

\begin{cor}\label{cor:eps-a}
For a generic knot $K\subset\R^3$ and $\ell\in\N$, each critical value
of $E^\ell$ corresponds to a unique $S_\ell$-orbit of critical
points. Moreover, for each $a\in\R$ there exists a constant $\eps_a>0$
with the following properties.

(a) If $a$ is not a critical value, then the space $\{L\leq a+\eps_a\}$
of $\ell$-tuples of $2$-gons of total length $\leq a+\eps$ deformation
retracts under the flow of $-\nabla E^\ell$ onto $\{L\leq a-\eps_a\}$.

(b) For each tuple $c=(c_1,\dots,c_\ell)$ of binormal cords of total
index $0$ and $L(c)=a$, the connected component of $c$ in the space
$\{a-\eps_a\leq L\leq a+\eps_a\}$ deformation retracts under the flow
of $-\nabla E^\ell$ onto $c$.

(c) For each tuple $c=(c_1,\dots,c_\ell)$ of binormal cords of total
index $1$ and $L(c)=a$, the connected component of
the unstable manifold $U_c$ of $c$ in the space of paths in $\{L\leq
a+\eps_a\}$ with boundary in $\{L\leq a-\eps_a\}$
deformation retracts under the flow of $-\nabla E^\ell$ onto $U_c$.
\hfill$\square$
\end{cor}
}

\subsection{Shortening linear $Q$-strings}\label{ss:chain-hom}
We will need some homological algebra.
Suppose we have the following algebraic situation:
\begin{itemize}
\item a chain complex $(\CC,D=\p+\delta)$ satisfying the
  relations
$$
\p^2=\delta^2=\p\delta+\delta\p=0, \quad\text{\rm and}
$$
\item a chain map $f:(\CC,\p)\to(\CC,\p)$ and a chain homotopy $H:(\CC,\p)\to(\CC,\p)$
  satisfying
\begin{equation}\label{eq:iso1}
   \p H+H\p = f-\id,
\end{equation}
such that for every $c\in\CC$ there exists a positive integer $S(c)$ with
\begin{equation}\label{eq:iso3}
   (\delta H)^{S(c)}(c)=0.
\end{equation}
\end{itemize}
In our applications below, we will have $\delta= \delta_Q+\delta_N$, and the
equation $\delta^2=0$ will follow from
$$
\delta_Q^2=\delta_N^2=[\delta_Q,\delta_N]=0,
$$
which is part of the statement that $D^2=0$ in our chain complex.
Here, as usual, we denote the graded commutator of two maps $A,B$ by
$$
   [A,B] := AB - (-1)^{|A||B|}BA.
$$
Set $H_0:=H$ and $f_0:=f$, and more generally for $d\geq 1$ define the maps
\begin{equation}\label{eq:Hd}
   H_d:= H (\delta H)^d, \quad
   f_d := \sum_{i=0}^{d} (H \delta)^i f (\delta H)^{d-i}.
\end{equation}
It is also convenient to set $H_{-1}=0$. Note that the maps $f_d$
satisfy the recursion relation $f_{d+1}=f_d \delta H + H_d\delta f$.

\begin{lemma}
For each $d\geq 1$ we have
\begin{equation}\label{eq:iso4}
   [\p,H_d]+[\delta,H_{d-1}]=f_d.
\end{equation}
\end{lemma}

\begin{proof}
We prove this by induction on $d$. The case $d=1$ is an immediate
consequence of \eqref{eq:iso1} and $[\delta,\p]=0$. For the induction step we observe that
\begin{align*}
   [\p,H_{d+1}] &= \p H_d\delta H + H_d\delta H\p\cr
   &= [\p,H_d]\delta H - H_d \p \delta H -H_d\delta \p H + H_d\delta f -H_d\delta\cr
   &= f_d \delta H -[\delta,H_{d-1}] \delta H + H_d\delta f -H_d\delta\cr
   &= f_{d+1} -\delta H_d - H_{d-1}\delta^2H - H_d\delta\cr
   &= f_{d+1} -[\delta,H_d].
\end{align*}
Here in the second equality we have used \eqref{eq:iso1}, in the third
equality the induction hypothesis and $[\delta,\p]=0$, in the fourth
equality the recursion relation above,
and in the fifth equality we have used $\delta^2=0$.
\end{proof}

In view of equation~\eqref{eq:iso3}, for each $c\in\CC$ we have
$H_dc=0$ and $f_dc=0$ for
$$
   d\geq S(c)+\max\bigl\{S(fc),S(f\delta H c),\dots,S(f(\delta H)^{S(c)-1}c)\bigr\}.
$$
So the sums
\begin{equation}\label{eq:BH}
   \BH:= \sum_{d=0}^\infty H_d,\qquad
   \Bf:= \sum_{d=0}^\infty f_d
\end{equation}
are finite on every $c\in\CC$. Summing up equation~\eqref{eq:iso4}
for $d=1,\dots,e$ and using equation~\eqref{eq:iso1}, we obtain
$$
   [\p,H_e] + [D,H_0+\cdots+H_{e-1}] = f_0+\cdots+f_e-\id
$$
for all $e$, and hence
\begin{equation*}
   [D,\BH] =\Bf -\id.
\end{equation*}
This concludes the homological algebra discussion.

We now apply this construction to the space $\Sigma_\lin$ of broken
strings with linear $Q$-strings as follows. We fix a large time
$T>0$ and consider a generic $i$-chain $\beta$ in $\Sigma_\lin$, for
$i=0,1$. Moving the $Q$-strings in $\beta$ by the flow of $-\nabla E$
for times $t\in[0,T]$ we obtain an $(i+1)$-chain in $(K\times
K)^\ell$. We make this an $(i+1)$-chain $H^T\beta$ in $\Sigma_\lin$ by dragging
along the $N$-strings without creating new intersections with the
knot. In the case $i=1$, we moreover grow new $N$-spikes starting
from the finitely many points $Z\beta$ where some $Q$-string becomes
tangent to the knot at one end point, as shown in
Figure~\ref{fig:Hbeta}. We define $f^T\beta$ as the boundary component
of $H^T\beta$ at time $T$.

\begin{remark}
Technically, we should be careful to arrange that $H$ maps generic
chains to generic chains. This is easy for $0$-chains, but some care
should be taken for $1$-chains, especially near the points $Z\beta$
where some $Q$-string becomes tangent to $K$ at one of its end
points.
\end{remark}

\begin{prop}
For a generic knot $K$, the operations defined above yield for $i=0,1$
maps
$$
   f^T:C_i(\Sigma_\lin)\to C_i(\Sigma_\lin),\qquad
   H^T:C_i(\Sigma_\lin)\to C_{i+1}(\Sigma_\lin)
$$
satisfying conditions~\eqref{eq:iso1} and~\eqref{eq:iso3}.
\end{prop}

\begin{proof}
Standard transversality arguments show that $f^T$ and $H^T$ map
generic chains to generic chains, provided that we impose suitable
genericity conditions on generic chains with respect to linear
strings. Now condition~\eqref{eq:iso1} is clear by construction.

For condition~\eqref{eq:iso3}, we use Lemma~\ref{lem:2-gons}(c).
It implies that there exists a neighborhood $U\subset K\times K$ of
the finitely many $2$-gons $\p S_Q$ that are tangent to $K$ at one end
point and an $\eps>0$ with the following property: Each $2$-gon in
$U\cap S_Q$ decreases in length by at least $\eps$ under the flow of
$-\nabla E$ before it meets $S_Q$ again, and the same holds for the
longer $2$-gon resulting from splitting it at its intersection with
the knot. On the other hand, if a $2$-gon in $S_Q\setminus U$ is split
at its intersection with the knot, then both pieces are shorter by at
least some fixed amount $\delta>0$. Hence each application of
$H^T\delta_Q$ decreases the total length of $Q$-strings by at least
$\min(\eps,\delta)$, and since $L(\beta)$ is finite this can happen
only finitely many times.
\end{proof}

Applying definition~\eqref{eq:BH} to the maps $f^T$ and $H^T$, we
obtain for $i=0,1$ length decreasing maps
$$
   \Bf^T:C_i(\Sigma_\lin)\to C_i(\Sigma_\lin),\qquad
   \BH^T:C_i(\Sigma_\lin)\to C_{i+1}(\Sigma_\lin)
$$
satisfying
\begin{equation}\label{eq:HT}
   D\BH^T_0 =\Bf^T_0 -\id,\qquad \BH^T_0D+D\BH^T_1 =\Bf^T_1 -\id
\end{equation}
We now use these maps to compute the homology of
$(C_i(\Sigma_\lin),D)$ in small length intervals.
For $a\in\R$ and $i=0,1$ we denote by $\AA_i^a$ the free $\Z$-module
generated by words $\gamma_1c_1\dots\gamma_\ell c_\ell\gamma_{\ell+1}$,
$\ell\geq 0$, where $c_1,\dots,c_\ell$ are binormal chords of total
length $a$ and of total index $i$, and the $\gamma_j$ are homotopy
classes of paths in $\p N$ connecting the $c_j$ to broken strings and
not intersecting $K$ in their interior. We define linear maps
$$
   \Theta:\AA_i^a\to H_i^{[a-\eps,a+\eps)}(\Sigma_\lin,D)
$$
as follows. For $i=0$, $\Theta$ sends $\gamma_1c_1\dots\gamma_\ell
c_\ell\gamma_{\ell+1}$ to the homology class of the broken string
$\wt\gamma_1c_1\dots\wt\gamma_\ell c_\ell\wt\gamma_{\ell+1}$, where
$\tilde\gamma_j$ are representatives of the classes $\gamma_j$.
For $i=1$, consider a word $\gamma_1c_1\dots\gamma_\ell
c_\ell\gamma_{\ell+1}$ with exactly one binormal chord $c_k$ of
index $1$ and all others of index $0$. Then $\Theta$ sends this word
to the homology class of the $1$-chain $\wt\gamma_1c_1\dots\wt
c_k\dots\wt\gamma_\ell c_\ell\wt\gamma_{\ell+1}$, where
$\tilde\gamma_j$ are representatives of the classes $\gamma_j$ and
$\wt c_k$ is the unstable manifold of $c_k$ in $(K\times K)\cap\{L\geq
a-\eps\}$, viewed as a $1$-chain by fixing some parametrization.

\begin{cor}\label{cor:lin-hom}
For $a\in\R$ let $\eps_a$ be the constant from Lemma~\ref{lem:eps}.
Then for each $\eps<\eps_a$ the map $\Theta:\AA_i^a\to
H_i^{[a-\eps,a+\eps)}(\Sigma_\lin,D)$ is an isomorphism for $i=0$ and
surjective for $i=1$.
\end{cor}

\begin{proof}
We first consider the case $i=1$.
Fix $\eps<\eps_a$ and $T>T_a$, where $\eps_a,T_a$ are the constants from
Lemma~\ref{lem:eps}. Consider a relative $1$-cycle $\beta\in
C_1^{[a-\eps,a+\eps)}(\Sigma_\lin)$. In view of~\eqref{eq:HT}, $\beta$
is homologous to $\Bf^T\beta$. Recall from its definition
in~\eqref{eq:Hd} and~\eqref{eq:BH} that each tuple of $Q$-strings
appearing in $\Bf^T\beta$ is obtained by flowing some tuple of
$Q$-strings for time $T$ (and maybe applying $H\delta_Q$ several times
to the resulting tuple). Now we distinguish two cases.

{\em Case 1: }$a$ is not the length of a word of binormal chords.
Then in Lemma~\ref{lem:eps} the set $V^a$ is empty and it follows that
all tuples of $Q$-strings in $\Bf^T\beta$ have length at most
$a-\eps$. This shows that $H_1^{[a-\eps,a+\eps)}(\Sigma_\lin)=0$ and
the map $\Theta$ is an isomorphism.

{\em Case 2: }$a$ is the length of a word of binormal chords.
For simplicity, let us assume that up to permutation there is only one
word $w$ of length $a$ (the general case differs just in notation).
By Lemma~\ref{lem:eps}, $\Bf^T\beta$ is a finite sum
$\beta_1'+\beta_2'+\dots$ of relative $1$-cycles $\beta_\ell'$ in
tubular neighborhoods $V^a$ of the unstable manifolds
$W^a\cap\{L^\ell\geq a-\eps\}$ of critical $\ell$-tuples of length $a$.
Recall that critical $\ell$-tuples consist of binormal chords and
$Q$-spikes (corresponding to constant $2$-gons). Using the operation
$\delta_N$, we can replace $Q$-spikes by differences of $N$-strings to
obtain a relative $1$-cycle $\beta''$ in $V^a$ homologous to
$\Bf^T\beta$ which contains no $Q$-spikes. So each $1$-simplex
$\beta_j''$ in $\beta''$ is a relative $1$-chain whose $Q$-strings lie
in the tubular neighborhood $V_j$ of the unstable manifold of some
permutation $w_j$ of $w$. Then the $N$-strings in $\beta''$ do not intersect
the knot in their interior, and by Lemma~\ref{lem:eps} neither do the
$Q$-strings. Thus each $\beta_j''$ is a relative cycle in $V_j$ with
respect to the singular boundary $\p$. We distinguish $2$ subcases.

(i) If the total degree of the word $w$ is bigger than $1$, then its stable
manifold for the flow of $-\nabla E$ has codimension bigger than
$1$. So, after a small perturbation, each $\beta_j''$ will avoid the stable
manifold of $w_j$ and will therefore have length at most $a-\eps$ for
sufficiently large $T$. This shows that, as in Case 1, both
groups vanish and $\Theta$ is an isomorphism.

(ii) If the degree of the word $w$ is $0$, then its unstable manifold is a
point and thus each $V_j$ is contractible relative to $\{L\leq a-\eps\}$.
It follows that each relative cycle $\beta_j''$ is $\p$-exact, and
since no $\delta_Q$ and $\delta_N$ occurs also $D$-exact. Again we see
that both groups vanish and $\Theta$ is an isomorphism.

(iii) If the degree of the word $w$ is $1$, then each $V_j$
deformation retracts relative to $\{L\leq a-\eps\}$ onto the
$1$-dimensional unstable manifold $\wt w_j$ of $w_j$.
It follows that each relative cycle $\beta_j''$ is $\p$-homologous,
and since no $\delta_Q$ and $\delta_N$ occurs also $D$-homologous,
to a multiple of the $1$-chain of $Q$-strings $\wt w_j$ connected by
suitable $N$-strings. By definition of $\Theta$, this shows that the
$D$-homology class $[\beta'']=[\beta]$ lies in the image of $\Theta$.
So $\Theta$ is surjective, which concludes the case $i=1$.

In the case $i=0$, the proof of surjectivity is analogous but simpler
than in the case $i=1$. For injectivity one considers $\Bf^T\beta$ for
a $1$-chain $\beta$ in $\Sigma_\lin$ with $D\beta=\alpha$ for a given
$0$-chain $\alpha$ and argues similarly. Note that this last step does
not work to prove injectivity for $i=1$ because it would require
considering $\Bf^T\beta$ for a $2$-chain $\beta$, which we have not
defined (although this should of course be possible).
\end{proof}

\subsection{Proof of the isomorphism}

Let $\Phi:(C_*(\RR),\p_\Lambda)\to (C_{\ast}(\Sigma),D)$ be the chain
map constructed in the previous section. We now use the fact
(Corollary~\ref{cor:respect-length}) that
the map $\Phi$ preserves the length filtrations. Thus for $a<b<c$ we
have the commuting diagram with exact rows of length filtered homology
groups
\begin{equation*}
\begin{CD}
   H_1^{[b,c)}(\RR) @>>> H_0^{[a,b)}(\RR) @>>> H_0^{[a,c)}(\RR)
         @>>> H_0^{[b,c)}(\RR) @>>> 0 \\
   @VV{\Phi_*}V @VV{\Phi_*}V @VV{\Phi_*}V @VV{\Phi_*}V @VVV \\
   H_1^{[b,c)}(\Sigma) @>>> H_0^{[a,b)}(\Sigma) @>>> H_0^{[a,c)}(\Sigma)
         @>>> H_0^{[b,c)}(\Sigma) @>>> 0\,.
\end{CD}
\end{equation*}
The main result of this section asserts that $\Phi_*$ is an
isomorphism (resp.~surjective) for sufficiently small action intervals:

\begin{prop}\label{prop:rel-hom}
For each $a\in\R$ there exists an $\eps_a>0$ such that for each
$\eps<\eps_a$ the map
$$
   \Phi_*:H_0^{[a-\eps,a+\eps)}(\RR) \to H_0^{[a-\eps,a+\eps)}(\Sigma)
$$
is an isomorphism and the map
$$
   \Phi_*:H_1^{[a-\eps,a+\eps)}(\RR) \to H_1^{[a-\eps,a+\eps)}(\Sigma)
$$
is surjective.
\end{prop}

This proposition implies Theorem~\ref{thm:main} as follows.

Since $H_0(\RR) = \lim_{R\to\infty}H_0^{[0,R)}(\RR)$ and $H_0(\Sigma) =
\lim_{R\to\infty}H_0^{[0,R)}(\Sigma)$, it suffices to show that
$\Phi_*:H_0^{[0,R)}(\RR)\to H_0^{[0,R)}(\Sigma)$ is an isomorphism for
each $R>0$. Now the compact interval $[0,R]$ is covered by finitely
many of the open intervals $(a-\eps_a,a+\eps_a)$, with $a\in[0,R]$ and
$\eps_a$ as in Proposition~\ref{prop:rel-hom}. Thus, according to
Proposition~\ref{prop:rel-hom}, there exists a partition
$0=r_0<r_1<\cdots<r_N=R$ such that the maps
$\Phi_*:H_0^{[r_{i-1},r_i)}(\RR)\to H_0^{[r_{i-1},r_i)}(\Sigma)$ are
isomorphisms and
$\Phi_*:H_1^{[r_{i-1},r_i)}(\RR)\to H_1^{[r_{i-1},r_i)}(\Sigma)$
are surjective for all $i=1,\dots,N$. To prove by induction that
$\Phi_*:H_0^{[0,r_i)}(\RR)\to H_0^{[0,r_i)}(\Sigma)$ is an isomorphism for
each $i=1,\dots,N$, consider the commuting diagram above with
$a=0$, $b=r_{i-1}$ and $c=r_i$. By induction hypothesis for $i-1$
the second, fourth and fifth vertical maps are isomorphisms and the
first one is surjective, so by the five lemma the third vertical map is an isomorphism as well.
This proves the inductive step and hence Theorem~\ref{thm:main}.

\begin{proof}[Proof of Proposition~\ref{prop:rel-hom}]
Let us denote the maps provided by Proposition~\ref{prop:pl} by
$\Bf_i^\pl,\BH_i^\pl$ and the maps in Proposition~\ref{prop:pl-lin} by
$\Bf_i^\lin,\BH_i^\lin$, $i=0,1$. A short computation shows that the
maps
\begin{align*}
   \Bf_i &:= \Bf_i^\lin\circ\Bf_i^\pl:C_i(\Sigma) \to
   C_i(\Sigma_\lin), \cr
   \BH_i &:= \BH_i^\pl+i_\pl\circ\BH_i^\lin\circ\Bf_i^\pl:C_i(\Sigma) \to C_{i+1}(\Sigma)
\end{align*}
for $i=0,1$ satisfy with the map $i:=i_\pl\circ
i_\lin:C_*(\Sigma_\lin)\hookrightarrow C_*(\Sigma)$:
\begin{enumerate}
\item $\Bf_0i=\id$ and $D\BH_0= i\Bf_0-\id$;
\item $\Bf_1i=\id$ and $\BH_0D + D\BH_1 = i\Bf_1-\id$;
\item $\Bf_0$, $\BH_0$, $\Bf_1$ and $\BH_1$ are (not necessarily strictly) length-decreasing.
\end{enumerate}
Conditions (i) and (ii) imply $D\Bf_1=\Bf_0D$ and
$i\Bf_1D=D(\id+\BH_1D)$, and therefore
$$
   \Bf_0(\im D)\subset\im D,\qquad
   \Bf_1(\ker D)\subset\ker D,\qquad
   \Bf_1(\im D)\subset i^{-1}(\im D).
$$
Hence the $\Bf_i$ define chain maps between the chain complexes (where
the left horizontal maps are the obvious inclusions)
\begin{equation*}
\begin{CD}
   \im D @>>> C_1(\Sigma) @>D>> C_0(\Sigma) \\
    @VV{\Bf_1}V @VV{\Bf_1}V @VV{\Bf_0}V \\
   i^{-1}(\im D) @>>> C_1(\Sigma_\lin) @>{D^\lin}>> C_0(\Sigma_\lin)\,.
\end{CD}
\end{equation*}
Note that the upper complex computes the homology groups
$H_0(\Sigma)$ and $H_1(\Sigma)$, while the lower complex has homology
groups $H_0(\Sigma^\lin)$ and
$$
   \wh H_1(\Sigma^\lin) := \ker D^\lin/i^{-1}(\im D).
$$
Conditions (i) and (ii) show that $\Bf_0,\Bf_1$ induce
isomorphisms between these homology groups (with inverses $i_*$), and
in view of condition (iii) the same holds for length filtered homology
groups. Setting
$$
   \Psi:=\Bf_i\circ\Phi:(C_i(\RR),\p_\Lambda)\to
   (C_i(\Sigma^\lin),D),\qquad i=0,1,
$$
it therefore suffices to prove:
{\em For each $a\in\R$ there exists an $\eps_a>0$ such that for each
$\eps<\eps_a$ the map
$$
   \Psi_*:H_0^{[a-\eps,a+\eps)}(\RR) \to H_0^{[a-\eps,a+\eps)}(\Sigma^\lin)
$$
is an isomorphism and the map
$$
   \Psi_*:H_1^{[a-\eps,a+\eps)}(\RR) \to \wh H_1^{[a-\eps,a+\eps)}(\Sigma^\lin)
$$
is surjective.}

We take for $\eps_a$ the constant from Lemma~\ref{lem:eps} and
consider $\eps<\eps_a$. Then we have canonical isomorphisms
$$
   \Gamma:H_i^{[a-\eps,a+\eps)}(\RR) \cong \AA^a_i,\qquad i=0,1
$$
to the groups $\AA^a_i$ introduced in the previous subsection. Recall
the maps $\Theta:\AA_i^a\to H_i^{[a-\eps,a+\eps)}(\Sigma_\lin,D)$ from
Corollary~\ref{cor:lin-hom} which are an isomorphism for $i=0$ and
surjective for $i=1$.

We consider first the case $i=0$. By
Proposition~\ref{prop:disklength}, for a binormal chord $c$ of
index $0$ and length $a$ the moduli space of holomorphic disks with positive puncture
$c$ and switching boundary conditions contains one component corresponding
to the half-strip over $c$, and on all other components the
$Q$-strings in the boundary have total length less than $a-\eps$.
This shows that the map
$\Psi_*:H_0^{[a-\eps,a+\eps)}(\RR) \to H_0^{[a-\eps,a+\eps)}(\Sigma^\lin)$
agrees with $\Theta\circ\Gamma$ and is therefore an isomorphism.

For $i=1$ we have a diagram
\begin{equation*}
\begin{CD}
   H_1^{[a-\eps,a+\eps)}(\RR) @>{\Psi_*}>>\wh H_1^{[a-\eps,a+\eps)}(\Sigma^\lin) \\
   @V{\cong}V{\Gamma}V @AA{\Pi}A \\
   \AA^a_1 @>{\Theta}>> H_1^{[a-\eps,a+\eps)}(\Sigma^\lin),
\end{CD}
\end{equation*}
where $\Pi:H_1(\Sigma^\lin) = \ker D^\lin/\im D^\lin \to \ker
D^\lin/i^{-1}(\im D) = \wh H_1(\Sigma^\lin)$ is the canonical
projection. Since $\Pi$ and $\Theta$ are surjective, surjectivity
of $\Psi_*$ follows once we show that the diagram commutes.

To see this, consider a word $w=b_1\cdots b_kc$ of binormal chords of
indices $|b_i|=0$ and $|c|=1$ and total length $a$. The $1$-dimensional moduli space
of holomorphic strips with positive puncture asymptotic to $c$ and one
boundary component on the zero section contains a unique component
$\MM_c$ passing through the trivial strip over $c$. By
Proposition~\ref{prop:disklength}, for each other element in
$\MM_c$ the boundary on the zero section has length strictly less
than $L(c)$. So, for $\eps$ sufficiently small, the moduli space
represents a generator of the local first homology at $c$. Since on all
other components of the moduli space the $Q$-strings in the boundary
have total length less than $a-\eps$, the product of $\MM_c$ with the
half-strips over the $b_j$ gives $\Phi(w)\in
C_1^{[a-\eps,a+\eps)}(\Sigma)$. Its image $\Psi(w)=\Bf_1\circ\Phi(w)\in
C_1^{[a-\eps,a+\eps)}(\Sigma_\lin)$ is obtained from $\Phi(w)$ by
shortening the $Q$-strings to linear ones. Since the tuples of
$Q$-strings in $\Phi(w)$ were either $C^1$-close to $w$ (depending on $\eps$)
or had total length less that $a-\eps$, the same holds for
$\Psi(w)$. Hence $\Psi(w)$ is homologous (with respect to $\p$, and
therefore with respect to $D$) in $C_1^{[a-\eps,a+\eps)}(\Sigma_\lin)$
to the unstable manifold of $w$ in $\Sigma_\lin$, which by definition
equals $\Pi\circ\Theta\circ\Gamma(w)$.

In the previous argument we have ignored the $N$-strings, always
connecting the ends of $Q$-strings to the base point by capping
paths. More generally, a generator of $H_1^{[a-\eps,a+\eps)}(\RR)\cong
\AA^a_1$ is given by a word $\gamma_1c_1\cdots\gamma_\ell
c_\ell\gamma_{\ell+1}$, where the $c_j$ are binormal chords with one
of them of index $1$ and all others of index $1$, and the $\gamma_j$
are homotopy classes of $N$-strings connecting the end points and not
intersecting $K$ in the interior. Now we apply the same arguments as
above to the $Q$-strings, dragging along the $N$-strings, to prove
commutativity of the diagram. This concludes the proof of
Proposition~\ref{prop:rel-hom}, and thus of Theorem~\ref{thm:main}.
\end{proof}

\section{Properties of holomorphic disks}\label{S:mdlisp}

In this section we begin our analysis of the holomorphic disks
involved in the definition of the chain map from Legendrian contact
homology to string homology. For the remainder of the paper, we
consider the following setup:
\begin{itemize}
\item $Q$ is a real analytic Riemannian $3$-manifold without
  closed geodesics and convex at infinity (the main example being
  $Q=\R^3$ with the flat metric);
\item $K\subset Q$ is a real analytic knot with nondegenerate binormal
  chords;
\item $L_K\subset T^{\ast}Q$ is the conormal bundle, $Q\subset
T^{\ast}Q$ is the $0$-section, and
$$
   L = L_K\cup Q
$$
is the singular Lagrangian with clean intersection $L_K\cap Q=K$.
\end{itemize}
The reader will notice that much of the discussion naturally extends
to higher dimensional manifolds $Q$ and submanifolds $K\subset Q$.

\subsection{Almost complex structures}\label{s:acs}
Consider
the subsets
$$
   S^*Q = \{(q,p)\;\bigl|\;|p|=1\}\subset
   D^*Q = \{(q,p)\;\bigl|\;|p|\leq 1\}\subset T^*Q
$$
of the cotangent bundle. The canonical isomorphism
\begin{equation*}
   \R\times S^*Q\to T^*Q\setminus Q,\qquad
   \bigl(s,(q,p)\bigr)\mapsto (q,e^s p)
\end{equation*}
intertwines the $\R$-actions given by translation resp.~rescaling. Let
$\lambda=p\,dq$ be the canonical Liouville form on $T^*Q$ with
Liouville vector field $p\p_p$. Its
restriction $\lambda_1$ to $S^*Q$ is a contact form with contact
structure $\xi=\ker\lambda_1$ and Reeb vector field $R$. We denote the
$\R$-invariant extensions of $\lambda_1,\xi,R$ to $T^*Q\setminus Q$ by
the same letters. In geodesic normal coordinates $q_i$ and dual
coordinates $p_i$ they are given by
$$
   \lambda_1=\frac{p\,dq}{|p|},\qquad
   R = \sum_i p_i\frac{\p}{\p q_i},\qquad
   \xi_{(q,p)} = \ker\lambda_1\cap\ker(p\,dp) = {\rm span}\Bigl\{R,p\frac{\p}{\p p}\Bigr\}^{\perp_{d\lambda_1}}.
$$
Around each Reeb chord $c:[0,T]\to S^*Q$ with end points on
$\Lambda_K=L_K\cap S^*Q$ we pick a neighborhood
$U\times(-\eps,T+\eps)\subset S^*Q$, where $U$ is a neighborhood of
the origin in $\C^{2}$, with the following properties:
\begin{itemize}
\item the Reeb chord $c$ corresponds to $\{0\}\times[0,T]$;
\item the Reeb vector field $R$ is parallel to $\p_t$, where $t$ is
  the coordinate on $(-\eps,T+\eps)$ and the contact planes project
  isomorphically onto $U$ along $R$;
\item along $\{0\}\times(-\eps,T+\eps)$ the contact planes agree with
  $\C^2\times\{0\}$ and the form $d\lambda_1$ with $\om_{\rm st}=dx_1\wedge
  dy_1+dx_2\wedge dy_2$;
\item the Legendrian $\Lambda_{K}$ intersects $U\times(-\eps,T+\eps)$ in two
linear subspaces contained in $U\times\{0\}$ and $U\times\{T\}$,
respectively, whose projections to $U$ are transversely intersecting
Lagrangian subspaces of $(\C^2,\om_{\rm st})$.
\end{itemize}

\begin{definition}\label{def:admissible}
An almost complex structure $J$ on $T^*Q$ is called {\em admissible}
if it has the following properties.
\begin{enumerate}
\item $J$ is everywhere compatible with the symplectic form $dp\wedge
  dq$. Moreover, $Q$ admits an exhaustion $Q_1\subset Q_2\subset\cdots$ by
  compact sets with smooth boundary such that the pullbacks
  $\pi^{-1}(\p Q_i)$ under the projection $\pi:T^*Q\to Q$ are
  $J$-convex hypersurfaces.
\item Outside $D^*Q$, $J$ agrees with an $\R$-invariant almost
  complex structure $J_1$ on the symplectization that takes the
  Liouville field $p\p_p$ to the Reeb vector field $R$, restricts
  to a complex structure on the contact distribution $\xi$, and is
  compatible with the symplectic form $d\lambda_1$ on $\xi$.
\item Outside the zero section, $J$ preserves the subspace
  ${\rm span}\{p\p_p,R\}$ as well as $\xi$ and is
  compatible with the symplectic form $d\lambda_1$ on $\xi$.
  Along the zero section, $J$ agrees with the canonical structure
  $\frac{\p}{\p p_i}\mapsto\frac{\p}{\p q_i}$.
\item $J$ is integrable near $K$ such that $Q$ and $K$ are real
  analytic.
\item On each neighborhood $U\times(-\eps,T+\eps)$ around a Reeb
  chord as above, the restriction of $J_1$ to the contact planes is
  the pullback of the standard complex structure on $U\subset\C^2$
  under the projection.
\end{enumerate}
\end{definition}

\begin{remark}
Conditions (i) and (ii) are standard conditions for studying
holomorphic curves in $T^*Q$ and its symplectization $\R\times S^*Q$.
Condition (iii) ensures the crucial length estimate for holomorphic
curves in the next subsection. Condition (iv) is needed for the
Finiteness Theorem~\ref{thm:finite} to hold.
Condition (v) is added to facilitate our study of spaces of
holomorphic disks and is convenient for fixing gauge when finding
smooth structures on moduli spaces; it can probably be removed with
a more involved analysis of asymptotics.
\end{remark}

\begin{remark}
Note that an admissible almost complex structure remains so under
arbitrary deformations satisfying (ii) that are supported outside $D^*Q$
and away from the Reeb chords. This gives us enough
freedom to achieve transversality within the class of admissible
structures in Section~\ref{sec:trans}.
\end{remark}

The Riemannian metric on $Q$ induces a canonical almost complex
structure $J_\st$ on $T^*Q$ which in geodesic normal coordinates $q_i$
at a point $q$ and dual coordinates $p_i$ is given by
$$
   J_\st\left(\frac{\p}{\p q_i}\right)=-\frac{\p}{\p p_i},\qquad J_\st\left(\frac{\p}{\p
   p_i}\right)=\frac{\p}{\p q_i}.
$$
More generally, for a positive smooth function $\rho\colon
[0,\infty)\to(0,\infty)$ we define an almost complex structure $J_{\rho}$ by
$$
   J_\rho\left(\frac{\p}{\p q_i}\right)=-\rho(|p|)\frac{\p}{\p p_i},\qquad
   J_\rho\left(\frac{\p}{\p p_i}\right)=\rho(|p|)^{-1}\frac{\p}{\p q_i}.
$$
If $\rho(r)=r$ for large $r$, then it is easy to check that $J_\rho$
satisfies the first part of condition (i) as well as conditions (ii)
and (iii) in Definition~\ref{def:admissible}.
If the metric is flat (i.e., $Q$ is $\R^3$ or a quotient of $\R^3$ by
a lattice), then $J_\st$ is integrable and $J_\rho$ also satisfies the
second part of (i) (choosing $Q_i$ to be round
balls) and condition (iv). Condition (v) can then be
arranged by deforming $J_\rho$ near infinity within the class of
almost complex structures satisfying (ii). So we have shown the following.

\begin{lemma}\label{lem:admissible}
For $Q=\R^3$ with the Euclidean metric there exist admissible almost
complex structures in the sense of Definition~\ref{def:admissible}.
\hfill $\square$
\end{lemma}

\begin{remark}
In fact, the almost complex structure $J_\st$ induced by the metric is
integrable if and only if the metric is flat (this observation is due
to M.~Gr\"uneberg, unpublished). So the preceding proof
of Lemma~\ref{lem:admissible} does not carry over to general manifolds
$Q$ (although the conclusion should still hold).
\end{remark}

The next result provides nice holomorphic coordinates near $K\subset
T^{\ast}Q$.

\begin{lemma}\label{l:knotnbhd}
Suppose that $J$ satisfies condition (iv) in Definition~\ref{def:admissible}.
Then for $\delta>0$ small enough there exists a holomorphic embedding from
$S^{1}\times (-\delta,\delta)\times B^{4}_\delta$, where
$B^{4}_\delta\subset \C^{2}$ is the ball of radius $\delta$, with its
standard complex structure onto a neighborhood of $K$ in
$T^{\ast}Q$ with complex structure $J$ with the following
properties:
\begin{itemize}
\item $S^{1}\times \{0\}\times\{0\}$ maps onto $K$;
\item $S^{1}\times\{0\}\times(\R^{2}\cap B^{4}_\delta)$ maps to $Q$;
\item $S^{1}\times\{0\}\times(i\R^{2}\cap B^{4}_\delta)$ maps to $L_K$.
\end{itemize}
Alternatively, we can arrange the last two properties with the roles
of $Q$ and $L_K$ interchanged. 
\end{lemma}

\begin{proof}
This is proved in more generality in~\cite[Remark 3.2]{CEL}; for
convenience we repeat the proof in the situation at hand.
Consider the real analytic embedding $\gamma\colon S^{1}\to Q$
representing $K$. Pick a real analytic vector field $v$ on $Q$
which is nowhere tangent to $K$ along $K$. Let $v_1$ be the unit
vector field along $K$ in the direction of the component of $v$
perpendicular to $\dot\gamma$. Then $v_1$ is a real analytic vector
field along $K$. Let $v_2= \dot\gamma \x v_1$ be the unit vector field
along $K$ which is perpendicular to both $\dot\gamma$ and $v_1$ and
which is such that $(\dot\gamma,v_1,v_2)$ is a positively oriented basis of
$TQ$. Consider $S^{1}\times D^{2}$ with coordinates
$(s,\sigma_1,\sigma_2)$, $s\in\R/\Z$, $\sigma_j\in\R$. Since $K$ is an
embedding there exists $\rho>0$ such that
\begin{equation}\label{e:emb1}
\phi(s,\sigma_1,\sigma_2)=\gamma(s)+\sigma_1v_1(s)+\sigma_2 v_2(s)
\end{equation}
is an embedding for $\sigma_1^{2}+\sigma_2^{2}<\rho$. Note that the
embedding is real analytic. Equip $S^{1}\times D^{2}$ with the flat
metric and consider the induced complex structure on
$T^{\ast}(S^{1}\times D^{2})$. The real analyticity of $\phi$ in
\eqref{e:emb1} implies that it extends to holomorphic embedding $\Phi$
from a neighborhood of $S^{1}\times D^{2}$ in
$T^{\ast}(S^1\times D^{2})$ to a neighborhood of $K$ in
$T^{\ast}Q$ (here we use integrability of $J$ near $K$). In fact, locally $\Phi$ is obtained by
replacing the real variables $(s,\sigma_1,\sigma_2)$ in the power
series corresponding in the right hand side of \eqref{e:emb1} by their
complexifications $(s+it,\sigma_1+i\tau_1,\sigma_2+i\tau_2)$.
This proves the first assertion of the lemma. The alternative assertion
follows from this one by precomposing $\Phi$ with multiplication by
$i$ on $B^4_\delta$. 
\end{proof}

\begin{remark}
The coordinate system gives a framing of $K$ determined by the normal
vector field $v$. By real analytic approximation we can take $v$ to
represent any class of framings.  
\end{remark}

\subsection{Length estimates}\label{sec:length-estimates2}
In this subsection we show that the chain map $\Phi$ respects the
length filtrations. This was shown in~\cite{CL}
for the absolute case, i.e.~without the additional boundary condition
$L_K$, and the arguments carry over immediately to the relative
case. For completeness, we provide the proof in this subsection and we
keep the level of generality of \cite{CL}, which is slightly more than what
we use in this paper.

For preparation, consider a smooth function $\tau\colon [0,\infty) \to
[0,\infty)$ with $\tau'(s) \geq 0$ everywhere and $\tau(s)=0$ near
$s=0$. Then
$$
   \lambda_\tau := \frac{\tau(|p|) p\,dq}{|p|}
$$
defines a smooth $1$-form on $T^*Q$.

\begin{lemma}\label{lem:nonneg}
Let $J$ be an admissible almost complex structure on $T^*Q$ and $\tau$
a function as above. Then for all $v\in T_{(q,p)}T^*Q$ we have
$$
   d\lambda_\tau(v,Jv)\geq 0.
$$
At points where $\tau(|p|)>0$ and $\tau'(|p|)>0$ equality holds only
for $v=0$, whereas at points where $\tau(|p|)>0$ and $\tau'(|p|)=0$
equality holds if and only if $v$ is a linear combination of the
Liouville field $p\,\p_{p}$ and the Reeb vector field $R=p\,\p_{q}$.
\end{lemma}

\begin{proof}
By condition (iii) in Definition~\ref{def:admissible}, $J$ preserves
the splitting
$$
   T(T^*Q)={\rm span}\{p\,\p_p,R\}\oplus \xi
$$
and is compatible with $d\lambda_1$ on $\xi$. Let us
denote by $\pi_1:T(T^*Q)\to {\rm span}\{p\,\p_p,R\}$ and
$\pi_2:T(T^*Q)\to \xi$ the projections onto the direct summands. Since
$\ker(d\lambda_1) = {\rm span}\{p\,\p_p,R\}$, for $v\in T_{(q,p)}T^*Q$ we
conclude
$$
   d\lambda_1(v,Jv) = d\lambda_1(\pi_2v,J\pi_2v)\geq 0,
$$
with equality iff $v\in {\rm span}\{p\,\p_p,R\}$. Next, we consider
$$
   d\lambda_\tau = \tau(|p|)d\lambda_1 +
   \frac{\tau'(|p|)}{|p|}p\,dp\wedge\lambda_1.
$$
Since the form $p\,dp\wedge\lambda_1$ vanishes on $\xi$ and is
positive on ${\rm span}\{p\,\p_p,R\}$, we conclude
$$
   d\lambda_\tau(v,Jv) = \tau(|p|)d\lambda_1(\pi_2v,J\pi_2v) +
   \frac{\tau'(|p|)}{|p|}p\,dp\wedge\lambda_1(\pi_1v,J\pi_1v) \geq 0,
$$
with equality iff both summands vanish. From this the lemma follows.
\end{proof}

\comment{
\begin{lemma}\label{lem:nonneg}
For any functions $\rho$ and $\tau$ as above, if $v\in
T_{(q,p)}T^*Q$, $|p|\le 1$ then
$$
   d\lambda_\tau(v,J_\rho v)\geq 0.
$$
Furthermore, at points where $\tau'(|p|)>0$, equality holds only for $v=0$, whereas at
points where $\tau'(|p|)=0$ and $\tau(|p|)>0$  equality holds if and only if
$v$ is a linear combination of $\sum_{i}p_{i}\p_{p_{i}}$ and $\sum_{i}p_{i}\p_{q_{i}}$.
\end{lemma}

\begin{proof}
We compute (in geodesic normal coordinates)
\begin{align*}
   d\lambda_\tau &= d\left(\sum_i\frac{\tau(|p|)p_idq_i}{|p|}\right) \\
   & =  \sum_i\frac{\tau(|p|)dp_i\wedge dq_i}{|p|} +
   \sum_{i,j}\frac{(\tau'(|p|)|p|-\tau(|p|))p_ip_jdp_i\wedge dq_j}{|p|^3}.
\end{align*}
For a vector of the form $v=\sum_ia_i\rho(|p|)\p_{p_i}$ we obtain
$J_\rho v=\sum_ia_i\p_{q_i}$ and hence by the Cauchy-Schwarz inequality
\begin{align*}
   d\lambda(v,Jv) &= \sum_i\frac{\tau(|p|)\rho(|p|)a_i^2}{|p|} -
   \sum_{i,j}\frac{(\tau'(|p|)|p|-\tau(|p|))\rho(|p|)p_ip_ja_ia_j}{|p|^3} \cr
   &= \frac{\tau(|p|)\rho(|p|)}{|p|^3}(|a|^2|p|^2-\la a,p\ra^2) +
   \frac{\tau'(|p|)\rho(|p|)}{|p|^2}\la a,p\ra^2 \cr
   &\geq 0.
\end{align*}
At points where $\tau'>0$, equality only holds for $a=0$, and at points where
$\tau'=0$ equality holds iff $a$ is a multiple of $p$. Similarly, for a general
vector $v=\sum_ia_i\rho(|p|)\p_{p_i} - \sum_ib_i\p_{q_i}$ we get
$d\lambda(v,J_{\rho}v)\geq 0$, with equality if and only if either $a=b=0$ or $\tau'=0$ and both
$a$ and $b$ are multiples of $p$.
\end{proof}
}

Let now $J$ be an admissible almost complex structure on $T^*Q$ and
$$
   u\colon (\Sigma,\p \Sigma)\to (T^*Q,Q\cup L_K)
$$
be a $J$-holomorphic curve with finitely many positive boundary
punctures asymptotic to Reeb chords $a_1,\dots,a_s$ and with switching
boundary conditions on $Q\cup L_K$. Let $\sigma_1,\dots,\sigma_k$ be
the boundary segments on $Q$. Recall that $L(\sigma_i)$ denotes the
Riemannian length of $\sigma_i$ and $L(a_j)=\int_{a_j}\lambda_1$
denotes the action of the Reeb chord $a_j$, which agrees with the
length of the corresponding binormal chord.

\begin{prop}\label{prop:length-estimate} 
With notation as above we have
$$
   \sum_{i=1}^kL(\sigma_i) \leq \sum_{j=1}^sL(a_j),
$$
and equality holds if and only if $u$ is a branched covering of a
half-strip over a binormal chord.
\end{prop}

\begin{proof}
The idea of the proof is straightforward: integrate $u^*d\lambda_1$
over $\Sigma$ and apply Stokes' theorem. However, some care is
required to make this rigorous because the $1$-form $\lambda_1$ is
singular along the zero section.

Fix a small $\delta>0$. For $i=1,\dots,s$ pick biholomorphic
maps $\phi_i:[0,\delta]\times[0,1]\to N_i\subset\Sigma$ onto
neighborhoods $N_i$ in $\Sigma$ of the $i^{th}$ boundary segment
mapped to $Q$, so that $\phi_i(0,t)$ is a parametrization of the
$i^{th}$ boundary segment. We choose $\delta$ so small that $N_i \cap
N_j = \varnothing$ if $i\neq j$ and $u\circ\phi_i(\delta,\cdot)$ does
not hit the zero section (the latter is possible because otherwise by
unique continuation $u$ would be entirely contained in the zero
section, which it is not by assumption). For fixed $i$ we denote the
induced parametrization of $\sigma_i$ by $q(t):=u\circ\phi_i(t)\in Q$,
so we can write
$$
   u\circ\phi_i(s,t)=(q(t)+v(s,t),s\dot{q}(t) + w(s,t))
$$
with $v(0,t)=0=w(0,t)$, and therefore $\frac {\p v}{\p t}(0,t)=0=\frac
{\p w}{\p t}(0,t)$. The hypothesis that $J$ is standard near the zero
section (condition (iii) in Definition~\ref{def:admissible}) implies
that $\frac {\p v}{\p s}(0,t)=0=\frac {\p w}{\p s}(0,t)$. Denoting
$v_\delta=v(\delta,\cdot)$ and $w_\delta=w(\delta,\cdot)$ we compute
\begin{align*}
   (u\circ\phi_i)^*\lambda_1|_{s=\delta}
   &= \frac{\langle \delta\dot{q} + w_\delta,\dot{q} + \dot{v_\delta} \rangle}
   {|\delta\dot{q} + w_\delta|} dt \\
   &= \frac {\langle \dot{q} + \frac{w_\delta}{\delta}, \dot{q} + \dot{v_\delta}\rangle}
   {|\dot{q}+ \frac{w_\delta}{\delta}|} dt\\
   &= \bigl(|\dot q| + O(\delta)\bigr)dt,
\end{align*}
where in the last line we have used that $\dot v_\delta=O(\delta)$ and
$w_\delta=O(\delta^2)$.

Pick $\eps>0$ smaller than the minimal norm of the $p$-components of
$u\circ\phi_i(\delta,\cdot)$ for all $i$.
Pick a function $\tau\colon [0,\infty) \to [0,1]$ with
$\tau'\geq 0$, $\tau(s)=0$ near $s=0$, and $\tau(s)=1$ for $s\geq \eps$.
By Lemma~\ref{lem:nonneg}, the form $\lambda_\tau =
\frac{\tau(|p|)}{|p|}pdq$ on $T^*Q$ satisfies $u^*(d\lambda_\tau) \geq
0$. Note that $\lambda_\tau$ agrees with $\lambda_1 = \frac p {|p|}dq$
on the subset $\{|p|\geq \epsilon\} \subset T^*Q$, so the preceding
computation yields
$$
   \int_{\{s=\delta\}}(u\circ\phi_i)^*\lambda_\tau =
   \int_{\{s=\delta\}}\bigl(|\dot q| + O(\delta)\bigr)dt =
   L(\sigma_i)+O(\delta)
$$
for all $i$. Next, consider polar coordinates $(r,\pHi)$ around $0$ in the
upper half plane $H^+$ near the $j^{th}$ positive puncture. Then the asymptotic
behavior of $u$ near the punctures yields
\[
   \int_{\{r=\delta\}\cap H^+}u^*\lambda_\tau = L(a_j)+O(\delta).
\]
Now let $\Sigma_\delta\subset\Sigma$ be the surface obtained by
removing the neighborhoods $\{r\leq \delta\}\cap H^+$ around the
positive punctures and the neighborhoods $N_i$ of the
boundary segments mapped to $Q$, see Figure \ref{fig:sigmadelta}. 
\begin{figure}
\labellist
\small\hair 2pt
\pinlabel $L_K$ at 94 305
\pinlabel $L_K$ at 277 300
\pinlabel $L_K$ at 12 158
\pinlabel $L_K$ at 118 13
\pinlabel $L_K$ at 323 87
\pinlabel $Q$ at 42 247
\pinlabel $Q$ at 49 61 
\pinlabel $Q$ at 228 10
\pinlabel $\Sigma$ at 342 18
\pinlabel $\Sigma_\delta$ at 178 162
\pinlabel $N_1$ at 69 230 
\pinlabel $N_2$ at 79 83
\pinlabel $N_3$ at 217 39
\endlabellist
	\centering
	\includegraphics[width=.6\linewidth]{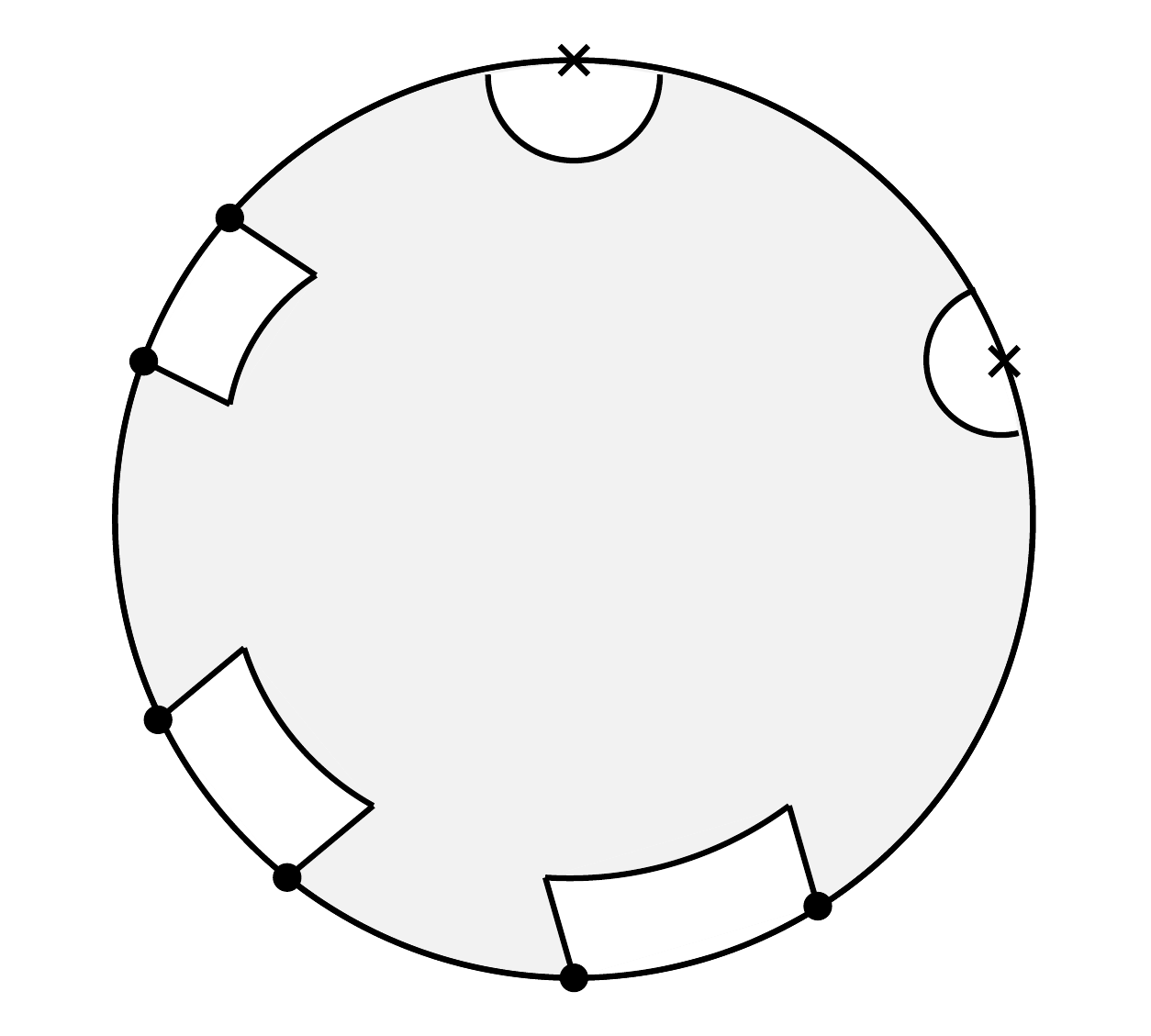}
	\caption{The domain $\Sigma_{\delta}$ is obtained from $\Sigma$ by removing small neighborhoods of the boundary arcs mapping to $Q$ and of the positive punctures. The punctures are denoted by x, and switches are denoted by dots.}
	\label{fig:sigmadelta}
\end{figure}

The boundary of $\Sigma_\delta$ consists
of the arcs $\{r=\delta\}\cap H^+$ around the positive punctures, the
arcs $\phi_i(\{s=\delta\})$ near the boundary segments mapped to $Q$
(negatively oriented), and the remaining parts of $\p\Sigma$ mapped to
$L_K$. Since $\lambda_\tau$ vanishes on $L_K$, the latter boundary
parts do not contribute to its integral and Stokes' theorem combined
with the preceding observations yields
$$
   0 \leq \int_{\Sigma_\delta}u^*d\lambda_\tau = \int_{\p\Sigma_\delta}u^*\lambda_\tau
   =  \sum_{j=1}^sL(a_j) - \sum_{i=1}^kL(\sigma_i) + O(\delta).
$$
Taking $\delta\to 0$ this proves the inequality in
Proposition~\ref{prop:length-estimate}. Equality holds iff
$u^*d\lambda_\tau$ vanishes identically, which by
Lemma~\ref{lem:nonneg} is the case iff $u$ is everywhere tangent to
${\rm span}\{p\,\p_p,R\}$. In view of the asymptotics at the positive
punctures, this is the case precisely for a half-strip over a binormal
chord.
\end{proof}

\subsection{Holomorphic half-strips}

We consider the half-strip $\R_+\times[0,1]$ with coordinates $(s,t)$ and
its standard complex structure. Let $J$ be an admissible almost
complex structure on $T^*Q$ and $J_1$ the associated structure on
$\R\times S^*Q$.
A {\em holomorphic half-strip} in $\R\times S^*Q$ is a holomorphic map
$$
   u\colon\R_+\times[0,1]\to (\R\times S^*Q,J_1)
$$
mapping the boundary segments $\R\times\{0\}$ and $\R\times\{1\}$ to
$\R\times\Lambda_K$. Similarly, a holomorphic half-strip in $T^*Q$ is
a holomorphic map
$$
   u\colon\R_+\times[0,1]\to (T^*Q,J)
$$
mapping the boundary to $L=L_{K}\cup Q$.
We write the components of a map $u$ into $\R\times S^*Q$ (or into
$T^*Q\setminus D^*Q\cong \R_+\times S^*Q$) as
$$
   u=(a,f).
$$
Recall from~\cite{BEHWZ} (see also~\cite{CEL}) that to any smooth
map $u$ from a surface to $\R\times S^*Q$ or $T^*Q$ we can associate
its {\em Hofer energy} $E(u)$. It is defined as the sum of two terms,
the $\om$-energy and the $\lambda$-energy, whose precise definition
will not be needed here.
The following result follows from \cite[Lemma B.1]{E_rsft}, see also
\cite[Proposition 6.2]{BEHWZ}, in combination with well-known results
in Lagrangian Floer theory, see e.g.~\cite{Floer_Lag_int}.

\begin{prop}\label{prop:asympt}
For each holomorphic half-strip $u$ in $\R\times S^*Q$ or $T^*Q$ of
finite Hofer energy exactly one of the following holds:
\begin{itemize}
\item There exists a Reeb chord $c\colon [0,T]\to S^*Q$ and
a constant $a_0\in\R$ such that
$$
   a(s,t)-Ts-a_0\to 0,\qquad f(s,t)\to c(Tt)
$$
uniformly in $t$ as $s\to\infty$. We say that the map has a \emph{positive puncture} at $c$.
\item There exists a Reeb chord $c\colon [0,T]\to S^*Q$ and
a constant $a_0\in\R$ such that
$$
   a(s,t)+Ts-a_0\to 0,\qquad f(s,t)\to c(-Tt)
$$
uniformly in $t$ as $s\to\infty$. We say that the map has a \emph{negative puncture} at $c$.
\item There exists a point $x_0$ on $\R\times\Lambda_K$ (resp.~$L$) such that
$$
   u(s,t)\to x_0
$$
uniformly in $t$ as $s\to\infty$. In this case $u\circ\chi^{-1}$,
where $\chi:\R_+\times[0,1]\to D^+$ is the map from~\eqref{eq:chi},
extends to a holomorphic map on the half-disk mapping the boundary to
$\R\times\Lambda_K$ (resp.~$L$). If $x_0\notin K$ then we say that $u$
has a \emph{removable puncture} at $x_0$, and if $x_0\in K$ then we say
that $u$ has a {\em Lagrangian intersection puncture} at $x_0$. (These are the standard situations in ordinary Lagrangian intersection Floer homology.)
\end{itemize}
\end{prop}

Because of our choice of almost complex structure we can say more
about the local forms of the maps as follows.

Consider first a Reeb chord puncture where the map approaches a Reeb
chord $c$. Let $U\times(-\eps,T+\eps)$ be the neighborhood of $c$ as
in Definition~\ref{def:admissible} (v) and note that the holomorphic
half-strip is uniquely determined by the local projection to
$U\subset\C^{2}$ where the complex structure is standard. By a complex
linear change of coordinates on $\C^{2}$ we can arrange that the two
branches of the Legendrian $\Lambda_K$ through the end points of $c$
project to $\R^{2}$ and to the subspace spanned by the vectors
$(e^{i\theta_1},0)$ and $(0,e^{i\theta_2})$, for some angles
$\theta_1,\theta_2$. The $\C^{2}$-component $v$ of the map $u$ then has
a Fourier expansion
\begin{equation}\label{eq:chordasympt}
v(z)=\sum_{n\ge 0} \left(c_{1;n}e^{-(\theta_1+n) z}, c_{2;n}e^{-(\theta_2+n) z}\right),
\end{equation}
where $c_{j;n}$ are real numbers. We call the smallest $n$ such that
$(c_{1;n},c_{2;n})\ne 0$ the \emph{order of convergence} to the Reeb
chord $c$.

We have similar expansions near the Lagrangian intersection
punctures. Lemma \ref{l:knotnbhd} gives holomorphic coordinates
$(z_0,z_1)=(x_0+iy_0,x_1+iy_1)$ in $\C\times\C^{2}$ around any
point $q_0\in K$ such that the Lagrangian
submanifold $Q\subset T^*Q$ corresponds to
$\{y_0=y_1=0\}$, the Lagrangian submanifold $L_K$ corresponds to
$\{y_0=x_1=0\}$, and the almost complex structure $J$ corresponds to the standard complex
structure $i$ on $\C^{3}$. Consider a holomorphic map $u\colon
[0,\infty)\times[0,1]\to T^*Q$ such that $u(z)\to q\in K$
as $z\to\infty$ where $q$ lies in a small neighborhood of $q_0$ in
$K$. We write $u$ in the local coordinates described above as
$v=(v_0,v_1)$. Now Remark~\ref{rem:series} yields the following Fourier
expansions for $v$. If $v([0,\infty)\times\{0\})\subset Q$
and $v([0,\infty)\times\{1\})\subset L_K$ then
\begin{equation}\label{e:nearK1}
v(z)=\left(\sum_{m\ge 0} c_{0,m} e^{-m\pi z}, \sum_{n+\frac12>0}
c_{1;n+\frac12} e^{-(n+\frac{1}{2})\pi z}\right),
\end{equation}
where $c_{0;m}\in\R$ for all $m\in\Z_{\ge 0}$ and where
$c_{1;n+\frac12}\in\R^{2}$ for all $n\in\Z_{\ge 0}$, in a neighborhood
of $\infty$.
If $v([0,\infty)\times\{0\})\subset L_K$ and
$v([0,\infty)\times\{1\})\subset Q$ then
\begin{equation}\label{e:nearK2}
v(z)=\left(\sum_{m\ge 0} c_{0,m} e^{-m\pi z}, i\sum_{n+\frac12>0}
c_{1;n+\frac12} e^{-(n+\frac{1}{2})\pi z}\right),
\end{equation}
where notation is as in \eqref{e:nearK1}. If
$v([0,\infty)\times\{0\})\subset Q$ and
$v([0,\infty)\times\{1\})\subset Q$ then
\begin{equation}\label{e:nearK3}
v(z)=\left(\sum_{n\ge 0} c_{0,m} e^{-m\pi z}, \sum_{n>0} c_{1;n}
e^{-n\pi z}\right),
\end{equation}
where $c_{0;m}$ is as in \eqref{e:nearK1} and $c_{1;n}\in\R^{2}$ all
$n\in\Z_{>0}$. If $v([0,\infty)\times\{0\})\subset L_K$ and
$v([0,\infty)\times\{1\})\subset L_K$ then
\begin{equation}\label{e:nearK4}
v(z)=\left(\sum_{n\ge 0} c_{0,m} e^{-m\pi z}, i\sum_{n>0} c_{1;n}
e^{-n\pi z}\right),
\end{equation}
where notation is as in \eqref{e:nearK3}. We say that the smallest
half-integer $n+\frac12$ in \eqref{e:nearK1} or \eqref{e:nearK2} such
that $c_{1,n+\frac12}\ne 0$ or the smallest integer $n$ in
\eqref{e:nearK3} or \eqref{e:nearK4} such that $c_{1;n}\ne 0$ is the
{\em asymptotic winding number} of $u$ at its Lagrangian intersection
puncture.

\subsection{Holomorphic disks}\label{s:disks}
Consider the closed unit disk $D\subset\C$ with $m+1$ cyclically ordered
distinct points $z_0,\dots,z_m$ on $\p D$. Set $\dot
D:=D\setminus\{z_0,\dots,z_m\}$. Consider a $J$-holomorphic map $u\colon\dot
D\to \R\times S^{\ast} Q$ resp.~$T^{\ast}Q$ which maps $\p D\setminus\{z_0,\dots,z_m\}$
to $\R\times\Lambda_{K}$ resp.~$L=Q \cup L_{K}$ and which has finite $\omega$-energy and $\lambda$-energy.
Proposition~\ref{prop:asympt} shows that near each puncture $z_j$ the map $u$ either extends continuously, or it is positively or negatively asymptotic to a Reeb chord. We will use the following notation for such disks.

A {\em symplectization disk (with $m \geq 0$ negative punctures)} is a
$J$-holomorphic map
$$
   u\colon(\dot D,\p \dot D)\to (\R\times S^{\ast} Q,\R\times\Lambda_{K})
$$
with positive puncture at $z_0$ and negative punctures at
$z_1,\dots,z_m$.
A {\em cobordism disk (with $m\ge 0$ Lagrangian intersection punctures)} is
a $J$-holomorphic map
$$
   u\colon(\dot D,\p \dot D)\to (T^{\ast}Q,L)
$$
with positive puncture at $z_0$ and Lagrangian intersection punctures
at $z_1,\dots,z_m$.

Let $\mathbf{b}=b_1b_2\dots b_m$ be a word of $m$ Reeb chords. We write
\[
   \MM^{\rm sy}(a,n_0;b_1,\dots,b_m)=\MM^{\rm sy}(a,n_0;\mathbf{b})
\]
for the moduli space of symplectization disks with positive puncture
asymptotic to the Reeb chord $a$ where the order of convergence is $n_0$ and $m$ negative punctures (in
counterclockwise order) asymptotic to the Reeb chords
$b_1,\dots,b_m$. Here the points $z_0,\dots,z_m$ on $\p D$
are allowed to vary and we divide by the action of M\"obius
transformations on $D$. Note that $\R$ acts by translation on
these moduli spaces.

Similarly, let $\mathbf{n}=(n_1,\dots,n_m)$ be a vector of
half-integers or integers. We write
\[
   \MM(a,n_0;n_1,\dots,n_m)=\MM(a,n_0;\mathbf{n})
\]
for the moduli space of cobordism disks with positive puncture
asymptotic to the Reeb chord $a$ with degree of convergence $n_0$ and $m \geq 0$ Lagrangian intersection
punctures with asymptotic winding numbers given by the integers or half-integers $n_j$.
Note that the number of half-integers must
be even for topological reasons (at each half-integer the boundary of
$u$ switches from $Q$ to $L_K$ or vice versa).

In both cases when $n_0=0$ we will suppress it from notation and simply write
\[
   \MM^{\rm sy}(a;\mathbf{b})\text{ and }\MM(a;\mathbf{n}),
\]
respectively.

For a Reeb chord $c\colon[0,T]\to S^*Q$ of length
$T$, the map $u_c\colon\R \x [0,1] \to \R \x S^*Q$ given by
$u_c(s+it)=(Ts, c(Tt))$ is a $J$-holomorphic parametrization of $\R \x
c$ and thus a symplectization disk with positive and negative puncture
asymptotic to $c$. We call it the {\em Reeb chord strip} over $c$.

\subsection{Compactness in $\R\times S^{\ast}Q$ and $T^{\ast} Q$}\label{S:cp}
In this subsection we review the compactness results proved in \cite{CEL} that concern compactness of the moduli spaces of holomorphic disks discussed in Section \ref{s:disks}.

Let us denote by a {\em source disk} $\D_m$ the unit disk with some number
$m+1\geq 1$ of punctures $z_0,\dots,z_m$ on its boundary; we call
$z_0$ the positive and $z_1,\dots,z_m$ the negative punctures.
A {\em broken source disk $\dot\D_m$ with $r\geq 1$ levels} with
$m+1$ boundary punctures is represented as a finite disjoint union of
punctured disks,
\[
\dot \D_m=\D^{1,1}\cup(\D^{2,1}\cup\dots\cup \D^{2,l_2})
\cup \dots\cup(\D^{r,1}\cup\dots\cup \D^{r,l_r}),
\]
where $(\D^{j,1}\cup\dots\cup\D^{j,l_j})$ are the disks in the
$j^{\rm th}$ level and we require the following properties:
\begin{itemize}
\item Each negative puncture $q$ of a disk $\D^{j,k}$ in the
  $j^{\rm th}$ level for $j<r$ is formally joined to the positive
  puncture of a unique disk $\D^{j+1,s}$ in the $(j+1)^{\rm th}$
  level. We say that $\D^{j+1,s}$ is attached to $\D^{j,k}$ at
  the negative puncture $q$.
\item The total number of negative punctures on level $r$ is $m$.
\end{itemize}
Note that a broken source disk with one level is just a source disk.

We consider first compactness for curves in the symplectization.
Let $\dot\D_m$ be a broken source disk as above. A {\em
broken symplectization disk with $r$ levels} with domain
$\dot \D_m$ is a collection $\dot v$ of $J$-holomorphic maps
$v^{j,k}$ defined on $\D^{j,k}$ with the following properties:
\begin{itemize}
\item For each $1 \leq j \leq r$ and $1\le k\le l_{j}$, $v^{j,k}$
  represents an element in
\[
\MM^{\rm sy}(a^{j,k};b^{j,k}_1,\dots,b^{j,k}_{s}).
\]
Moreover, for $j>1$, the Reeb chord $a^{j,k}$ at the positive
puncture of $v^{j,k}$ matches the Reeb chord $b^{j-1,k'}$ at the
negative puncture of $v^{j-1,k'}$ in $\D^{j-1,k'}$ at which
$\D^{j,k}$ is attached.
\item For each level $1 \leq j \leq r$, at least one
  of the maps $v^{j,k}$ is not a Reeb chord strip.
\end{itemize}

An {\em arc} in a source disk is an embedded curve that intersects the
boundary only at its end points and away from the punctures.
We say that a sequence of symplectization disks
\[
   \{u_j\}\subset\MM^{\rm sy}(a;b_1,\dots,b_m)
\]
{\em converges to a broken symplectization disk} if there are disjoint
arcs $\gamma_1,\dots,\gamma_k$ in the domains of $u_j$ which give the
decomposition of the domain into a broken source disk in the limit and
such that in the complement of these arcs, the maps $u_j$
converge to the corresponding map of the broken disk uniformly on
compact subsets.

\begin{thm}\label{t:cp_sy}
Any sequence $\{u_{j}\} \subset \MM^{\rm sy}(a,b_1,\dots,b_m)$ of
symplectization disks has a subsequence which converges to a
broken symplectization disk $\dot v$ with $r\geq 1$ levels.
\end{thm}

\begin{proof}
Follows from \cite{BEHWZ} (see also \cite[Theorem 1.1]{CEL}).
\end{proof}

In order to describe the compactness result for moduli spaces of
holomorphic disks in $T^{\ast}Q$ we first introduce a class of
constant holomorphic disks and then the notion of convergence to a constant disk. A \emph{constant holomorphic disk} is a
source disk $\D_m$, $m\ge 3$, a constant map into a point $q\in K$,
and the following extra structure: Each boundary component is labeled
by $L_K$ or by $Q$ and at each puncture $z_j$ there is an asymptotic
winding number $n_j\in \{\frac12, 1,\frac32,\dots\}$ such that $n_{j}$
is a half-integer if the adjacent boundary components of
$\dot{\D}_{m}$ are labeled by different components of $L=L_K\cup Q$ and an
integer otherwise, and such that $n_0=\sum_{j=1}^{m}n_j$. 

A sequence of holomorphic maps $v_j\colon \dot \D_{m}\to T^{\ast}Q$ with boundary on
$L$ \emph{converges to a constant holomorphic disk} if it converges uniformly
to the constant map on any compact subset and if for all sufficiently
large $j$, $v_j$ takes any boundary component labeled by $L_K$ or $Q$
to $L_K$ or $Q$, respectively, and if the asymptotic winding numbers
at the negative punctures of the maps $v_j$ agree with those of the
constant limit map at corresponding punctures.

Let $\dot\D_m$ be a broken source disk with $r$ levels and suppose
$1 \leq r_0\le r$. A {\em broken cobordism disk
with $r_0$ non-constant levels} and domain $\dot \D_m$ is a
collection $\dot v$ of $J$-holomorphic maps $v^{j,k}$ defined on
$\D^{j,k}$ with the following properties.
\begin{itemize}
\item For $j<r_0$ and $1\le k\le l_{j}$, $v^{j,k}$ represents an element in
\[
\MM^{\rm sy}(a^{j,k};b^{j,k}_1,\dots,b^{j,k}_{s}).
\]
Moreover, for $j>1$, the Reeb chord $a^{j,k}$ at the positive
puncture of $v^{j,k}$ matches the Reeb chord $b^{j-1,k'}$ at the
negative puncture of $v^{j-1,k'}$ in $\D^{j-1,k'}$ at which
$\D^{j,k}$ is attached.
\item For each level $j<r_0$, at least one
  of the maps $v^{j,k}$ is not a Reeb chord strip.
\item For $j=r_0$ and $1\le k\le l_{j}$, $v^{j,k}$ represents an
element in \[\MM(a^{j,k};n^{j,k}_1,\dots,n^{j,k}_{s})\]
and the Reeb chord at the positive puncture of $v^{j,k}$ matches the
Reeb chord at the negative puncture of $v^{j-1,k'}$ in
$\D^{j-1,k'}$ at which $\D^{j,k}$ is attached.
\item For $j>r_0$, $v^{j,k}$ is a constant map to $q\in K$, where
  $q\in K$ is the image of the negative puncture of $v^{j-1,k'}$
in $\D^{j-1,k'}$, at which $\D^{j,k}$ is attached. Moreover,
  $\D^{j,k}$ has at least $3$ punctures and the winding number and
labels at its positive puncture agree with those of the negative
puncture where it is attached. (From the point of view of the source disk these constant levels encode degenerations of the conformal structure corresponding to colliding Lagrangian intersection punctures, see Section \ref{ss:constglu} for more details.)
\end{itemize}
We say that the disks in levels $j<r_0$ are the {\em symplectization
disks}, that the disks in level $r_0$ are the {\em cobordism disks},
and that disks in levels $j>r_0$ are the {\em constant disks} of the broken
disk.

We define convergence to a broken cobordism disk completely
parallel to the symplectization case.

\begin{thm}\label{t:cp}
Let $\{u_{j}\} \subset \MM(a;n_1,\dots,n_m)$ be a sequence of
cobordism disks. Then $\{u_j\}$ has a subsequence which
converges to a broken cobordism disk.
\end{thm}

\begin{proof}
This is a consequence of \cite[Theorem 1.1]{CEL}. Note that the levels
of constant disks are recovered by the sequence of source disks that
converges to a broken source disk.
\end{proof}

\begin{remark}\label{r:breakingmodelclose}
We consider the convergence implied by the Compactness Theorem
\ref{t:cp} in more detail in a special case relevant to the
description of our moduli spaces below. Consider a sequence of
holomorphic disks $u_j$ as in the theorem that converges to a broken
cobordism disk with top level $v$ and such that all disks on lower levels are
constant. Let $q_\ell$ be a negative puncture of the top level $v$ and
let $D_\ell$ be the (possibly broken) constant disk attached with its
positive puncture at $q_\ell$.

Consider the sequence of domains of $u_j$ as a sequence of strips with slits $S_j$, see the discussion of standard domains in Section \ref{ssec:confrep} and Figure \ref{fig:stdom}.
It follows from the proof of \cite[Theorem 1.1]{CEL}
that there is a strip region $[-\rho_j,0]\times[0,1]\subset S_j$,
where $\rho_j\to\infty$ as $j\to\infty$ such that in the limit the
negative puncture $q_\ell$ of $v$ corresponds to
$(-\infty,0]\times[0,1]$ and the positive puncture of the domain $D_\ell$
corresponds to $[0,\infty)\times[0,1]$ attached at this
puncture. Assume that $q_\ell$ maps to $x\in K$ and consider the
Fourier expansion of $v$ near $q_\ell$ in the local coordinates
near $K$ perpendicular to the knot:
\[
   v(s+it)=e^{k_0\pi(s+it)}\sum_{k=0}^{\infty} c_k e^{k\pi(s+it)},
\]
where $k_{0}\ge\frac12$ is a half-integer and $c_k$ are vectors in
$\R^{2}$ or $i\R^{2}$, $c_0\ne 0$. We say that the complex line
spanned by $c_0$ is the limiting tangent plane of $v$
at $q_\ell$. Writing $v$ using Taylor expansion as a map from the upper
half plane with the puncture $q_\ell$ at the origin and taking the
complex line of $c_0$ as the first coordinate we find that the normal
component of $v$ at $x$ is given by
\[
   v(z) = \left(z^{k_0},\Ordo(z^{k_0+1})\right),
\]
after suitable rescaling of the first coordinate.

We next restrict to the case relevant to our applications, of a
sequence of disks $u_j$ with a  constant disk with three or four
punctures splitting off. The three punctured disk is simpler, so we
consider the case of a disk with four punctures splitting off. In this
case, consider a vertical segment $\{\rho^{0}\}\times[0,1]$
in the stretching strip $[-\rho_j,0]\times [0,1]$. It
subdivides the domain of $u_j$ in two components $D_+$ containing the
positive puncture and its complement $D_-$. Consider the Fourier
expansion of $u_j$ near this vertical segment. We have
\[
   u_j(s+it)=\sum_{k\ge k_0} c_{j;k} e^{-k\pi(s+it)},
\]
where $k$ are half-integers and $c_{j;k}\in\R^{2}$ (or
$i\R^{2}$). Since the winding number along the vertical segment
is equal to the sum of the winding numbers of the negative punctures in the component of $D_{-}$ that it bounds, we find that, for $j$ sufficiently large, $c_{j;k}=0$ for all $k<\frac32$,
hence $k_0\ge \frac32$. Moreover, $c_{j;k_0}$ converges to a vector in the limiting
tangent plane of $u_0$ at the newborn negative puncture. In the
generic case, see Lemma \ref{l:tv},
this limiting vector is non-zero. We
assume for definiteness in what follows that it is equal to $(1,0)$.

Pick a conformal map taking $D_-$ to the half disk of radius $1$ in the
upper half plane, with the vertical segment corresponding to the half
circular arc and with the middle boundary puncture mapping to
$0$. Then as $j\to\infty$ the locations of the other two punctures
both converge to $0$ and, for large $j$, the projection to the first complex coordinate determines the location of the other two punctures. Moreover, the sum of the winding numbers at these three punctures equals $\frac32$ (i.e.~the winding number along the half circle of radius $1$). Consequently, we have, with $z$ a coordinate on the upper half plane, for all $j$ large enough
\[
   u_j(z)=\sqrt{z(z-\delta_j)(z-\epsilon_j)}\Bigl((1,0)+v_j+\Ordo(z)\Bigr),
\]
where $\delta_j,\epsilon_j\to 0$ and $v_j\to 0\in\R^{2}$ as
$j\to\infty$. It follows that disks in a limiting sequence eventually
lie close to the model disk \eqref{eq:2-dimmodel} discussed in Section~\ref{ss:spikes}.

There is a completely analogous and simpler analysis of the case when
two punctures collide which shows that disks in a limiting sequence
are close to the model disk \eqref{eq:1-dimmodel} of Section~\ref{ss:spikes} in the same
sense.
\end{remark}

\section{Transversely cut out solutions and orientations}\label{sec:trans}

In this section we show that the moduli spaces in Section
\ref{S:mdlisp} are manifolds for generic almost complex structure
$J$. To accomplish this, we first express each moduli space as
the zero locus of a section of a bundle over a Banach manifold and then show,
using an argument from \cite{EES2}, that one may make any section
transverse to the $0$-section by perturbing the almost complex
structure. Here cases of disks with unstable domains require extra care: we stabilize their domains using extra marked points on the boundary. We control these marked points using disks with higher order of convergence to Reeb chords.

\subsection{Conformal representatives and Banach manifolds}\label{ssec:confrep}
In order to define suitable Banach spaces for our study of holomorphic curves we endow the domains of our holomorphic disks with cylindrical ends. For convenience we choose a particular such model for each conformal structure on the punctured disks. (The precise choice is not important since the space of possible choices of cylindrical ends is contractible.)

A \emph{standard domain} $\Delta_{0}$ with one puncture is the unit disk in the complex plane with a puncture at $1$ and fixed cylindrical end $[0,\infty)\times[0,1]$ at this puncture.

A \emph{standard domain} $\Delta_{1}$ with two punctures is the strip $\R\times[0,1]$.

A \emph{standard domain} $\Delta_m([a_1,\dots,a_{m-1}])$  with $m+1\geq 2$ boundary punctures
is a strip $\R\times[0,m]\subset\C$ with slits of small fixed width
(and fixed shape) around half-infinite lines $(-\infty,a_j]\times
\{j\}$, where $0<j<m$ is an integer, removed. See Figure
\ref{fig:stdom}.
\begin{figure}
\centering
\includegraphics[width=.6\linewidth, angle=180]{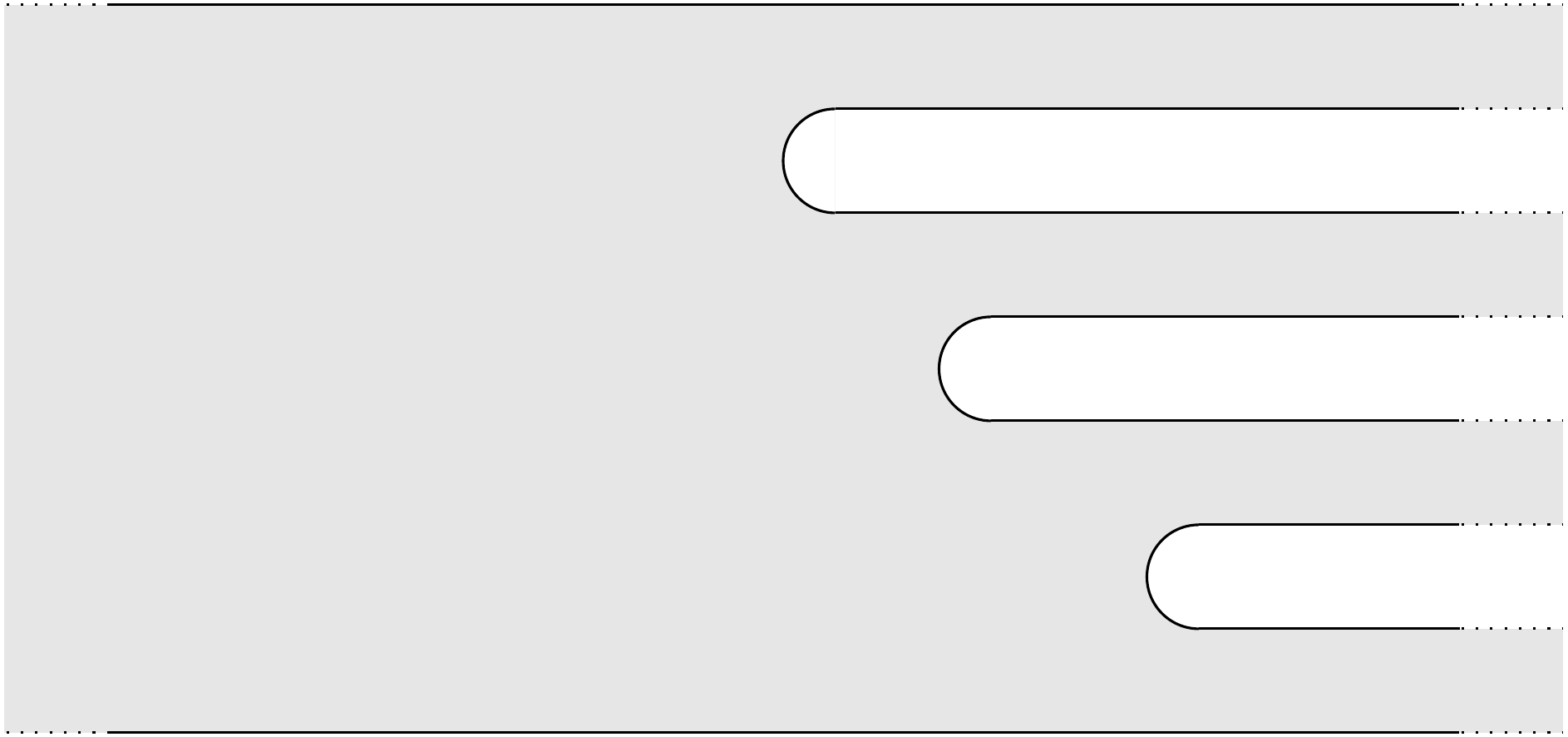}
\caption{A standard domain.}
\label{fig:stdom}
\end{figure}
We say that $a_j\in\R$ is the $j^{\rm th}$ boundary maximum of
$\Delta_m([a_1,\dots,a_{m-1}])$.

The space of conformal structures $\CC_m$ on the $(m+1)$-punctured
disk is then represented as $\R^{m}/\R$ where $\R$ acts on vectors of
boundary maxima by overall translation, see \cite[Section 2.1.1]{E}.
The boundary of the space of conformal structures on an
$(m+1)$-punctured disk in its compactification
$\pa\CC_m\subset\overline{\CC}_m$ can then be understood as consisting
of the several level disks which arise as some differences $|a_j-a_k|$
between boundary maxima approach $\infty$. We sometimes write
$\Delta_m$ for a standard domain, suppressing its conformal structure
$[a_1,\dots,a_{m-1}]$ from the notation.

The breaking of a standard domain into a standard domain of several levels is compatible with the compactness results Theorems \ref{t:cp_sy} and \ref{t:cp}. In the proof of these results given in \cite{CEL}, after adding a finite number of additional punctures the derivatives of the maps are uniformly bounded and each component in the limit has at least two punctures and can thus be represented as a standard domain. In particular, the domain right before the limit is the standard domain obtained by gluing these in the natural way and the arcs in
the definition of convergence can be represented by vertical segments. Here a \emph{vertical segment} in a standard domain $\Delta_m\subset\C$ is a line segment in $\Delta_m$ parallel to
the imaginary axis which connects two boundary components of $\Delta_m$.

\subsection{Configuration spaces}\label{s:confcob}
In this section we construct Banach manifolds which are configuration spaces for
holomorphic disks. In order to show that all moduli spaces we use are manifolds we need to stabilize disks with one and two punctures by adding punctures in a systematic way. To this end we will use Sobolev spaces with extra weights. This is the reason for introducing somewhat more complicated spaces below. The constructions in this section parallels corresponding constructions in \cite{EES2} and \cite{ES}.

We first define the configuration space for holomorphic disks in $T^{\ast}Q$ and then find local coordinates for this space showing that it is a Banach manifold. We then repeat this construction for disks in the symplectization.

Below we are interested in the moduli spaces $\MM(a,n_0;\mathbf{n})$ of holomorphic disks for $n_0=0$ or $n_0=1$ and $\mathbf{n}=(n_1,\dots,n_m)$, which we will describe as subsets of suitable configuration spaces $\mathcal{W}=\mathcal{W}(a;\delta_0;\mathbf{n})$. Here $\delta_0>0$ and $n_0$ are related as follows: Consider the standard neighborhood $(-\eps,T+\eps) \times U$ (with $U \subset \C^2$) of the Reeb cord $a:[0,T] \to S^*Q$ which we introduced on page \pageref{s:acs}. The projections of the contact planes at the two end points of $a$ to $\C^2$ intersect transversally, and we denote by $0<\theta'\pi\le \theta''\pi<\pi$ the two complex angles between them. Now for $n_0=0$ we choose $0<\delta_0<\theta'$ and for $n_0=1$ we choose $\theta''<\delta_0<1$.

The space $\mathcal{W}$ fibers over the product space
\[
B=\R^{m-2}\times \R\times J(K).
\]
The first factor $\R^{m-2}$ is the space of conformal structures on the disk with $m+1$ boundary punctures. We represent the disk as a standard domain with the first boundary maximum at $0$ and $\R^{m-2}$ as the coordinates of the remaining $m-2$ boundary maxima. The second factor $\R$ corresponds to the shift in parameterization of the asymptotic trivial strips at the positive puncture. The third factor is itself a product with one factor for each negative puncture:
\[
J(K)=J^{(r_1)}(K)\times\dots\times J^{(r_m)}(K).
\]
Here $r_{j}$ is the smallest integer $<n_j$ and $J^{(r_{j})}(K)$ denotes the $r_{j}^{\rm th}$ jet-space of $K$. A point $(q_0,q_1,\dots,q_{r_j}) \in J^{(r_j)}(K)$ corresponds to the first Fourier (Taylor) coefficients of the map at the $j^{\rm th}$ negative puncture. Note that $J(K)$ depends on $\mathbf{n}=(n_1,\dots,n_m)$, but we omit this dependence from the notation.

Fix a parameterization of each Reeb chord strip. If $\gamma\in\R^{m-2}$ then we write $\Delta[\gamma]$ for the standard domain with first boundary maximum at $0$ and the following boundary maxima according to the components of $\gamma$. If $\mathbf{n}=(n_1,\dots,n_m)\in (\frac12\Z)^{m}$ with $\sum_j n_j\in\Z$ then we decorate the boundary components of $\Delta[\gamma]$ according to $\mathbf{n}$ as follows. Start at the positive puncture and follow the boundary of $\Delta[\gamma]$ in the positive direction. Decorate the first boundary component by $L_K$ and then when we pass the $j^{\rm th}$ negative puncture we change Lagrangian (from $L_K$ to $Q$ or vice versa) if $n_j$ is a half integer and do not change if it is an integer.

Fix a smooth family of smooth maps
\[
w_{\beta}\colon (\Delta[\beta_1],\partial\Delta[\beta_1])\to (T^{\ast}Q,L), \quad\beta=(\beta_1,\beta_2,\beta_3)\in B,
\]
with the following properties:
\begin{itemize}
\item $w_\beta$ respects the boundary decoration, i.e., it takes boundary components decorated by $L_K$ resp. $Q$ to the corresponding Lagrangian submanifold.
\item $w_{\beta}$ agrees with the Reeb chord strip of $a$ shifted by $\beta_2$ in a neighborhood of the positive puncture.
\item Consider standard coordinates $\C\times\C^{2}$ near the first component of $\beta_3^{j}\in J^{(r_{j})}(K)$. Then in a strip neighborhood of the $j^{\rm th}$ negative puncture, the $\C^{2}$-component of $w_{\beta}$ vanishes and the $\C$-component is given by
\[
w_{\beta}(z)= \sum_{l=0}^{r_{j}} q_l e^{l\pi z},
\]
where the $j^{\rm th}$ component $\beta_3^{j}$ of $\beta_{3}$ is
\[
\beta_3^{j}=(q_0,q_1,\dots,q_{r_j})\in J^{(r_{j})}(K).
\]
\end{itemize}

Let $0<\delta<\frac12$ and as before let either $0<\delta_0<\theta'$ or $\theta''<\delta_0<1$, where $\theta'$ describes the smallest non-zero complex angle at the Reeb chord $a$ and $\theta''$ the largest. Let $\sblv_{\delta_0,\delta}(\beta_1)$ denote the Sobolev space of maps
\[
w\colon\Delta[\beta_1]\to T^*\R^3\cong \R^6
\]
with two derivatives in $L^{2}$ and finite weighted 2-norm with respect to the weight function $\eta_{\delta}$ with the following properties.
\begin{itemize}
\item $\eta_{\delta_0,\delta}$ equals $1$ outside a neighborhood of the punctures.
\item $\eta_{\delta_0,\delta}(s+it)=e^{\delta_{0}\pi|s|}$ near the positive puncture.
\item $\eta_{\delta_0,\delta}(s+it)=e^{(n_j-\delta)\pi|s|}$ near the $j^{\rm th}$ negative puncture.
\end{itemize}

Consider the bundle $E\to B$ with fiber over $\beta\in B$ given by $\sblv_{\delta_0,\delta}(\beta_1)$.
Define the configuration space $\mathcal{W}=\mathcal{W}(a;\delta_0;\mathbf{n})\subset E$ of $(\beta,w)$ such that $u=w_\beta+w$ satisfies the following
\begin{itemize}
\item $u$ takes the boundary of $\Delta[\beta_1]$ to $L$ respecting the boundary decoration.
\item $u$ is holomorphic on the boundary, i.e. the restriction (trace) of $\bar{\pa}_{J}u$ to $\partial\Delta[\beta_1]$ vanishes.
\end{itemize}

It is not hard to see that $\mathcal{W}$ is a closed subspace of $E$. In fact it is a Banach submanifold of the Banach manifold $E$. We will next explain how to find local coordinates on $\mathcal{W}$. Let $(\beta,w)\in E$, and assume that $u=w_\beta+w$ is a map in $\mathcal{W}$. 

In order to find local coordinates around $u$ we first consider the finite dimensional directions.
Pick diffeomorphisms of the source $\Delta[\beta_1]$,  
\begin{equation}\label{eq:confvardiffeos}
\phi_{\gamma},\; \gamma\in\R;\qquad \psi_{\eta_{1}},\;\eta_{1}\in\R^{m-1},
\end{equation}
corresponding to the second and first finite dimensional factors. Here $\phi_{\gamma}$ equals the identity outside a neighborhood of the positive end where it equals translation by $\gamma$, and $\psi_{\eta_1}\colon \Delta[\beta_1]\to\Delta[\beta_1+\eta_1]$ moves the boundary maxima according to $\eta_1$, see \cite[Section 6.2.3]{E}.

We next turn to the translations along the knot and the infinite dimensional component of the space. Using the coordinate map of Lemma \ref{l:knotnbhd} we import the flat
metric on $T^{\ast}(S^{1}\times D^{2})$ to $T^{\ast}Q$, we extend this metric to a metric $h^1$ on all of $T^\ast Q$ so that $L_K$ is totally geodesic and flat near Reeb chord endpoints, see \ref{def:admissible} (v), and such that $h^1=ds^{2}+g$ on $T^*Q\setminus D^*Q\cong\R_+\times S^{\ast}Q$, where $g$ is a metric on $S^{\ast}Q$. Consider the standard almost complex structure in a neighborhood of the zero section of $Q=\R^{3}$ in $T^{\ast} Q$. Note that this almost complex structure agrees with the standard almost complex structure in the holomorphic neighborhood of $K$. Using the construction in \cite[Proposition 5.3]{EES1}, we extend it to an almost complex structure $J$ over all of $T^{\ast}Q$ with the following additional property near $L_{K}$. If $V$ is a vector field along a geodesic in the metric $h^{1}$ in $L_{K}$ then $V$ satisfies the Jacobi equation if and only if the vector field $JV$ does. To achieve this we might have to be alter $h^{1}$ slightly near but not on $L_K$, see \cite[Equation (5.7)]{EES1} for the precise form of $h^{1}$ (corresponding to $\hat g$ in that equation). Note that this construction gives the standard almost complex structure near the knot. Let $h^0$ denote the standard flat metric on $T^{\ast}Q$ and note that it has the Jacobi field property discussed above along $Q$. Let 
\begin{equation}\label{eq:interpolmetric}
h^\sigma, \quad 0\le \sigma\le 1
\end{equation}
be the linear interpolation between the metrics $h^0$ and $h^1$. 

Consider the pullback bundle $u^{\ast}T(T^{\ast}Q)$. Note that the Riemannian metrics $h^t$ on $T^{\ast}Q$ induce connections on this bundle which we denote by $\nabla^t$.

Let $\dot\sblv_{\delta}(u)$ denote the linear space of sections $v$ of
$u^{\ast}T(T^{\ast}Q)$ with the following properties
\begin{itemize}
\item The partial derivatives of $v$ up to second order lie in $L^{2}_{\rm loc}
  (\Delta[\beta],u^{\ast}T(T^{\ast}Q))$.
\item The restriction of $\nabla^{\sigma} v + J\circ\nabla^{\sigma} v\circ i$ to the
  boundary (sometimes called the trace of $\nabla^{\sigma} v + J\circ\nabla^{\sigma} v\circ i$) vanishes, where $\sigma=1$ for a boundary component mapping to $L_K$ and $\sigma=0$ for a component mapping to $Q$.
\item With $\|\cdot\|_{\delta,\delta_0,\mathbf{n}}$ denoting the Sobolev $2$-norm
  weighted by $\eta_{\delta,\delta_0,\mathbf{n}}$, $\|v\|_{\delta,\mathbf{n}}<\infty$.
\end{itemize}
Then $\dot\sblv_{2,{\delta,\delta_0,\mathbf{n}}}(w)$ equipped with the norm
$\|\cdot\|_{\delta,\delta_0,\mathbf{n}}$ is a Banach space.

Also fix $m+\sum_{j=1}^{m} r_j$ smooth vector fields $s^{j}_{k}$, $1\le j\le m$ and $0\le k\le r_j$ along $u$ with properties as above and with the following additional properties
\begin{itemize}
\item The vector field $s^{j}_k$ is supported only near the $j^{\rm th}$ negative puncture in a half strip neighborhood which maps into the analytic neighborhood of the knot.
\item In standard coordinates along the knot $\C\times\C^{2}$, the $\C^{2}$-component of $s^j_{k}$ equals $0$ and the $\C$-component is $s^{j}_{k}= e^{k\pi z}$.
\end{itemize}

We are now ready to define the local coordinate system. Write $\exp^{\sigma}$ for the exponential map in the Riemannian metric $h^{\sigma}$, $0\le \sigma\le 1$, from \eqref{eq:interpolmetric}. The local coordinate system around $u$ has the form
\[
\Phi_{u}\colon U_1\times U_{2}\times U_{3} \times \mathcal{U}\to \mathcal{W},
\]
where $U_1\subset\R^{m-2}$, $U_{2}\subset\R$, $U_3\subset\Pi_{j=1}^{m}\R^{r_j+1}$, and $\mathcal{U}\subset \dot\sblv_{2,{\delta,\delta_0,\mathbf{n}}}(w)$ are small neighborhoods of the origin with coordinates $\gamma_j\in U_j$. Let $\sigma\colon \Delta[\beta_1]\to [0,1]$ be a smooth function that equals $0$ resp.~$1$ in a neighborhood of any boundary component that maps to $Q$ resp.~$L_K$ and that equals $0$ on $u^{-1}(D^{\ast} Q)$. For $u$ as above we then consider
\[
\Psi_{u}(\gamma_1,\gamma_2,\gamma_3,v)(z)=
\exp^{\sigma(z')}_{u(z')}\left(v(z')+\sum_{j=1}^{m}\sum_{k=0}^{r_j}{\gamma_3}^{j}_{k}s^{j}_{k}(z')\right),\quad
z'=\phi_{\gamma_1}(\psi_{\gamma_{2}}(z)),
\]
see \eqref{eq:confvardiffeos} for the diffeomorphisms $\phi_{\gamma_{1}}$ and $\psi_{\gamma_{2}}$. Here $\gamma_{1}$ corresponds to shifts near the positive puncture, $\gamma_{2}$ corresponds to variations of the conformal structure, $\gamma_{3}$ is related to variations of the map near Lagrangian intersection punctures, and $v$ is a vector field along the curve. We use the exponential map to go from linearized variations to actual maps.
\begin{lemma}
The space $\mathcal{W}$ is a Banach manifold with local coordinates around $u$ given by $\Psi_{u}$.
\end{lemma}
\begin{proof}
This is straightforward, see \cite[Lemma 3.2]{EES2} for an analogous result.
\end{proof}

Consider the bundle $\mathcal{E}$ over the configuration space $\mathcal{W}$ with fiber over $u$ the complex anti-linear maps
\[
T\Delta[\beta_1]\to T(T^{\ast}Q).
\]
The $\bar{\pa}_{J}$-operator gives a section of this bundle $u\mapsto (du+J\circ du\circ i)$ and the moduli space $\MM(a;n_0,\mathbf{n})$ is the zero locus of this section, where $n_0=0$ if $0<\delta_0<\theta'$ and $n_0=1$ if $\theta''<\delta_0<1$. The section is Fredholm and the formal dimension of the solution spaces is given by its index. We have the following dimension formula.

\begin{lemma}\label{l:weightdimcob}
The formal dimension of $\mathcal{M}(a,n_0;\mathbf{n})$ is given by
\[
\dim(\mathcal{M}(a,n_0;\mathbf{n}))=|a|-2n_0-|\mathbf{n}|.
\]
\end{lemma}
\begin{proof}
The case $n_0=0$ follows from \cite[Theorem A.1 and Remark A.2]{CEL}. The fact that the index jumps when the exponential weight crosses the eigenvalues of the asymptotic operator is well known and immediately gives the other case, see e.g.~\cite[Proposition 6.5]{EES1}.
\end{proof}

We next consider a completely analogous construction of a configuration space for holomorphic disks in $\MM^{\rm sy}(a,n_0;\mathbf{b})$. We discuss mainly the points where this construction differs from that above. Consider first the finite dimensional base. Here the situation is simpler and we take instead
\[
B=\R^{m-2}\times\R^{m+1},
\]
where the first factor corresponds to conformal structures on the domain exactly as before and where the second factor corresponds to re-parameterizations of the trivial Reeb chord strips exactly as for the positive puncture before. We fix a smooth family of maps $w_{\beta}\colon\Delta[\beta_1]\to\R\times S^{\ast}Q$ which agrees with the prescribed Reeb chord strips near the punctures. We next fix an isometric embedding of $S^{\ast}Q$ into $\R^{N}$ and consider the bundle of weighted Sobolev spaces with fiber over $\beta\in B$ the Sobolev space $\sblv_{n_0,\delta}$ of functions with two derivatives in $L^{2}$ with respect to the norm weighed by a function which equals $e^{\delta|s|}$ in the negative ends and $e^{(\delta+n_0)|s|}$ in the positive end.

In analogy with the above we then fix (commuting) re-parameterization diffeomorphisms $\psi_{\beta_1}$ corresponding to changes of the conformal structure and $\phi_{\beta_2}$ corresponding to translation in the half strip neighborhoods. Again this then leads to a Fredholm section and its index gives the formal dimension of the moduli space.

\begin{lemma}\label{l:weightdimsymp}
The formal dimension of $\mathcal{M}^{\rm sy}(a,n_0;\mathbf{b})$ is given by
\[
\dim\mathcal{M}^{\rm sy}(a,n_0;\mathbf{b})=|a|-2n_0-|\mathbf{b}|.
\]
\end{lemma}
\begin{proof}
See \cite[Theorem A.1 and Remark A.2]{CEL} and use the relation between weights and index, see e.g.~\cite[Proposition 6.5]{EES1}.
\end{proof}

\begin{remark}\label{rmk:confvaratpuncture}
We consider for future reference the conformal variations of the domain with more details. In the local coordinates around a map $w\colon \Delta_{m+1}\to T^{\ast} Q$ or $w\colon \Delta_{m+1}\to \R\times S^{\ast} Q$  defined above, the conformal variations correspond to a diffeomorphism that moves the boundary maxima of the domain. We take such a diffeomorphism to be a shift along a constant (and hence holomorphic) vector field $\tau$ in the real direction around the boundary maximum and then cut it off in nearby strip regions. Hence the corresponding linearized variation $L\bar{\pa_{J}}(\gamma)$ at $w$, where $\gamma$ is the first order variation of the complex structure corresponding in the domain is
\[
L\bar{\pa}_{J}(\gamma)=\pa_{J}w\circ \bar{\pa} \tau.
\]

We will sometimes use other ways of expressing conformal variations, where the variations are supported near a specific negative puncture rather than near a specific boundary maximum. To this end we first note that we may shift the conformal variation by any element $L\bar{\pa}_{J}(v)$ where $v$ is a vector field along $w$ in the Sobolev space $\sblv_{\delta}$. In particular we can shift $\gamma$ by $\bar{\pa}\sigma$ where $\sigma$ is a vector field along $\Delta_{m+1}$ that is constant near the punctures. In this way we get equivalent conformal variations $\gamma_{q}$ of the form
\[
L\bar{\pa}_{J}(\gamma_{q})=\pa_{J}w\circ \bar{\pa}\tau_{q}.
\]
where $\tau_{q}$ is a vector field of the form
\[
\tau_{q}(z) = \beta(s+it) e^{\pi(s+it)},
\]
where $s+it$ is a standard coordinate in the strip neighborhood of the negative puncture $a$ and $\beta$ is a cut-off function equal to $1$ near the puncture and $0$ outside a strip neighborhood of the puncture. We refer to \cite[Section 2.1.1]{E} for details.
\end{remark}

\subsection{Transversality}\label{ss:trans}
We next use the special form of our almost complex structure near Reeb chords in combination with an argument from \cite[Lemma 4.5]{EES2} to show that we can achieve transversality for $\bar{\pa}_{J}$-section of $\mathcal{E}$ over $\mathcal{W}$ by perturbing the almost complex structure. In other words we need to show that the linearization $L\bar\pa_{J}$ of the section $\bar\pa_{J}$ is surjective.

\begin{lemma}\label{l:tv}
For generic $J$ any solutions in $\mathcal{M}(a,n_0;\mathbf{n})$ and $\mathcal{M}^{\rm sy}(a, n_0;\mathbf{b})$ are transversely cut out.
\end{lemma}

\begin{proof}
To see this we perturb the almost complex structure near the positive puncture. Consider the local projection to $\C^{2}$ near the Reeb chord. Here the Lagrangians correspond to two Lagrangian planes. Furthermore the holomorphic disks admit local Taylor expansions near the points that map to their intersection. The lemma now follows from the proof of \cite[Lemma 4.5]{EES2}. We sketch the argument.

Let $U$ denote a neighborhood of the Reeb chord strip $C_{a}$ of $a$ for $\mathcal{M}^{\rm sy}$ or of the Reeb chord strip in $T^{\ast}Q- D^{\ast}Q$ for $\mathcal{M}$. If $u$ is a holomorphic disk then $u^{-1}(U\cap C_{a})$ is the pre-image under $u$ composed with the projection to $\C^{2}$ of the intersection point of the two Lagrangian planes. It follows by monotonicity that the preimage is a finite collection of points $\{q_{0},q_{1},\dots, q_{r}\}$, where $q_{0}$ is the positive puncture. If $q_{i}$ is an interior point, let $E_{i}$ denote a small disk around $q_{i}$, if $q_{j}$ is a boundary point let $E_{j}$ denote a half-disk neighborhood of $q_{i}$. If the map $u$ has an injective point near the double point then a standard argument perturbing the almost complex structure there establishes the necessary transversality. We therefore assume that this is not the case. Consider the image of a small half disk $E_{0}$ near the positive puncture $q_{0}$, and note that the boundary arcs end at the positive punctures. Since the map is not injective there are neighborhoods (after renumbering) $E_{1},\dots,E_{m}$ where $u$ agrees with the image $\gamma$ under $u$ of one of the boundary arcs of $E_{0}$. By analytic continuation, the images of these neighborhoods then intersect the Lagrangian sheet of the boundary arc $\gamma'$ that contains $\gamma$. Consequently, the map has multiplicity $m+1$ along $\gamma$ and multiplicity $m$ along $\gamma'-\gamma$. Consider a vector field in the cokernel of the linearized operator $L\bar{\pa}$. Perturbing the almost complex structure near $\gamma'-\gamma$ we see that the contributions from the anti-holomorphic cokernel vector field on $E_{1},\dots,E_{m}$ must vanish. By unique continuation, the contributions from $E_{1},\dots, E_{m}$ must then also cancel along $\gamma$ and it follows that there is nothing that cancels the perturbation in $E_{0}$ (just as if the map was injective in $E_{0}$). The desired transversality follows.
\end{proof}

\subsection{Stabilization of domains}\label{s:stab}
For disks with more than three punctures the transversality results in
Section \ref{ss:trans}
directly give the solution spaces the structure of $C^{1}$-smooth manifolds. For the case of unstable domains this is not as direct since the solutions admit re-parameterizations that do not act with any uniformity on the associated configuration spaces. This is a well-known phenomenon and we resolve the problem by a gauge fixing procedure, adding marked points near the positive puncture. This construction was studied in detail in \cite[Appendix A.2]{ESrev} and in \cite[Sections 5.2 and 6]{ES} and we will refer to these articles for details.

As we shall see below we need only consider moduli spaces of dimension $\le 2$. Recall the neighborhood $a\times U$, $U\subset\C^{2}$ of the Reeb chord $a$, see the discussion before Definition~\ref{def:admissible} on page~\pageref{def:admissible}, and the corresponding Fourier expansion of the $\C^{2}$-component of any holomorphic disk near $a$, see \eqref{eq:chordasympt} on page~\pageref{eq:chordasympt}.

Consider a space $\mathcal{M}(a,\mathbf{n})$ of formal dimension $\le 1$. Then by Lemmas \ref{l:weightdimcob} and \ref{l:tv} the corresponding space $\mathcal{M}(a,1;\mathbf{n})$ is empty.
Consequently, for any solution $u\in \MM(a,\mathbf{n})$, the first Fourier coefficient of the $\C^{2}$-component of the map near $a$ is non-vanishing. Let $S_{0;\epsilon}$ and $S_{1;\epsilon}$ be spheres in $\Lambda_{K}$ of radii $\epsilon>0$ around the Reeb chord endpoints of $a$. Non-vanishing of the first Fourier coefficient in combination with compactness then implies that for each solution $u$ there are two unique points in the boundary of the domain closest to the positive puncture that map to $S_{j;\epsilon}$, $j=0,1$, see Figure \ref{fig:marked}. We add punctures at these points. More precisely, we consider standard domains with two more punctures and require that the maps are asymptotic to points in $S_{j;\epsilon}$ at the extra punctures. In the above notation these would be ``Lagrangian intersection punctures'' in $S_{j;\epsilon}$ of local winding number $1$ in the direction normal to $S_{j;\epsilon}$. The transversality result \ref{l:tv} holds as before also for the solution spaces with extra punctures, so that they are $C^{1}$-manifolds. The asymptotic properties above then imply that the solutions with extra punctures capture all holomorphic disks.

Consider next a space $\mathcal{M}^{\rm sy}(a;\mathbf{b})$ of formal dimension $\le 2$. Since any holomorphic curve in the symplectization can be translated we find that the corresponding space $\mathcal{M}^{\rm sy}(a,1;\mathbf{b})$ is again empty and we get a manifold structure by adding two marked points near the Reeb chord endpoints exactly as above.

It remains then to consider the case of spaces $\MM(a;\mathbf{n})$ of formal dimension $2$. Here the corresponding space $\MM(a,1;\mathbf{n})$ has dimension $0$. There are then a finite number of solutions with this decay condition. Considering the Fourier expansion we can fix unique marked points for all solutions in a neighborhood $\mathcal{V}$ (in the configuration space) of these isolated solutions as above. For solutions outside $\mathcal{V}$ the Fourier coefficients do not vanish and we can fix marked points as above. Note however, that these will generally not be the same marked points. This way we however get two types of manifold charts: one for solutions inside $\mathcal{V}$ and one for solutions in a neighborhood of any map $u'$ with nonvanishing first Fourier coefficient which lies outside a smaller neighborhood $\mathcal{V}'$ of $u$. To get a manifold structure for the moduli space we then need to study the transition maps, and to that end we use four marked points, see Figure \ref{fig:marked} and \cite[Section 5.2]{ES} for details.

\begin{figure}[h]
\centering
\includegraphics[width=0.5\textwidth]{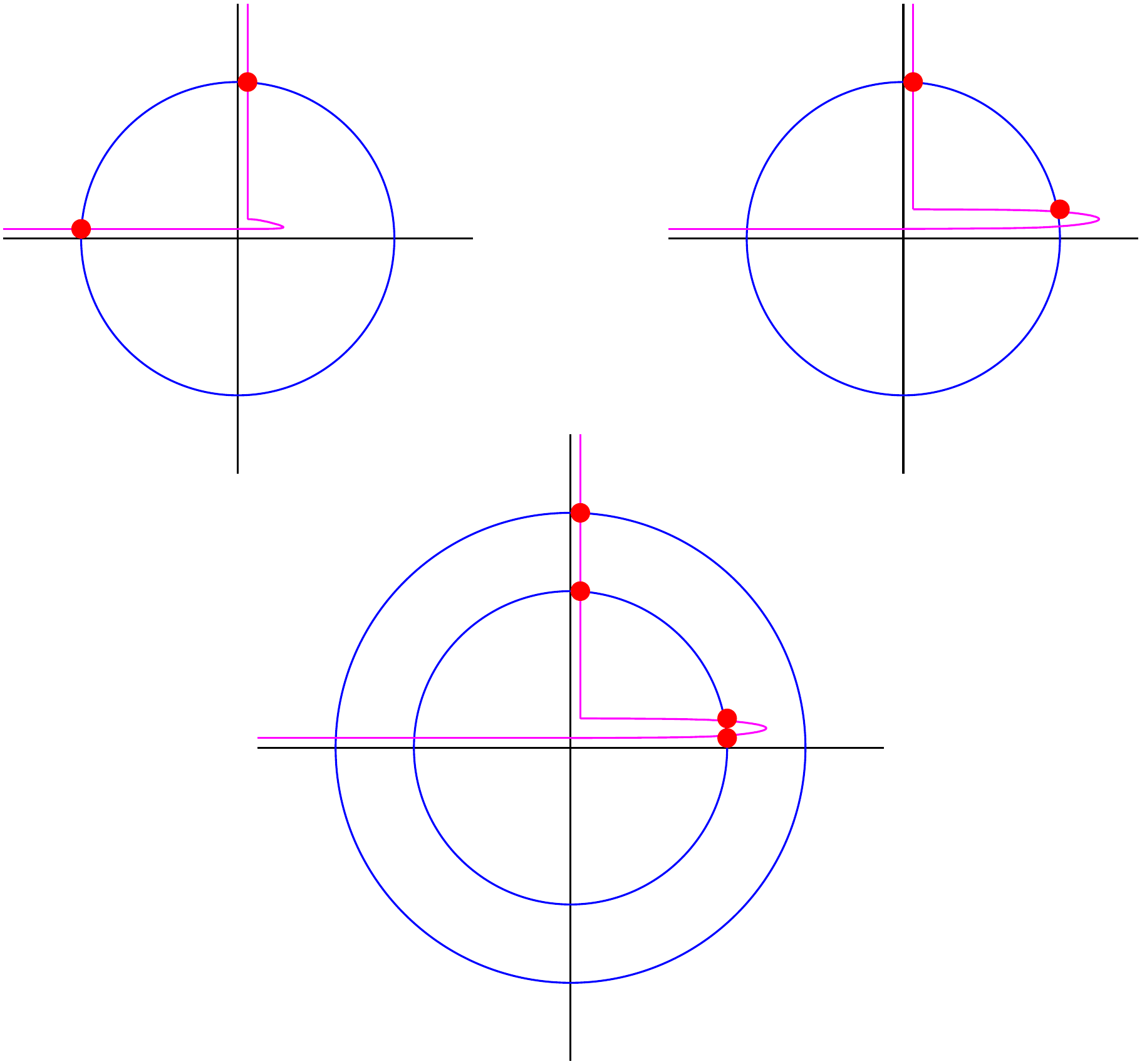}
\caption{Top left: marked points for a disk near the disk with degenerate asymptotics. Top right: marked points for a disk outside a neighborhood of the disk with degenerate asymptotics. Lower: the four marked points in the intermediate region used to define the coordinate change. The (black) lines represent the projections of the branches of $L_K \cup Q$ to $\C^2$, and the (blue) circles represent 3-spheres $S^3_\eps$, which cut these local branches along the circles $S_{j,\eps}$ appearing in the text.}
\label{fig:marked}
\end{figure}

A priori, the smooth structures on the moduli spaces above depend on the choice of gauge condition. However, using the fact that the $C^{0}$-norm of a holomorphic map controls all other norms, it is not hard to see that different gauge conditions lead to the same smooth structure.

We also need to show that the compactness result where sequences of curves converge to several level curves are compatible with additional marked points. This is similar to the above. The compactness result we already have implies uniform convergence on compact sets and in particular it is possible to add marked points on the curves near the limit that correspond to the extra marked points on the unstable curves in the limit. As before we show that these extra marked points do not affect the moduli spaces. See \cite[Section A.3]{ES} for details. In conclusion, by adding marked points also on curves near broken limits we obtain versions of the compactness results Theorems \ref{t:cp_sy} and \ref{t:cp} where all domains involved are stable with marked points compatible with the several level breaking.

\subsection{Index bundles and orientations}\label{ss:indexbndles}
Viewing the $\bar{\pa}_{J}$-operator as a Fredholm section of a Banach
bundle, its linearization defines an index bundle over the
configuration space and an orientation of this index bundle gives an
orientation of transverse solution spaces. Following
Fukaya, Oh, Ohta, Ono~\cite[Section 8.1]{FO3I} one defines a coherent system of such orientations as
follows. Fix spin structures of the two Lagrangians $L_K$ and $Q$,
which we here can think of as trivializations of the respective tangent
bundles. Consider a closed disk with boundary in one of the two
Lagrangians and the linearized $\bar{\pa}_{J}$-operator acting on vector fields
along this disk that are tangent to the Lagrangian along the
boundary. Using the trivialization of the boundary condition, such an
operator can be deformed to an operator on the disk with values in
$\C^3$ and constant
$\R^{3}$ boundary condition, with a copy of $\C P^{1}$ attached at the
center with a complex linear operator. The first operator has trivial
cokernel and a kernel that consists only of constant vector fields, and the orientation of
$\C$ induces an orientation on the determinant bundle of the operator
over $\C P^{1}$. This gives a canonical orientation over closed disks
with trivialized boundary condition (that depends only on the homotopy
class of the trivialization).

Here we need to orient moduli spaces of disks with punctures. This was
done in the setting of Legendrian contact homology in
\cite{EESori}; we will give a sketch and refer to that reference for 
details. We reduce to the case of closed disks by picking so-called
capping operators at all Reeb chords and along the Lagrangian
intersection $K$ with an orientation of the corresponding determinant bundles. Here it is
important that the capping operators are chosen in a consistent
way. At Reeb chords there is a positive and a negative capping
operator and we require that they glue to the standard orientation on
the closed disk. We also pick positive and negative capping operators
at the Lagrangian intersection punctures satisfying the same
conditions. Now, given a holomorphic disk in the
symplectization or in $T^{\ast}Q$ we glue the capping operators
to it and produce a closed disk. The standard orientation of the
closed disk and the chosen orientations on the capping operators then
give an orientation of the determinant line of the linearized operator
over the disk, which, together with an orientation of the finite dimensional space of conformal structures on the punctured disk, in turn gives an orientation of the moduli
space if it is transversely cut out. The gluing condition for the capping operators ensures that the
resulting orientations of the moduli spaces are compatible with
splittings into multi-level curves.

In what follows we assume that spin structures on the Lagrangians and capping operators have been fixed and thus all our moduli spaces are oriented manifolds.

\subsection{Signs and the chain map equation}\label{ss:signsandchainmap}
Recall the chain map  
\[ 
\Phi\colon (C_*(\RR),\p_\Lambda)\to 
(C_*(\Sigma),\partial+\delta_{Q}+\delta_{N})
\]
from Theorem \ref{t:Phichainmap}. Here we consider the signs of the operations $\delta_Q$
and $\delta_N$ in this formula. These operations are defined on chains
of broken strings by taking the oriented preimage of $K$ under the evaluation
map. In the map $\Phi$, the oriented chain is given by a moduli space
of holomorphic disks. In order to deal with the evaluation maps on
such spaces we present them as bundles over $Q$ as follows. Consider
first the operation $\delta_Q$. Fix a point $q\in Q$ and an additional
puncture on the boundary that we require maps to $q$. Concretely, we
work on strips with slits and add a small positive exponential weight
at the puncture mapping to $q$. Then we consider the bundle of such
maps over $Q$ when we let $q$ vary in $Q$. The orientation of this
space is induced from capping operators as described above. When we
consider the corresponding boundary condition on the closed disk we
find a vanishing condition for linearized variations at the marked
point corresponding to the positive exponential weight. Thus if
$\sigma$ denotes the orientation of the index bundle induced as above,
then the orientation on the bundle with marked point mapping to $q$ is
given by the orientation of the formal difference
$\sigma \ominus TQ$. (The formal difference should be interpreted as in $K$-theory: the difference $\xi\ominus \eta$ of two bundles $\xi$ and $\eta$ is represented by a bundle $\zeta$ such that the direct sum $\zeta\oplus\eta$ is equivalent to $\xi$.) 

We point out that here and throughout this section orientations depend on ordering conventions, whether the point condition goes before or after the index bundle, etc. In calculations below we put point conditions after the index bundle, and put the fiber of bundles over $Q$ before the base.

The orientation of the bundle corresponding to a point constraint $q$ varying over $Q$ is then given given by $\sigma\ominus TQ\oplus TQ$. Finally, the orientation of the chain given by the
preimage of $K$ under the evaluation map is then
\begin{equation}\label{eq:cutori0}
\sigma\ominus TQ\oplus TQ\oplus TK\ominus TQ=\sigma\oplus TK\ominus TQ.
\end{equation}

In order to show that the chain map equation holds we must then show
that there are choices of capping operators and orientations on $Q$
and $N$ so that this orientation agrees with the boundary orientation
of the disk viewed as the boundary in the moduli space of disks with
two colliding Lagrangian punctures.

Consider the capping operators $c_{QN}$ and $c_{NQ}$ for such a puncture going from $Q$ to $N$ and vice versa. These capping operators are standard $\bar{\pa}$-operators on a once punctured disk $D_{1}$ acting on $\C^{3}$-valued functions in a weighted Soboloev space that satisfy a Lagrangian boundary condition. 

We first describe the boundary conditions. For $c_{QN}$ the Lagrangian boundary condition $\lambda\colon \partial D_{1}\to \mathrm{Lag}_{3}$, where $\mathrm{Lag}_{3}$ denotes the Lagrangian Grassmannian of Lagrangian subspaces of $\C^{3}$, starts at the tangent space of $Q$ and ends at the tangent space of $N$. For $c_{NQ}$ the boundary condition instead starts at the tangent space of $N$ and ends at the of $Q$. More specifically, the tangent spaces of $Q$ and $N$ intersect along $TK$ and are perpendicular in the normal directions of $K$. We think of the normal directions to $K$ as $\C^{2}$ and the tangent spaces of $Q$ and $N$ as $i\R^{2}$ and $\R^{2}$, respectively. We take both capping operators $c_{QN}$ and $c_{NQ}$ to fix $TK$, to be a rotation by $\frac{\pi}{2}$ in one of the complex lines normal to the knot, and a rotation by $\frac{3\pi}{2}$ in the other. 

We next describe the weights at the puncture in $D_{1}$. We use a half strip neighborhood of the puncture and a Sobolev space with small positive exponential weight $\delta$, $0<\delta<\frac{\pi}{2}$, in this strip neighborhood. 

The index of the $\bar{\partial}$-operator with this boundary condition and weight equals $3$, see e.g.~\cite[Proposition 6.5]{EES1}.

Recall from Section \ref{ss:indexbndles} that an orientation of the moduli space is induced from the capping operators together with an orientation on the space of conformal structures on the punctured disk. Here we think of variations of the conformal structure as vector fields moving the punctures along the boundary of the disk. We have one such vector field for each puncture which give an additional one dimensional oriented vector space associated to each puncture, see \cite[Section 3.4.1]{EESori} for details. For simplicity we write simply $c_{QN}$ and $c_{NQ}$ for the sum of the index bundles of the capping operator described above and one dimensional conformal variations associated to the respective punctures. Thus, in the calculations below $c_{QN}$ and $c_{NQ}$ have index $3+1=4$.

We choose the orientations on $Q$ and $N$ so that the linear transformations between tangent spaces $TQ$ and $TN$ induced by the Lagrangian boundary conditions of $c_{QN}$ and $c_{NQ}$ take the orientation on $Q$ to that on $N$ and vice versa.

The boundary orientation of the two-level disk (second level constant)
is the fiber product over $K$ of the orientations of its levels. We
view the top level disk as having a small positive exponential weight
at the puncture mapping to $K$ and a cut-off local solution in the
direction of $K$. In analogy with the above, its orientation is thus
given by $\sigma\ominus TQ\oplus TK$.
The orientation of the constant disk (which
has small negative weights at its positive puncture) is then
$\sigma'\oplus c_{QN}\oplus c_{NQ}$, where $\sigma'$ is the standard orientation
on the closed up boundary condition of the constant three punctured
disk. The boundary orientation is thus
\begin{equation}\label{eq:cutori1}
(\sigma\ominus TQ\oplus TK)\oplus (\sigma'\oplus c_{QN}\oplus c_{NQ})\ominus TK.
\end{equation}
Now choose the orientation on $c_{QN}$ and $c_{NQ}$ so that the
orientation of the index one problem on the constant disk with kernel
in direction of the knot induced by $\sigma'\oplus c_{QN}\oplus c_{NQ}$ is
opposite to the orientation of $TK$. Then the orientation in
\eqref{eq:cutori1} is $\sigma\oplus TK\ominus TQ$ (there is an orientation change
when one permutes the odd-dimensional summands $TK$ and $TQ$), in
agreement with \eqref{eq:cutori0}.

For the sign of the operation $\delta_N$ we argue exactly as above
replacing $Q$ with $N$ and we must compare the orientations
$\sigma\oplus TK\ominus TN$ and
\[
(\sigma\ominus TN\oplus TK)\oplus (\sigma'\oplus c_{NQ}\oplus c_{QN})\ominus TK.
\]
Compared to the above the main difference is that the summands
$c_{NQ}$ and $c_{QN}$ have been permuted. However, as explained above, the index of each
of these operators is $4$, so the orientation remains as
before and the positive sign for $\delta_N$ is correct for the chain
map.

\section{Compactification of moduli spaces and gluing}\label{sec:gluing}

In this section we show that the moduli spaces $\MM(a;\mathbf{n})$ and $\MM^{\rm sy}(a;\mathbf{b})/\R$
admit compactifications as manifolds with boundary with corners. Furthermore, we
describe the boundary explicitly in terms of broken holomorphic
disks. The smoothness of individual strata of the compactified moduli
spaces are governed by the Transversality Lemma \ref{l:tv}.
The Compactness Theorems \ref{t:cp} and \ref{t:cp_sy} describe
disk configurations in the boundary of the compactification. The main
purpose of this section is thus to show how to glue these
configurations on the boundary to curves in the smooth part of the
moduli space and thereby obtain boundary charts in the sense of
manifolds with boundary with corners. Such gluing theorems were proved
before in closely related situations and we will discuss details only
when they differ from the standard cases.

We first state the structural theorems in Section \ref{s:mfdbdrycrn}
and then turn to the gluing results and their proofs in the following subsections.

We work throughout this section with an almost complex structure $J$ so that Lemma \ref{l:tv} holds. Furthermore we assume that the domains of all holomorphic disks are stable, which can be achieved by adding marked points as explained in Section \ref{s:stab}.

\subsection{Structure of the moduli spaces}\label{s:mfdbdrycrn}
In this subsection we state the results on moduli spaces of holomorphic
disks. As before there are two cases to consider, disks in the
symplectization and disks in the cotangent bundle. The structural
results all have the same flavor. Basically we show that a specified
moduli space is a manifold with boundary with corners of dimension
$\le 2$, and we describe the boundary strata as well as certain
submanifolds important for our study. The proofs of the results are
the main goal for the rest of the section.

Recall from Sections~\ref{s:confcob} and~\ref{ss:trans} (with
$n_0=0$) that for generic $J$ the moduli spaces
$\MM(a;\mathbf{n})$ and $\MM^{\rm sy}(a;\mathbf{b})$ are manifolds of
dimensions
$$
   \dim\mathcal{M}(a;\mathbf{n})=|a|-|\mathbf{n}|,\qquad
   \dim\mathcal{M}^{\rm sy}(a;\mathbf{b})=|a|-|\mathbf{b}|.
$$
Here $|a|=\ind(a)$ is the degree of the Reeb chord $a$ (which takes
only values $0,1,2$), and to the vector of local winding numbers
$\mathbf{n}=(n_1,\dots,n_{m})$ (where the $n_j$ are positive
half-integers or integers) we have associated the nonnegative integer
\[
   |\mathbf{n}|=\sum_{j=1}^{m}2(n_j-\tfrac12)\ge 0.
\]
If either $\mathbf{n}$ or $\mathbf{b}$ is empty, the corresponding contribution to the index formula is $0$.
If $a$ is a Reeb chord of $\Lambda_{K}\subset S^{\ast}Q$, then $0\le |a|\le 2$. Since $J_1$ is $\R$-invariant, $0$-dimensional moduli spaces in the symplectization consists only of Reeb chord strips. Thus the only non-empty moduli spaces $\MM^{\rm sy}(a;\mathbf{b})$ of dimension $d^{\rm sy}$ are the following (write $\mathbf{b}=b_1\dots b_m$), see Figure \ref{fig:sympdisk}:
\begin{itemize}
\item
{$[2,0]^{\rm sy}$}:
If $|a|=2$ and $|\mathbf{b}|=0$ (i.e.~$|b_j|=0$ for all $j$) then $d^{\rm sy}=2$.
\item
{$[2,1]^{\rm sy}$}:
If $|a|=2$ and $|\mathbf{b}|=1$ (i.e.~$|b_j|=0$ for all $j\ne s$ and $|b_s|=1$) then $d^{\rm sy}=1$.
\item
{$[1,0]^{\rm sy}$}:
If $|a|=1$ and $|\mathbf{b}|=0$ then $d^{\rm sy}=1$.
\end{itemize}

\begin{figure}
\labellist
\small\hair 2pt
\pinlabel $D$ at 64 91
\pinlabel $u$ at 184 107
\pinlabel ${\color{red} a}$ at 298 177
\pinlabel ${\color{red} b_1}$ at 256 4
\pinlabel ${\color{red} b_2}$ at 300 4
\pinlabel ${\color{red} b_3}$ at 342 4
\pinlabel $+$ at 64 162
\pinlabel $-$ at 0 61
\pinlabel $-$ at 63 22
\pinlabel $-$ at 127 61
\endlabellist	
\centering
	\includegraphics[width=.7\linewidth]{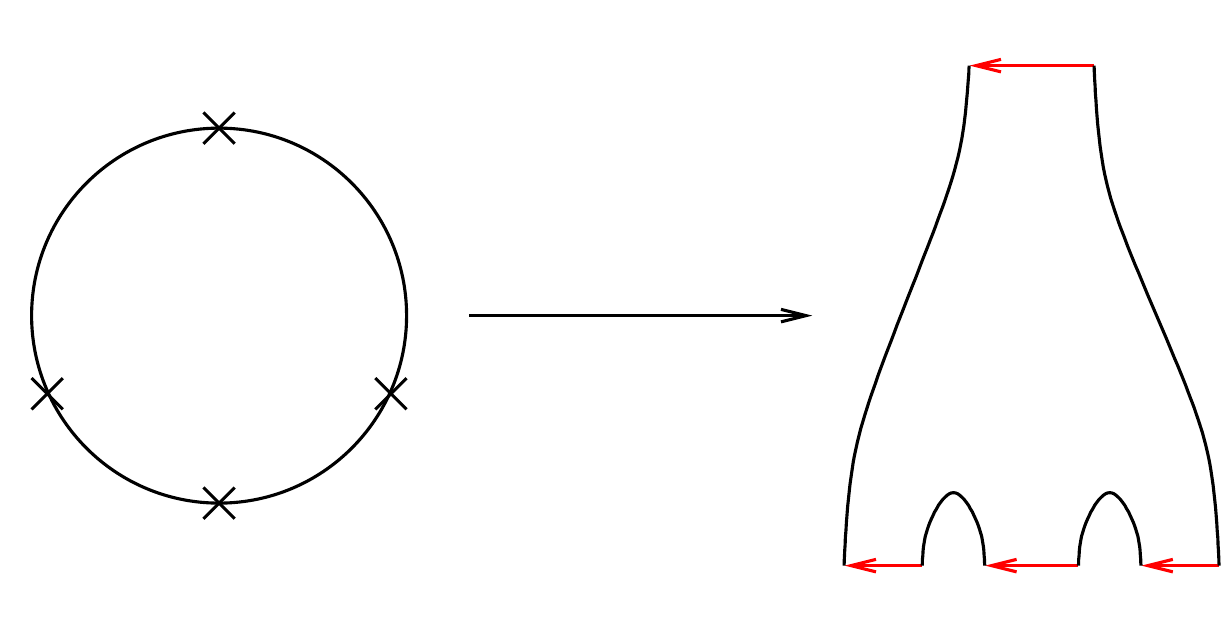}
	\caption{Disks $u :\thinspace (D,\partial D) \to (\R\times S^*Q,\R\times\Lambda_K)$ in the symplectization.}
	\label{fig:sympdisk}
\end{figure}

Similarly, the only non-empty moduli spaces $\MM(a;\mathbf{n})$ of dimension $d$ are the following
(write $\mathbf{n}=n_1\cdots n_m$), see Figures \ref{fig:pihalf}, \ref{fig:pi}, \ref{fig:pipihalf}, \ref{fig:pipi}:
\begin{itemize}
\item
{$[2,0]$}: If $|a|=2$ and all $n_j=\frac12$, then $|\mathbf{n}|=0$ and $d=2$.
\item
{$[2,1]$}: If $|a|=2$ and $n_j=\frac12$ for all $j\ne s$ and $n_s=1$, then $|\mathbf{n}|=1$ and $d=1$.
\item
{$[2,\frac32]$}: If $|a|=2$ and $n_j=\frac12$ for all $j\ne s$ and $n_s=\frac32$, then $|\mathbf{n}|=2$ and $d=0$.
\item
{$[2,2]$}: If $|a|=2$ and $n_j=\frac12$ for all $j\ne s,t$, and $n_s=n_t=1$, then $|\mathbf{n}|=2$ and $d=0$.
\item
{$[1,0]$}: If $|a|=1$ and all $n_j=\frac12$, then $|\mathbf{n}|=0$ and $d=1$.
\item
{$[1,1]$}: If $|a|=1$ and $n_j=\frac12$ for all $j\ne s$ and $n_s=1$, then $|\mathbf{n}|=1$ and $d=0$.
\item
{$[0,0]$}: If $|a|=0$ and $n_j=\frac12$ all $j$, then $|\mathbf{n}|=0$ and $d=0$.
\end{itemize}

\begin{figure}
\labellist
\small\hair 2pt
\pinlabel ${\color{red} a}$ at 64 134
\pinlabel ${\color{red} Q}$ at 174 17
\pinlabel ${\color{blue} L_K}$ at 28 65 
\pinlabel ${\color{blue} L_K}$ at 100 65
\pinlabel $\textstyle{\frac{1}{2}}$ at 42 10
\pinlabel $\textstyle{\frac{1}{2}}$ at 96 10
\endlabellist	
	\centering
	\includegraphics[width=.6\linewidth]{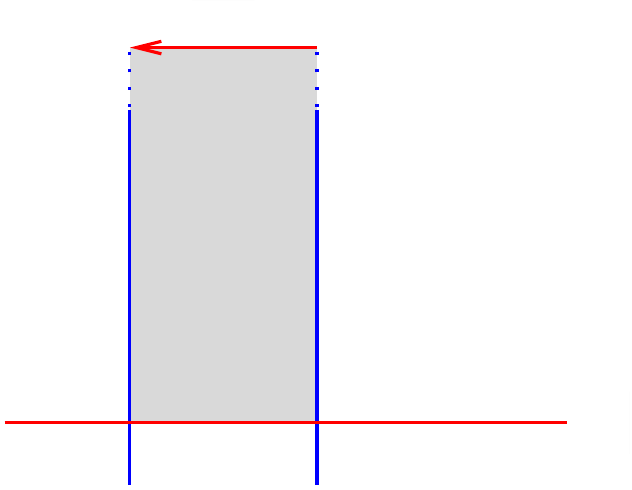}
	\begin{align*}
	[2,0]: \hspace{4ex} & |a|=2, \hspace{2ex} \dim=2 \\
	[1,0]:  \hspace{4ex} & |a|=1, \hspace{2ex}  \dim=1 \\
	[0,0]:  \hspace{4ex} & |a|=0, \hspace{2ex}  \dim=0
	\end{align*}
	\caption{Curves with $|\mathbf{n}|=0$.}
	\label{fig:pihalf}
\end{figure}

\begin{figure}
\labellist
\small\hair 2pt
\pinlabel ${\color{red} a}$ at 75 134
\pinlabel ${\color{red} Q}$ at 174 17
\pinlabel ${\color{blue} L_K}$ at 28 65 
\pinlabel ${\color{blue} L_K}$ at 82 65 
\pinlabel ${\color{blue} L_K}$ at 116 65
\pinlabel $\textstyle{\frac{1}{2}}$ at 42 10
\pinlabel $1$ at 79 10
\pinlabel $\textstyle{\frac{1}{2}}$ at 114 10
\endlabellist	
	\centering
	\includegraphics[width=.6\linewidth]{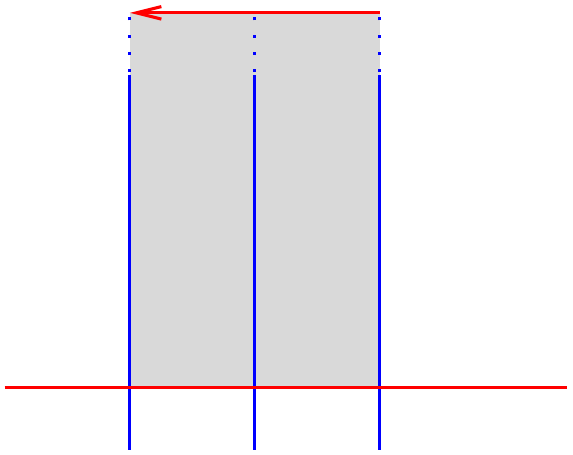}
		\begin{align*}
	[2,1]: \hspace{4ex} & |a|=2, \hspace{2ex} \dim=1 \\
	[1,1]:  \hspace{4ex} & |a|=1, \hspace{2ex}  \dim=0
	\end{align*}
	\caption{Curves with $|\mathbf{n}|=1$.}
	\label{fig:pi}
\end{figure}

\begin{figure}
\labellist
\small\hair 2pt
\pinlabel ${\color{red} a}$ at 86 142
\pinlabel ${\color{red} Q}$ at 174 28
\pinlabel ${\color{blue} L_K}$ at 28 73 
\pinlabel ${\color{blue} L_K}$ at 82 73 
\pinlabel ${\color{blue} L_K}$ at 138 73
\pinlabel $\textstyle{\frac{1}{2}}$ at 42 18
\pinlabel $\textstyle{\frac{3}{2}}$ at 68 18
\endlabellist	
	\centering
	\includegraphics[width=.6\linewidth]{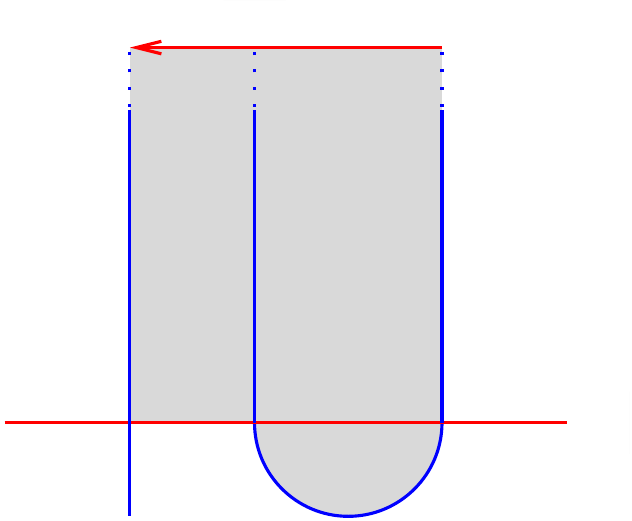}
			\begin{align*}
	[2,\textstyle{\frac{3}{2}}]: \hspace{4ex} & |a|=2, \hspace{2ex} \dim=0
	\end{align*}
	\caption{Curves with $|\mathbf{n}|=2$.}
	\label{fig:pipihalf}
\end{figure}

\begin{figure}
\labellist
\small\hair 2pt
\pinlabel ${\color{red} a}$ at 74 134
\pinlabel ${\color{red} Q}$ at 174 17
\pinlabel ${\color{blue} L_K}$ at 20 65 
\pinlabel ${\color{blue} L_K}$ at 128 65
\pinlabel $\textstyle{\frac{1}{2}}$ at 34 10
\pinlabel $1$ at 65 10
\pinlabel $1$ at 92 10
\pinlabel $\textstyle{\frac{1}{2}}$ at 124 10
\endlabellist		\centering
	\includegraphics[width=.6\linewidth]{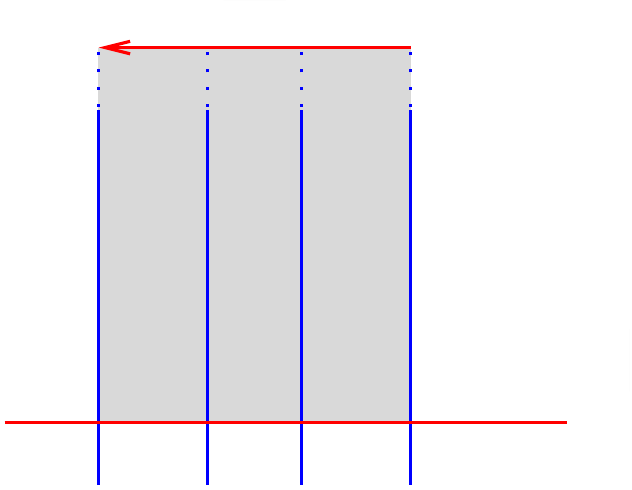}
			\begin{align*}
	[2,2]: \hspace{4ex} & |a|=2, \hspace{2ex} \dim=0
	\end{align*}
	\caption{Curves with $|\mathbf{n}|=2$.}
	\label{fig:pipi}
\end{figure}
It follows from Theorem \ref{t:cp} and Lemma \ref{l:tv} (see also Section \ref{s:stab}) that the $0$-dimensional moduli spaces listed above are transversely cut out compact $0$-manifolds. The corresponding structure theorems for moduli spaces of dimension one and two are the following.

Recall that $\R$ acts on holomorphic disks in the symplectization
$\R\times S^{\ast}Q$ by translation. Dividing out this action, we
obtain moduli spaces of dimension zero and one in the symplectization
which have the following structure.

\begin{thm}\label{t:sy}
Moduli spaces of holomorphic disks in the symplectization satisfy the following.
\begin{itemize}
\item[{\rm (i)}]
If $\MM^{\rm sy}(a;\mathbf{b})$ is a moduli space of type $[2,1]^{\rm
  sy}$ or of type $[1,0]^{\rm sy}$, then $\MM^{\rm
  sy}(a;\mathbf{b})/\R$ is a compact $0$-manifold.
\item[{\rm (ii)}]
If $\MM^{\rm sy}(a;\mathbf{b})$ is a moduli space of type $[2,0]^{\rm
  sy}$, then $\MM^{\rm sy}(a;\mathbf{b})/\R$ admits a natural
compactification $\overline{{\MM}^{\rm sy}}(a,\mathbf{b})$ which is a
compact $1$-manifold with boundary. Boundary points of $\overline{{\MM}^{\rm
    sy}}(a;\mathbf{b})$ correspond to two-level disks $\hat v$ where
the level one disk $v^{1}$ is of type $[2,1]^{\rm sy}$, and where
exactly one level two disk $v^{2,s}$ is of type $[1,0]^{\rm sy}$ and
all other level two disks $v^{2,j}$, $j\ne s$, are trivial Reeb chord
strips.
\end{itemize}
\end{thm}

In the cotangent bundle we have moduli spaces of dimension zero, one,
or two. We start with the $0$-dimensional case.

\begin{thm}\label{t:0dimcob}
Moduli spaces $\MM(a;\mathbf{n})$ of holomorphic disks of types
$[2,\frac32]$, $[2,2]$, $[1,1]$, or $[0,0]$ are compact
$0$-dimensional manifolds.
\end{thm}

In the $1$-dimensional case we consider two cases separately. We first
consider the case when $|\mathbf{n}|=0$.

\begin{thm}\label{t:[1,0]}
Moduli spaces $\MM(a;\mathbf{n})$ of disks of type $[1,0]$ admit
natural compactifications $\overline{\MM}(a;\mathbf{n})$ which are
$1$-manifolds with boundary. Boundary points of $\overline{\MM}(a;\mathbf{n})$
correspond to the following.
\begin{itemize}
\item[{\rm (a)}] Two-level disks $\dot v$ where the level one disk $v^{1}$ has type $[1,1]$ and where the second level is a three punctured constant disk $v^{2}$ attached at the Lagrangian intersection puncture of $v^{1}$ where the asymptotic winding number equals $1$.
\item[{\rm (b)}] Two-level disks $\dot v$ where the top level disk $v^{1}$ is a symplectization disk of type $[1,0]^{\rm sy}$ and where all the second level disks $v^{2,j}$, $1\le j\le k$ are of type $[0,0]$.
\item[{\rm (c)}] If there are no entries in $\mathbf{n}$, then all points of the reduced moduli space $\MM^{\rm sy}(a;\varnothing)/\R$ containing disks of type $[1,0]^{\rm sy}$ appear as boundary points.
\end{itemize}
\end{thm}

In the second $1$-dimensional case $|\mathbf{n}|=1$ and we have the following.

\begin{thm}\label{t:[2,1]}
Moduli spaces $\MM(a;\mathbf{n})$ of disks of type $[2,1]$ admit
natural compactifications $\overline{\MM}(a;\mathbf{n})$ which are
$1$-manifolds with boundary. Boundary points of $\overline{\MM}(a;\mathbf{n})$
correspond to the following.
\begin{itemize}
\item[{\rm (a)}] Two-level disks $\dot v$ where the level one disk $v^{1}$ has type $[2,2]$ and where the second level is a three punctured constant disk $v^{2}$ attached at the new-born Lagrangian intersection puncture of $v^{1}$ with winding number $1$.
\item[{\rm (b)}] Two-level disks $\dot v$ where the top level disk $v^{1}$ is a symplectization disk of type $[2,1]^{\rm sy}$ and where the second level consists of disks $v^{2,j}$, $1\le j\le k$ such that for some $s$, $v^{2,s}$ has type $[1,1]$ and $v^{2,j}$ has type $[0,0]$ for $j\ne s$.
\item[{\rm (c)}] Two-level disks $\dot v$ where the level one disk is
  of type $[2,\frac32]$ and where the second level disk is a constant
  three punctured disk attached at the Lagrangian intersection
  puncture with winding number $\frac32$. (Here the constant disk has winding number $\frac32$ at its positive puncture, and $1$ and $\frac12$ at its negative punctures.)
\end{itemize}
\end{thm}

\begin{remark}\label{r:1-dimcornercoord}
In order to parameterize a neighborhood of the boundary points in
Theorem~\ref{t:[1,0]}(a) and Theorem \ref{t:[2,1]}(a) one can use
the local model~\eqref{eq:1-dimmodel} from Section~\ref{ss:spikes}.
Here the location $\eps>0$ of the puncture on the real axis can be
used as local coordinate for the moduli space. Furthermore,
the maps in the moduli space differ from the map
in~\eqref{eq:1-dimmodel} by terms of order $\Ordo(z^{2})$, so they
have a spike that vanishes as $\eps\to 0$ as shown at the top of
Figure~\ref{fig:spike-families}.
Similarly, in order to parameterize a neighborhood of the boundary points in
Theorem \ref{t:[2,1]} (c) one can use the local
model~\eqref{eq:2-dimmodel} from Section~\ref{ss:spikes} with
$\delta=0$. Here the location $\eps>0$ of the puncture on the real
axis can be used as local coordinate for the moduli space. Furthermore,
the maps in the moduli space differ from the map
in~\eqref{eq:2-dimmodel} by terms of order $\Ordo(z^{5/2})$, so they
have a spike that vanishes as $\eps\to 0$ as shown at the bottom of
Figure~\ref{fig:spike-families} (with $\delta=0$).
See Remark~\ref{r:1-dimfinal} for details.
\end{remark}

In the 2-dimensional case we have the following description of the
structure of the moduli space which is naturally more involved.

\begin{thm}\label{t:[2,0]}
Moduli spaces $\MM(a;\mathbf{n})$ of disks of type $[2,0]$ admit
natural compactifications $\overline{\MM}(a;\mathbf{n})$ which are
$2$-manifolds with boundary with corners.
The top-dimensional strata of the boundary have codimension $1$ in
$\overline{\MM}(a;\mathbf{n})$ and correspond to the following.
\begin{itemize}
\item[{\rm (a1)}] Two-level disks $\dot v$ where the top level disk
  $v^{1}$ has type $[2,1]$ and where the second level is a three
  punctured constant disk $v^{2}$ attached at the Lagrangian
  intersection puncture of $v^{1}$ where the asymptotic winding number
  equals $1$.
\item[{\rm (b1)}] Two-level disks $\dot v$ where the top level disk $v^{1}$ is a symplectization disk of type $[2,0]^{\rm sy}$ and where all the second level disks $v^{2,j}$, $1\le j\le k$ are of type $[0,0]$.
\item[{\rm (c1)}] Two-level disks $\dot v$ where the top level disk
  $v^{1}$ is a symplectization disk of type $[2,1]^{\rm sy}$ and where
  the second level consists of disks $v^{2,j}$, $1\le j\le k$ such that for some $s$, $v^{2,s}$ has type $[1,0]$ and $v^{2,j}$ has type $[0,0]$ for $j\ne s$.
\end{itemize}

The corner points on the boundary (i.e., the codimension two strata) of
$\overline{\MM}(a;\mathbf{n})$ correspond to the following.
\begin{itemize}
\item[{\rm (a2)}] Two-level disks $\dot v$ where the top level disk
  $v^{1}$ has type $[2,2]$ and where the second level consists of two
  three punctured constant disks $v^{2,1}$ and $v^{2,2}$ attached at
  the Lagrangian intersection punctures of $v^{1}$ where the winding numbers are $1$.
\item[{\rm (b2)}] Three-level disks $\dot v$ where the top level disk $v^{1}$ is a symplectization disk of type $[2,1]^{\rm sy}$, where the second level disk $v^{2,s}$ is of type $[1,1]$ and all other second level disks $v^{2,j}$, $j\ne s$ are of type $[0,0]$, and where the third level consists of a constant three punctured disk $v^{3}$ attached at the Lagrangian intersection puncture of $v^{2,s}$ with winding number $1$.
\item[{\rm (c2)}] Three-level disks $\dot v$ where the top level disk $v^{1}$ is a symplectization disk of type $[2,1]^{\rm sy}$, where the second level disk $v^{2,s}$ is of type $[1,0]^{\rm sy}$ and all other second level disks $v^{2,j}$ are Reeb chord strips, and where the third level consists of disks $v^{3,j}$ all of type $[0,0]$.
\item[{\rm (d2)}] Two-level disks $\dot v$ where the top level disk
  $v^{1}$ has type $[2,\tfrac32]$ and where the second level consists of a
  4- punctured constant disk $v^{2}$ attached at
  the Lagrangian intersection puncture of $v^{1}$ where the
  asymptotic winding number is $\tfrac32$.
\end{itemize}
\end{thm}

\begin{remark}\label{r:2-dimcornercoord}
In order to parameterize a neighborhood of the corner points in
Theorem \ref{t:[2,0]} (d2) one can use the local
model~\eqref{eq:2-dimmodel} from Section~\ref{ss:spikes}
Here the locations $(\eps,\delta)$ of the punctures on the real
axis can be used as local coordinates for the moduli space. Furthermore,
the maps in the moduli space differ from the map
in~\eqref{eq:2-dimmodel} by terms of order $\Ordo(z^{5/2})$, so they
have two spikes that vanish as $\eps,\delta\to 0$ as shown at the bottom of
Figure~\ref{fig:spike-families}.
See Remark~\ref{r:2-dimfinal} for details.
\end{remark}

Some of the moduli spaces above admit natural maps into others by forgetting some Lagrangian intersection punctures. We next describe such maps. It is convenient to write
\[
\mathbf{\tfrac12}^{s}=\tfrac12,\stackrel{s}{\dots},\tfrac12.
\]
We consider first the case when the target is a one dimensional moduli space.
\begin{thm}\label{t:emb}
Consider a moduli space $\MM(a;\mathbf{\tfrac12}^{s},1,\mathbf{\tfrac12}^{t})$ of disks of type $[1,1]$. Forgetting the $(s+1)^{\rm th}$ Lagrangian intersection puncture we get a map
\[
\MM(a;\mathbf{\tfrac12}^{s},1,\mathbf{\tfrac12}^{t})\to
\overline{\MM}(a;\mathbf{\tfrac12}^{s+t})
\]
into the compactified moduli space of disks of type $[1,0]$. This map is an embedding of a $0$-dimensional manifold into the interior of a $1$-manifold.
\end{thm}

Finally, we consider similar maps when the target space is two dimensional.

\begin{thm}\label{t:imm}
Consider a compactified moduli space $\overline{\MM}(a;\mathbf{\tfrac12}^{s},1,\mathbf{\tfrac12}^{t})$ of disks of type $[2,1]$. Forgetting the $(s+1)^{\rm th}$ Lagrangian intersection puncture we get a map
\[
\iota\colon\overline{\MM}(a;\mathbf{\tfrac12}^{s},1,\mathbf{\tfrac12}^{t})\to
\overline{\MM}(a;\mathbf{\tfrac12}^{s+t})=\overline{\MM}
\]
into the compactified moduli space of disks of type $[2,0]$. This map is an immersion of a $1$-dimensional manifold into a $2$-manifold with boundary with corners. Let $\overline\MM_{s+1}$ denote the image of this immersion. Then $\overline{\MM}_{s+1}$ consists of those disks for which some point in the $(s+1)^{\rm th}$ boundary arc hits $K$. Then $\overline{\MM}_{s+1}$ and $\overline{\MM}_{t+1}$ intersect (self-intersect if $s=t$) transversely at disks with two points hitting $K$ (this corresponds to disks of type $[2,2]$). The boundary of $\overline{\MM}_s$ consists of points in the codimension one boundary of $\overline{\MM}$ corresponding to disks as in Theorem \ref{t:[2,1]} (a) and (b) as well as to interior points corresponding to disks of type $[2,\frac32]$ as in Theorem \ref{t:[2,1]} (c). Furthermore $\overline{\MM}_{s+1}$ and $\overline{\MM}_{s+2}$ with a common boundary point corresponding to a disk of type $[2,\frac32]$ fit together smoothly at this point.
\end{thm}

\subsection{Floer's Picard lemma}\label{s:Floer-Picard}
In the following subsections we show that the broken disks in Theorems \ref{t:[1,0]}--\ref{t:[2,1]} can be glued in a unique way to give disks in the interior of the moduli space thus providing a standard neighborhood of the boundary of the moduli space inside the compactified moduli space. Our approach here is standard and starts from Floer's Picard lemma, see \cite{Fmem} for a proof.

\begin{lemma}\label{l:FloerPicard}
Let $f\colon B_1\to B_2$ be a smooth map of Banach spaces which satisfies
$$
f(v)=f(0)+df(0)v+ N(v),
$$
where $df(0)$ is Fredholm and has a right inverse $Q$ satisfying
\begin{equation}\label{e:nlFloer}
\|QN(u)-QN(v)\|\le G(\|u\|+\|v\|)\|u-v\|
\end{equation}
for some constant $G$. Let $B(0,r)$ be the $r$-ball centered at $0\in B_1$ and assume that
\begin{equation}\label{e:appFloer}
\|Qf(0)\|\le \frac{1}{8G}.
\end{equation}
Then for $r<\frac{1}{4G}$, the zero-set $f^{-1}(0)\cap B(0,r)$ is a smooth submanifold of dimension $\dim(\ker(df(0)))$ diffeomorphic to the $r$-ball in $\ker(df(0))$.
\end{lemma}

We will apply this result as well as a parameterized version of it, see \cite[Lemma 5.13]{ESrev}.
In our case $f$ will be the $\bar\pa_J$-operator. To show existence of solutions near a broken solution we must thus establish three things: a sufficiently good approximate solution $w$ near the broken solution corresponding to $0$ in Lemma \ref{l:FloerPicard}, a right inverse for the linearization of the $\bar\pa_J$-operator at $w$, corresponding to $Q$ in Lemma \ref{l:FloerPicard}, and a quadratic estimate for the non-linear term in the Taylor expansion, corresponding to \eqref{e:nlFloer}. Here the Banach space $B_1$ will be a product of a weighted Sobolev space and a certain finite dimensional space that will serve as a neighborhood of the broken configuration and the Banach space $B_2$ will be a space of fields of complex antilinear maps. In addition to verifying uniform invertibility of the differential and the non-linear estimate we must also check that the gluing construction captures all solutions near the broken solution and that the natural change of coordinates (from the Banach space around the broken solution to the standard charts in the interior of the moduli space) is smooth.

\subsection{Gluing constant disks}\label{ss:constglu}
The boundary strata of the moduli spaces we study involve splitting off of constant disks and splitting off of disks in the symplectization. In this section we consider gluing constant disks.

We first consider a configuration $\dot v$ as in Theorem \ref{t:[1,0]} (a), \ref{t:[2,1]} (a), or \ref{t:[2,0]} (a1). In all these cases the broken configuration is a two level disk where the second level consists of a constant $3$-punctured disk $v^{2}$ that is attached to the first level disk at a Lagrangian intersection puncture with asymptotic winding number $1$. After we have carried out the gluing argument in this case we will discuss modifications needed for the other cases of constant disk gluing.

Assume that the first level disk $v^{1}$ has $m$ Lagrangian intersection punctures. We take the domain of $v^{1}$ to be the standard domain $\Delta^{1}\approx \Delta_{m}$. (As explained in Section \ref{s:stab}, we may assume that the domain is stable by adding extra marked points near the positive puncture.)
Recall that we defined a functional analytic neighborhood $\mathcal{W}(a;\mathbf{n})$ of $v^{1}$, where
$\mathcal{W}(a;\mathbf{n})$ is a product of an infinite dimensional weighted Sobolev space $\mathcal{H}(a;\delta,\mathbf{n})$ and a finite dimensional space which is an open neighborhood $B$ of the origin in $\R^{m-2}\times \R\times \R^{m}$, see Section \ref{s:confcob}. Here the first $\R^{m-2}$-component of an element in $B$ corresponds to variations of the conformal structure of $\Delta^{1}$, the second $\R$-factor to shifts of the map in the symplectization direction near the positive puncture, and the last $\R^{m}$-factor corresponds to shifts along the knot near the Lagrangian intersection punctures. Here we will write
$\mathcal{W}^{1}$ for this neighborhood $\mathcal{W}(a;\mathbf{n})$ and think of it as a product
\[
\mathcal{W}^{1}= \mathcal{W}^{1}_{0}\times B^{1}
\]
where $B^{1}$ is an open subset of $\R$, as follows. Let $q$ denote the negative puncture where the second level disk is attached. Then $B^{1}$ corresponds to shifts along the knot at $q$.

Consider the negative puncture $q$ at which the constant three punctured disk  $v^{2}\colon\Delta^2\to K$, where $\Delta^{2}\approx\Delta_{3}$ is a standard domain with three punctures, is attached and fix a half-strip neighborhood $Q=(-\infty,0]\times[0,1]$ of it such that $v^{1}(Q)$ lies entirely in the standard neighborhood of $K$ with complex analytic coordinates.

For $\rho>0$, define a standard domain $\Delta_{\rho}\approx \Delta_{m+1}$ as follows. Remove the neighborhood $(-\infty,-\rho)\times[0,1]$ of $q$ from $\Delta^{1}$ and the neighborhood $(\rho,\infty)\times[0,1]$ of the positive puncture in $\Delta^{2}$, getting domains $\Delta_{\rho}^{1}$ and $\Delta_{\rho}^{2}$. The domain $\Delta_{\rho}$ is then obtained by identifying the boundary segments $\{-\rho\}\times[0,1]\subset\Delta_{\rho}^{1}$ and $\{\rho\}\times[0,1]\subset\Delta_{\rho}^{2}$. Then $\Delta_{\rho}$ contains the strip $Q_{\rho}\approx [-\rho,\rho]\times[0,1]$:
\[
Q_{\rho}=\Delta_{\rho}-(\Delta_{0}^{1}\cup\Delta_{0}^{2}).
\]

We next define a pre-gluing $w_{\rho}\colon \Delta_{\rho}\to T^{\ast}Q$ (i.e., an approximate solution close to the broken disk $\dot v$) and a neighborhood of it in a suitably weighted space of maps. We start with the map. Fix complex analytic coordinates $\C\times\C^{2}$ around $p\in K$ on $T^{\ast}Q$, where $p$ is the point where the constant disk $v^{2}$ sits. Let $\phi\colon \Delta_{\rho}\to\C$ be a smooth function which equals $1$ on $\Delta_{{\rho/2}}^{1}$, equals $0$ on $\Delta_{\rho}^{2}$ and is real-valued and holomorphic on the boundary. (Holomorphic on the boundary just means that the restriction of $\bar{\partial}$ to the boundary vanishes. For example, if $s+it\in \mathbb{R}\times[0,1]$ are coordinates on the strip and $\phi(s)$ is an ordinary real valued cut-off function then a corresponding complex valued cut-off function that is holomorphic on the boundary is $\phi(s)+i\psi(t)\frac{d\phi}{ds}(s)$, where $\psi(t)$ is a small function with support near $\partial [0,1]$ such that $\psi(0)=\psi(1)=0$ and $\frac{d\psi}{dt}(0)=\frac{d\psi}{dt}(1)=1$.)

 Define
\[
w_{\rho}(z)=
\begin{cases}
v^{1}(z), & z\in\Delta_{{\rho/2}}^{1},\\
\phi(z)v^{1}(z) & z\notin\Delta_{{\rho/2}}^{1},
\end{cases}
\]
where the last expression refers to the analytic coordinates around
$q$ corresponding to $0$ in the coordinate system. Note then that
$w_{\rho}$ takes the boundary $\partial \Delta_{\rho}$ to $L=L_{K}\cup
Q$ and that $\bar{\pa}_{J}w_\rho$ is supported in $Q_{\rho}$. Furthermore, using the Fourier expansion of $v^{1}$ near $q$,
\[ 
v^{1}(z)=\sum_{k\ge 1} c_{k} e^{-k\pi z},
\]
we find that
\[
|\bar{\pa}_{J}w_{\rho}|_{C^{1}}=\Ordo(e^{-\pi\rho}).
\]

Define a weight function $\lambda_\rho\colon\Delta_{\rho}\to\R$ as follows, where $\eta_\delta\colon\Delta_{m+1}\to\R$ denotes the weight function on $\Delta^{1}$,
\[
\lambda_\rho(z)=
\begin{cases}
\eta_\delta(z) &\text{for }z\in \Delta^{1}_0,\\
e^{\delta(\rho-|\tau|)} &\text{for }z\in Q_\rho\approx[-\rho,\rho]\times[0,1],\\
1 &\text{for }z\in\Delta^{2}_0.
\end{cases}
\]
Let $\|\cdot\|_{k,\rho}$ denote the Sobolev norm with $k$ derivatives on $\Delta_{\rho}$ and weight function $\lambda_{\rho}$. From the above we then find
\begin{equation}
\|\bar\pa_J w_\rho\|_{1,\rho}\le |\bar\pa_J w_\rho|_{C^{1}}\int_{-\rho}^{\rho}e^{\delta(\rho-|\tau|)} d\tau=
\Ordo(e^{-(\pi-\delta)\rho}).
\end{equation}

We next define configuration spaces of maps giving neighborhoods of the approximate solutions $w_\rho$. As in Subsection \ref{s:confcob}
this space is a direct sum of an infinite dimensional space and two finite dimensional summands. We first discuss the infinite dimensional summand.

Define $\sblv_{2,\rho}(w_\rho)$ as the Sobolev space of vector fields
$v$ along $w_\rho$ (i.e., sections of $w_\rho^{\ast}T(T^\ast Q)\to\Delta_{\rho}$) which satisfies the following requirements.
\begin{itemize}
\item If $\zeta\in\pa\Delta_{\rho}$ maps to $L_K$ (maps to $Q$) then $v(\zeta)$ is tangent to $L_K$ (resp. to $Q$).
\item $\nabla v + J\circ \nabla v\circ i=0$ along $\pa\Delta_{\rho}$.
\item Fix an endpoint $\zeta_0\in\pa\Delta_\rho$ of the vertical segment which separates the part of $\Delta_{\rho}$ which corresponds to $\Delta^{1}$ from that corresponding to $\Delta^{2}$. We require that $v(\zeta_0)=0$.
\end{itemize}
Here the first two requirements have counterparts in Section~\ref{s:disks} and the third is connected to the addition of certain cut-off solutions in the gluing region.
We endow $\sblv_{2,\rho}(w_\rho)$ with the weighted Sobolev $2$-norm $\|\cdot\|_{2,\rho}$.

Second, we discuss the finite dimensional factor $B_{\rho}=B_{\rho}^{1}\times B_{\rho}^{2}$. Here
$B_{\rho}^{1}$ is an open neighborhood of the origin in $\R^{m-2}\times \R\times\R^{m-1}$ and agrees with the finite dimensional factor of $\mathcal{W}^{1}_{0}$ in the following sense. The first $\R^{m-2}$-factor corresponds to the conformal variations of $\Delta_{\rho}$ inherited from $\Delta^{1}$, the second $\R$-factor corresponds to shifts at the positive puncture, and the last $\R^{m-1}$-factor corresponds to shifts along the knot $K$ at Lagrangian intersection punctures that are also punctures of $\Delta^{1}$. The second factor $B_{\rho}^{2}$ is an open neighborhood of the origin in $\R^{3}\times \R^{2}\times\R$, where the first $\R^{3}$-factor corresponds to constant vector fields supported in $Q_{\rho}$ along the Lagrangian in a neighborhood of $p$ that are cut off in finite regions near the ends of $Q_{\rho}$, where the weight function $\lambda_{\rho}$ is uniformly bounded and where the second factor corresponds to the shifts along $K$ supported at the Lagrangian intersection punctures that are also punctures of $\Delta^{2}$. Finally, the third $\R$-factor is a newborn conformal variation defined as follows.

Consider the domain of the constant disks as a strip $\R\times[0,1]$ with positive puncture at $+\infty$, one negative puncture at $0$, and one at $-\infty$. Let $v$ be the constant vector field $\pa_\tau$ and note that its flow moves the puncture at $0$ and that in the standard model of the $3$-punctured disk this vector field looks like $c_1+\Ordo(e^{-\pi|\tau|})$ at the puncture at $+\infty$ and at one of the punctures at $-\infty$, whereas it looks like $c_2 e^{-\pi\tau}+\Ordo(1)$ at the other puncture at $-\infty$, where $c_j$, $j=1,2$ are real constants. We extend this vector field $v$ holomorphically over the gluing region $Q_\rho$ and then cut it off using a cut-off function $\beta$ with derivative supported near the end of $Q_{\rho}$ that comes from $\Delta^{1}$ where the weight function $\lambda_{\rho}$ is close to $1$. The conformal variation is then the complex anti-linear map $i\bar\pa(\beta v)$.

Note that the conformal variation in $\Delta_{\rho}$ that is inherited from the conformal variation at $q$ in $\Delta^{1}$ can be identified with the linear combination of conformal variations as above for the two punctures from $\Delta^{2}$ which looks like $0+\Ordo(e^{-\pi|\tau|})$ at the positive puncture. We take the $\R$-factor to correspond to this variation. (Note that this conformal variation is supported in $\Delta^{1}_{\rho}$ and agrees with the conformal variation in $\Delta^{1}$ corresponding to the negative puncture $q$ where the constant disk is attached.)

\begin{remark}\label{rmk:gluingparameter}
We note that there is a complementary linear combination of the two newborn conformal variations with non-zero leading constant term at the positive puncture of the constant disk that corresponds to the gluing parameter $\rho$ which, from the point of view of the domain, shifts the boundary maximum between the two new punctures, see Remark \ref{rmk:confvaratpuncture}.
\end{remark}

Let $\EE_{1,\rho}$ denote the space of complex anti-linear maps $T\Delta^{\rho}_{m+2}\to w_\rho^{\ast}T(T^{\ast}Q)$, again weighted by $\lambda_\rho$.  The linearization of the $\bar\pa_J$-operator at $w_\rho$ is then an operator
\[
L\bar\pa_J\colon\sblv_{2,\rho}(w_\rho)\times B_{\rho}\to \EE_{1,\rho},
\]

\begin{lemma}\label{l:unifinv}
The operators $L\bar\pa_J$ admit right inverses which are uniformly bounded as $\rho\to\infty$.
\end{lemma}

\begin{proof}
The argument here is standard.
Let $k_1,\dots, k_l$ be a basis of the kernel $K$ of the linearized operator on $v^{1}$. Fix a cut-off function $\beta$ which equals $1$ on the part of $\Delta_{\rho}$ corresponding to $\Delta^{1}$ and with first and second derivatives supported in $Q_\rho$ of size $\Ordo(\rho^{-1})$. (Such a cut-off function exists since the length of $Q_\rho$ equals $2\rho$.) We will establish an estimate
\begin{equation}\label{e:estonL2c}
\|v\|_{2,\rho}\le C\|L\bar\pa_j v\|_{1,\rho},
\end{equation}
where $C>0$ is a constant, for $v$ in the $L^{2}$-complement of the subspace $\tilde K$ spanned by the cut-off solutions $\beta k_1,\dots,\beta k_l$. The lemma follows from this estimate.

We argue by contradiction: assume that the estimate does not hold. Then there is a sequence $v_\rho$ in this $L^{2}$-complement with
\begin{equation}\label{e:contra}
\|v_\rho\|_{2,\rho}= 1,\quad\text{and}\quad \|L\bar\pa_J v_\rho\|_{1,\rho}\to 0.
\end{equation}

We write $v_\rho=u_\rho + b_\rho^{1} + b_\rho^{2}$, where $u_\rho\in \sblv_{2,\rho}(w_\rho)$, $b_\rho^{1}\in B^{1}_{\rho}$, and $b^{2}_\rho\in B^{2}_{\rho}$.  Fix cut-off functions $\beta_j$ on $\Delta_{\rho}$, $j=1,2$ with the following properties. The function $\beta_1$ equals $1$ on $\Delta^{1}_{\rho/2}\subset \Delta_{\rho}$ and equals $0$ on $\Delta^{2}_{\rho}\subset\Delta_{\rho}$. Furthermore, $\beta_1 u_\rho$ is holomorphic on the boundary, and $|D\beta_1|=\Ordo(\rho^{-1})$. The function $\beta_2$ has similar properties but with support in $\Delta_{\rho}^{2}\subset\Delta_\rho$. We also let $\alpha$ be a similar cut-off function on $\Delta_{\rho}$, equal to $1$ on $Q_{\frac12\rho}$ and equal to $0$ outside $Q_{\rho-1}$.

Since $|\bar\pa\beta_j^{\rho}|\to 0$ as $\rho\to 0$ we then have
\begin{align*}
\left\|L\bar{\pa}_{J}(\beta_1 u_\rho+b_{\rho}^{1}+b_{\rho}^{2}|_{\Delta_{\rho}^{1}})\right\|_{1,\rho}
&\le|\bar\pa\beta_j^{\rho}|\left\| u_\rho \right\|_{1,\rho}
+\left\|L\bar{\pa}_{J}(u_\rho+b_{\rho}^{1}+b_{\rho}^{2}|_{\Delta_{\rho}^{1}})\right\|_{1,\rho}\\
&\le |\bar\pa\beta_j^{\rho}|\left\| u_\rho \right\|_{1,\rho}
+ \|L\bar\pa_J v_\rho\|_{1,\rho}
\to 0,
\end{align*}
as $\rho\to\infty$. We then conclude from transversality of $v^{1}$ (i.e., invertibility of the linearized operator off of its kernel) that there exists a constant $M>0$ such that
\begin{equation}\label{eq:tozero1}
\left\|\beta_1 u_\rho+b_{\rho}^{1}+b_{\rho}^{2}|_{\Delta_{\rho}^{1}}\right\|_{2,\rho}\le 
M \left\|L\bar{\pa}_{J}(\beta_1 u_\rho+b_{\rho}^{1}+b_{\rho}^{2}|_{\Delta_{\rho}^{1}})\right\|_{1,\rho}
\to 0.
\end{equation}
In particular the cut-off constant solution in the gluing region goes to $0$.

Similarly we have
\[
\left\|L\bar{\pa}_{J}(\beta_2u_\rho
+b_{\rho}^{2}|_{\Delta_{\rho}^{2}})\right\|_{1,\rho}\to 0.
\]
We conclude from the invertibility of the standard operator on the three punctured disk that
\begin{equation}\label{eq:stconst3punct}
\left\|\beta_2u_\rho
+b_{\rho}^{2}|_{\Delta_{\rho}^{2}}\right\|_{2,\rho}\to 0.
\end{equation}

After dividing the weight function in the gluing region $Q_{\rho/2}\approx [-\frac{\rho}{2},\frac{\rho}{2}]\times[0,1]$ by its maximum the problem on the gluing region converges to the $\bar\pa$-problem on the strip with $\R^{3}$-boundary condition and negative exponential weights at both ends (i.e.~weight function $\rho(s+it)=e^{-\delta|s|}$). This problem has a three-dimensional kernel spanned by constant solutions in $\R^{3}$. As mentioned above, the estimates \eqref{eq:tozero1} and \eqref{eq:stconst3punct} imply that the components along the constant solutions go to zero. This gives first that
\[
\left\|L\bar\pa_J(\alpha u_\rho)\right\|_{1,\rho}\to 0,
\]
and then, by invertibility of $L\bar{\pa}_{J}$ on the complement of the kernel, 
also that $\|\alpha u_{\rho}\|_{2,\rho}\to 0$.  Our assumption thus implies that $\|v_\rho\|_{2,\rho} \to 0$. This contradicts \eqref{e:contra}. The lemma follows.
\end{proof}

The next thing to establish is the quadratic estimate for the non-linear term in the Taylor expansion of $\bar\pa_J$ around $w_{\rho}$, i.e., around the origin in $\sblv_{2,\rho}\times B_{\rho}$. We use the exponential map as in Section \ref{s:confcob} to define the local coordinate system around $w_{\rho}$ and the estimate for the non-linear term follows from a standard argument that uses the uniform bounds on the derivatives of the exponential map in our metric, see \cite[Lemma A.18]{ESrev} and also \cite{EES1, EES2}.
In fact the standard argument gives the corresponding unweighted estimate but then the case of positive weights follows since the left hand side of the inequality is linear in the weight whereas the right hand side is quadratic. So the inequality follows for weights bounded from below. Note also that variations along the cut-off solutions in $B_{\rho}$ give contributions to the non-linear term only in the regions where the derivatives of the cut-off functions are supported and in such regions the weight functions have finite size.

\begin{remark}
It is essential here that the cut-off solutions are real solutions to the non-linear equation since a small error term would give a large norm contribution because of the large weight function in $N_\rho$, which in turn is key for the proof of the uniform invertibility of the differential in Lemma \ref{l:unifinv}.
\end{remark}

The final step is then to show surjectivity of the construction. More concretely, this means that we must show that any sequence of disks which converges in the sense of Subsection \ref{S:cp} to a broken disk eventually lies in a small $\|\cdot\|_{2,\rho}$-neighborhood of $w_\rho$. This follows once we show that any holomorphic disk in a $C^{0}$-neighborhood of the approximate solution is also close in $\|\cdot\|_{2,\rho}$-norm. The proof of that fact follows from the knowledge of explicit solutions in the region where the weight is big. Here $C^{0}$-control at the ends gives norm control, see
\cite[Proof of Theorem A.21]{ESrev} or \cite[Proof of Theorem 1.3]{E}. This finishes the gluing results needed in the cases when we glue one constant 3-punctured disks at a Lagrangian intersection puncture of winding number $1$.

The remaining cases for gluing constant disks are proved by modifications of the above argument that we describe next. Consider first Theorem \ref{t:[2,0]} (a2). Here we replace the gluing parameter $\rho$ with two independent gluing parameters $(\rho_1,\rho_2)\in[0,\infty)^{2}$,  one for each constant disk. Likewise we have two copies of the new finite dimensional factors in the configuration space. The gluing argument is then a word by word repetition of the above.

Next consider broken disks as in Theorem \ref{t:[2,1]}(c). Here the exponential weight at the winding $\frac32$-puncture of $v^{1}$ is $\delta\in (\frac{\pi}{2},\pi)$ and the boundary condition in the strip $Q_{\rho}$ has different constant Lagrangians along the two boundary components. The cut-off solutions in $B^{2}_{\rho}$ change accordingly: instead of an $\R^{3}$-factor of cut-off solutions we have an $\R^{5}$-factor, $\R^{5}=\R\times\R^{2}\times\R^{2}$. The $\R$-factor is a constant solution in the direction of the knot. The first $\R^{2}$-factor contains cut-off solutions near the positive puncture of $\Delta_{\rho}^{2}$ of the form $c e^{\frac{\pi z}{2}}$ for $c$ a vector in the appropriate Lagrangian 2-space perpendicular to the knot, the second $\R^{2}$-factor consists of cut-off solutions of the form $c e^{\frac{-\pi z}{2}}$. Then in Lemma \ref{l:unifinv} we replace \eqref{eq:stconst3punct} with the estimate on the three punctured disk with boundary condition corresponding to the constant disk. I.e., in directions perpendicular to the knot the boundary condition are two perpendicular Lagrangian planes at the two boundary components near the positive puncture and one of these planes between the two negative punctures. There is a small positive exponential weight at the negative punctures, the weight $\delta$ and two cut-off solutions at the positive puncture. In the directions perpendicular to the knot the $\bar{\pa}$-operator is then an isomorphism and the argument above proceeds as before.

\begin{remark}\label{rmk:twoconst}
In Theorem \ref{t:[1,0]} (c) there are two different constant disks and the corresponding boundary points cancel out. Geometrically this corresponds to pushing a winding $\frac12$ puncture through a winding $1$ puncture.
\end{remark}

Finally, we consider Theorem \ref{t:[2,0]} (d2). The argument here is the same as that just described for Theorem \ref{t:[2,1]} (c) with the only difference being that the 3-punctured constant disk should be replaced by a 4-punctured disk and that we invert the operator on the $L^{2}$ complement of the additional conformal variation in the 4-punctured disk. In fact, when the 4-punctured disk is broken into two levels it corresponds to the 3-level configuration with the two top levels as in Theorem \ref{t:[2,1]} (c) and a third level constant disk attached at the winding 1 puncture of the second level constant disk.

\subsection{Symplectization gluing}\label{ss:sympglu}
Consider a disk with two non-constant levels as in Theorem \ref{t:[2,0]} (b) or (c), Theorem \ref{t:[1,0]} (b) or (c), or Theorem \ref{t:[2,1]} (b). The argument needed to glue such configurations is similar to the one in Subsection \ref{ss:constglu} and we only sketch the details. There are again four steps: define an approximate solution, prove uniform invertibility of the differential, establish a quadratic estimate for the non-linear term, and show surjectivity of the construction.

We consider first the case when we glue a symplectization disk to a disk in $T^{\ast}Q$ and discuss modifications needed when the second level also lives in the symplectization later.
Denote the top-level disk in the symplectization $v^{1}\colon \Delta^{1}=\Delta_{m}\to \R\times S^{\ast}Q$ and the $m$ second level disks $v^{2,j}\colon \Delta_{m_j}\to T^{\ast}Q$, $j=1,\dots,m$. Recall that by adding marked points we reduce to the case when all domains involved are stable, see Section \ref{s:stab}.

Each symplectization disk lies in a natural $\R$-family. Let $t$ denote a standard coordinate on the $\R$-factor. Fix the unique map $v^{1}$ in this family that takes the largest boundary maximum in $\Delta^{1}$ to the slice $\{t=0\}$. By asymptotics at the negative punctures, for all $T>0$ sufficiently large $(v^{1})^{-1}(\{t\le -T\})$ consists of $m$ half strip regions with one component around each negative puncture of $v^{1}$. Furthermore, as $T\to\infty$ the inverse image of the slice $\{t=-T\}$ converges to vertical segments at an exponential rate (since the map agrees with trivial Reeb chord strips up to exponential error). We fix such a slice and consider the vertical segments through its end point. Parameterize the neighborhoods of all the punctures cut at these vertical segments by $(-\infty,0]\times[0,1]$. For $\rho>0$, let $\Delta^{1}_{\rho}\subset \Delta^{1}$ be the subset obtained by removing $(-\infty,-\rho)\times[0,1]$ from the neighborhood $(-\infty,0]\times[0,1]$ of each negative puncture.

Fix neighborhoods $[0,\infty)\times[0,1]$ of the positive puncture in each $\Delta^{2,j}$, $j=1,\dots,m$ in which the map is well approximated by the trivial strip at the positive puncture and let $\Delta^{2,j}_{\rho}\subset\Delta^{2,j}$ denote the subset obtained by removing $(\rho,\infty)\times[0,1]$ from this neighborhood. Let $\Delta_{\rho}$ denote the domain obtained by adjoining $\Delta^{2,j}_{\rho}$ to  $\Delta^{1}_{\rho}$ by identifying the vertical segment at the positive puncture of $\Delta^{2,j}_{\rho}$ with the vertical segment of the negative puncture in $\Delta^{1}_{\rho}$ where $v^{2,j}$ is attached to $v^{1}$. Then we get $m$ strip regions $Q^{j}_{\rho}=[-\rho,\rho]\times[0,1]\subset\Delta_{\rho}$ around each vertical segment where the disks were joined.

By interpolating between the two maps joined at each negative puncture using the standard coordinates near the Reeb chords we find a pregluing
\[
w_{\rho}\colon \Delta_{\rho}\to T^{\ast}Q
\]
such that $\bar{\pa}_{J}w_{\rho}$ is supported only in the middle $[-1,1]\times[0,1]$ of each $Q_{\rho}^{j}$ and such that
\[
|\bar{\pa}_{J}w_{\rho}|_{C^{1}}=\Ordo(e^{-\alpha\rho}),
\]
where $\alpha>0$ depends on the angle between the Lagrangian subspaces of the contact hyperplane obtained by moving the tangent space of $\Lambda_K$ at the Reeb chord start point to the tangent space of $\Lambda_K$ at the Reeb chord end point by the linearized Reeb flow.

As in Section \ref{ss:constglu} we use a configuration space of maps in a neighborhood of $w_{\rho}$ that is a product of an infinite and a finite dimensional space of functions. We first consider the infinite dimensional factor.
Define weight functions $\lambda_\rho\colon\Delta_{\rho}\to\R$ by patching (suitably scaled) weight functions $\eta_\delta$ of the domains of the broken disks where we take $0<\delta<\alpha$. In particular, we have $\lambda_\rho(\tau+it)=c_j e^{\delta|\tau|}$ for $\tau+it\in Q_\rho^{j}$. Then, writing $\|\cdot\|_{k,\rho}$ for the Sobolev $k$-norm with this weight, we have
\[
\|\bar\pa_J w_\rho\|_{1,\rho}=\Ordo(e^{(\delta-\alpha)\rho}).
\]

We let $\sblv_{2,\rho}(w_\rho)$ denote the $\lambda_\rho$-weighted Sobolev space of vector fields along $w_\rho$ which are tangent to the Lagrangians, holomorphic on the boundary, and which satisfy the following vanishing condition. The map $w_{\rho}$ maps the strip regions $Q^{j}_{\rho}$ into small neighborhoods of the Reeb chord strips where we have standard coordinates $\R\times (-\epsilon, L+\epsilon)\times\C^{2}$ and we require that the $\R$-component of the vector field vanishes at one of the endpoints of the vertical segments where the disks were joined. Thus there are in total $m$ vanishing conditions.

Next we discuss the finite dimensional factor $B_{\rho}=B^{0}_{\rho}\times B^{1}_{\rho}\times B^{2}_{\rho}$. The second factor $B^{1}_{\rho}$ is an open subset of the origin in $\R$ corresponding to the shift at the positive puncture of $w_{\rho}$. The third factor $B^{2}_{\rho}$ contains all the conformal variations and the shifts inherited from the negative punctures of the second level disks. Thus $B^{2}_{\rho}$ is a neighborhood of the origin in
\[
\Pi_{j=1}^{m} (\R^{m_j-2}\times\R^{m_j}).
\]
Finally, the first factor $B_{\rho}^{0}$ is an open subset of the origin in a codimension one subspace of
\[
(\R\times\R^{2})^{m},
\]
where each ($\R\times\R^{2}$)-factor corresponds to a specific second level disk. The $\R$-component of the $j^{\rm th}$ puncture of $v^{1}$ corresponds to a cut-off shifting vector field $a_j$ in the $\R$-direction of the symplectization supported in $Q^{j}_{\rho}$. The $\R^{2}$-component corresponds to the two newborn conformal variations in $\Delta^{2,j}_{\rho}$. As before these conformal variations have the form $\gamma = \bar{\pa} V$ where $V$ is a vector field along $\Delta_{\rho}$. The first factor of $\R^{2}$ corresponds to a variation $\gamma^{j}_{1}$ that agrees with the conformal variation at the negative puncture in $\Delta^{1}$ where $v^{2,j}$ is attached. The second factor is spanned by  $\gamma^{2,j}=\bar{\pa} V_{2}$ where $V_{2}$ is the vector field in $\Delta^{2,j}_{\rho}\cup Q_{\rho}^{j}$ that corresponds to translations along the real axis that moves all the boundary maxima in $\Delta^{2,j}_{\rho}$ cut off near the end of $Q_{\rho}$ in $\Delta_{\rho}^{1}$. The codimension one subspace is the orthogonal complement of the line given by the equation
\[
\gamma^{2,1}=\gamma^{2,2}=\dots =\gamma^{2,m}.
\]
Note that this later conformal variation corresponds to changing $\rho$.

\begin{remark}
The nature of the conformal variations $\gamma^{1,j}$ and $\gamma^{2,j}$ are easy to see using a different conformal model for the domain $\Delta_{\rho}$ as follows. Consider the domain of $\Delta^{1}$ as the upper half plane $H$ with positive puncture at $\infty$ and negative punctures along the real axis. The conformal variations of this domain can be viewed as translating the negative punctures along the real axis. To construct the domain $\Delta_{\rho}$ we think also of the domains $\Delta^{2,j}$ as upper half planes. Cut out small half disks of radius $c_j e^{-\alpha\rho}$ near the negative punctures of $\Delta^{1}$ and glue in the half disks in the domain $\Delta^{2,j}$ of radius $c_j e^{\alpha\rho}$ scaled by $e^{-2\alpha\rho}$. Now the conformal variation $\gamma^{1,j}$ corresponds to translating the whole half disk at the $j^{\rm th}$ negative puncture of $\Delta^{1}$ rigidly in the real direction and the conformal variation $\gamma^{2,j}$ corresponds to keeping the small half disk fixed but scaling it so that its negative punctures move closer together.
\end{remark}

We use the neighborhood $\mathcal{W}_{\rho}=\sblv_{2,\rho}\times B_{\rho}$ of $w_{\rho}$. In order to apply Lemma \ref{l:FloerPicard} we must first establish the counterpart of Lemma \ref{l:unifinv}. Here we invert the linearized operator on the $L^{2}$-complement of the subspace spanned by cut-off kernel elements in $\Delta^{1}$ and $\Delta^{2,j}$ defined as follows. The infinite dimensional components are indeed just a cut-off vector field. For the finite dimensional components we identify the conformal variation at the $j^{\rm th}$ negative puncture of $\Delta^{1}$ with $\gamma^{1,j}$, the shift at this negative puncture with $a_j$, and the shift at the positive puncture of $\Delta^{2,j}$ with $\gamma^{2,j}$. To show uniform invertibility we then argue by contradiction as in the proof of Lemma \ref{l:unifinv}. Using the above identifications of finite dimensional factors, the result follows in a straightforward way.

Finally, the two remaining steps, the quadratic estimate for the non-linear term and the surjectivity of the construction are completely analogous to their counterparts in Subsection \ref{ss:constglu} and will not be discussed further.

In the case that the second level disk lies in the symplectization as well we start as above by fixing a representative for $v^{1}$ and a slice $\{t=-T\}$ after which this representative is well approximated by Reeb chord strips. We then fix representatives for all the non-trivial second level curves $v^{2,j}$ (of which there is only one in our case) that are translated sufficiently much so that they are well approximated by Reeb chord strips in the slice $\{t=-T\}$ at their positive punctures. We then repeat the argument above.

\subsection{Point constraints on the knot}\label{ss:branchmark}

An analogous construction allows us to express neighborhoods of disks with Lagrangian intersection punctures of winding number $1$ inside the space of disks with these punctures removed. In the analytical $\C\times\C^{2}$-coordinates around the knot a disk $v$ with such a puncture looks like
\[
v(z)=
\sum_{n\ge 0} c_n e^{-n\pi z}, \quad z\in[0,\infty)\times[0,1], \quad c_n\in\R^{3} \text{ or } c_n\in\R\times i\R^{2}
\]
with $c_0=(c_0',0)$, whereas a general disk looks the same way but has unrestricted $c_0$. We can thus construct a configuration space $\mathcal{W}$ for unrestricted disks in a neighborhood of $v$ as
\[
\mathcal{W}=\mathcal{W}'\oplus \R^{2},
\]
where $\mathcal{W}'$ is the configuration space for disks in a neighborhood of $v$ with Lagrangian intersection puncture of winding number $1$ and $\R^{2}$ is spanned by two cut-off constant solutions in the Lagrangian perpendicular to $K$. The zero-set of the $\bar{\pa}_{J}$-operator acting on $\mathcal{W}$ then gives a neighborhood of $v$ in the space of unrestricted disks.

\subsection{Proofs of the structure theorems}
The proof of all the theorems on the structure of the compactified moduli spaces as manifolds with boundary with corners now follow the same pattern. Transversality and compactness results give the possible degenerations and gluing give neighborhoods of several level disks in the boundary. The manifold structure in the interior is a consequence of standard Fredholm theory, whereas charts near the boundary are obtained from the conformal structures of the domains.

\begin{proof}[Proof of Theorem \ref{t:sy}]
Part (i) follows immediately from Lemma \ref{l:tv} and Theorem \ref{t:cp}. Consider part (ii).  Lemma \ref{l:tv} and Theorem \ref{t:cp} imply that the broken disks listed are the only possible configurations in the boundary of the compactified moduli space. It follows from (the parameterized version of) Lemma \ref{l:FloerPicard} that the gluing parameter gives a parameterization of the boundary of the reduced moduli space. Recall that we identified the gluing parameter with a certain conformal variation (that shifts all the boundary maxima in the second level disk) and we topologize a neighborhood of the broken configuration using the induced map to the compactified space of conformal structures. This establishes (ii).
\end{proof}

\begin{proof}[Proof of Theorem \ref{t:0dimcob}]
The theorem follows immediately from Lemma \ref{l:tv} and Theorem \ref{t:cp}.
\end{proof}

\begin{proof}[Proof of Theorem \ref{t:[1,0]}]
The proof is analogous to the proof of Theorem \ref{t:sy} (ii) except for (c). Here a disk without Lagrangian intersection punctures moves out as a rigid disk in the symplectization into the $\R$-invariant region and the translations along $\R$ give a neighborhood of the boundary.
\end{proof}

\begin{proof}[Proof of Theorem \ref{t:[2,1]}]
The argument is analogous to the proofs above and we explain only how to parameterize the boundary in the cases that differ from the above. Consider (b). Recall that we identified the gluing parameter with the conformal variation that translates all the boundary maxima in the second level disks uniformly. As above we use this to parameterize a neighborhood of the boundary. Finally, consider (c). Here again the boundary can be parameterized by the gluing parameter which corresponds to a conformal variation. In particular, the boundary point corresponds to a three punctured disk splitting off. As explained in Remark \ref{rmk:twoconst} there are two such disks and the corresponding boundaries of the moduli space naturally fit together to a smooth $1$-manifold.
\end{proof}

\begin{remark}\label{r:1-dimfinal}(cf.~Remark~\ref{r:1-dimcornercoord}).
Consider a holomorphic disk near the codimension one boundary as in
Theorem~\ref{t:[2,1]} (c). Remark \ref{r:breakingmodelclose} gives a
local model~\eqref{eq:1-dimmodel} for the disk, parameterized by
a half disk in the upper half plane near the two colliding corners
with one puncture at $0$ and one at $\epsilon>0$. The above
proof shows that the newborn conformal variation which here is the
length of the stretching strip can be used as local coordinate in the
moduli space near the corner. A conformal map that takes a 
vertical segment in the stretching strip to the upper arc in the unit
circle and the boundary of the domain in the disk splitting off to
the real line gives a smooth change of coordinates from this parameter
to the coordinates given by $\epsilon$. Thus the local model
\eqref{eq:1-dimmodel} used in the definition of the string operations
is $C^{k}$-close to the actual moduli space, when both are viewed as
parameterized by the coordinates $\epsilon$.
A similar discussion applies to Theorem~\ref{t:[2,1]} (c), using the local
model~\eqref{eq:2-dimmodel} with $\delta=0$.
\end{remark}

\begin{proof}[Proof of Theorem \ref{t:[2,0]}]
Arguments for producing neighborhoods of codimension one boundary strata are similar to the above, so we discuss the codimension two parts.

Consider a broken disk as in (a2). The gluing result needed in this case is analogous to the argument in Subsection \ref{ss:constglu}. Here however we attach two constant disks, producing approximate solutions $w_{\rho_1,\rho_2}$ depending on two independent variables $\rho_1,\rho_2\to\infty$. In this case there are two independent newborn conformal variations and the linearized $\bar\pa_J$-operator is inverted on the complement of their linear span. It follows as above that the projection of the moduli space is an embedding into the space of conformal structures and we induce the corner structure from there. Note that this is coherent with our treatment of nearby codimension one boundary disks.

The arguments in cases (b2), (c2), and (d2) follow the same lines. We produce approximate solutions depending on two independent variables. In case (b2) the linearized operator is inverted on the complement of the $2$-dimensional spaces spanned by the cut off shift of the symplectization disk and the newborn conformal structure of the constant disk. In case (c2) the linearized operator is inverted on the complement of the (independent) shifts of the first and second level disks, and in case (d2) on the complement of the newborn conformal structure and the additional conformal structure in the constant 4-punctured disk. In all cases, the corner structure is induced from the corresponding structure on the space of conformal structures and the construction is compatible with nearby strata of lower codimension.
\end{proof}

\begin{remark}\label{r:2-dimfinal}(cf.~Remark~\ref{r:2-dimcornercoord}).
Consider a holomorphic disk near the codimension two corner as in
Theorem~\ref{t:[2,0]} (d2). Remark \ref{r:breakingmodelclose} gives a
local model~\eqref{eq:2-dimmodel} for the disk, parameterized by
a half disk in the upper half plane near the three colliding corners
with one puncture at $0$ and the two others at boundary points
$\delta< 0$ and $\epsilon>0$. The above proof shows that the newborn
conformal variation (which here is the length of the stretching strip)
together with the difference between the boundary maxima in the
4-punctured disk splitting off can be used as local coordinates in the
moduli space near the corner. A conformal map that takes a
vertical segment in the stretching strip to the upper arc in the unit
circle and the boundary of the domain in the disk splitting off to the
real line gives a smooth change of coordinates from these two
parameters to the coordinates given by $(\epsilon,\delta)$. Thus the
local model \eqref{eq:2-dimmodel} used in the definition of the string
operations is $C^{k}$-close to the actual moduli space, when both are
viewed as parameterized by the coordinates $(\epsilon,\delta)$.
\end{remark}

\begin{proof}[Proof of Theorem \ref{t:emb}]
The theorem follows from the discussion in Subsection \ref{ss:branchmark}.
\end{proof}

\begin{proof}[Proof of Theorem \ref{t:imm}]
The theorem follows from the discussion in Subsection \ref{ss:branchmark} in combination with the argument in the proof of Theorem \ref{t:[2,1]} (c).
\end{proof}


\bibliographystyle{plain}
\bibliography{biblio}

\begin{thebibliography}{10}

\bibitem{AENV}
Mina Aganagic, Tobias Ekholm, Lenhard Ng, and Cumrun Vafa.
\newblock Topological strings, {D}-model, and knot contact homology.
\newblock {\em Adv. Theor. Math. Phys.}, 18(4):827--956, 2014.

\bibitem{BMSS}
Somnath Basu, Jason McGibbon, Dennis Sullivan, and Michael Sullivan.
\newblock Transverse string topology and the cord algebra.
\newblock {\em J. Symplectic Geom.}, 13(1):1--16, 2015.

\bibitem{BEHWZ}
F.~Bourgeois, Y.~Eliashberg, H.~Hofer, K.~Wysocki, and E.~Zehnder.
\newblock Compactness results in symplectic field theory.
\newblock {\em Geom. Topol.}, 7:799--888, 2003.

\bibitem{CS:1}
Moira Chas and Dennis Sullivan.
\newblock String topology.
\newblock arXiv:math.GT/9911159.

\bibitem{CEL}
K.~Cieliebak, T.~Ekholm, and J.~Latschev.
\newblock Compactness for holomorphic curves with switching {L}agrangian
  boundary conditions.
\newblock {\em J. Symplectic Geom.}, 8(3):267--298, 2010.

\bibitem{CL}
Kai Cieliebak and Janko Latschev.
\newblock The role of string topology in symplectic field theory.
\newblock In {\em New perspectives and challenges in symplectic field theory},
  volume~49 of {\em CRM Proc. Lecture Notes}, pages 113--146. Amer. Math. Soc.,
  Providence, RI, 2009.

\bibitem{Rizell}
Georgios Dimitroglou~Rizell.
\newblock Lifting pseudo-holomorphic polygons to the symplectisation of
  {$P\times\Bbb{R}$} and applications.
\newblock {\em Quantum Topol.}, 7(1):29--105, 2016.

\bibitem{DG}
Nathan~M. Dunfield and Stavros Garoufalidis.
\newblock Non-triviality of the {$A$}-polynomial for knots in {$S^3$}.
\newblock {\em Algebr. Geom. Topol.}, 4:1145--1153 (electronic), 2004.

\bibitem{E}
Tobias Ekholm.
\newblock Morse flow trees and {L}egendrian contact homology in 1-jet spaces.
\newblock {\em Geom. Topol.}, 11:1083--1224, 2007.

\bibitem{E_rsft}
Tobias Ekholm.
\newblock Rational symplectic field theory over {$\Bbb Z_2$} for exact
  {L}agrangian cobordisms.
\newblock {\em J. Eur. Math. Soc. (JEMS)}, 10(3):641--704, 2008.

\bibitem{EENStransverse}
Tobias Ekholm, John Etnyre, Lenhard Ng, and Michael Sullivan.
\newblock Filtrations on the knot contact homology of transverse knots.
\newblock {\em Math. Ann.}, 355(4):1561--1591, 2013.

\bibitem{EES1}
Tobias Ekholm, John Etnyre, and Michael Sullivan.
\newblock The contact homology of {L}egendrian submanifolds in
  {$\mathbb{R}^{2n+1}$}.
\newblock {\em J. Differential Geom.}, 71(2):177--305, 2005.

\bibitem{EESori}
Tobias Ekholm, John Etnyre, and Michael Sullivan.
\newblock Orientations in {L}egendrian contact homology and exact {L}agrangian
  immersions.
\newblock {\em Internat. J. Math.}, 16(5):453--532, 2005.

\bibitem{EES2}
Tobias Ekholm, John Etnyre, and Michael Sullivan.
\newblock Legendrian contact homology in {$P\times\mathbb{R}$}.
\newblock {\em Trans. Amer. Math. Soc.}, 359(7):3301--3335 (electronic), 2007.

\bibitem{EENS}
Tobias Ekholm, John~B. Etnyre, Lenhard Ng, and Michael~G. Sullivan.
\newblock Knot contact homology.
\newblock {\em Geom. Topol.}, 17(2):975--1112, 2013.

\bibitem{EK}
Tobias Ekholm and Tam\'as K\'alm\'an.
\newblock Isotopies of {L}egendrian 1-knots and {L}egendrian 2-tori.
\newblock {\em J. Symplectic Geom.}, 6(4):407--460, 2008.

\bibitem{ENSh}
Tobias Ekholm, Lenhard Ng, and Vivek Shende.
\newblock A complete knot invariant from contact homology.
\newblock arXiv:1606.07050.

\bibitem{ESrev}
Tobias Ekholm and Ivan Smith.
\newblock Exact {L}agrangian immersions with one double point revisited.
\newblock {\em Math. Ann.}, 358(1-2):195--240, 2014.

\bibitem{ES}
Tobias Ekholm and Ivan Smith.
\newblock Exact {L}agrangian immersions with a single double point.
\newblock {\em J. Amer. Math. Soc.}, 29(1):1--59, 2016.

\bibitem{EGH}
Y.~Eliashberg, A.~Givental, and H.~Hofer.
\newblock Introduction to symplectic field theory.
\newblock {\em Geom. Funct. Anal.}, (Special Volume, Part II):560--673, 2000.
\newblock GAFA 2000 (Tel Aviv, 1999).

\bibitem{Floer_Lag_int}
Andreas Floer.
\newblock Morse theory for {L}agrangian intersections.
\newblock {\em J. Differential Geom.}, 28(3):513--547, 1988.

\bibitem{FO3I}
Kenji Fukaya, Yong-Geun Oh, Hiroshi Ohta, and Kaoru Ono.
\newblock {\em Lagrangian intersection {F}loer theory: anomaly and obstruction.
  {P}art {I}}, volume~46 of {\em AMS/IP Studies in Advanced Mathematics}.
\newblock American Mathematical Society, Providence, RI; International Press,
  Somerville, MA, 2009.

\bibitem{GL}
C.~McA. Gordon and J.~Luecke.
\newblock Knots are determined by their complements.
\newblock {\em J. Amer. Math. Soc.}, 2(2):371--415, 1989.

\bibitem{GLid}
Cameron Gordon and Tye Lidman.
\newblock Knot contact homology detects cabled, composite, and torus knots.
\newblock arXiv:1509.01642.

\bibitem{Hir}
Morris~W. Hirsch.
\newblock {\em Differential topology}.
\newblock Springer-Verlag, New York-Heidelberg, 1976.
\newblock Graduate Texts in Mathematics, No. 33.

\bibitem{HWZ1}
H.~Hofer, K.~Wysocki, and E.~Zehnder.
\newblock A general {F}redholm theory. {I}. {A} splicing-based differential
  geometry.
\newblock {\em J. Eur. Math. Soc. (JEMS)}, 9(4):841--876, 2007.

\bibitem{Fmem}
Helmut Hofer, Clifford~H. Taubes, Alan Weinstein, and Eduard Zehnder, editors.
\newblock {\em The {F}loer memorial volume}, volume 133 of {\em Progress in
  Mathematics}.
\newblock Birkh\"auser Verlag, Basel, 1995.

\bibitem{KM}
P.~B. Kronheimer and T.~S. Mrowka.
\newblock Dehn surgery, the fundamental group and {SU{$(2)$}}.
\newblock {\em Math. Res. Lett.}, 11(5-6):741--754, 2004.

\bibitem{LR}
R.~H. Lagrange and A.~H. Rhemtulla.
\newblock A remark on the group rings of order preserving permutation groups.
\newblock {\em Canad. Math. Bull.}, 11:679--680, 1968.

\bibitem{Ng:2a}
Lenhard Ng.
\newblock Knot and braid invariants from contact homology. {I}.
\newblock {\em Geom. Topol.}, 9:247--297 (electronic), 2005.

\bibitem{Ng:2b}
Lenhard Ng.
\newblock Knot and braid invariants from contact homology. {II}.
\newblock {\em Geom. Topol.}, 9:1603--1637 (electronic), 2005.
\newblock With an appendix by the author and Siddhartha Gadgil.

\bibitem{Ng:1}
Lenhard Ng.
\newblock Framed knot contact homology.
\newblock {\em Duke Math. J.}, 141(2):365--406, 2008.

\bibitem{Ngtransverse}
Lenhard Ng.
\newblock Combinatorial knot contact homology and transverse knots.
\newblock {\em Adv. Math.}, 227(6):2189--2219, 2011.

\bibitem{Ngsurvey}
Lenhard Ng.
\newblock A topological introduction to knot contact homology.
\newblock In {\em Contact and symplectic topology}, volume~26 of {\em Bolyai
  Soc. Math. Stud.}, pages 485--530. J\'anos Bolyai Math. Soc., Budapest, 2014.

\bibitem{Shende}
Vivek Shende.
\newblock The conormal torus is a complete knot invariant.
\newblock arXiv:1604.03520.

\bibitem{S:1}
Dennis Sullivan.
\newblock Open and closed string field theory interpreted in classical
  algebraic topology.
\newblock In {\em Topology, geometry and quantum field theory}, volume 308 of
  {\em London Math. Soc. Lecture Note Ser.}, pages 344--357. Cambridge Univ.
  Press, Cambridge, 2004.

\bibitem{Su}
Dennis Sullivan.
\newblock String topology background and present state.
\newblock In {\em Current developments in mathematics, 2005}, pages 41--88.
  Int. Press, Somerville, MA, 2007.

\end{thebibliography}

\end{document}